\documentclass[12pt]{amsart}

\usepackage{amsmath,amsthm,amscd,amsfonts,amssymb,epic,eepic,bbm,mathtools,mathrsfs}
\usepackage{longtable, enumitem, enumerate}
\usepackage{amstext}
\usepackage[dvipsnames]{xcolor}
\usepackage{stmaryrd}
\usepackage{enumitem}

\usepackage{tikz}
\usetikzlibrary{er,positioning}
\usepackage{tikz-cd}
\usepackage[all]{xy} 

\usepackage[margin=1in]{geometry} 
\usepackage{lineno}
\usepackage[pagebackref,colorlinks=true,linkcolor=blue,citecolor=blue]{hyperref}
\usepackage{csquotes}
\allowdisplaybreaks

\newtheorem{thm}{Theorem}[subsection]
\newtheorem{cor}[thm]{Corollary}

\newtheorem{lem}[thm]{Lemma}

\newtheorem{prop}[thm]{Proposition}

\theoremstyle{remark}
\newtheorem*{rem}{Remark}
\newtheorem*{example}{Example}

\newcounter{remarkscounter}

\numberwithin{equation}{subsection} 

\newcommand{\ddbar}{\overline\partial}

\newcommand{\abs}[1]{\left\vert#1\right\vert}
\newcommand{\set}[1]{\left\{#1\right\}}
\newcommand{\To}{\rightarrow}

\makeatletter
\newcommand*\bigcdot{\mathpalette\bigcdot@{.5}}
\newcommand*\bigcdot@[2]{\mathbin{\vcenter{\hbox{\scalebox{#2}{$\m@th#1\bullet$}}}}}
\makeatother

\theoremstyle{definition}
\newtheorem{defn}[thm]{Definition}

\newcommand\ol[1]{{\overline{#1}}}

\begin{document}
\title{Semi-classical Szeg\H{o} Kernel Asymptotics on CR Manifolds of High codimension}

\author[Chin-Yu Hsiao]{Chin-Yu Hsiao}
	\address{Department of Mathematics, National Taiwan University, Taipei, Taiwan}
	\email{chinyuhsiao@ntu.edu.tw or chinyu.hsiao@gmail.com}
	\thanks{Chin-Yu Hsiao was partially supported by National Science and Technology Council project 113-2115-M-002-011-MY3.} 

     \author[Chin Jui Yang]{Chin Jui Yang} \address{Department of Mathematics, National Taiwan University, Taipei, Taiwan} 
\email{yangchinjui@ntu.edu.tw or p9433256@gmail.com}


\begin{abstract}
In this paper, we consider a CR manifold $X$ of high codimension and a CR line bundle $L$ over $X$. For each $k \in \mathbb{N}$, let $L^k$ denote the $k$-th power of $L$. We establish that the spectral projection $\Pi_{k,\leq k\mu_k} := \mathbf{1}_{(-\infty, k\mu_k]}(\Box_{b,k})$ converges to the weighted Szeg\H{o} projection on the high codimension Heisenberg group in the sense of distributions, where $\Box_{b,k}$ denotes the Kohn Laplacian with values in $L^k$. Moreover, when $\Box_{b,k}$ possesses a local spectral gap, we show that the Szeg\H{o} projection $\Pi_k$ converges to the weighted Szeg\H{o} projection on the high codimension Heisenberg group in the sense of distributions. We also provide an explicit formula for the Szeg\H{o} projection on the high codimension Heisenberg group. Furthermore, we consider a high codimension CR manifold $X$ equipped with a transversal and CR $\mathbb{R}^m$-action. In this setting, we prove that the partial Szeg\H{o} projection $\Pi_{k,({\mathbb R}^m\curvearrowright,\leq k\delta)}$
converges to the weighted partial Szeg\H{o} projection on the high codimension Heisenberg group in the $\mathcal{C}^\infty$ topology. As an application of our results, we establish a Kodaira embedding theorem for high codimension CR manifolds with a transversal and CR $\mathbb{R}^m$-action.
	\end{abstract}

	\maketitle
    \tableofcontents

\section{introduction}

The study of CR manifolds has long been one of the central directions in several complex variables. A primary reason for this is that the boundary of a domain in $\mathbb{C}^n$ forms a CR manifold; by analyzing the geometric and analytic properties of the boundary—such as its Szeg\H{o} kernel—one can gain deeper insights into the geometry of the domain itself and its associated analytic problems. Since the 1990s, research on CR manifolds has developed profound connections with key topics in complex geometry, geometric quantization, and related fields. For instance, let $(L, h^L)$ be a holomorphic line bundle over a complex manifold $M$. Then the circle bundle $X := \{v \in L^*;\, \vert{}v\vert{}^2_{h^{L^*}} = 1\}$ is a CR manifold. Many geometric properties of $X$ and analytic problems associated with it (such as Szeg\H{o} kernels) are intimately tied to fundamental questions regarding the complex manifold $M$ and the line bundle $L$. For example, when $X$ is compact and strongly pseudoconvex, One can apply the results of Boutet de Monvel and Sj\"{o}strand
\cite{BS76} on the Szeg\H{o} kernel on $X$ to obtain the Tian-Yau-Catlin-Zelditch asymptotic expansion of the Bergman kernel associated with high tensor powers of the line bundle over $M$ (see, for instance, \cite{T90}, \cite{Z98}, and \cite{C99}). Consequently, modern research on CR manifolds is not only vital to several complex variables but also deeply intertwined with current developments in complex geometry. Conversely, insights from complex geometry frequently inspire new directions in CR geometry. Take the aforementioned circle bundle $X$, for instance: it belongs to a class of CR manifolds equipped with a group action. Many recent developments in complex geometry can be generalized to CR manifolds with group actions, yielding non-trivial generalizations. Although CR manifolds with group actions constitute a very specific class within CR geometry, their rich structure often allows for more refined results than those achievable on general CR manifolds. From this perspective, CR manifolds carrying group actions are well worth studying. 

Almost all research on CR manifolds focuses on the codimension-one case (hypersurface type), with very few results available for CR manifolds of higher codimension. Notably, in 2019, Fritsch, Herrmann, and Hsiao~\cite{FHH22} established a Boutet de Monvel-type embedding result for CR manifolds of high codimension. This implies that a certain class of CR manifolds of high codimension possesses a rich space of CR functions. Furthermore, their paper proved that any complex manifold arises as the quotient of a CR manifold of high codimension by a compact Lie group. In 2020, Fritsch and Heinzner~\cite{FH22} construct universal equivariant complex extensions of CR manifolds of high codimension admitting proper transversal CR actions of Lie groups that are closed subgroups of their universal complexifications. Using holomorphic slices and local finite-group quotients, they equip the orbit spaces with complex orbifold structures whose holomorphic functions correspond precisely to invariant CR functions. They then applied methods from complex geometry to prove a Kodaira-type $K$-equivariant CR embedding theorem for a compact CR manifold $X$ of high codimension with a transversal CR action of a compact Lie group $K$, assuming that $X/K$ carries a weakly negative holomorphic line bundle.  Therefore, the study of high-codimension CR manifolds is closely related to the investigation of complex manifolds with singularities. In addition, research on CR manifolds of high codimension is closely tied to numerous classical problems in several complex variables, such as CR local embedding problems. For instance, to construct CR functions locally, one needs to study the $\overline{\partial}_b$-Neumann problem on CR manifolds with boundary, where the boundary of a codimension-one CR manifold is a CR manifold of codimension two. Moreover, CR geometry of high codimension is intimately connected with geometric quantization; for example, the level set $\mu^{-1}(0)$ is typically a CR manifold of high codimension (see~\cite{FHH22}), where $\mu$ denotes the CR moment map. Thus, we believe that the study of CR manifolds of high codimension is a highly interesting and significant direction. 

In this paper, we consider a CR line bundle $L$ over $X$, a CR manifold of high codimension. We prove that when $L$ has positive CR curvature, the spectral kernel converges to the Szeg\H{o} kernel of the model case in the sense of distributions. Although smooth convergence is generally difficult to achieve in the high-codimension setting, we demonstrate that if the high codimension of the CR manifold originates from a CR $\mathbb{R}^m$-action, the partial Szeg\H{o} kernel converges to the Szeg\H{o} kernel of the model case  in the $\mathcal{C}^\infty$ topology. As an application of this result, we establish an equivariant Kodaira embedding theorem. 

We now formulate the main results. We refer to Section~\ref{s-gue260804yyd} for some notations and terminology used here. Let $(X,T^{1,0}X)$ be a paracompact orientable CR manifold of dimension $2n+m$ of high codimension $m$, $n, m\in\mathbb N$, $n, m\geq1$ (see Definition~\ref{d-gue260502yyd} and Definition~\ref{d-gue260502yydI}). Fix a Hermitian metric on $\mathbb{C}TX$ such that \eqref{e-gue260818yyda} hold. The Hermitian metric $\langle\,\cdot\,|\,\cdot\,\rangle$ on 
$\mathbb CTX$ induces by duality, a Hermitian metric $\langle\,\cdot\,|\,\cdot\,\rangle$ on $\Lambda^{2n+m}\mathbb CT^*X$. Let $dv_X$ be the volume form on $X$ induced by $\langle\,\cdot\,|\,\cdot\,\rangle$ and for every $q=0,1,\ldots,n$, let $(\,\cdot\,|\,\cdot\,)$ be the $L^2$ inner product on $\Omega^{0,q}_c(X)$ induced by $dv_X$ and $\langle\,\cdot\,|\,\cdot\,\rangle,$ and let $L^2_{(0,q)}(X)$ be the completion of $\Omega^{0,q}_c(X)$ with respect to $(\,\cdot\,|\,\cdot\,)$.

Let $(L,h^L)\To X$ be a CR line bundle and for $k\in\mathbb N$, let $(L^k,h^{L^k})$ be the $k$-th power of $(L,h^L)$, where $h^L$ denotes the Hermitian metric of $L$. Let $(\,\cdot\,|\,\cdot\,)_k$ be the $L^2$ inner product on $\Omega^{0,q}_c(X,L^k)$ induced by $dv_X$, $\langle\,\cdot\,|\,\cdot\,\rangle$ and $h^{L^k},$ and let $L^2_{(0,q)}(X,L^k)$ be the completion of $\Omega^{0,q}_c(X,L^k)$ with respect to $(\,\cdot\,|\,\cdot\,)_k$. Let 
\[\ddbar_b: \Omega^{0,q}(X,L^k)\To\Omega^{0,q+1}(X,L^k)\]
be the tangential Cauchy-Riemann operator with values in $L^k$. We extend $\ddbar_b$ to $L^2$ space:
\[\ddbar_b: {\rm Dom\,}\ddbar_b\subset L^2_{(0,q)}(X,L^k)\To L^2_{0,q+1}(X,L^k),\]
where ${\rm Dom\,}\ddbar_b:=\set{u\in L^2_{(0,q)}(X,L^k);\, \ddbar_bu\in L^2_{(0,q+1)}(X,L^k)}$. Let 
\[\ddbar^*_b: {\rm Dom\,}\ddbar^*_b\subset L^2_{(0,q+1)}(X,L^k)\To L^2_{0,q}(X,L^k)\]
be the $L^2$ adjoint of $\ddbar_b$ with respect to $(\,\cdot\,|\,\cdot\,)_k$. The Kohn Laplacian for $(0,q)$ forms is given by 
\[\Box^{(q)}_{b,k}:=\ddbar^*_b\,\ddbar_b+\ddbar_b\,\ddbar^*_b: {\rm Dom\,}\Box^{(q)}_{b,k}\subset L^2_{(0,q)}(X,L^k)\To L^2_{(0,q)}(X,L^k),\]
where 
\[\begin{split}
&{\rm Dom\,}\Box^{(q)}_{b,k}
=\{u\in L^2_{(0,q)}(X,L^k);\, u\in{\rm Dom\,}\ddbar_b\cap{\rm Dom\,}\ddbar^*_b,\\
&\quad\quad \ddbar_bu\in{\rm Dom\,}\ddbar^*_b,\ \ \ddbar^*_bu\in{\rm Dom\,}\ddbar_b.\}
\end{split}\]
In this work, we will focus on $q=0$ and we write $\Box_{b,k}:=\Box_{b,k}$.

Let 
\[\Pi_k: L^2(X,L^k)\To{\rm Ker\,}\Box_{b,k}\]
be the orthogonal projection (Szeg\H{o} projection). More general, 
for $\mu\in\mathbb R_+:=[0,\infty]$, let 
\begin{equation}\label{e-gue260818ycdm}
\Pi_{k,\leq\mu}:=1_{(-\infty,\mu]}(\Box_{b,k}): L^2(X,L^k)\To E_{\leq\mu}(X,L^k)
\end{equation}
be the functional calculus of $\Box_{b,k}$ with the indicate function $1_{(-\infty,\mu]}$, where $E_{\leq\mu}(X,L^k):={\rm Range\,}1_{(-\infty,\mu]}(\Box_{b,k})$. When no confusion arises, we sometimes write $\Pi_{k,\leq\mu}$ simply as $\Pi_{\leq\mu}$. When $\lambda=0$, $E_{\leq\mu}(X,L^k)={\rm Ker\,}\Box_{b,k}$ and 
$1_{(-\infty,\mu]}(\Box_{b,k})=\Pi_k$.

Our first result establishes, without any conditions, the asymptotic behavior of the spectral kernel $\Pi_{\leq k\mu_k}$ in the sense of distributions, where $\mu_k>0$ is a sequence with $\lim_{k\To+\infty}\mu_k=0$. In order to describe our result, we need to introduce special local charts. 

We have $\mathbb{C}TX = T^{1,0}X\oplus T^{0,1}X\oplus N(X),$ where 
\[N(X)=\set{\sum^m_{j=1}\lambda_jT_j;\, \lambda_j\in\mathbb C, j=1,\ldots,m},\]
where $T_1,\ldots,T_m\in\mathcal{C}^{\infty}(X,TX)$ are as in Definition~\ref{d-gue260502yydI}.  Let $\omega_j(x)\in\mathcal{C}^\infty(X,T^*X)$, $j=1,\ldots,m$, be the one forms given by \eqref{e-gue260502yyd}. Fix $x_0\in X$. Let $U_1(x),\dots,U_n(x)$ be an orthonormal frame of $T^{1,0}X$ over a neighborhood $D$ of $x_0$. 
For all $j,s = 1,\ldots,n$, $\ell = 1,\dots,m$, let
\[
\lambda^{\ell}_{js} := \frac{1}{2i}\langle\,[U_j,\overline{U}_s]\,,\,\omega_\ell\,\rangle(x_0)
\]
which is the Levi form at $x_0$ in the direction $\omega_\ell$.
Then $(\lambda^{\ell}_{js})$ are Hermitian matrices. Let $s$ be a local CR trivialization of $L$ defined on an open neighborhood of $x_0$, $\abs{s}^2_{h^L}=e^{-\phi}$. We can find a CR local trivialization $s$, $\abs{s}^2_{h^L}=e^{-\phi}$, defined on a small open set $D$ of $x_0$, a local chart $(x_1,\dots,x_{2n},x_{2n+1},\dots,x_{2n+m}) =(x',x'') = (z,x'') = (z_1,\dots,z_n,x'')$, defined on $D$, 
$z_j = x_{2j-1}+ix_{2j}, j=1,\dots,n, x' = (x_1,\dots, x_{2n}), x''=(x_{2n+1},\dots,x_{2n+m})$, such that, for $1\leq j,s\leq 2n, \ 1\le \ell, u\le m$,
\begin{equation}\label{e-gue260829ycd}
(z(x_0),x''(x_0)) = 0,\quad \Big\langle\frac{\partial}{\partial x_j}(x_0)\mid\frac{\partial}{\partial x_s}(x_0)\Big\rangle = 2\delta_{js}, \quad \Big\langle\frac{\partial}{\partial x_{2n+\ell}}(x_0)\mid\frac{\partial}{\partial x_{2n+u}}(x_0)\Big\rangle = \delta_{\ell u},
\end{equation}
\begin{equation}\label{e-gue260819yyd}
\begin{split}
&U_j = \frac{\partial}{\partial z_j} - i\sum_{\ell=1}^{m}\Big(\sum_{s=1}^{n}\lambda^{\ell}_{js}\overline{z}_s\Big)\frac{\partial}{\partial x_{2n+\ell}}  + O(|(z,x'')|^2), \quad j=1,\dots,n,\\
&T_\ell(x_0) =-\frac{\partial}{\partial x_{2n+\ell}}(x_0), \quad \ell=1,\dots,m,
\end{split}
\end{equation}
\begin{equation}\label{e-gue260829ycda}
\begin{split}
\phi(x) &= \sum_{j=1}^{n}(\alpha_jz_j+\overline{\alpha}_j\overline{z}_j) + \sum_{\ell=1}^{m}\beta_\ell x_{2n+\ell} + \sum_{j,s=1}^{n}(a_{js}z_jz_s + \overline{a}_{js}\overline{z}_j\overline{z}_s) + \sum_{j,s=1}^{n}\mu_{js}z_j\overline{z}_s \\
  &\quad + O(|z||x''|) + O(|x''|^2) + O(|(z,x'')|^3),
\end{split}
\end{equation}
where $\alpha_j\in\mathbb C$, $\beta_j\in\mathbb R$, $a_{js}\in\mathbb C$, $\mu_{js}\in\mathbb C$, $\mu_{js}=\ol{\mu_{sj}}$, $j, s=1,\ldots,n$. Let
\begin{equation}\label{e-gue260829ycdb}
R(x):=\sum_{j=1}^{n}\alpha_jz_j+\sum_{j,s=1}^{n}a_{js}z_jz_s
\end{equation}
and put
\begin{equation}\label{e-gue260830yyd}
\psi_0(x):=\sum_{\ell=1}^{m}\beta_\ell x_{2n+\ell}+\sum_{j,s=1}^{n}\mu_{js}z_j\overline{z}_s.
\end{equation}

Fix $\mu\geq0$. Let $\Pi_{k,\leq\mu}(x,y)$ be the distribution kernel of $\Pi_{k,\leq\mu}$. On $D$, write 
\begin{equation}\label{e-gue260825yyd}
\Pi_{k,\leq\mu}(x,y)=s^k(x)\otimes\Pi_{k,\leq\mu,s}(x,y)\otimes (s^k(y))^*,\end{equation}
where $\Pi_{k,\leq\mu,s}(x,y)\in\mathcal{D}'(D\times D)$, that is,
\begin{equation}\label{e-gue260829yyd}(\,\Pi_{k,\leq\mu}(s^k\otimes u)\,|\,s^k\otimes v\,)_k=\langle\,\Pi_{k,\leq\mu,s}(x,y)\,,\,e^{-k\phi(x)}u(y)\overline{v(x)}\,\rangle,\end{equation}
for all $u, v\in\mathcal{C}^\infty_c(D)$, where the right side of \eqref{e-gue260829yyd} means that there are $f_j\in\mathcal{C}^\infty(D\times D)$, $j=1,2,\ldots$, such that, 
\[\begin{split}
\langle\,\Pi_{k,\leq\mu,s}(x,y)\,,\,e^{-k\phi(x)}u(y)\ol{v(x)}\,\rangle
=\lim_{j\To+\infty}\int_{D\times D}f_j(x,y)u(y)\ol{v(x)}e^{-k\phi(x)}dv_X(x)dv_X(y),
\end{split}\]
for all $u, v\in\mathcal{C}^\infty_c(D)$. If $\mu=0,$ we write $\Pi_{k,s}=\Pi_{k,\leq 0,s}$. When no confusion arises, we sometimes write $\Pi_{k,\leq\mu,s}$ simply as $\Pi_{\leq\mu,s}$.

The first main result of this work is the following

\begin{thm}\label{t-gue260829yyd}
With the notations used above, fix $0<\epsilon<\frac{1}{8}$. Let $x_0\in X$ and let $x=(x_1,\ldots,x_{2n+m})=(x',x'')$ be the local coordinates of $X$ centered at $x_0$ defined on an open neighborhood $D$ of $x_0$ such that \eqref{e-gue260829ycd}, \eqref{e-gue260819yyd} and \eqref{e-gue260829ycda} hold. Let $\mu_k>0$ be a sequence with $\lim_{k\To+\infty}\mu_k=0$ and $\lim_{k\To+\infty}\frac{1}{\mu_k k^{2\epsilon}}=0$. Let 
\[\widetilde\Pi_{\leq k\mu_k,s}(x,y):=e^{-kR(x)}\Pi_{\leq k\mu_k,s}(x,y)e^{kR(y)}\in\mathcal{D}'(D\times D),\]
where $R(x)$ is as in \eqref{e-gue260829ycdb}. Then
\begin{equation}\label{e-gue260829yydm}
\lim_{k\To+\infty}k^{-(n+m)}\widetilde\Pi_{\leq k\mu_k,s}\Big(\Big(\frac{x'}{\sqrt{k}},\frac{x''}{k}\Big),
\Big(\frac{y'}{\sqrt{k}},\frac{y''}{k}\Big)\Big)=\Pi^{(0),\psi_0}_b(x,y)\ \ \mbox{in $\mathcal{D}'(\mathbb H^{n,m}\times\mathbb H^{n,m})$},
\end{equation}
where $\mathbb H^{n,m}:=\mathbb C^n\times\mathbb R^m$ is the CR manifold with CR structure given by \eqref{e-gue260930yydb} and $\Pi^{(0),\psi_0}_b(x,y)$ is the Szeg\H{o} kernel on $\mathbb H^{n,m}$ with respect to the volume form $e^{-\psi_0(x)}d\mu_{\mathbb H^{n,m}}$ (see \eqref{modwgtszeker} and \eqref{e-gue261010yyd}), $d\mu_{\mathbb H^{n,m}}$ is given by \eqref{e-gue260830yydm}. 
\end{thm}
\begin{rem}
The limit on the left hand side of \eqref{e-gue260829yydm} means that  for any $u\in \mathcal{C}^\infty_c(\mathbb{H}^{m,n}),$ if we scale $u$ to become a smooth function $u_{[k]}$ on $D$ (see \eqref{scaledu}):
\begin{equation*}
    u_{[k]}(x) := k^{\frac{n+m}{2}}\chi\Big(\frac{\sqrt{k}z}{k^\epsilon},\frac{kx''}{k^{2\epsilon}}\Big)u(\sqrt{k}z,kx'')e^{kR}s^k,
\end{equation*}
and then rescale $\Pi_{\leq k\mu_k}(u_{[k]})$(see \eqref{scaledmap} and \eqref{scaledPimuu}):
\begin{equation}\label{scaledmuuin}
k^{-\frac{n+m}{2}}F_k^*\Big(\widehat{\Pi_{\leq k\mu_k} u_{[k]}}\cdot e^{-kR}\Big),
\end{equation}
then the sequence in \eqref{scaledmuuin} converges to $\Pi^{(0),\psi_0}_bu$ in the sense of distributions with respect to the volume form $e^{-\psi_0(x)}d\mu_{\mathbb H^{n,m}};$ more precisely, for all $v\in \mathcal{C}^\infty_c(\mathbb{H}^{m,n}),$
\begin{align}
&\lim_{k\to\infty} \int_{\mathbb{H}^{m,n}} k^{-\frac{n+m}{2}}F_k^*\Big(\widehat{\Pi_{\leq k\mu_k} u_{[k]}}\cdot e^{-kR}\Big)(x)\overline{v(x)}e^{-kF^*_k\phi_0(x)}(F^*_kV)(x)d\mu_{\mathbb{H}^{m,n}}(x)\\
&= \int_{\mathbb{H}^{m,n}\times\mathbb{H}^{m,n}}\Pi^{(0),\psi_0}_b(x,y)u(y)\overline{v(x)}e^{-\psi_0(x)}d\mu_{\mathbb{H}^{m,n}\times\mathbb{H}^{m,n}}(x,y).\nonumber
\end{align}
Formally we can write, for all $u,v\in\mathcal{C}^\infty_c(\mathbb{H}^{m,n}),$ we have
\begin{align*}
&(\frac{1}{k})^{n+m}\int_{\mathbb{H}^{m,n}\times\mathbb{H}^{m,n}}\widetilde\Pi_{\leq k\mu_k,s}((\frac{x'}{\sqrt{k}},\frac{x''}{k}),
(\frac{y'}{\sqrt{k}},\frac{y''}{k}))u(y)\overline{v(x)}e^{-kF^*_k\phi_0(x)}(F^*_kV)(x)d\mu_{\mathbb{H}^{m,n}\times\mathbb{H}^{m,n}}(x,y)\\
&\xrightarrow[k\to\infty]{} \int_{\mathbb{H}^{m,n}\times\mathbb{H}^{m,n}}\Pi^{(0),\psi_0}_b(x,y)u(y)\overline{v(x)}e^{-\psi_0(x)}d\mu_{\mathbb{H}^{m,n}\times\mathbb{H}^{m,n}}(x,y).
\end{align*}
For more details, see Section 5.
\end{rem}

Our second main theorem is the asymptotic of Szeg\H{o} projection under the local specrtal gap condition.
\begin{thm}\label{mainthmsp}
With the notations used above, fix $0<\epsilon<\frac{1}{8}$. Let $x_0\in X$ and let $x=(x_1,\ldots,x_{2n+m})=(x',x'')$ be the local coordinates of $X$ centered at $x_0$ defined on an open neighborhood $D$ of $x_0$ such that \eqref{e-gue260829ycd}, \eqref{e-gue260819yyd} and \eqref{e-gue260829ycda} hold. Let 
\[\widetilde\Pi_{k,s}(x,y):=e^{-kR(x)}\Pi_{k,s}(x,y)e^{kR(y)}\in\mathcal{D}'(D\times D),\]
where $R(x)$ is as in \eqref{e-gue260829ycdb}. Suppose that for all large $k,$ we have 
\[
\big\|\big(I - \Pi_k\big)u\big\|_k^2 \leq (k^{1-2\epsilon})^{-1} \big(\square_{b,k}u|u\big)_k
\]
for all $u\in \mathcal{C}^\infty_c(D, L^k).$
 Then
\begin{equation}\label{locspgap}
\lim_{k\To+\infty}k^{-(n+m)}\widetilde\Pi_{k,s}\Big(\Big(\frac{x'}{\sqrt{k}},\frac{x''}{k}\Big),
\Big(\frac{y'}{\sqrt{k}},\frac{y''}{k}\Big)\Big)=\Pi^{(0),\psi_0}_b(x,y)\ \ \mbox{in $\mathcal{D}'(\mathbb H^{n,m}\times\mathbb H^{n,m})$},
\end{equation}
where $\mathbb H^{n,m}:=\mathbb C^n\times\mathbb R^m$ is the CR manifold with CR structure given by \eqref{e-gue260930yydb} and $\Pi^{(0),\psi_0}_b(x,y)$ is the Szeg\H{o} kernel on $\mathbb H^{n,m}$ with respect to the volume form $e^{-\psi_0(x)}d\mu_{\mathbb H^{n,m}}$ (see \eqref{modwgtszeker}), $d\mu_{\mathbb H^{n,m}}$ is given by \eqref{e-gue260830yydm}. 
\end{thm}
We give an example that satisfies the local spectral gap assumption in Theorem \ref{mainthmsp}.
\begin{example}
Let $X$ be a $2n+m$ dimensional compact orientable Levi-flat CR manifold of CR codimension $m$ and let $(L,h^L)$ be a positive CR complex line bundle over $X.$ (See Definition \ref{pos} for positivity of CR line bundles.) Then for large $k,$ the Kohn-Laplacian $\square_{b,k}$ associated to the $k$-th tensor power of $(L,h^L)$ has spectral gap $\operatorname{Spec}(\square_{b,k})\subset \{0\}\cup[Ck,\infty)$ for some constant $C>0.$ In particular, $\square_{b,k}$ satisfies \eqref{locspgap}. The proof is given in  Section 5 in \cite{HM17} with mirror change.
\end{example}

In 2022, Herrmann, Hsiao and Li \cite{HHL22} considered a compact connected CR manifold of CR codimension 1  with a transversal locally free CR $\mathbb{R}$-action endowed with a rigid positive CR line bundle. They showed that the weighted Fourier-Szeg\H{o} kernel of the CR sections in the high tensor powers admits a full asymptotic expansion. Consequently, they established an $\mathbb{R}$-equivariant Kodaira embedding theorem. Motivated by their results, we  consider a CR manifold $X$ of CR codimension $m$ equipped with a locally free transversal CR $\mathbb{R}^m$-action together with a rigid positive $\mathbb{R}^m$-equivariant CR line bundle $(L,h^L)$. We refer to Section \ref{gpact} for the notion of CR $\mathbb{R}^m$-action.
Although we can just obtain semi-classical Szeg\H{o} kernel asymptotic in the sense of distributions under the local spectral gap assumption, with the help of the CR $\mathbb{R}^m$-action and positivity of the CR line bundle, we can derive the convergence of partial Szeg\H{o} kernels to the partial Szeg\H{o} kernel of the model case in the $\mathcal{C}^\infty$ topology. As an application, we prove an $\mathbb{R}^m$-equivariant Kodaira embedding theorem.

Let $Y_1,\dots,Y_m$ be the infinitesimal generators of the $\mathbb{R}^m$-action induced from the standard basis for $\mathbb{R}^m.$ In this paper, we will make systematic use of appropriate coordinates introduced by Baouendi-Rothschild-Treves \cite{BRT85}. For each point $p \in X,$ there exists a coordinate neighborhood $U$ with coordinates $(x_1, \dots, x_{2n+m})$, centered at $p=0$, and $\eta > 0$, $\eta_0 > 0$ such that, by setting $z_j = x_{2j-1} + ix_{2j}$, $j=1, \dots, n,$ $x'' = (x_{2n+1},\dots, x_{2n+m})$ and $D = \{(z,x'') \in U;\, |z| < \eta,\, |x_{2n+\ell}| < \eta_0,\, \ell =1,\dots, m\} \subset U$, we have 
\begin{equation}
Y_\ell = \frac{\partial}{\partial x_{2n+\ell}} \text{ on } D,
\end{equation}
$\ell=1,\dots, m,$ and the vector fields 
\begin{equation}\label{inbrtframe}
Z_j = \frac{\partial}{\partial z_j} - i\sum_{\ell=1}^m \frac{\partial \varphi_\ell}{\partial z_j}(z) \frac{\partial}{\partial x_{2n+\ell}},
\end{equation}
$j=1, \dots, n$, form a basis of $T_x^{1,0}X$ for each $x \in D$, where $\varphi_\ell \in \mathcal{C}^\infty(D, \mathbb{R})$ is independent of $x''$. By \cite{BRT85}, we can assume
\begin{equation}
    \varphi_\ell(z) = \sum_{j,s=1}^n \frac{\partial^2\varphi_\ell}{\partial z_j\partial \overline{z}_s}(0)z_j \overline{z}_s + O(|z|^3).
\end{equation}
We call $(x_1, \dots, x_{2n+m})$ canonical coordinates, $D$ a canonical coordinate patch and\\ $(D, (z, x''), \varphi)$ a BRT trivialization where $\varphi = (\varphi_1,\dots,\varphi_m)$. The frames in (\ref{inbrtframe}) are called BRT frames.

Let $L$ be a rigid CR line bundle over $X.$ Suppose that $L$ is endowed with a rigid Hermitian metric $h^L$ such that the curvature $R^L$ is positive definite. We take an $\mathbb{R}^m$-equivariant rigid Hermitian metric $\langle\,\cdot\,|\,\cdot\,\rangle$ on $\mathbb{C}TX$ such that
\begin{align*}
    T^{1,0}X \perp T^{0,1}X, \ Y_\ell\perp (T^{1,0}X \oplus T^{0,1}X), \ \langle Y_\ell \mid Y_q \rangle = \delta_{\ell q}, \quad \ell,q=1,\dots,m.
\end{align*}
For every $\alpha\in \text{Spec}(-iY_\ell),\ \ell = 1,\dots, m,$ put
\[
\Omega^{0,q}_{Y_\ell,\alpha}(X,L^k) := \{ u\in \Omega^{0,q}(X, L^k);\, -iY_\ell u = \alpha u \}.
\]
Let $L^{2,(0,q)}_{Y_\ell,\alpha}(X,L^k)\subset L^2_{(0,q)}(X,L^k)$ be the eigenspace of $-iY_\ell$ with eigenvalue $\alpha.$ We can check that $L^{2,(0,q)}_{Y_\ell,\alpha}(X,L^k)$ is the completion of $\Omega^{0,q}_{Y_\ell,\alpha}(X,L^k)$ with respect to $(\,\cdot\,|\,\cdot\,)_k.$

Put
\[
\mathcal{H}^{0,q}_{b,(Y_\ell,\alpha)}(X,L^k) := \{u\in L^{2,(0,q)}_{Y_\ell,\alpha}(X,L^k);\,\overline{\partial}_b u = 0 \},
\]
and, for $\boldsymbol{\lambda} = (\lambda_1,\dots,\lambda_m)\in (\mathbb{R}_{>0})^m,$ put
\begin{align*}
\Omega^{0,q}_{Y,\leq\boldsymbol{\lambda}}(X,L^k) &:= \bigoplus_{\alpha_\ell\in \text{Spec}(-iY_\ell), |\alpha_\ell|\leq \lambda_\ell, 1\leq\ell\leq m} \left(\bigcap_{1\leq\ell\leq m}\Omega^{0,q}_{Y_\ell,\alpha_\ell}(X,L^k)\right),\\
L^{2,(0,q)}_{Y,\leq\boldsymbol{\lambda}}(X,L^k) &:= \bigoplus_{\alpha_\ell\in \text{Spec}(-iY_\ell), |\alpha_\ell|\leq \lambda_\ell, 1\leq\ell\leq m} \left(\bigcap_{1\leq\ell\leq m}L^{2,(0,q)}_{Y_\ell,\alpha_\ell}(X,L^k)\right),\\
\mathcal{H}^{0,q}_{b,\leq\boldsymbol{\lambda}}(X,L^k) &:= \bigoplus_{\alpha_\ell\in \text{Spec}(-iY_\ell), |\alpha_\ell|\leq \lambda_\ell, 1\leq\ell\leq m} \left(\bigcap_{1\leq\ell\leq m}\mathcal{H}^{0,q}_{b,(Y_\ell,\alpha_\ell)}(X,L^k)\right).
\end{align*}
If each $\lambda_\ell=\lambda$ for all $\ell$, we simply write
\begin{align*}
&\Omega^{0,q}_{Y,\leq \lambda}(X,L^k) := \Omega^{0,q}_{Y,\leq(\lambda,\dots,\lambda)}(X,L^k),\quad
L^{2,(0,q)}_{Y,\leq \lambda}(X,L^k) := L^{2,(0,q)}_{Y,\leq(\lambda,\dots,\lambda)}(X,L^k),\\
&\mathcal{H}^{0,q}_{b,\leq \lambda}(X,L^k) := \mathcal{H}^{0,q}_{b,\leq (\lambda,\dots,\lambda)}(X,L^k).
\end{align*}
For every $\boldsymbol{\lambda} \in (\mathbb{R}_{>0})^m,$ we have
$\mathcal{H}^{0,q}_{b,\leq\boldsymbol{\lambda}}(X,L^k)\subset\Omega^{0,q}(X, L^k)\quad \text{and}\quad  \dim\mathcal{H}^{0,q}_{b,\leq\boldsymbol{\lambda}}(X,L^k) < \infty$ (See Section \ref{findim}.)
For every $\lambda > 0.$ Let 
\begin{equation*}
Q^{(0)}_{X,\leq \lambda}(u) : L^2(X,L^k) \to L^{2}_{Y,\leq \lambda}(X,L^k)
\end{equation*}
be the natural orthogonal projection.
For every $\lambda>0,$ we consider the partial Szeg\H{o} projector
\begin{equation*}
\Pi_{k,({\mathbb R}^m\curvearrowright,\leq \lambda)}=\Pi^{(0)}_{k,({\mathbb R}^m\curvearrowright,\leq \lambda)}:L^{2}(X,L^{k})\rightarrow\mathcal{H}_{b,\le \lambda}^{0}(X,L^{k})
\end{equation*}
which is the orthogonal projection onto the space of $\mathbb{R}^m$-equivariant CR functions of degree less than $\lambda.$ Here we denote $\mathcal{H}_{b,\le \lambda}^{0,0}(X,L^{k})$ by $\mathcal{H}_{b,\le \lambda}^{0}(X,L^{k})$. Note that we have 
\begin{equation*}
\Pi^{(0)}_{k,({\mathbb R}^m\curvearrowright,\leq \lambda)} = Q^{(0)}_{X,\leq \lambda} \circ \Pi^{(0)}_{b,k} = \Pi^{(0)}_{b,k} \circ Q^{(0)}_{X,\leq \lambda}.
\end{equation*}
Fix $\delta>0.$  Since $\mathcal{H}_{b,\le k\delta}^{0}(X,L^{k})\subset \mathcal{C}^\infty(X,L^k)$ and $\dim \left(\mathcal{H}_{b,\le k\delta}^{0}(X,L^{k})\right)<\infty,$ we have $\Pi^{(0)}_{k,({\mathbb R}^m\curvearrowright,\leq k\delta)}$ as a smoothing operator. Denote the smooth kernel of $\Pi^{(0)}_{k,({\mathbb R}^m\curvearrowright,\leq k\delta)}$ by $P_{k,\delta}(x,y)\in \mathcal{C}^\infty(X\times X, L^k_x\boxtimes L^{{*}k}_y).$ Let $\{f_1,\dots,f_{d_k}\}$ be an orthonormal basis of $\mathcal{H}_{b,\le k\delta}^{0}(X,L^{k}).$ Then
\begin{equation}
    P_{k,\delta}(x,y) = \sum_{j=1}^{d_k} f_j(x)\otimes(f_j(y))^*.
\end{equation}
We next localize $P_{k,\delta}(x,y).$ Fix a point $p\in X$ and take a BRT frame, $(D, (z, x''), \varphi),$ around $p.$ 
 Let $s$ be a local rigid CR section of $L$ over $D.$ Write $|s|^2_{h^L} = e^{-\phi},\ \phi \in \mathcal{C}^\infty(D;\mathbb R).$ After a holomorphic change of variables if necessary, we may assume the follwing.
\begin{equation*}
    \phi(z) = \sum_{j,r=1}^{n}\mu_{jr}z_j\overline{z}_r  + O(|z|^3),
\end{equation*}
which is independent of $x''.$ 
 Consider the isometry
\begin{equation*}
    R_{k,s} : L^2_{(0,q)}(D, e^{-k\phi}) \rightarrow L^2_{(0,q)}(D, L^k),\quad u \mapsto u\otimes s^k,
\end{equation*}
with inverse $R_{k,s}^{-1} : L^2_{(0,q)}(D, L^k)\rightarrow L^2_{(0,q)}(D, e^{-k\phi}),\ \alpha \mapsto s^{-k}\alpha.$
The localization of $P_{k,\delta}$ with respect to the trivializing rigid CR section $s$ is given by
\[
P^s_{k,\delta} : L^2_{comp}(D, e^{-k\phi}) \to L^2(D, e^{-k\phi}), \quad P^s_{k,\delta} = R^{-1}_{k,s}P_{k,\delta}R_{k,s},
\]
where $L^2_{comp}(D, e^{-k\phi})$ is the subspace of elements of $L^2(D, e^{-k\phi})$ with compact support in $D.$
In $D,$ write $f_j(x) = \hat{f}_j\otimes s^k,\ \hat{f}_j\in \mathcal{C}^\infty(D),\ j=1,\dots,d_k.$ We can express $P_{k,\delta}(x,y)$ in $D$ as
\begin{equation}
    P^s_{k, \delta}(x,y) = \sum_{j=1}^{d_k} \hat{f}_j(x)\overline{\hat{f}_j(y)}e^{-k\phi(y)} \in \mathcal{C}^\infty(D\times D).
\end{equation}
Hence we have, for all $\alpha\in L^2_{comp}(D, L^k),$ write $\alpha=\hat{\alpha}\otimes s^k,$
\begin{equation}
R^{-1}_{k,s}(P_{k,\delta}\alpha)(x) = \int_D P^s_{k,s}(x,y)\hat{\alpha}(y)dv_X(y).
\end{equation}
Let 
\begin{equation*}
\begin{cases} 
\phi(x) = \psi_0(x) + \psi_1(x), \\ 
\psi_0(x) = \sum_{j,r=1}^{n}\mu_{jr}z_j\overline{z}_r, \\ 
\psi_1(x) = O(|z|^3). 
\end{cases} 
\end{equation*}
Let $\lambda^\ell_{j,r} = \frac{\partial^2\varphi_\ell}{\partial z_j\partial \overline{z}_r}(0)$ and let $M_\ell = (\lambda^{\ell}_{jr}),$ which is a  Hermitian matrix, $\ell=1,...,m.$ Fix $\delta>0$ small enough such that $(\mu_{jr}) - 2\sum_{\ell=1}^m\eta_\ell(\lambda^{\ell}_{jr})$ is positive definite for all $|\eta_\ell|\leq\delta,\ \ell = 1,\dots, m.$ For $\eta = (\eta_1,\dots,\eta_m)\in {\mathbb R}^m,$ let $M_\eta = (\mu_{jr}) - 2\sum_{\ell=1}^m\eta_\ell M_\ell.$
Put
\begin{equation*}
P^s_{(k),\delta}(x,y) := k^{-n-m} P^s_{k,\delta} \left( \Big(\frac{z}{\sqrt{k}},\frac{x''}{k}\Big), \Big(\frac{w}{\sqrt{k}},\frac{y''}{k}\Big) \right) \in \mathcal{C}^\infty(D(k^\epsilon,k^{2\epsilon})\times D(k^\epsilon,k^{2\epsilon})). 
\end{equation*}
We can now state our third main theorem.
\begin{thm}\label{3rdmain1}
For $\delta$ small enough, the sequence $P^s_{(k),\delta}(x,y)$ converges to $P_\delta(x,y)$ locally uniformly in the $\mathcal{C}^\infty$-topology, where
\begin{align*}
P_{\delta}(x,y) &:= \Big(\frac{1}{2\pi}\Big)^{n+m} \int_{[-\delta,\delta]^m} \exp\bigg[i\langle\eta,x''\rangle-i\langle\eta,y''\rangle +\sum_{j,r=1}^{n}(M_\eta)_{jr}z_j\overline{w}_r-\sum_{j,r=1}^{n}(M_\eta)_{jr}w_j\overline{w}_r\\ 
&\quad -\sum_{\ell=1}^m\eta_\ell\sum_{j,r=1}^{n}\lambda^{\ell}_{jr}w_j\overline{w}_r+\sum_{\ell=1}^m\eta_\ell\sum_{j,r=1}^{n}\lambda^{\ell}_{jr}z_j\overline{z}_r\bigg]\det(M_\eta) d\mu(\eta).\nonumber
\end{align*}
\end{thm}

For convenience of later use, we will prove the following equivalent form of Theorem \ref{3rdmain1}.
Consider the isometry
\begin{equation*}
    U_{k,s} : L^2_{(0,q)}(D) \rightarrow L^2_{(0,q)}(D, L^k),\quad u \mapsto ue^{\frac{k}{2}\phi}s^k,
\end{equation*}
with inverse $U_{k,s}^{-1} : L^2_{(0,q)}(D, L^k)\rightarrow L^2_{(0,q)}(D),\ \alpha \mapsto e^{-\frac{k}{2}\phi}s^{-k}\alpha.$
The localization of $P_{k,\delta}$ with respect to the trivializing rigid CR section $s$ is given by
\[
P_{k,\delta,s} : L^2_{comp}(D) \to L^2(D), \quad P_{k,\delta,s} = U^{-1}_{k,s}P_{k,\delta}U_{k,s},
\]
where $L^2_{comp}(D)$ is the subspace of elements of $L^2(D)$ with compact support in $D.$
In $D,$ write $f_j(x) = \hat{f}_j\otimes s^k,\ \hat{f}_j\in \mathcal{C}^\infty(D),\ j=1,\dots,d_k.$ We can express $P_{k,\delta}(x,y)$ in $D$ as
\begin{equation*}
    P_{k, \delta,s}(x,y) = \sum_{j=1}^{d_k} e^{-\frac{k}{2}\phi(x)} \hat{f}_j(x)\overline{\hat{f}_j(y)}e^{-\frac{k}{2}\phi(y)} \in \mathcal{C}^\infty(D\times D).
\end{equation*}
Note that we have
\[
P_{k,\delta,s}(x,y) = e^{-\frac{k}{2}\phi(x)} P^s_{k,\delta}(x,y) e^{\frac{k}{2}\phi(y)}.
\]
Put
\begin{equation*}
P_{(k),\delta,s}(x,y) := k^{-n-m} P_{k,\delta,s} \left( \Big(\frac{z}{\sqrt{k}},\frac{x''}{k}\Big), \Big(\frac{w}{\sqrt{k}},\frac{y''}{k}\Big) \right) \in \mathcal{C}^\infty(D(k^\epsilon,k^{2\epsilon})\times D(k^\epsilon,k^{2\epsilon})). 
\end{equation*}

\begin{thm}\label{3rdmain}
For $\delta$ small enough, the sequence $P_{(k),\delta,s}(x,y)$ converges to\\ $e^{-\frac{1}{2}\psi_0(x)}P_\delta(x,y)e^{\frac{1}{2}\psi_0(y)}$ locally uniformly in the $\mathcal{C}^\infty$-topology.
\end{thm}

As an application of Theorem \ref{3rdmain}, we prove an $\mathbb{R}^m$-equivariant CR embedding theorem of $X$ into a complex projective space. We recall the definition of CR embeddings.
\begin{defn}
Let $X$ and $Y$ be two CR manifolds. A map $\Psi: X \rightarrow Y$ is called a CR embedding if $\Psi$ is a smooth embedding and $\Psi$ satisfies
\[
d\Psi(T^{1,0}X) = d\Psi(\mathbb{C}TX)\cap T^{1,0}Y.
\]
\end{defn}
On an open subset $D\subset X$ satisfying Corollary \ref{invnbd}, we write
\begin{align*}
    f_j = \hat{f}_j\otimes s^k, \quad \hat{f}_j \in \mathcal{C}^\infty(D), \quad j=1,\dots,d_k.
\end{align*}
By Lemma \ref{wedef}, for $\delta$ small enough and $k$ large enough, we have a well-defined global map
\begin{equation}\label{introembedmap}
    \Phi_{k,\delta} : X \to \mathbb{CP}^{d_k-1}, \quad x \mapsto \big[\hat{f}_1(x) : \cdots : \hat{f}_{d_k}(x)\big].
\end{equation}
We can choose $f_j,\ j=1,\dots, d_k,$ such that $f_j\in \bigcap_{1\leq\ell\leq m}\mathcal{H}^0_{b,(Y_\ell,\alpha^j_\ell)}(X,L^k)$  where each $\alpha^j_\ell \in \text{Spec}(-iY_\ell)$ with $\big|\alpha^j_\ell\big| \leq k\delta.$

We now define an $\mathbb{R}^m$-action on $\mathbb{CP}^{d_k-1}.$ For each $j = 1,\dots, d_k,$ we assign a multi-weight $\alpha^j=(\alpha^j_1,\dots,\alpha^j_{m})\in \mathbb{R}^m.$ For all $\eta = (\eta_1,\dots,\eta_m)\in \mathbb{R}^m,$ we define, for $[z_1 : \cdots : z_{d_k}]\in \mathbb{CP}^{d_k-1},$
\begin{align}\label{intrordact}
    \eta \circ [z_1 :\cdots :z_{d_k}] 
    &:= \big[e^{i\langle\alpha^1,\eta\rangle} z_1 : \cdots :e^{i\langle\alpha^{d_k},\eta\rangle} z_{d_k}\big],\\
    & = \big[e^{i(\alpha^1_1\eta_1+\cdots+\alpha^1_m\eta_m)}z_1 : \cdots : e^{i(\alpha^{d_k}_1\eta_1+\cdots+\alpha^{d_k}_m\eta_m)}z_{d_k}\big].\nonumber
\end{align}

\begin{thm}\label{embedthmintro}
 Let $(X,T^{1,0})$ be a compact CR manifold of dimension $2n+m$ and of CR-codimension $m,$ with a transversal locally free $\mathbb{R}^m$-action. Assume there is a rigid positive CR line bundle $L$ over $X.$ Then there exists $\delta_0 > 0$ such that for all $\delta\in (0,\delta_0)$ there exists a positive integer $k(\delta)$ so that for $k>k(\delta)$ and any orthonormal basis $\{f_j\}_{j=1}^{d_k}$ of $\mathcal{H}^0_{b,\leq k\delta}(X, L^k)$ with respect to $(\,\cdot\,|\,\cdot\,)_k$ such that $f_j\in \bigcap_{1\leq\ell\leq m}\mathcal{H}^0_{b,(Y_\ell,\alpha^j_\ell)}(X,L^k)$ where each $\alpha^j_\ell \in \text{Spec}(-iY_\ell)$ with $\big|\alpha^j_\ell\big| \leq k\delta$, the map $\Phi_{k,\delta}$ introduced by (\ref{introembedmap}) is a smooth CR embedding which is $\mathbb{R}^m$-equivariant with respect to the multi-weighted $\mathbb{R}^m$-action on $\mathbb{CP}^{d_k-1}$ defined by $\alpha^1=(\alpha^1_1,\dots,\alpha^1_m),\dots,\alpha^{d_k}=(\alpha^{d_k}_1,\dots,\alpha^{d_k}_m)$ as in (\ref{intrordact}).
\end{thm}

This paper is organized as follows. In Section 2, we introduce some notation and recall the necessary notions concerning CR manifolds of high codimension, tangential Cauchy--Riemann operators, and CR vector bundles. In Section 3, we introduce the Heisenberg group $\mathbb{H}^{n,m}:=\mathbb{C}^n\times\mathbb{R}^m$, equipped with the weighted volume form $e^{-\psi_0(x)}d\mu_{\mathbb{H}^{n,m}}(x)$, as our model space, and recall the theory of weighted Bergman kernels. We then define the Kohn Laplacian $\square^{(q),\psi_0}_b$ on $\mathbb{H}^{n,m}$ and, following the ideas of \cite{HM12} and \cite{HHL18}, derive an explicit representation of the Szeg\H{o} kernel associated with $\square^{(0),\psi_0}_b$. In Section 4, we introduce the scaled Kohn Laplacian $\square^{(q)}_{b,(k),s}$ and the scaled Szeg\H{o} projection $\Pi_{(k),s}$, and prove Theorem \ref{mainthmsp} in Section 4.3. In Section 5, we prove Theorem \ref{t-gue260829yyd} by suitably modifying the proof of Theorem \ref{mainthmsp}. In Section 6, we discuss one special case in which a substantially weaker spectral-gap assumption suffices to establish the asymptotics of the Szeg\H{o} kernels. In Section 7, we define a locally free transversal CR $\mathbb{R}^m$-action on a compact oriented CR manifold of high codimension and introduce the partial Szeg\H{o} kernel $P_{k,\delta}$, defined using a cutoff associated with the $\mathbb{R}^m$-action. We show that $P_{k,\delta}$ is a smoothing operator. We then prove in Section 7.4 that the scaled partial Szeg\H{o} kernel $P_{(k),\delta,s}$ converges in the $\mathcal{C}^{\infty}$ topology to the partial Szeg\H{o} kernel of the model case; this is Theorem \ref{3rdmain}. Finally, in Section 8, we prove Theorem \ref{embedthmintro} and give some examples to illustrate Theorem \ref{embedthmintro}.
\section{Preliminaries}\label{s-gue260804yyd} 

\subsection{Standard notations} \label{s-ssna}

We use the following notations: $\mathbb N=\{1,2,\ldots\}$ is the set of natural numbers excluding $0$ and $\mathbb N_0=\mathbb N\cup\{0\}$, $\mathbb R$ is the set of real numbers, 
$\dot{\mathbb R}=\{x\in\mathbb R;\, x\neq0\}$,
${\mathbb R}_+=\{x\in\mathbb R;\, x>0\}$, $\overline{\mathbb R}_+=\{x\in\mathbb R;\, x\geq0\}$. Furthermore we adopt the standard multi-index notation: we write $\alpha=(\alpha_1,\ldots,\alpha_\ell)\in\mathbb N^\ell_0$ 
if $\alpha_j\in\mathbb N_0$, $j=1,\ldots,\ell$. For $x = (x_1, \dots, x_\ell) \in \mathbb{R}^\ell$, we write
\begin{align*}
x^\alpha &= x_1^{\alpha_1} \dots x_\ell^{\alpha_\ell}, \quad \partial_{x_j} = \frac{\partial}{\partial x_j}, \quad \partial_x^\alpha = \partial_{x_1}^{\alpha_1} \dots \partial_{x_\ell}^{\alpha_\ell} = \frac{\partial^{|\alpha|}}{\partial x^\alpha}, \\
D_{x_j} &= \frac{1}{i}\partial_{x_j}, \quad D_x^\alpha = D_{x_1}^{\alpha_1} \dots D_{x_\ell}^{\alpha_\ell}, \quad D_x = \frac{1}{i}\partial_x
\end{align*}
Let $z = (z_1, \dots, z_\ell)$, $z_j = x_{2j-1} + i x_{2j}$, $j = 1, \ldots, \ell$, be coordinates of $\mathbb{C}^\ell$, where $x = (x_1, \dots, x_{2\ell}) \in \mathbb{R}^{2\ell}$ are coordinates in $\mathbb{R}^{2\ell}$.  We write
\begin{align*}
z^\alpha &= z_1^{\alpha_1} \dots z_\ell^{\alpha_\ell}, \quad \bar{z}^\alpha = \bar{z}_1^{\alpha_1} \dots \bar{z}_\ell^{\alpha_\ell}, \\
\partial_{z_j} &= \frac{\partial}{\partial z_j} = \frac{1}{2} \Big( \frac{\partial} {\partial x_{2j-1}} - i \frac{\partial} {\partial x_{2j}} \Big), \quad \partial_{\bar{z}_j} = \frac{\partial}{\partial \bar{z}_j} = \frac{1}{2} \Big( \frac{\partial}{\partial x_{2j-1}} + i \frac{\partial}{\partial x_{2j}} \Big), \\
\partial_z^\alpha &= \partial_{z_1}^{\alpha_1} \dots \partial_{z_\ell}^{\alpha_\ell} = \frac{\partial^{|\alpha|}}{\partial z^\alpha}, \quad \partial_{\bar{z}}^\alpha = \partial_{\bar{z}_1}^{\alpha_1} \dots \partial_{\bar{z}_\ell}^{\alpha_\ell} = \frac{\partial^{|\alpha|}}{\partial \bar{z}^\alpha}.
\end{align*}

Let $M$ be a smooth paracompact manifold. We let $TM$ and $T^*M$ denote respectively the tangent bundle of $M$ and the cotangent bundle of $M$. The complexified tangent bundle $TM \otimes \mathbb{C}$ of $M$ will be denoted by $\mathbb CTM$, similarly we write $\mathbb CT^*M$ for the complexified cotangent bundle of $M$. Consider $\langle\,\cdot\,,\cdot\,\rangle$ to denote the pointwise
duality between $TM$ and $T^*M$; we extend $\langle\,\cdot\,,\cdot\,\rangle$ bi-linearly to $\mathbb C TM\times\mathbb CT^*M$. Let $B$ be a smooth vector bundle over $M$. The fiber of $B$ at $x\in M$ will be denoted by $B_x$. Let $E$ be a vector bundle over a smooth paracompact manifold $N$. We write
$B\boxtimes E^*$ to denote the vector bundle over $M\times N$ with fiber over $(x, y)\in M\times N$ consisting of the linear maps from $E_y$ to $B_x$.  

Let $Y\subset M$ be an open set. From now on, the spaces of distribution sections of $B$ over $Y$ and smooth sections of $B$ over $Y$ will be denoted by $\mathcal D'(Y, B)$ and $\mathcal{C}^\infty(Y, B)$, respectively.
Let $\mathcal E'(Y, B)$ be the subspace of $\mathcal D'(Y, B)$ whose elements have compact support in $Y$. Let $\mathcal{C}^\infty_c(Y,B):=\mathcal{C}^\infty(Y,B)\cap\mathcal{E}'(Y,B)$. 
For $m\in\mathbb R$, let $H^m(Y, B)$ denote the Sobolev space
of order $m$ of sections of $B$ over $Y$. Let us denote
\begin{align*}
H^m_{\rm loc\,}(Y, B)=\big\{u\in\mathcal{D}'(Y, B);\, \varphi u\in H^m(Y, B),
    \, \forall\varphi\in \mathcal{C}^\infty_c(Y)\big\}\,,
\end{align*}
and
\begin{align*}
       H^m_{\rm comp\,}(Y, B)=H^m_{\rm loc}(Y, B)\cap\mathcal{E}'(Y, B)\,.
\end{align*}

Let $B$ and $E$ be smooth vector
bundles over paracompact orientable manifolds $M$ and $N$, respectively, equipped with smooth densities of integration. If
$A: \mathcal{C}^\infty_c(N,E)\To\mathcal D'(M,B)$
is continuous, we write $A(x, y)$ to denote the distribution kernel of $A$.
The following two statements are equivalent
\begin{itemize}
	\item[(i)] $A$ is continuous: $\mathcal E'(N,E)\To\mathcal{C}^\infty(M,B)$,
	\item[(ii)] $A(x,y)\in\mathcal{C}^\infty(M\times N,B\boxtimes E^*)$.
\end{itemize}
If $A$ satisfies (1) or (2), we say that $A$ is smoothing on $M \times N$. 
We say that $A$ is properly supported if the restrictions of the two projections 
$(x,y)\mapsto x$, $(x,y)\mapsto y$ to ${\rm supp\,}(A(x,y))$
are proper.

Let $H(x,y)\in\mathcal D'(M\times N,B\boxtimes E^*)$. We write $H$ to denote the unique continuous operator $\mathcal \mathcal{C}^\infty_c(N,E)\To\mathcal D'(M,B)$ with distribution kernel $H(x,y)$. In this work, we identify $H$ with $H(x,y)$. 

Let
$B: \mathcal{C}^\infty_c(N,E)\To\mathcal D'(M,B)$
be a continuous operator. We write $A(x, y)\equiv B(x,y)$ or $A\equiv B$ if $A-B$ is a smoothing operator.  

\subsection{CR manifolds of high codimension} \label{s-gue250331}
	
	Let $X$ be a smooth manifold of dimension $2n+m$, $n, m\geq 1$, where $T^{1,0}X$ is a CR structure of $X$, that is. 

    \begin{defn}\label{d-gue260502yyd}
Let $T^{1,0}X$ be a subbundle of $\mathbb CTX$. We say that $T^{1,0}X$ is a CR structure if 
\begin{itemize}
\item[(i)] ${\rm dim\,}T^{1,0}_xX=n$, for all $x\in X$,
\item[(ii)] $T^{1,0}X\cap T^{0,1}X=\set{0}$, where $T^{0,1}X:=\ol{T^{1,0}X}$, 
\item[(iii)] $[\mathcal V,\mathcal V]\subset\mathcal V$, where $\mathcal V=\mathcal{C}^\infty(X, T^{1,0}X)$.
\end{itemize}
If $T^{1,0}X$ is a subbundle, we call the pair $(X,T^{1,0}X)$ a CR manifold of codimension $m$. 
    \end{defn}

\begin{defn}\label{d-gue260502yydI}
Let $(X,T^{1,0}X)$ be a CR manifold of codimension $m$. 
We say that $(X,T^{1,0}X)$ is orientable if we can find global vector fields $T_1,\ldots,T_m\in\mathcal{C}^\infty(X,TX)$ such that 
\[\mathbb CT_xX=T^{1,0}_xX\oplus T^{0,1}_xX\oplus\set{\lambda_1T_1(x)+\cdots+\lambda_mT_m(x);\, \lambda_j\in\mathbb C, j=1,\ldots,m}.\]

We call $T_1,\ldots,T_m$ transversal vector fields. 
    \end{defn}

From now on, we assume that $(X,T^{1,0}X)$ is an orientable paracompact 
CR manifold of dimension $2n+m$, codimension $m$, and we fix transversal vector fields $T_1,\ldots,T_m\in\mathcal{C}^\infty(X,TX)$. 
For every $j\in\set{1,\ldots,m}$, let $\omega_j(x)\in\mathcal{C}^\infty(X,T^*X)$ be the global one form given by 
\begin{equation}\label{e-gue260502yyd}
\langle\,T_\ell\,,\,\omega_j\,\rangle=-\delta_{j\ell}\ \ \mbox{on $X$},\ \ \ell=1,\ldots,m.   
\end{equation} 

Fix $(\lambda_1,\ldots,\lambda_m)\in\mathbb R^m\setminus\set{0}$. Let $\omega:=\sum^m_{j=1}\lambda_j\omega_j$. 
	For each $x \in X$, we define the Hermitian quadratic form on $T^{1,0}_xX$ with respect to $\omega(x)$ by
	\begin{equation}\label{e-gue260502ycd}
		\mathcal{L}_{\omega(x)}(U,\overline{V}) := -\frac{1}{2i}d\omega(x)(U,\overline{V}),\ \ U, V\in T^{1,0}_xX.
	\end{equation}
	The Hermitian quadratic form $\mathcal{L}_{\omega(x)}$ on $T^{1,0}_xX$ is called Levi form at $x$ with respect to $\omega(x)$. 

Fix a smooth Hermitian metric $\langle\,\cdot\,|\, \cdot\,\rangle$ on $\mathbb{C}TX$ so that  
    \begin{equation}\label{e-gue260818yyda}\begin{split}
    &T^{1,0}X\perp T^{0,1}X,\\
    &\langle\,u\,|\,v\,\rangle\in\mathbb R,\ \ u, v\in TX,\\
&\langle\, T_j\, | \, T_\ell\, \rangle=\delta_{j\ell},\ \ j, \ell=1,\ldots,m,\\
&\langle\,T_j\,|\,u\,\rangle=0,\ \ \mbox{for all $j=1,\ldots,m$, $u\in T^{1,0}X$}.
    \end{split}\end{equation}
For $u \in \mathbb{C}TX$, we write $|u|^2 := \langle u | u \rangle$. Denote by $T^{*1,0}X$ and $T^{*0,1}X$ the dual bundles of $T^{1,0}X$ and $T^{0,1}X$, respectively. They can be identified with subbundles of the complexified cotangent bundle $\mathbb{C}T^*X$. Define the vector bundle of $(0,q)$-forms by $T^{*0,q}X := \wedge^qT^{*0,1}X$. Let
    $\Lambda^\bullet(\mathbb CT^*X):=\oplus^{2n+m}_{j=0}\Lambda^j(\mathbb CT^*X)$.
    The Hermitian metric $\langle\,\cdot\,|\, \cdot\,\rangle$ on $\mathbb{C}TX$ induces, by duality, a Hermitian metric on $\mathbb{C}T^*X$ and also on $\Lambda^\bullet(\mathbb CT^*X)$. We shall also denote all these induced metrics by $\langle\,\cdot\,|\,\cdot\,\rangle$. For $u\in\Lambda^\bullet(\mathbb CT^*X)$, let $\abs{u}^2:=\langle\,u\,|\,u\,\rangle$. Let $T^{*0,\bullet}X:=\oplus^n_{q=0}T^{*0,q}X$. 

    Let $D\subset X$ be an open subset. Let $\Omega^{0,q}(D)$ denote the space of smooth sections 
	of $T^{*0,q}X$ over $D$ and let $\Omega^{0,q}_c(D)$ be the subspace of
	$\Omega^{0,q}(D)$ whose elements have compact support in $D$. For each $q=0,1,\ldots,n$, let 
    \[\tau^q: \Lambda^\bullet(\mathbb CT^*X)\To T^{*0,q}X\]
    be the orthogonal projection with respect to $\langle\,\cdot\,|\,\cdot\,\rangle$. As codimension one case, the tangential Cauchy-Riemann operator is given by
    
	\begin{equation} \label{e-suIV}
		\ddbar_b:=\tau^q\circ d: \Omega^{0,q}(X)\To\Omega^{0,q+1}(X),
	\end{equation}
    where $d$ denotes the standard derivative. For a function $f\in\mathcal{D}'(D)$, we say that $f$ is a CR function (on $D$) if $\ddbar_bf=0$ in the sense of distribution. 

    Let $dv_X$ be the volume form on $X$ induced by the given Hermitian metric $\langle\,\cdot\,|\,\cdot\,\rangle$ and let $(\,\cdot\,|\,\cdot\,)$ be the $L^2$ inner product on $\Omega^{0,q}_c(X)$ induced by $\langle\,\cdot\,|\,\cdot\,\rangle$ and $dv_X$ and let $L^2_{(0,q)}(X)$ be the completion of $\Omega^{0,q}_c(X)$ with respect to $(\,\cdot\,|\,\cdot\,)$. 

\begin{defn}\label{d-gue260804yyd}
Let $E$ be a complex vector bundle over $X$ of rank $r$. We say that $E$ is a CR vector bundle over $X$ if 
there exists an open covering $\{ U_j \}^{+\infty}_{j=1}$, $g_{U_jU_\ell} \in \mathcal{C}^\infty(U_j \cap U_\ell, \mathbb{C}^r)$, for all $j, \ell \in\mathbb N$, such that 
$g_{U_jU_\ell}$ is an invertiable matrix of CR functions, and 
\begin{equation}\label{e-gue260818yyd}
\begin{split}
&\mbox{$g_{U_jU_\ell} = g_{U_\ell U_j}^{-1}$ on $U_j\cap U_\ell$, if $U_j\cap U_\ell\neq\emptyset$},\\
&\mbox{$g_{U_jU_\ell}g_{U_\ell U_s}g_{U_sU_j}=I$ 
on $U_j\cap U_\ell\cap U_s$, if $U_j\cap U_\ell\cap U_s\neq\emptyset$},
\end{split}
\end{equation}
and
$E = \amalg U_j \times \mathbb{C}^r/ \equiv$, where $\amalg$ denotes the disjoint union, $(z, \lambda) \equiv (w, \mu)$ if $z=w \in U_j \cap U_\ell$, for some $j, \ell \in\mathbb N$, then $\mu = g_{U_\ell U_j} \lambda$. We write $[(z, \lambda)]$ to denote the equivalence class of $(z, \lambda) \in X \times \mathbb{C}^r$. Then $E= \cup_{(z, \lambda) \in X \times \mathbb{C}^r}[(z, \lambda)]$ and $E$ has natural smooth and CR structures induced by the equivalence relation $\equiv$ and $E$ is a CR manifold of dimesion $2(n+r)+m$ of codimension $m$. 

When $r=1$, we call $E$ a CR line bundle over $X$.
\end{defn}

Let $E$ be a CR vector bundle over $X$ and let $D\subset X$ be an open set. Let $\Omega^{0,q}(D,E)$ denote the space of smooth sections 
	of $T^{*0,q}X\otimes E$ over $D$ and let $\Omega^{0,q}_c(D,E)$ be the subspace of
	$\Omega^{0,q}(D,E)$ whose elements have compact support in $D$. We can define 
\[\ddbar_b: \Omega^{0,q}(X,E)\To\Omega^{0,q+1}(X,E)\]
in the standard way. 

Let $(L,h^L)$ be a CR line bundle over $X$ and for $k\in\mathbb N$, let $(L^k,h^{L^k})$ be the $k$-th power of $(L,h^L)$, where $h^L$ denotes the Hermitian metric of $L$. Let $(\,\cdot\,|\,\cdot\,)_k$ be the $L^2$ inner product on $\Omega^{0,q}_c(X,L^k)$ induced by $\langle\,\cdot\,|\,\cdot\,\rangle$, $h^{L^k}$ and $dv_X$ and let $L^2_{(0,q)}(X,L^k)$ be the completion of $\Omega^{0,q}_c(X,L^k)$ with respect to $(\,\cdot\,|\,\cdot\,)_k$.

\section{The Szeg\H{o} kernel on the Heisenberg groups}

In this section, we will give the explicit representation of the Szeg\H{o} kernel for the weighted Heisenberg group $\mathbb{H}^{n,m}.$ The idea comes from \cite{HM12} and \cite{HHL18}, in particular Section 7 in \cite{HHL18}.
\subsection{Notations}\label{heisenberggpcrstr}

We identify $\mathbb{R}^{2n+m}$ with the Heisenberg group $\mathbb H^{n,m} := \mathbb{C}^n\times\mathbb{R}^m$. We write $x = (x',x'') = (z,x'') = (z_1,\dots,z_n,x_{2n+1},\dots,x_{2n+m})$ to denote the coordinates of $\mathbb H^{n,m}$, $x'= (x_1,\dots,x_{2n}),$ $z=(z_1,\dots,z_n),$ $z_j = x_{2j-1} + ix_{2j}, j=1,\dots,n,$ and $x''=(x_{2n+1},\dots,x_{2n+m})\in \mathbb{R}^m.$   Let $(\lambda^1_{\ell s})^n_{\ell ,s=1},\cdots, (\lambda^m_{\ell s})^n_{\ell,s=1}$ be $m$ Hermitian $n\times n$ matrices. We define a CR structure  $T^{1,0}\mathbb H^{n,m}$ on $\mathbb H^{n,m}$ by
\begin{equation}\label{e-gue260930yydb}
T^{1,0}\mathbb H^{n,m} := \text{Span}_{\mathbb{C}}\bigg\{U_j=U_{j,\mathbb H^{n,m}} = \frac{\partial}{\partial z_j} - i\sum_{\ell=1}^{m}\Big(\sum_{s=1}^{n}\lambda^{\ell}_{js}\overline{z}_s\Big)\frac{\partial}{\partial x_{2n+\ell}};\, j =1,\dots,n\bigg\}.
\end{equation}
We denote by $T^{0,1}\mathbb H^{n,m}$ the complex conjugate of $T^{1,0}\mathbb H^{n,m}.$ Set $T_{\ell}=-\frac{\partial}{\partial x_{2n+\ell}}, \ell =1,\dots,m.$ Fix an Hermitian metric on $\mathbb{C}T\mathbb H^{n,m}$ denoted by $\langle\,\cdot\,|\,\cdot\,\rangle$ such that
\begin{align*}
&T_\ell\perp T^{1,0}\mathbb H^{n,m}\perp T^{0,1}\mathbb H^{n,m}, \ \ell=1,\dots,m,\\
&\langle T_\ell \mid T_j\rangle = \delta_{\ell j},\ \ell,j = 1,\dots,m,\\
&\langle U_j\mid U_s\rangle = \langle \overline{U}_j\mid \overline{U}_s\rangle = \delta_{js},\ j,s = 1,\dots,n.
\end{align*}
Take 
\begin{equation}\label{e-gue260830yydm}
d\mu_{\mathbb H^{n,m}}:=2^ndx_1\wedge\cdots\wedge dx_{2n+m}
\end{equation}
to be the volume form on $\mathbb H^{n,m},$ $d\mu(z) = 2^n dx_1\wedge\cdots\wedge dx_{2n}$ to be the volume form on $\mathbb{C}^n$ and $d\mu(x'') = dx_{2n+1}\wedge\cdots\wedge dx_{2n+m}$ to be the volume form on $\mathbb{R}^m.$

Denote by $T^{*1,0}\mathbb H^{n,m}$ and $T^{*0,1}\mathbb H^{n,m}$ the dual bundles of $T^{1,0}\mathbb H^{n,m}$ and $T^{0,1}\mathbb H^{n,m},$ respectively. Let $T^{*0,q}\mathbb H^{n,m}:=\Lambda^q (T^{*0,1}\mathbb H^{n,m})$ be the bundle of $(0,q)$-forms on $\mathbb H^{n,m}$. Let
\begin{equation*}
    T^{*0,\bullet}\mathbb H^{n,m}:= \bigoplus_{q=0}^nT^{*0,q}\mathbb H^{n,m}.
\end{equation*}
The Hermitian metric $\langle\cdot| \cdot\rangle$ on $\mathbb{C}T\mathbb H^{n,m}$ induces an Hermitian metric on $\mathbb{C}T^*\mathbb H^{n,m}$ and
$T^{*0,\bullet}\mathbb H^{n,m}$ by duality. We will also denote all these induced metrics by $\langle\,\cdot\,|\,\cdot\,\rangle.$
We can see that
\[
\{dz_1,\dots,dz_n, d\overline{z}_1,\dots,d\overline{z}_n,\omega_1,\dots,\omega_m\},
\]
where
\begin{equation*}
    \omega_q=d x_{2n+q}+i\sum_{j=1}^{n}\Big(\sum_{s=1}^{n}\lambda^q_{js}\overline{z}_s\Big)dz_j-i\sum_{j=1}^n\Big(\sum_{s=1}^n\lambda^q_{sj}z_s\Big)d\overline{z}_j,
\end{equation*}
$q=1,\dots,m,$
is the orthonormal basis of $\mathbb{C}T^*\mathbb H^{n,m}$ which is dual to \[
\{U_{j,\mathbb H^{n,m}}, \overline{U}_{j,\mathbb H^{n,m}}, -T_\ell\,;\, j=1,\dots,n,\ \ell=1,\dots,m\}.
\]

\subsection{Kohn Laplacian}
Let $D\subset\mathbb H^{n,m}$ be an open subset. Denote by $\Omega^{0,q}(D)$ the space of all smooth sections of $ T^{*0,q}\mathbb H^{n,m}$ over $D.$ Let $\Omega^{0,q}_c(D)$ be the space of all smooth sections of $T^{*0,q}\mathbb H^{n,m}$ over $D$ with compact support in $D$.
Let
\[
\psi_0 (z,x'') = \beta_1x_{2n+1} + \cdots + \beta_m x_{2n+m} + \sum_{j,s=1}^{n}\mu_{js}z_j\overline{z}_s,
\]
where $\left(\mu_{js}\right)^n_{j,s=1}$ is an Hermitian matrix.
We define the weighted inner product $(\,\cdot\,|\,\cdot\,)_{\psi_0}$ on $\Omega^{0,q}_c(\mathbb H^{n,m})$ induced from $\langle\,\cdot\,|\,\cdot\,\rangle,$ the volume form $d\mu_{\mathbb H^{n,m}}$ and $\psi_0$ by
\begin{equation}\label{e-gue260830yydn}
(f\mid g)_{\psi_0} := \int_{\mathbb H^{n,m}}\langle\,f(x)\,|\,g(x)\,\rangle e^{-\psi_0(x)}d\mu_{\mathbb H^{n,m}(x)}
\end{equation}
for all $f,g\in\Omega^{0,q}_c(\mathbb H^{n,m}).$ Let $\|\cdot\|_{\psi_0}$ denote the corresponding norm. Let $L^2_{(0,q)}(D,\psi_0)$ be the completion of $\Omega^{0,q}_c(D)$ with respect to $(\,\cdot\,|\,\cdot\,)_{\psi_0}.$ We write $L^2(\mathbb H^{n,m},\psi_0) := L^2_{(0,0)}(\mathbb H^{n,m},\psi_0).$ For the non-weighted case, i.e, $\psi_0 = 0,$ we just write $(\,\cdot\,|\,\cdot\,),$ $\|\cdot\|,$ $L^2_{(0,q)}(\mathbb H^{n,m})$ and $L^2(\mathbb H^{n,m})$ for the corresponding $L^2$ inner product, norm, the space of $L^2$ $(0,q)$-forms and the space of $L^2$ functions, respectively.
Let
\[
\overline{\partial}_b : \Omega^{0,q}(\mathbb H^{n,m}) \rightarrow \Omega^{0,q+1}(\mathbb H^{n,m})
\]
be the tangential Cauchy-Riemann operator. We extend $\overline{\partial}_b$ to $L^2_{(0,q)}(\mathbb H^{n,m},\psi_0),\ q=0,\dots,n,$ by
\begin{equation*}
    \overline{\partial}_b : \text{Dom}\, \overline{\partial}_b \subset L^2_{(0,q)}(\mathbb H^{n,m},\psi_0) \rightarrow L^2_{(0,q+1)}(\mathbb H^{n,m},\psi_0),
\end{equation*}
where $\text{Dom}\, \overline{\partial}_b := \{u\in L^2_{(0,q)}(\mathbb H^{n,m},\psi_0)\,;\,\overline{\partial}_bu\in L^2_{(0,q+1)}(\mathbb H^{n,m},\psi_0)\}$ and, for any \\$u\in  L^2_{(0,q)}(\mathbb H^{n,m},\psi_0),$ $ \overline{\partial}_b u$ is defined in the sense of distributions.
Let
\[
\overline{\partial}^{*,\psi_0}_b : \text{Dom}\,\overline{\partial}^{*,\psi_0}_b\subset L^2_{(0,q+1)}(\mathbb H^{n,m},\psi_0) \rightarrow L^2_{(0,q)}(\mathbb H^{n,m},\psi_0)
\]
be the Hilbert space adjoint of $\overline{\partial}_b$ with respect to $(\,\cdot\,|\,\cdot\,)_{\psi_0}.$
Define the Kohn-Laplacian
\[
\square^{(q),\psi_0}_b = \overline{\partial}_b\overline{\partial}^{*,\psi_0}_b + \overline{\partial}^{*,\psi_0}_b\overline{\partial}_b,
\]
where
\begin{align*}
&\text{Dom}\,\square^{(q),\psi_0}_b = \{s\in L^2_{(0,q)}(\mathbb{H}^{n,m},\psi_0) \,;\, s\in \text{Dom}\,\overline{\partial}_b\cap\text{Dom}\,\overline{\partial}^{*,\psi_0}_b, \overline{\partial}_b s\in \text{Dom}\,\overline{\partial}^{*,\psi_0}_b, \\ &\quad\quad\quad\quad\quad\quad\quad\overline{\partial}^{*,\psi_0}_b s\in \text{Dom}\,\overline{\partial}_b\}.
\end{align*}
By a result of Gaffney, we have $\square^{(q),\psi_0}_b$ is a positive self-adjoint operator, $q=0,\dots,n.$ Let
\begin{equation*}
    \Pi^{(q),\psi_0}_{b} : L^2_{(0,q)}(\mathbb H^{n,m}, \psi_0) \rightarrow \operatorname{Ker}\square^{(q),\psi_0}_b
\end{equation*}
be the orthogonal projection. Sometimes we write $\mathcal{H}^0_b(\mathbb{H}^{n,m}) := \operatorname{Ker}\square^{(0),\psi_0}_b.$
We have
\[
\overline{\partial}_b = \sum_{j=1}^{n}d\overline{z}_j\wedge\overline{U}_{j,\mathbb H^{n,m}},
\]
\[
\overline{\partial}^{*,\psi_0}_b = \sum_{j=1}^{n}(d\overline{z}_j\wedge)^*\overline{U}_{j,\mathbb H^{n,m}}^{*,\psi_0},
\]
where $(d\overline{z}_j\wedge)^*$ is the adjoint of $d\overline{z}_j\wedge$ with respect to $\langle\,\cdot\,|\,\cdot\,\rangle$ and $\overline{U}_{j,\mathbb H^{n,m}}^{*,\psi_0}$ is the adjoint of $\overline{U}_{j,\mathbb H^{n,m}}$ with respect to $(\,\cdot\,|\,\cdot\,)_{\psi_0}.$ We can check that
\[
\overline{U}_{t,\mathbb H^{n,m}}^{*,\psi_0} = -U_{t,\mathbb H^{n,m}} + \sum_{s=1}^{n}\mu_{ts}\overline{z}_s - i\sum_{\ell=1}^{m}\Big(\sum_{s=1}^{n}\lambda^{\ell}_{ts}\overline{z}_s\Big)\beta_{\ell}.
\]
By calculation, we have
\begin{prop}\label{lapfma}
\begin{align*}
\square^{(q),\psi_0}_b 
&=\sum_{j=1}^{n} \overline{U}_{j,\mathbb{H}^{n,m}}^{*,\psi_0} \overline{U}_{j,\mathbb{H}^{n,m}} + \sum_{j,t=1}^{n}d\overline{z}_j\wedge(d\overline{z}_t\wedge)^*\left[\overline{U}_{j,\mathbb{H}^{n,m}}, \overline{U}_{t,\mathbb{H}^{n,m}}^{*,\psi_0}\right]\\
&= \sum_{j=1}^{n}\bigg[\bigg(-\frac{\partial}{\partial z_j} + i\sum_{q=1}^{m}\Big(\sum_{s=1}^{n}\lambda^q_{js}\overline{z}_s\Big)\frac{\partial}{\partial x_{2n+q}} + \sum_{s=1}^{n}\mu_{js}\overline{z}_s - i\sum_{q=1}^{m}\Big(\sum_{s=1}^{n}\lambda^q_{js}\overline{z}_s\Big)\beta_q\bigg)\\
&\quad\quad\circ\bigg(\frac{\partial}{\partial \overline{z}_j} + i\sum_{q=1}^{m}\Big(\sum_{s=1}^{n}\lambda^q_{sj}z_s\Big)\frac{\partial}{\partial x_{2n+q}} \bigg)\bigg]\\
&\quad + \sum_{j,t=1}^{n}d\overline{z}_j\wedge(d\overline{z}_t\wedge)^*\bigg(\mu_{tj} - i\sum_{\ell=1}^{m}\lambda^{\ell}_{tj}\beta_{\ell} + 2i\sum_{\ell=1}^{m}\lambda^{\ell}_{tj}\frac{\partial}{\partial x_{2n+\ell}}\bigg).
\end{align*}
\end{prop}

\subsection{The partial Fourier transform}
Let $u(z,x_{2n+1},\dots,x_{2n+m})\in \Omega^{0,q}(\mathbb{H}^{n,m})$ with $\|u\|_{\psi_0}=1$ and $\square^{(q),\psi_0}_b u = 0.$ Put $v(z,x'') = u(z,x'')e^{-\frac{1}{2}\sum_{t=1}^m\beta_t x_{2n+t}}$ and set $\Phi_0=\sum_{j,s=1}^{n}\mu_{js}z_j\overline{z}_s.$ We have
\[
\int_{\mathbb{H}^{n,m}}|v(z,x'')|^2e^{-\Phi_0}d\mu(z)d\mu(x'') = 1.
\]
Denote by $L^2(\mathbb{H}^{n,m}, \Phi_0)$ the completion of $\Omega^{0,q}_c(\mathbb{H}^{n,m})$ with respect to $\|\cdot\|_{\Phi_0}.$ 
Choose $\chi(x'') = \chi(x_{2n+1},\dots,x_{2n+m}) \in C^{\infty}_c(\mathbb{R}^m)$ with $\chi(x'')=1$ when $|x''|<1$ and $\chi(x'')=0$ when $|x''|>2$ and set $\chi_j(x'') = \chi(\frac{x''}{j}),\ j\in \mathbb{N}.$ Set $v_j(x',x'') = \chi_j(x'')v(x',x'').$
Let
\begin{align*}
\mathcal{F}_{x''\to\eta}[v_j](z,\eta)=\mathcal{F}_{x''\to\eta}[v_j](z,\eta_1,\dots,\eta_m)
&= \int_{\mathbb{R}^m} v(z,x'')\chi_j(x'')e^{-i\langle x'',\eta\rangle}d\mu(x'')\\
&= \int_{\mathbb{R}^m} v(z,x'')\chi_j(x'')e^{-i\sum\limits_{t=1}^m x_{2n+t}\eta_t}d\mu(x'').
\end{align*} 
Note that $\mathcal{F}_{x''\to\eta}[{v}_j](z,\eta) \in \Omega^{0,q}(\mathbb{H}^{n,m}).$ By the Parseval's formula, we have 
\begin{align*}
&\int_{\mathbb{H}^{n,m}}|\mathcal{F}_{x''\to\eta}[v_j](z,\eta) - \mathcal{F}_{x''\to\eta}[v_k](z,\eta)|^2e^{-\Phi_0}d\mu(z)d\mu(\eta) \\
&= (2\pi)^m\int_{\mathbb{H}^{n,m}}|v(z,x'')|^2|\chi_j(x'')-\chi_k(x'')|^2e^{-\Phi_0}d\mu(z)d\mu(x'') \rightarrow 0\quad \text{as} \ j,k \rightarrow\infty.
\end{align*}
Hence there exists $\mathcal{F}_{x''\to\eta}[v](z,\eta) \in L^2(\mathbb{H}^{n,m}, \Phi_0)$ such that $\mathcal{F}_{x''\to\eta}[v_j](z,\eta)\rightarrow \mathcal{F}_{x''\to\eta}[v](z,\eta)\in L^2(\mathbb{H}^{n,m}, \Phi_0).$
We call $\mathcal{F}_{x''\to\eta}[v](z,\eta)$ the partial Fourier transform of $v(z,x'')$ with respect to $x''.$ Formally,
\begin{equation*}
\mathcal{F}_{x''\to\eta}[v](z,\eta) = \int_{\mathbb{R}^m} v(z,x'')e^{-i\langle x'',\eta\rangle}d\mu(x'').
\end{equation*}

Moreover, we have
\begin{align*}
\| \mathcal{F}_{x''\to\eta}[v](z, \eta) \|_{\Phi_0}^2 &= \int_{\mathbb{H}^{n,m}} |\mathcal{F}_{x''\to\eta}[v](z, \eta)|^2 e^{-\Phi_0(z)} d\mu(z) d\mu(\eta)\\
&= \lim_{j \to \infty} \int_{\mathbb{H}^{n,m}} |\mathcal{F}_{x''\to\eta}[v_j](z, \eta)|^2 e^{-\Phi_0(z)} d\mu(z) d\mu(\eta) \\
&= (2\pi)^m \lim_{j \to \infty} \int_{\mathbb{H}^{n,m}} \bigg|u(z, x'') e^{-\frac{1}{2} \sum\limits_t \beta_t x_{2n+t}} \chi_j(x'')\bigg|^2 e^{-\Phi_0(z)} d\mu(z) d\mu(x'') \\
&= (2\pi)^m \int_{\mathbb{H}^{n,m}} |u(z, x'')|^2 e^{-\Phi_0(z)} d\mu(z) d\mu(x'') \\
&= (2\pi)^m < \infty.
\end{align*}
By Fubini's theorem, we have
\begin{equation}\label{ae}
\int_{\mathbb{C}^n} |\mathcal{F}_{x''\to\eta}[v](z, \eta)|^2 e^{-\Phi_0(z)} d\mu(z) < \infty
\end{equation}
for a.e. $\eta = (\eta_1, \dots, \eta_m) \in \mathbb{R}^m.$

Let $s \in L^2_{(0,q)}(\mathbb{H}^{n,m}, \Phi_0)$. Assume that $\int |s(z, \eta)|^2 d\mu(\eta) < \infty$ and $\int |s(z, \eta)| d\mu(\eta) < \infty$ for all $z \in \mathbb{C}^n$.
By Parseval's formula, we can check that
\begin{align*}
&\int_{\mathbb{C}^n} \int_{\mathbb{R}^m} \big(\mathcal{F}_{x''\to\eta}[v](z, \eta)\mid s(z, \eta)\big) e^{-\Phi_0(z)} d\mu(\eta) d\mu(z)\\
&= \int_{\mathbb{C}^n} \int_{\mathbb{R}^m} \Big(u(z, x'') e^{-\frac{1}{2} \sum\limits_t \beta_t x_{2n+t}}\, \Big|\, \int_{\mathbb{R}^m} e^{i \langle x'', \eta\rangle} s(z, \eta) d\mu(\eta) \Big) e^{-\Phi_0(z)} d\mu(x'') d\mu(z) \\
&= (2\pi)^m \int_{\mathbb{C}^n} \int_{\mathbb{R}^m} (v(z, x'')\mid \mathcal{F}^{-1}_{x''\to\eta}[s](z, x'')) e^{-\Phi_0(z)} d\mu(x'') d\mu(z)
\end{align*}

Fix $\eta = (\eta_1,\dots,\eta_m) \in \mathbb{R}^m.$ Set 
\[
\Phi_\eta = \sum_{j,s=1}^{n}\mu_{js}z_j\overline{z}_s -2\eta_1\sum_{j,s=1}^{n}\lambda^{1}_{js}z_j\overline{z}_s - \cdots -2\eta_m\sum_{j,s=1}^{n}\lambda^m_{js}z_j\overline{z}_s.
\]
Let $(\cdot|\cdot )_{\Phi_{\eta}}$ be the inner product on $\Omega^{0,q}_c(\mathbb{C}^n)$ defined by
\[
(f\mid g)_{\Phi_{\eta}} := \int_{\mathbb{C}^n} (f\mid g) e^{-\Phi_{\eta}} d\mu(z),
\]
for $f,g \in \Omega^{0,q}_c(\mathbb{C}^n).$ Let
\[
\square^{(q),\Phi_{\eta}} = \overline{\partial}\overline{\partial}^{*,\Phi_{\eta}} + \overline{\partial}^{*,\Phi_{\eta}}\overline{\partial}
\]
be the complex Laplacian with respect to $(\cdot|\cdot )_{\Phi_{\eta}},$ where $\overline{\partial}^{*,\Phi_{\eta}}$ is the formal adjoint of $ \overline{\partial}$ with respect to $(\cdot|\cdot )_{\Phi_{\eta}}.$
We have
\begin{align*}
  &\overline{\partial} = \sum_{j=1}^{n}d\overline{z}_j\wedge \frac{\partial}{\partial \overline{z}_j},\\
  &\overline{\partial}^{*,\Phi_{\eta}} = \sum_{j=1}^{n}(d\overline{z}_j\wedge)^* \Big(\frac{\partial}{\partial \overline{z}_j}\Big)^{*,\Phi_{\eta}},
\end{align*}
where $(d\overline{z}_j\wedge)^*$ is the adjoint of $d\overline{z}_j\wedge$ with respect to $\langle\,\cdot\,|\,\cdot\,\rangle$ and $\big(\frac{\partial}{\partial \overline{z}_j}\big)^{*,\Phi_{\eta}}$ is the formal adjoint of $\frac{\partial}{\partial \overline{z}_j}$ with respect to $(\cdot|\cdot )_{\Phi_{\eta}}.$
We can check that 
\[
\Big(\frac{\partial}{\partial \overline{z}_j}\Big)^{*,\Phi_{\eta}} = -\frac{\partial}{\partial z_j} + \sum_{s=1}^{n}\Big(\mu_{js}-2\sum_{\ell=1}^m\eta_{\ell}\lambda^{\ell}_{js}\Big)\overline{z}_s.
\]
\begin{prop}\label{holapfma}
\begin{align*}
\square^{(q),\Phi_{\eta}} &= \sum_{j=1}^{n}\Big(\frac{\partial}{\partial \overline{z}_j}\Big)^{*,\Phi_{\eta}}\Big(\frac{\partial}{\partial \overline{z}_j}\Big) + \sum_{j,t=1}^{n}d\overline{z}_j\wedge(d\overline{z}_t\wedge)^*\Big(\mu_{tj}-2\sum_{\ell=1}^m\eta_{\ell}\lambda^{\ell}_{tj}\Big)\\
&= \sum_{j=1}^{n} \bigg[-\frac{\partial}{\partial z_j} + \sum_{k=1}^{n}\Big(\mu_{jk}-2\sum_{\ell=1}^m\eta_{\ell}\lambda^{\ell}_{jk}\Big)\overline{z}_k\bigg]\Big(\frac{\partial}{\partial \overline{z}_j}\Big)\\
&\quad+ \sum_{j,t=1}^{n}d\overline{z}_j\wedge(d\overline{z}_t\wedge)^*\Big(\mu_{tj}-2\sum_{\ell=1}^m\eta_{\ell}\lambda^{\ell}_{tj}\Big).
\end{align*}
\end{prop}

Let $u \in L^2_{(0,q)}(\mathbb{H}^{n,m}, \psi_0)$ with $\Box^{(0),\psi_0}_{b} u = 0$. Let $v(z, x'') =  u(z,x'') e^{-\frac{1}{2} \sum\limits_t \beta_t x_{2n+t}}.$
Let $\mathcal{F}_{x''\to\eta}[v](z, \eta)$ be the partial Fourier transform of $v(z, x'')$ with respect to $(x_{2n+1}, \dots, x_{2n+m})$.

Set
\begin{equation*}
\alpha(z,\eta) = \alpha(z,\eta_1,\dots,\eta_m) = \mathcal{F}_{x''\to\eta}[v](z,\eta) \exp\bigg[\sum_{\ell=1}^m\Big(\frac{i\beta_\ell}{2}-\eta_\ell\Big)\sum_{j,k=1}^{n}\lambda^{\ell}_{jk}z_j\overline{z}_k\bigg].
\end{equation*}

\begin{prop}\label{alphaholo}
For almost every $\eta\in \mathbb{R}^m,$ we have $\int_{\mathbb{C}^n}|\alpha(z,\eta)|^2e^{-\Phi_\eta}d\mu(z) < \infty$ and
\[
\square^{(q),\Phi_{\eta}}\alpha(z,\eta) = 0
\]
in the sense of distribution. Thus, $\alpha(z,\eta) \in \Omega^{0,q}(\mathbb{C}^n)$ for almost every $\eta\in \mathbb{R}^m.$
\end{prop}

Before proving Proposition \ref{alphaholo}, we prove the following lemma.

\begin{lem}\label{alphahololem}
Let $f \in \Omega^{0,q}(\mathbb{C}^n)$. Let $\varphi(\eta) = \varphi(\eta_1, \dots, \eta_m) \in L^2(\mathbb{R}^m)$ with compact support. Then we have
\begin{align*}
&\int_{\mathbb{R}^m} e^{i \langle x'', \eta\rangle}\, e^{-\sum\limits_{\ell=1}^m(\frac{i\beta_\ell}{2}
  -\eta_\ell)\sum\limits_{j,k=1}^{n}\lambda^{\ell}_{jk}z_j\overline{z}_k + \frac{1}{2} \sum\limits_{t=1}^m \beta_t x_{2n+t}} \big(\Box^{(q),{\Phi}_\eta} f(z)\big) \varphi(\eta) d\mu(\eta) \\
&= \Box^{(q),\psi_0}_{b} \bigg(\int_{\mathbb{R}^m} e^{i \langle x'', \eta\rangle} e^{-\sum\limits_{\ell=1}^m(\frac{i\beta_\ell}{2}
  -\eta_\ell)\sum\limits_{j,k=1}^{n}\lambda^{\ell}_{jk}z_j\overline{z}_k + \frac{1}{2} \sum\limits_{t=1}^m \beta_t x_{2n+t}} f(z) \varphi(\eta) d\mu(\eta)\bigg)
\end{align*}
\end{lem}

\begin{proof}
We first calculate, for $s=1,\dots,n,$
\begin{align}\label{1}
&\overline{U}_{s, \mathbb{H}^{n,m}} \bigg(\int_{\mathbb{R}^m} e^{i \langle x'', \eta\rangle} e^{-\sum\limits_{\ell=1}^m(\frac{i\beta_\ell}{2}
  -\eta_\ell)\sum\limits_{j,k=1}^{n}\lambda^{\ell}_{jk}z_j\overline{z}_k + \frac{1}{2} \sum\limits_t \beta_t x_{2n+t}} f(z) \varphi(\eta) d\mu(\eta)\bigg)\\
&= \Big(\frac{\partial}{\partial \overline{z}_s} + i \sum_{q=1}^m \Big(\sum_{p=1}^n \lambda_{ps}^q z_p\Big) \frac{\partial}{\partial x_{2n+q}}\Big)\nonumber \\
&\quad \ \bigg(\int_{\mathbb{R}^m} e^{i \langle x'', \eta\rangle} e^{-\sum\limits_{\ell=1}^m(\frac{i\beta_\ell}{2}
  -\eta_\ell)\sum\limits_{j,k=1}^{n}\lambda^{\ell}_{jk}z_j\overline{z}_k + \frac{1}{2} \sum\limits_t \beta_t x_{2n+t}} f(z) \varphi(\eta) d\mu(\eta)\bigg)\nonumber \\
&= \int_{\mathbb{R}^m} e^{i \langle x'', \eta\rangle} \, e^{-\sum\limits_{\ell=1}^m(\frac{i\beta_\ell}{2}
  -\eta_\ell)\sum\limits_{j,k=1}^{n}\lambda^{\ell}_{jk}z_j\overline{z}_k + \frac{1}{2} \sum\limits_t \beta_t x_{2n+t}}\nonumber\\
&\quad\times \bigg[ -\Big(\sum_{\ell=1}^m \Big(\frac{i \beta_\ell}{2} - \eta_\ell\Big)\sum\limits_{j=1}^n \lambda_{js}^\ell z_j\Big) \, f(z) + \frac{\partial f}{\partial \overline{z}_s} \bigg] \varphi(\eta)\nonumber \\
&\quad + \bigg[ i \sum_{q=1}^m \Big(\sum_{p=1}^n \lambda_{ps}^q z_p\Big) \, e^{i \langle x'', \eta\rangle}  e^{-\sum\limits_{\ell=1}^m(\frac{i\beta_\ell}{2}-\eta_\ell)\sum\limits_{j,k=1}^{n}\lambda^{\ell}_{jk}z_j\overline{z}_k + \frac{1}{2} \sum\limits_t \beta_t x_{2n+t}}  \Big(i \eta_q + \frac{1}{2} \beta_q\Big) \bigg] f(z) \varphi(\eta) d\mu(\eta)\nonumber \\
&= \int_{\mathbb{R}^m} e^{i \langle x'', \eta\rangle} \, e^{-\sum\limits_{\ell=1}^m(\frac{i\beta_\ell}{2}
  -\eta_\ell)\sum\limits_{j,k=1}^{n}\lambda^{\ell}_{jk}z_j\overline{z}_k + \frac{1}{2} \sum\limits_t \beta_t x_{2n+t}} \bigg[  -\Big(\sum_{\ell=1}^m \Big(\frac{i \beta_\ell}{2} - \eta_\ell\Big)\sum\limits_{j=1}^n \lambda_{sj}^\ell z_j\Big) \, f(z) + \frac{\partial f}{\partial \overline{z}_s}\nonumber\\
&\quad + \sum_{q=1}^m \Big(\sum_{p=1}^n \lambda_{ps}^q z_p\Big) \Big(\frac{i \beta_q}{2} - \eta_q\Big) \, f(z) \bigg] \varphi(\eta) d\mu(\eta)\nonumber \\
&= \int_{\mathbb{R}^m} e^{i \langle x'', \eta\rangle} \, e^{-\sum\limits_{\ell=1}^m(\frac{i\beta_\ell}{2}
  -\eta_\ell)\sum\limits_{j,k=1}^{n}\lambda^{\ell}_{jk}z_j\overline{z}_k + \frac{1}{2} \sum\limits_t \beta_t x_{2n+t}} \, \frac{\partial f}{\partial \overline{z}_s}\, \varphi(\eta) d\mu(\eta)\nonumber
\end{align}

We next calculate, for $s=1,\dots,n,$
\begin{align}\label{2}
&\overline{U}^{*, \psi_0}_{s, \mathbb{H}^{n,m}} \bigg(\int_{\mathbb{R}^m} e^{i \langle x'', \eta\rangle} e^{-\sum\limits_{\ell=1}^m(\frac{i\beta_\ell}{2}
  -\eta_\ell)\sum\limits_{j,k=1}^{n}\lambda^{\ell}_{jk}z_j\overline{z}_k + \frac{1}{2} \sum\limits_t \beta_t x_{2n+t}}\, f(z) \varphi(\eta) d\mu(\eta)\bigg)\\
&= \bigg(-\frac{\partial}{\partial z_s} + \sum_{p=1}^n \mu_{sp} \overline{z}_p + i \sum_{q=1}^m \Big(\sum_{p=1}^n \lambda_{sp}^q \overline{z}_p\Big) \frac{\partial}{\partial x_{2n+q}} - i \sum_{q=1}^m \sum_{p=1}^n \lambda_{sp}^q \overline{z}_p \beta_q\bigg)\nonumber\\
&\quad\quad\bigg(\int_{\mathbb{R}^m} e^{i \langle x'', \eta\rangle} e^{-\sum\limits_{\ell=1}^m(\frac{i\beta_\ell}{2} -\eta_\ell)\sum\limits_{j,k=1}^{n}\lambda^{\ell}_{jk}z_j\overline{z}_k + \frac{1}{2} \sum\limits_t \beta_t x_{2n+t}} f(z) \varphi(\eta) d\mu(\eta)\bigg)\nonumber \\
&= \int_{\mathbb{R}^m} e^{i \langle x'', \eta\rangle} \, e^{-\sum\limits_{\ell=1}^m(\frac{i\beta_\ell}{2}
  -\eta_\ell)\sum\limits_{j,k=1}^{n}\lambda^{\ell}_{jk}z_j\overline{z}_k + \frac{1}{2} \sum\limits_t \beta_t x_{2n+t}}  \bigg[ \Big(\sum_{\ell=1}^m \Big(\frac{i \beta_\ell}{2} - \eta_\ell\Big)\sum\limits_{k=1}^n \lambda_{sk}^\ell \overline{z}_k\Big)  f(z) - \frac{\partial f}{\partial z_s} \bigg] \varphi(\eta)\nonumber \\
&\quad + \bigg[ i \sum_{q=1}^m \Big(\sum_{p=1}^n \lambda_{sp}^q \overline{z}_p\Big)  e^{i \langle x'', \eta\rangle} e^{-\sum\limits_{\ell=1}^m(\frac{i\beta_\ell}{2}
  -\eta_\ell)\sum\limits_{j,k=1}^{n}\lambda^{\ell}_{jk}z_j\overline{z}_k + \frac{1}{2} \sum\limits_t \beta_t x_{2n+t}}  \Big(i \eta_q + \frac{1}{2} \beta_q\Big) \bigg] f(z) \varphi(\eta)\nonumber \\
&\quad + e^{i \langle x'', \eta\rangle} \, e^{-\sum\limits_{\ell=1}^m(\frac{i\beta_\ell}{2}
  -\eta_\ell)\sum\limits_{j,k=1}^{n}\lambda^{\ell}_{jk}z_j\overline{z}_k + \frac{1}{2} \sum\limits_t \beta_t x_{2n+t}}  \Big(\sum_{p=1}^n \mu_{sp} \overline{z}_p - i \sum_{q=1}^m \sum_{p=1}^n \lambda_{sp}^q \overline{z}_p \beta_q \Big) f(z) \varphi(\eta) d\mu(\eta)\nonumber \\
&= \int_{\mathbb{R}^m} e^{i \langle x'', \eta\rangle} \, e^{-\sum\limits_{\ell=1}^m(\frac{i\beta_\ell}{2}
  -\eta_\ell)\sum\limits_{j,k=1}^{n}\lambda^{\ell}_{jk}z_j\overline{z}_k + \frac{1}{2} \sum\limits_t \beta_t x_{2n+t}}  \bigg[ \Big(\sum_{\ell=1}^m \Big(\frac{i \beta_\ell}{2} - \eta_\ell\Big)\sum\limits_{k=1}^n \lambda_{sk}^\ell \overline{z}_k\Big)  f(z) - \frac{\partial f}{\partial z_s}\nonumber\\
&\quad+ \Big(\sum_{q=1}^m \Big(\sum_{p=1}^n \lambda_{sp}^q \overline{z}_p\Big) \Big(\frac{i \beta_q}{2} - \eta_q\Big)\Big) f(z) + \Big(\sum_{p=1}^n \mu_{sp} \overline{z}_p\Big) f(z)\nonumber\\
&\quad- \Big(i \sum_{q=1}^m \sum_{p=1}^n \lambda_{sp}^q \overline{z}_p \beta_q\Big) f(z) \bigg] \varphi(\eta) d\mu(\eta)\nonumber \\
&= \int_{\mathbb{R}^m} e^{i \langle x'', \eta\rangle} e^{-\sum\limits_{\ell=1}^m(\frac{i\beta_\ell}{2}
  -\eta_\ell)\sum\limits_{j,k=1}^{n}\lambda^{\ell}_{jk}z_j\overline{z}_k + \frac{1}{2} \sum\limits_t \beta_t x_{2n+t}}  \bigg[ \Big(\sum_{q=1}^m i \beta_q \Big(\sum_{p=1}^n \lambda_{sp}^q \overline{z}_p\Big) \Big) f(z)\nonumber\\
&\quad- \Big( 2\sum_{q=1}^m \Big(\sum_{p=1}^n \lambda_{sp}^q \overline{z}_p\Big) \eta_q \Big) f(z) - \frac{\partial f}{\partial z_s}
+ \Big(\sum_{p=1}^n \mu_{sp} \overline{z}_p\Big) f(z) - \Big( i \sum_{q=1}^m \sum_{p=1}^n \lambda_{sp}^q \overline{z}_p \beta_q \Big) f(z) \bigg]\nonumber\\
&\quad\times\varphi(\eta) d\mu(\eta)\nonumber \\
&= \int_{\mathbb{R}^m} e^{i \langle x'', \eta\rangle} e^{-\sum\limits_{\ell=1}^m(\frac{i\beta_\ell}{2}
  -\eta_\ell)\sum\limits_{j,k=1}^{n}\lambda^{\ell}_{jk}z_j\overline{z}_k + \frac{1}{2} \sum\limits_t \beta_t x_{2n+t}}  \bigg[ -\frac{\partial f}{\partial z_s} + \Big(\sum_{p=1}^n \mu_{sp} \overline{z}_p\Big) f(z)\nonumber\\
&\quad- \Big( 2 \sum_{q=1}^m \Big(\sum_{p=1}^n \lambda_{sp}^q \overline{z}_p\Big) \eta_q \Big) f(z) \bigg] \varphi(\eta) d\mu(\eta),\nonumber
\end{align}

Finally, we calculate, for $s, p = 1,\dots,n,$
\begin{align}\label{3}
&\Big(\mu_{sp} - i \sum_{q=1}^m \lambda_{sp}^q \beta_q + 2i \sum_{q=1}^m \lambda_{sp}^q \frac{\partial}{\partial x_{2n+q}}\Big)\\ 
&\Big(\int_{\mathbb{R}^m} e^{i \langle x'', \eta\rangle} e^{-\sum\limits_{\ell=1}^m(\frac{i\beta_\ell}{2}
  -\eta_\ell)\sum\limits_{j,k=1}^{n}\lambda^{\ell}_{jk}z_j\overline{z}_k + \frac{1}{2} \sum\limits_t \beta_t x_{2n+t}} f(z) \varphi(\eta) d\mu(\eta)\Big)\nonumber\\
&= \int_{\mathbb{R}^m} e^{i \langle x'', \eta\rangle} e^{-\sum\limits_{\ell=1}^m(\frac{i\beta_\ell}{2}
  -\eta_\ell)\sum\limits_{j,k=1}^{n}\lambda^{\ell}_{jk}z_j\overline{z}_k + \frac{1}{2} \sum\limits_t \beta_t x_{2n+t}} \Big(\mu_{sp} - i \sum_{q=1}^m \lambda_{sp}^q \beta_q\Big) f(z) \varphi(\eta)\nonumber \\
&\quad + \bigg[ 2i \Big(\sum_{q=1}^m \lambda_{sp}^q\Big) e^{i \langle x'', \eta\rangle} e^{-\sum\limits_{\ell=1}^m(\frac{i\beta_\ell}{2}
  -\eta_\ell)\sum\limits_{j,k=1}^{n}\lambda^{\ell}_{jk}z_j\overline{z}_k + \frac{1}{2} \sum\limits_t \beta_t x_{2n+t}} \Big(i \eta_q + \frac{1}{2} \beta_q\Big) \bigg] f(z) \varphi(\eta) d\mu(\eta)\nonumber \\
&= \int_{\mathbb{R}^m} e^{i \langle x'', \eta\rangle} e^{-\sum\limits_{\ell=1}^m(\frac{i\beta_\ell}{2}
  -\eta_\ell)\sum\limits_{j,k=1}^{n}\lambda^{\ell}_{jk}z_j\overline{z}_k + \frac{1}{2} \sum\limits_t \beta_t x_{2n+t}}\bigg[ \mu_{sp} - i \sum_{q=1}^m \lambda_{sp}^q \beta_q\nonumber\\
&\quad+ 2i \sum_{q=1}^m \lambda_{sp}^q \Big(i \eta_q + \frac{1}{2} \beta_q\Big) \bigg] f(z) \varphi(\eta) d\mu(\eta)\nonumber \\
&= \int_{\mathbb{R}^m} e^{i \langle x'', \eta\rangle} e^{-\sum\limits_{\ell=1}^m(\frac{i\beta_\ell}{2}
  -\eta_\ell)\sum\limits_{j,k=1}^{n}\lambda^{\ell}_{jk}z_j\overline{z}_k + \frac{1}{2} \sum\limits_t \beta_t x_{2n+t}} \Big( \mu_{sp} - 2 \sum_{q=1}^m \lambda_{sp}^q \eta_q \Big) f(z) \varphi(\eta) d\mu(\eta)\nonumber
\end{align}
From (\ref{1}), (\ref{2}), (\ref{3}) and the explicit formulas for $\Box^{(q),\psi_0}_b$ (Proposition \ref{lapfma}) and $\Box^{(q),\Phi_\eta}$ (Proposition \ref{holapfma}), the lemma follows.
\end{proof}

We can now prove Proposition \ref{alphaholo}.
\begin{proof}
From the discussion in (\ref{ae}), we can see that there is a Lebesgue measure zero set $A_0 \subset \mathbb{R}^m$ such that for all $\eta = (\eta_1, \dots, \eta_m) \notin A_0$,
\begin{align*}
&\int_{\mathbb{C}^n} |\alpha(z, \eta)|^2 e^{-\Phi_\eta(z)} d\mu(z)\\ 
&= \int_{\mathbb{C}^n} |\mathcal{F}_{x''\to\eta}[v](z, \eta)|^2 e^{ -2\sum\limits_{\ell=1}^m\eta_\ell\big(\sum\limits_{j,s=1}^{n}\lambda^{\ell}_{js}z_j\overline{z}_s\big) } \, e^{-\sum\limits_{j,s} \mu_{js} z_j \overline{z}_s + 2 \sum\limits_{\ell=1}^m \eta_\ell\big(\sum\limits_{j,s} \lambda_{js}^\ell z_j \overline{z}_s\big)} d\mu(z) \\
&= \int_{\mathbb{C}^n} |\mathcal{F}_{x''\to\eta}[v](z, \eta)|^2 e^{-\Phi_0(z)} d\mu(z) < \infty.
\end{align*}

Next we show $\Box^{(q),{\Phi}_\eta} \alpha(z, \eta) = 0$ for a.e. $\eta \in \mathbb{R}^m$. Let $f \in \Omega_c^{0,q}(\mathbb{C}^n)$. Put 
\begin{equation*}
h(\eta) = \int_{\mathbb{C}^n} (\alpha(z, \eta) \mid \Box^{(q),{\Phi}_\eta} f ) e^{-\Phi_\eta(z)} d\mu(z)
\end{equation*}
if $\eta \notin A_0$ and $h(\eta)=0$ if $\eta \in A_0$.

By Cauchy's inequality, we have
\begin{equation*}
|h(\eta)|^2 \leq \Big(\int_{\mathbb{C}^n} |\alpha(z, \eta)|^2 e^{-\Phi_\eta(z)}\, d\mu(z)\Big) \Big(\int_{\mathbb{C}^n} \big|\Box^{(q),{\Phi}_\eta} f\big|^2 e^{-\Phi_\eta(z)}\, d\mu(z)\Big).
\end{equation*}
For $R > 0$, put $\varphi_R(\eta) = 1_{B(R)}(\eta) \cdot \overline{h(\eta)}$ where $B(R) = \{ \eta \in \mathbb{R}^m \mid |\eta| \leq R \}$.
We have
\begin{align*}
\int |\varphi_R(\eta)|^2\, d\mu(\eta) &= \int_{B(R)} |h(\eta)|^2\, d\mu(\eta) \leq C \int_{B(R)} \int_{\mathbb{C}^n} |\alpha(z, \eta)|^2 e^{-\Phi_\eta(z)} \,d\mu(z) d\mu(\eta) \\
&\leq C \int \int |\mathcal{F}_{x''\to\eta}[v](z, \eta)|^2 e^{-\Phi_0(z)}\, d\mu(z) d\mu(\eta) < \infty,
\end{align*}
where $C>0.$
Hence $\varphi_R(\eta) \in L^2(B(R)) \cap L^1(B(R))$.

\begin{align*}
&\int_{\mathbb{R}^m} h(\eta) \varphi_R(\eta) d\mu(\eta)\\
&= \int_{B(R)} |h(\eta)|^2 d\mu(\eta) = \int_{\mathbb{R}^m} \int_{\mathbb{C}^n} (\alpha(z, \eta) \mid \Box^{(q),{\Phi}_\eta} f ) e^{-\Phi_\eta(z)} \varphi_R(\eta)\, d\mu(\eta) d\mu(z) \\
&= \int \int \big(\mathcal{F}_{x''\to\eta}[v](z, \eta) \mid e^{-\sum\limits_{\ell=1}^m(\frac{i\beta_\ell}{2}-\eta_\ell)\sum\limits_{j,s=1}^{n}\lambda^{\ell}_{js}z_j\overline{z}_s} \big(\Box^{(q),{\Phi}_\eta} f\big)(z) \, \overline{\varphi_R(\eta)}\big) e^{-\Phi_0(z)}\, d\mu(\eta) d\mu(z) \\
&= \int \int \Big(u(z, x'') \,\Big|\, \int_{\mathbb{R}^m} e^{i \langle x'', \eta\rangle} e^{-\sum\limits_{\ell=1}^m(\frac{i\beta_\ell}{2}
  -\eta_\ell)\sum\limits_{j,s=1}^{n}\lambda^{\ell}_{js}z_j\overline{z}_s + \frac{1}{2} \sum\limits_t \beta_t x_{2n+t}}\,  \Box^{(q),{\Phi}_\eta}\big( f \, \overline{\varphi_R(\eta)}\big)  d\mu(\eta) \Big)\\
&\quad \times e^{-\psi_0(z, x'')}\, d\mu(x'') d\mu(z)
\end{align*}
By Lemma \ref{alphahololem}, we have
\begin{align*}
&\int_{\mathbb{R}^m} e^{i \langle x'', \eta\rangle} e^{-\sum\limits_{\ell=1}^m(\frac{i\beta_\ell}{2}
  -\eta_\ell)\sum\limits_{j,s=1}^{n}\lambda^{\ell}_{js}z_j\overline{z}_s + \frac{1}{2} \sum\limits_t \beta_t x_{2n+t}} \big(\Box^{(q),{\Phi}_\eta} f(z) \big)\overline{\varphi_R(\eta)}\, d\mu(\eta) \\
&= \square^{(q),\psi_0}_b \bigg( \int_{\mathbb{R}^m} e^{i \langle x'', \eta\rangle} e^{-\sum\limits_{\ell=1}^m(\frac{i\beta_\ell}{2}
  -\eta_\ell)\sum\limits_{j,s=1}^{n}\lambda^{\ell}_{js}z_j\overline{z}_s + \frac{1}{2} \sum\limits_t \beta_t x_{2n+t}} f(z) \overline{\varphi_R(\eta)}\, d\mu(\eta) \bigg).
\end{align*}
Put 
\begin{equation*}
S(z, x'') = \int_{\mathbb{R}^m} e^{i \langle x'', \eta\rangle} e^{-\sum\limits_{\ell=1}^m(\frac{i\beta_\ell}{2}
  -\eta_\ell)\sum\limits_{j,s=1}^{n}\lambda^{\ell}_{js}z_j\overline{z}_s + \frac{1}{2} \sum\limits_t \beta_t x_{2n+t}}\, f(z)\, \overline{\varphi_R(\eta)}\,d\mu(\eta).
\end{equation*}
Then
\begin{equation*}
\int_{B(R)} |h(\eta)|^2 d\mu(\eta) = \int \int \big(u(z, x'') \mid \square^{(q),\psi_0}_b S(z, x'')\big) e^{-\psi_0(z, x'')}\, d\mu(x'') d\mu(z)
\end{equation*}
Take a smooth function $\chi(x'') = \chi(x_{2n+1}, \dots, x_{2n+m}) \in \mathcal{C}^\infty_c(\mathbb{R}^m)$ with $\chi(x'')=1$ when $|x''|<1$ and $\chi(x'')=0$ when $|x''|>2$. Then
\begin{align}\label{heta}
\int_{B(R)} |h(\eta)|^2 d\mu(\eta) &= \lim_{j \to \infty} \int \int \big(u(z, x'') \mid \chi\big(\frac{x''}{j}\big) \square^{(q),\psi_0}_b S(z, x'')\big) e^{-\psi_0(z, x'')}\, d\mu(x'') d\mu(z) \\
&= \lim_{j \to \infty} \bigg( \int \int \big(\square^{(q),\psi_0}_b u(z, x'') \mid \chi\big(\frac{x''}{j}\big) S(z, x'')\big) e^{-\psi_0(z, x'')}\, d\mu(x'') d\mu(z) \nonumber \\
&\quad\quad  + \int \int \Big(u(z, x'') \mid \Big[\chi(\frac{x''}{j}), \square^{(q),\psi_0}_b\Big] S(z, x'')\Big) e^{-\psi_0(z, x'')}\, d\mu(x'') d\mu(z) \bigg)\nonumber \\
&= \lim_{j \to \infty} \int \int \Big(u(z, x'') \mid \Big[\chi\big(\frac{x''}{j}\big), \square^{(q),\psi_0}_b\Big] S(z, x'')\Big) e^{-\psi_0(z, x'')} \,d\mu(x'') d\mu(z).\nonumber
\end{align}
We can check that $\Big[\chi(\frac{x''}{j}), \square^{(q),\psi_0}_b\Big]$ is a 1st order partial differential operator and all the coefficients of $\Big[\chi(\frac{x''}{j}), \square^{(q),\psi_0}_b\Big]$ converge to $0$ as $j \to \infty$ uniformly in $x''$ and locally uniformly in $z$.

By Parseval's formula, we have
\begin{align*}
&\sum_{|\alpha| \leq 1} \int \int |\partial^\alpha_{z, x''} S|^2 e^{-\psi_0(z, x'')} \,d\mu(x'')\, d\mu(z)\\
&= \sum_{|\alpha| \leq 1} \int \int \bigg| \partial^\alpha_{z, x''} \Big(\int_{\mathbb{R}^m} e^{i \langle x'', \eta\rangle} e^{-\sum\limits_{\ell=1}^m(\frac{i\beta_\ell}{2}-\eta_\ell)\sum\limits_{j,s=1}^{n}\lambda^{\ell}_{js}z_j\overline{z}_s + \frac{1}{2} \sum\limits_t \beta_t x_{2n+t}}f(z)\overline{\varphi_R(\eta)}\,d\mu(\eta) \Big) \bigg|^2 \\
&\hspace{3cm} \times e^{-\psi_0(z, x'')} \, d\mu(x'') d\mu(z) \\
&= \sum_{|\alpha| \leq 1} \int \int \bigg| \partial^\alpha_{z, x''} \Big(\int_{\mathbb{R}^m} e^{i \langle x'', \eta\rangle} e^{-\sum\limits_{\ell=1}^m(\frac{i\beta_\ell}{2} -\eta_\ell)\sum\limits_{j,s=1}^{n}\lambda^{\ell}_{js}z_j\overline{z}_s}f(z)\overline{\varphi_R(\eta)}\,d\mu(\eta) \Big) \bigg|^2 \\
&\hspace{3cm} \times e^{-\Phi_0(z)} d\mu(x'') d\mu(z) \\
&= (2\pi)^{2d}\sum_{|\alpha| \leq 1} \int \int \bigg| \partial^\alpha_{z, x''} \mathcal{F}^{-1}_{\eta\rightarrow x''}\Big[ e^{-\sum\limits_{\ell=1}^m(\frac{i\beta_\ell}{2} -\eta_\ell)\sum\limits_{j,s=1}^{n}\lambda^{\ell}_{js}z_j\overline{z}_s} f(z) \overline{\varphi_R(\eta)} \Big] (z, x'') \bigg|^2 \\
&\hspace{3cm} \times e^{-\Phi_0(z)} d\mu(x'')  d\mu(z) \\
&= (2\pi)^{2d}\sum_{|\alpha| \leq 1} \int \int \bigg| \mathcal{F}^{-1}_{\eta\rightarrow x''}\Big[ \eta^{\alpha_{x''}}\partial^{\alpha_z}_z \Big(e^{-\sum\limits_{\ell=1}^m(\frac{i\beta_\ell}{2} -\eta_\ell)\sum\limits_{j,s=1}^{n}\lambda^{\ell}_{js}z_j\overline{z}_s} f(z) \overline{\varphi_R(\eta)}\Big) \Big] (z, x'') \bigg|^2 \\
&\hspace{3cm} \times e^{-\Phi_0(z)} d\mu(x'')  d\mu(z) \\
&= (2\pi)^m \sum_{|\alpha| \leq 1} \int \int \bigg| \eta^{\alpha_{x''}} \partial^\alpha_z \Big( e^{-\sum\limits_{\ell=1}^m(\frac{i\beta_\ell}{2} -\eta_\ell)\sum\limits_{j,s=1}^{n}\lambda^{\ell}_{js}z_j\overline{z}_s} f(z) \overline{\varphi_R(\eta)} \Big) \bigg|^2 \\
&\hspace{3cm} \times d\mu(\eta) e^{-\Phi_0(z)} d\mu(z) \\
&\leq C \sum_{|\alpha| \leq 1} \int \int \big(1 + |z| + |\eta| + |z||\eta|\big)^2 |\partial^\alpha_z f(z)|^2 |\varphi_R(\eta)|^2 e^{-\Phi_\eta(z)} d\mu(\eta) d\mu(z) \\
&\leq C \int |\varphi_R(\eta)|^2 d\mu(\eta) < \infty.
\end{align*}

Thus, we have 
\begin{equation*}
\lim_{j \to \infty} \int \int \Big(u(z, x'') \mid \Big[\chi\big(\frac{x''}{j}\big), \square^{(q),\psi_0}_b\Big] S(z, x'')\Big) e^{-\psi_0(z, x'')} \,d\mu(x'') d\mu(z) = 0.
\end{equation*}
Together with (\ref{heta}), we obtain $\int_{B(R)} |h(\eta)|^2 d\mu(\eta) = 0$ for all $R > 0.$ Hence  $h(\eta) = 0$ for a.e. $\eta \in \mathbb{R}^m$.

Let us consider the Sobolev space $W^2(\mathbb{C}^{n})$ of distributions in $\mathbb{C}^{n}$ whose derivatives of
order $\leq 2$ are in $L^2$. The space of forms of type $(0, q)$ with coefficients in this space is accordingly denoted $W^2_{(0,q)}(\mathbb{C}^{n})$. Since $W^2_{(0,q)}(\mathbb{C}^{n})$ is separable and $\Omega^{0,q}_c(\mathbb{C}^{n})$ is dense in
$W^2_{(0,q)}(\mathbb{C}^{n})$, we can find $f_j \in \Omega^{0,q}_c(\mathbb{C}^{n}), j = 1, 2, \dots$, such that $\{f_1, f_2, \dots\}$ is a dense
subset of $W^2_{(0,q)}(\mathbb{C}^{n})$. Moreover, we can take $\{f_1, f_2, \dots\}$ so that for all $g \in \Omega^{0,q}_c(\mathbb{C}^{n})$
with $\text{supp } g \subset B_r,\ r > 0$, we can find $f_{j_1}, f_{j_2}, \dots, \text{supp } f_{j_t} \subset
B_r, t = 1, 2, \dots$, such that $f_{j_t} \to g$ for $t \to \infty$ in $W^2_{(0,q)}(\mathbb{C}^{n})$.

Now, for each $j$, we can repeat the method above and find a measurable set $A_j \supset
A_0, |A_j| = 0$ ($A_0$ is as in the beginning of the proof), such that $(\alpha(z, \eta) \mid \Box^{(q),{\Phi}_\eta} f_j(z))_{\Phi_\eta} = 0$
for all $\eta \notin A_j$. Put $A = \bigcup_j A_j$. Then, $|A| = 0$ and for all $\eta \notin A, (\alpha(z, \eta) \mid \Box^{(q),{\Phi}_\eta} f_j(z))_{\Phi_\eta} = 0$
for all $j$. Let $g \in \Omega^{0,q}_c(\mathbb{C}^{n})$ with $\text{supp } g \subset B_r$. From the discussion above, we can find
$f_{j_1}, f_{j_2}, \dots, \text{supp } f_{j_t} \subset B_r, t = 1, 2, \dots$, such that $f_{j_t} \to g$ in $W^2_{(0,q)}(\mathbb{C}^{n}), t \to \infty$.
Then, for $\eta \notin A$,
\begin{align*}
(\alpha(z, \eta) \mid \Box^{(q),{\Phi}_\eta} g)_{\Phi_\eta} &= (\alpha(z, \eta) \mid \Box^{(q),{\Phi}_\eta} (g - f_{j_t}))_{\Phi_\eta} + (\alpha(z, \eta) \mid \Box^{(q),{\Phi}_\eta} f_{j_t})_{\Phi_\eta} \\
&= (\alpha(z, \eta) \mid \Box^{(q),{\Phi}_\eta} (g - f_{j_t}))_{\Phi_\eta}.
\end{align*}

Now,
\begin{align*}
\left| (\alpha(z, \eta) \mid \Box^{(q),{\Phi}_\eta} (g - f_{j_t}))_{\Phi_\eta} \right| &= \left| \int_{B_r} (\alpha(z, \eta) \mid \Box^{(q),{\Phi}_\eta} (g - f_{j_t})) e^{-\Phi_\eta(z)} d\mu(z) \right| \\
&\leq C \sum_{|\alpha| \leq 2} \int |\partial^\alpha_x (g - f_{j_t})|^2 d\mu(z) \to 0, \quad t \to \infty.
\end{align*}

Thus, for $\eta \notin A, (\alpha(z, \eta) \mid \Box^{(q),{\Phi}_\eta} g)_{\Phi_\eta} = 0$ for all $g \in \Omega^{0,q}_c(\mathbb{C}^{n})$. The proposition follows.
\end{proof}

Set $M_\ell = (\lambda^{\ell}_{jk}), \ell = 1,\dots,m.$  Let
\[
I_{M_1,\dots,M_m}^{\Phi_0} := \Big\{\eta=(\eta_1,\dots,\eta_m)\in \mathbb{R}^m;\, M_\eta = (\mu_{jk}) - 2\sum_{\ell=1}^m\eta_\ell M_\ell ~~ \text{is positive definite}\Big\}.
\]

\begin{cor} \label{vanish}
Let $u \in \mathcal{H}^0_b(\mathbb{H}^{n,m})$. Then $\alpha(z, \eta) = 0$ for a.e. $\eta \notin I^{\Phi_0}_{M_1, \dots, M_m}$.
\end{cor}

\begin{proof}
Suppose $u \in \mathcal{H}^0_b(\mathbb{H}^{n,m})$. By Proposition \ref{alphaholo}, we have $\alpha(z, \eta)$ is holomorphic on $\mathbb{C}^n$ and
\begin{equation*}
\int_{\mathbb{C}^n} |\alpha(z, \eta)|^2 e^{-\Phi_\eta(z)} d\mu(z) < \infty
\end{equation*}
for a.e. $\eta \in \mathbb{R}^m$.

Fix $\eta_0 \notin I^{\Phi_0}_{M_1, \dots, M_m}$ such that $\int_{\mathbb{C}^n} |\alpha(z, \eta_0)|^2 e^{-\Phi_{\eta_0}(z)} d\mu(z) < \infty$ and $\alpha(z, \eta_0)$ is holomorphic. We may assume
that, with a holomorphic change of coordinates if necessary,
\begin{equation*}
\Phi_{\eta_0}(z) = \lambda_1 |z_1|^2 + \dots + \lambda_n |z_n|^2
\end{equation*}
where $\lambda_1, \dots, \lambda_n \in \mathbb{R}$ and at least one of $\lambda_1, \dots, \lambda_n$ is less than or equal to 0.

Write $\alpha(z, \eta_0) = \sum_{\gamma \in \mathbb{N}_0^n} c_\gamma(\eta_0) z^\gamma$. Fix $\gamma_0 \in \mathbb{N}_0^n$. It is easy to see that
\begin{equation*}
|c_{\gamma_0}(\eta_0)|^2 \int_{\mathbb{C}^n} |z^{\gamma_0}|^2 e^{-\sum_j \lambda_j |z_j|^2} d\mu(z) \leq \int_{\mathbb{C}^n} |\alpha(z, \eta_0)|^2 e^{-\sum_j \lambda_j |z_j|^2} d\mu(z) < \infty.
\end{equation*}

But $\int_{\mathbb{C}^n} |z^{\gamma_0}|^2 e^{-\sum_j \lambda_j |z_j|^2} d\mu(z) = \infty$ since some $\lambda_j \leq 0$. This implies $c_{\gamma_0}(\eta_0) = 0$ and hence $\alpha(z, \eta_0) = 0$.
\end{proof}

\subsection{Weighted Bergman kernel}
In the following, we establish the explicit representation of the Szeg\H{o} kernel for the weighted Heisenberg group. We first recall the formulas for the weighted Bergman kernel.
Let $\phi = (\gamma_{jk})$ be a positive definite matrix. Let $\phi(z) = \sum_{j,k=1}^{n}\gamma_{jk}z_j\overline{z}_k.$ We denote by $L^2(\mathbb{C}^n, 2\phi)$ the weighted $L^2$ space with the norm
\[
\|f\|^2_\phi = \int_{\mathbb{C}^n} |f(z)|^2 e^{-2\phi(z)}d\mu(z).
\]
Let
\[
\overline{\partial}_\phi := \overline{\partial} + (\overline{\partial}\phi)\wedge : \mathcal{C}^\infty(\mathbb{C}^n) \rightarrow \Omega^{0,1}(\mathbb{C}^n).
\]
Let \[
P_\phi : L^2(\mathbb{C}^n) \rightarrow \operatorname{Ker}(\overline{\partial}_\phi)
\] be the Bergman projection with respect to $\overline{\partial}_\phi.$
Then the Bergman kernel function with respect to $\overline{\partial}_\phi$ is given by
\[
P_\phi(z,w) = \frac{\det(\phi)}{\pi^n}e^{-\sum\limits_{j,k=1}^{n}\gamma_{jk}(z_j-w_j)(\overline{z}_k-\overline{w}_k) + \sum\limits_{j,k=1}^{n}\gamma_{jk}(z_j\overline{w}_k-\overline{z}_kw_j)}.
\]
That is,
\[
(P_\phi f)(z) = \int_{\mathbb{C}^n} P_\phi(z,w)f(w)d\mu(w)
\]
for all $f\in L^2(\mathbb{C}^n).$
Let 
\[
B_\phi(z,w) = e^{\phi(z)}P_\phi(z,w)e^{-\phi(w)} = \frac{\det(\phi)}{\pi^n}e^{2\sum\limits_{j,k=1}^{n}\gamma_{jk}z_j\overline{w}_k - 2\sum\limits_{j,k=1}^{n}\gamma_{jk}w_j\overline{w}_k}.
\]
Then 
\begin{equation}\label{berg}
u(z) = \int_{\mathbb{C}^n} B_\phi(z,w) u(w) d\mu(w)
\end{equation}
for all $u\in L^2(\mathbb{C}^n, 2\phi)\cap \operatorname{Ker}(\overline{\partial}).$

\subsection{Szeg\H{o} kernel function for the model case}
We recall some notations.
\begin{align*}
  \psi_0(z,x'') &= \beta_1x_{2n+1} + \cdots + \beta_mx_{2n+m} + \sum_{j,k=1}^{n}\mu_{jk}z_j\overline{z}_k \\
  \Phi_0(z) &= \sum_{j,k=1}^{n}\mu_{jk}z_j\overline{z}_k \\
  \Phi_\eta(z) &= \sum_{j,k=1}^{n}\mu_{jk}z_j\overline{z}_k -2\eta_1\sum_{j,k=1}^{n}\lambda^{1}_{jk}z_j\overline{z}_k - \cdots -2\eta_m\sum_{j,k=1}^{n}\lambda^m_{jk}z_j\overline{z}_k.
\end{align*}

Let $M_\ell = (\lambda^{\ell}_{jk}), \ell = 1,\dots,m$ be $d$ Hermitian matrices. Let
\[
I_{M_1,\dots,M_m}^{\Phi_0} := \Big\{\eta=(\eta_1,\dots,\eta_m)\in \mathbb{R}^m;\, M_\eta = (\mu_{jk}) - 2\sum_{\ell=1}^m\eta_\ell M_\ell ~~ \text{is positive definite}\Big\}.
\]
Let 
\[
\Pi^{(0),\psi_0}_b : L^2(\mathbb{H}^{n,m}, \psi_0) \rightarrow \operatorname{Ker}(\overline{\partial}_{b})
\]
be the Szeg\H{o} projection. Let $x'' = (x_{2n+1},\dots,x_{2n+m})$ and $y'' = (y_{2n+1},\dots,y_{2n+m}).$ Motivated by \cite{HHL18}, and from Proposition \ref{alphaholo}, Corollary \ref{vanish} and \eqref{berg}, we have for $u\in\mathcal{H}^0_b(\mathbb{H}^{n,m}),$  formally,
\begin{align}\label{weightfml}
  &\big(\Pi^{(0),\psi_0}_b u\big)(z,x'')\\ 
  &= e^{\frac{1}{2}\sum\limits_{t=1}^m\beta_t x_{2n+t}}\Big(\frac{1}{2\pi}\Big)^m\int_{\mathbb{R}^m}e^{i\langle\eta,x''\rangle}\Bigg(\int_{\mathbb{C}^n}B_{\frac{\Phi_\eta}{2}}(z,w)\bigg(\Big(\int_{\mathbb{R}^m}
  e^{-i\langle\eta,y''\rangle}u(w,y'')e^{{-\frac{1}{2}}\sum\limits_{t=1}^m\beta_t y_{2n+t}}d\mu(y'')\Big)\nonumber\\
  &\quad\quad\cdot e^{\sum\limits_{\ell=1}^m(\frac{i\beta_\ell}{2}-\eta_\ell)\sum\limits_{j,k=1}^{n}\lambda^{\ell}_{jk}w_j\overline{w}_k}\bigg)d\mu(w)\Bigg) e^{-\sum\limits_{\ell=1}^m(\frac{i\beta_\ell}{2}-\eta_\ell)\sum\limits_{j,k=1}^{n}\lambda^{\ell}_{jk}z_j\overline{z}_k}d\mu(\eta)\nonumber \\
  &=\Big(\frac{1}{2\pi}\Big)^{n+m} \int_{I_{M_1,\dots,M_m}^{\Phi_0}}\int_{\mathbb{C}^n}\int_{\mathbb{R}^m} \det(M_\eta) \exp\bigg[i\langle\eta,x''\rangle-i\langle\eta,y''\rangle +\frac{1}{2}\sum_{t=1}^m\beta_t x_{2n+t}-\frac{1}{2}\sum_{t=1}^m\beta_t y_{2n+t}\nonumber\\
&\quad\quad+\sum_{j,s=1}^{n}(M_\eta)_{js}z_j\overline{w}_s-\sum_{j,s=1}^{n}(M_\eta)_{js}w_j\overline{w}_s+\sum_{\ell=1}^m\Big(\frac{i\beta_\ell}{2}-\eta_\ell\Big)\sum_{j,s=1}^{n}\lambda^{\ell}_{js}w_j\overline{w}_s\nonumber\\
&\quad\quad-\sum_{\ell=1}^m\Big(\frac{i\beta_\ell}{2}
  -\eta_\ell\Big)\sum_{j,s=1}^{n}\lambda^{\ell}_{js}z_j\overline{z}_s\bigg]\, u(w,y'')\,d\mu(y'')d\mu(w)d\mu(\eta).\nonumber
\end{align}
Let $\chi(\eta)\in C^{\infty}_c(\mathbb{R}^m)$ be a smooth function such that $\chi\geq 0,$ $\chi=1$ on $B(1)$ and $\operatorname{supp}(\chi)\subset B(2).$ Motivated by (\ref{weightfml}), we consider the following integral operator on $C^{\infty}_c(\mathbb{H}^{n,m}).$ For all $u\in C^{\infty}_c(\mathbb{H}^{n,m}),$
\begin{align}\label{modwgtszeker}
&(Pu)(z,x'')\\
&= \lim_{\epsilon\rightarrow 0} \Big(\frac{1}{2\pi}\Big)^{n+m} \int_{I_{M_1,\dots,M_m}^{\Phi_0}}\int_{\mathbb{C}^n}\int_{\mathbb{R}^m} \chi(\epsilon\eta)\det(M_\eta) \exp\bigg[i\langle\eta,x''\rangle-i\langle\eta,y''\rangle+\frac{1}{2}\sum_{t=1}^m\beta_t x_{2n+t}\nonumber\\
&\quad\quad\quad\quad-\frac{1}{2}\sum_{t=1}^m\beta_t y_{2n+t}+\sum_{j,s=1}^{n}(M_\eta)_{js}z_j\overline{w}_s-\sum_{j,s=1}^{n}(M_\eta)_{js}w_j\overline{w}_s
+\sum_{\ell=1}^m\Big(\frac{i\beta_\ell}{2}-\eta_\ell\Big)\sum_{j,s=1}^{n}\lambda^{\ell}_{js}w_j\overline{w}_s\nonumber\\
&\quad\quad\quad\quad-\sum_{\ell=1}^m\Big(\frac{i\beta_\ell}{2}
  -\eta_\ell\Big)\sum_{j,s=1}^{n}\lambda^{\ell}_{js}z_j\overline{z}_s\bigg]\, u(w,y'')d\mu(y'')d\mu(w)d\mu(\eta).\nonumber
\end{align}

Let \[
\Pi^{(0)}_{b,\psi_0 } : L^2(\mathbb{H}^{n,m}) \rightarrow \operatorname{Ker}(\overline{\partial}_{b, \psi_0}),
\] where 
$\overline{\partial}_{b,\psi_0} = \overline{\partial}_{b} + \frac{1}{2}\overline{\partial}_{b}\psi_0\wedge,$ be the Szeg\H{o} projection with respect to $\overline{\partial}_{b, \psi_0}.$ Let $I=I_{M_1,\dots,M_m}^{\Phi_0}.$ 
Let $\chi(\eta)\in C^{\infty}_c(\mathbb{R}^m)$ be a smooth function such that $\chi\geq 0,$ $\chi=1$ on $B(1)$ and $\operatorname{supp}(\chi)\subset B(2).$ We consider the following integral operator on $\mathcal{C}^{\infty}_c(\mathbb{H}^{n,m}).$ For all $u\in \mathcal{C}^{\infty}_c(\mathbb{H}^{n,m}),$
\begin{align}
&(\tilde{P}u)(z,x'')\\
 &=\lim_{\epsilon\rightarrow 0} \Big(\frac{1}{2\pi}\Big)^{n+m} \int_{I}\int_{\mathbb{C}^n}\int_{\mathbb{R}^m} \chi(\epsilon\eta) \det(M_\eta) \exp\bigg[i\langle\eta,x''\rangle-i\langle\eta,y''\rangle-\frac{1}{2}\sum_{j,s=1}^{n}(M_\eta)_{js}(z_j-w_j)\nonumber\\
 &\quad\quad\quad\cdot(\overline{z}_s-\overline{w}_s)+\frac{1}{2}\sum_{j,s=1}^{n}(M_\eta)_{js}(z_j\overline{w}_s-\overline{z}_s w_j)+\sum_{\ell=1}^m\frac{i\beta_\ell}{2}
  \sum_{j,s=1}^{n}\lambda^{\ell}_{js}(w_j\overline{w}_s-z_j\overline{z}_s)\bigg]\nonumber\\
  &\quad\quad \quad \cdot u(w,y'')\,d\mu(y'')d\mu(w)d\mu(\eta).\nonumber
\end{align}

\medskip

\begin{rem}\label{r-gue261010yyd}
Observe that we have
\begin{equation}\label{equofsze}
\tilde{P}(x,y) = e^{-\frac{\psi_0}{2}(x)}P(x,y) e^{\frac{\psi_0}{2}(y)} \qquad\text{and}\qquad \Pi^{(0)}_{b,\psi_0}=e^{-\frac{\psi_0}{2}(x)}\Pi^{(0),\psi_0}_b(x,y)e^{\frac{\psi_0}{2}(y)}.
\end{equation}
\end{rem}

\medskip

We will show that 
\begin{equation}\label{e-gue261010yyd}
\mbox{$P =\Pi^{(0),\psi_0}_b$ and $\tilde{P} = \Pi^{(0)}_{b,\psi_0}$.}
\end{equation}
Since the two statements are equivalent by \eqref{equofsze}, we only need to prove one of them. We will prove that $P =  \Pi^{(0),\psi_0}_b.$ We first show that $\tilde{P}u$ converges.
Put
\begin{align*}
&\varphi(\eta,z,w)\\
 &= -\frac{1}{2}\sum_{j,s=1}^{n}(M_\eta)_{js}(z_j-w_j)(\overline{z}_s-\overline{w}_s)+\frac{1}{2}\sum_{j,s=1}^{n}(M_\eta)_{js}(z_j\overline{w}_s-\overline{z}_sw_j)\\
 &\quad\, +\sum_{\ell=1}^m\frac{i\beta_\ell}{2}\sum_{j,s=1}^{n}\lambda^{\ell}_{js}(w_j\overline{w}_s-z_j\overline{z}_s).
\end{align*}

\begin{prop}\label{HSzegocon}
  We have $\tilde{P}u$ converges, for all $u\in C^{\infty}_c(\mathbb{H}^{n,m}).$
\end{prop}
\begin{proof}
Let $I = I_{M_1,\dots,M_m}^{\Phi_0}.$ For $\xi=(\xi_1,\dots,\xi_m)\in \mathbb{R}^m,$ let $\langle \xi\rangle = (1+|\xi|^2)^{\frac{1}{2}}$ and $\langle D_\xi\rangle$ be the pseudodifferential operator associated to the symbol $\langle \xi\rangle.$ We will show that the integral of
\begin{equation}\label{kercon}
\int_{\mathbb{C}^n}\int_{\mathbb{R}^m} 1_I(\eta)\chi(\epsilon\eta) \det(M_\eta) e^{i\langle\eta,x''\rangle-i\langle\eta,y''\rangle}e^{\varphi(\eta,z,w)}u(w,y'')d\mu(y'')d\mu(w)
\end{equation}
as a function in $\eta$ is bounded by a constant $C$ independent of $\epsilon.$ Throughout the proof, the constant $C$ may change line by line.
Fix arbitrary $\epsilon > 0.$ Fix a positive integer $q>m.$ By integrating by parts with respect to $y''$ several times, we have
\begin{align*}
&\Bigg|\int_I\int_{\mathbb{C}^n}\int_{\mathbb{R}^m} \chi(\epsilon\eta) \det(M_\eta) e^{i\langle\eta,x''\rangle-i\langle\eta,y''\rangle}e^{\varphi(\eta,z,w)}u(w,y'')d\mu(y'')d\mu(w)d\mu(\eta)\Bigg| \\
&= \Bigg|\int_I\int_{\mathbb{C}^n}\int_{\mathbb{R}^m} \chi(\epsilon\eta) \det(M_\eta) \big(\langle \eta \rangle^{-2q}\langle D_{y''}\rangle^{2q}e^{-i\langle\eta,y''\rangle}\big)e^{i\langle\eta,x''\rangle}e^{\varphi(\eta,z,w)}
u(w,y'')d\mu(y'')d\mu(w)d\mu(\eta)\Bigg| \\
&= \Bigg|\int_I\int_{\mathbb{C}^n}\int_{\mathbb{R}^m} \chi(\epsilon\eta) \det(M_\eta) \langle \eta \rangle^{-2q}e^{-i\langle\eta,y''\rangle}e^{i\langle\eta,x''\rangle}e^{\varphi(\eta,z,w)}
\big(\langle D_{y''}\rangle^{2q}u(w,y'')\big)d\mu(y'')d\mu(w)d\mu(\eta)\Bigg| \\
&\leq \int_I\int_{\mathbb{C}^n}\int_{\mathbb{R}^m} \chi(\epsilon\eta) \det(M_\eta) \langle \eta \rangle^{-2q}e^{-\frac{1}{2}\sum\limits_{j,s=1}^{n}(M_\eta)_{js}(z_j-w_j)(\overline{z}_s-\overline{w}_s)}
\big|\langle D_{y''}\rangle^{2q}u(w,y'')\big|d\mu(y'')d\mu(w)d\mu(\eta)\\
&\leq C\int_I\int_{\mathbb{C}^n}\det(M_\eta) \langle \eta \rangle^{-2q}e^{-\frac{1}{2}\sum\limits_{j,s=1}^{n}(M_\eta)_{js}(z_j-w_j)(\overline{z}_s-\overline{w}_s)}\,d\mu(w)d\mu(\eta)
\end{align*}
where $C$ is a positive constant depending only on $u.$ Fix an $\eta \in I.$ We have $M_\eta$ is positive definite. So we can take a suitable unitary linear change of coordinates $w(\eta) = (w_1(\eta),\dots,w_n(\eta))$ of $(w_1,\dots ,w_n)$ such that we can convert $\sum_{j,s=1}^{n}(M_\eta)_{js}w_j\overline{w}_s$ into the form $\sum_{j=1}^{n}\lambda_j(\eta)|w_j(\eta)|^2,$ where $\lambda_1(\eta),\dots,\lambda_n(\eta) >0$ are eigenvalues of $M_\eta,$ under the coordinates $w(\eta).$
Hence
\begin{align*}
   & \int_I\int_{\mathbb{C}^n}\det(M_\eta) \langle \eta \rangle^{-2q}e^{-\frac{1}{2}\sum\limits_{j,s=1}^{n}(M_\eta)_{js}(z_j-w_j)(\overline{z}_s-\overline{w}_s)}\,d\mu(w)d\mu(\eta) \\
  &= \int_I\int_{\mathbb{C}^n} \langle \eta \rangle^{-2q}\, \lambda_1(\eta)\cdots\lambda_n(\eta) e^{-\frac{1}{2}\sum_{j=1}^{n}\lambda_j(\eta)|w_j(\eta)|^2}d\mu(w(\eta))d\mu(\eta) \\
  &= \int_I \langle \eta \rangle^{-2q}\,\lambda_1(\eta)\cdots\lambda_n(\eta) \int_{\mathbb{C}^n}e^{-\frac{1}{2}\sum_{j=1}^{n}\lambda_j(\eta)|w_j(\eta)|^2}d\mu(w(\eta))d\mu(\eta) \\
  &= C\int_I \langle \eta \rangle^{-2q}\,\lambda_1(\eta)\cdots\lambda_n(\eta)\frac{(2\pi)^n}{\lambda_1(\eta)\cdots\lambda_n(\eta)}d\mu(\eta)\\
  &= C\int_I \langle \eta \rangle^{-2q}\,d\mu(\eta)\\
  &\leq C\int_{\mathbb{R}^m} \langle \eta \rangle^{-2q}\,d\mu(\eta)\\
  &\leq C
\end{align*}
Thus, (\ref{kercon}) is integrable on $\mathbb{R}^m$ and is bounded by a constant independent of $\epsilon.$ By the dominated convergence theorem, $\tilde{P}u$ converges. 
\end{proof}

By similar arguments, we can check that $\tilde{P}(u)\in C^{\infty}(\mathbb{H}^{n,m}).$ We also have $P(u)\in C^{\infty}(\mathbb{H}^{n,m}).$ For all $\epsilon > 0$ and $u\in C^{\infty}_c(\mathbb{H}^{n,m}),$
Put
\begin{align*}
&\tilde{P}_\epsilon u (x)\\
&=\Big(\frac{1}{2\pi}\Big)^{n+m}\int_{I_{M_1,\dots,M_m}^{\Phi_0}}\int_{\mathbb{H}^{n,m}} \chi(\epsilon\eta) \det(M_\eta) e^{i\langle\eta,x''\rangle-i\langle\eta,y''\rangle}e^{\varphi(\eta,z,w)}u(w,y'')\,d\mu_{\mathbb{H}^{n,m}}(y)d\mu(\eta).
\end{align*}
We have
\begin{prop}\label{opbd}
For all $\epsilon > 0$ and $u\in \mathcal{C}^{\infty}_c(\mathbb{H}^{n,m}),$  
\[
\int_{\mathbb{H}^{n,m}}\big|\tilde{P}_\epsilon u (x)\big|^2d\mu(x) \leq \|u\|^2.
\]
\end{prop}
\begin{proof}
Let $I = I_{M_1,\dots,M_m}^{\Phi_0}.$ Set $\hat{u}(w,\eta) = \int_{\mathbb{R}^m}e^{-i\langle y'', \eta\rangle}u(w,y'')\,d\mu(y'').$ Then
\[
\int_{\mathbb{R}^m} |\hat{u}(w,\eta)|^2d\mu(\eta) = (2\pi)^m \int_{\mathbb{R}^m} |u(w,y'')|^2\,d\mu(y'').
\]
Let
\[
g_\epsilon(z,\eta) = \int_{\mathbb{C}^n} 1_I(\eta) \chi(\epsilon\eta)\hat{u}(w,\eta)e^{\varphi(\eta,z,w)}\,d\mu(w)
\]
where $1_I(\eta)$ is the characteristic function with respect to $I.$
Then we have 
\begin{align*}
&\int_I\int_{\mathbb{H}^{n,m}} \chi(\epsilon\eta) \det(M_\eta) e^{i\langle\eta,x''\rangle-i\langle\eta,y''\rangle}e^{\varphi(\eta,z,w)}u(w,y'')d\mu_{\mathbb{H}^{n,m}}(y)d\mu(\eta) \\
  &= \int_{\mathbb{R}^m} \det(M_\eta)g_\epsilon(z,\eta)e^{i\langle\eta,x''\rangle}d\mu(\eta).
\end{align*}
By Parseval's formula, we have
\[
\int_{\mathbb{R}^m}\bigg|\int_{\mathbb{R}^m} \det(M_\eta)g_\epsilon(z,\eta)e^{i\langle\eta,x''\rangle}d\mu(\eta)\bigg|^2d\mu(x'')
= (2\pi)^m \int_{\mathbb{R}^m}\Big|\det(M_\eta)g_\epsilon(z,\eta)\Big|^2d\mu(\eta).
\]
Hence,
\begin{align*}
&\int_{\mathbb{H}^{n,m}}\bigg|\int_I\int_{\mathbb{H}^{n,m}} \chi(\epsilon\eta) \det(M_\eta) e^{i\langle\eta,x''\rangle-i\langle\eta,y''\rangle}e^{\varphi(\eta,z,w)}u(w,y'')d\mu_{\mathbb{H}^{n,m}}(y)d\mu(\eta)\bigg|^2d\mu(x)\\
&=\int_{\mathbb{H}^{n,m}}\bigg|\int_{\mathbb{R}^m} \det(M_\eta)g_\epsilon(z,\eta)e^{i\langle\eta,x''\rangle}d\mu(\eta)\bigg|^2d\mu(x)\\
&=(2\pi)^m\int_{\mathbb{C}^n}\int_{\mathbb{R}^m}\Big|\det(M_\eta)g_\epsilon(z,\eta)\Big|^2d\mu(\eta)d\mu(z)\\
&=(2\pi)^m\int_{\mathbb{C}^n}\int_{\mathbb{R}^m}\Big|\det(M_\eta) \int_{\mathbb{C}^n}1_I(\eta)\chi(\epsilon\eta)\hat{u}(w,\eta)e^{\varphi(\eta,z,w)}d\mu(w)\Big|^2d\mu(\eta)d\mu(z)\\
&\leq(2\pi)^m\int_{\mathbb{C}^n}\int_{\mathbb{R}^m}|\det(M_\eta)|^2\chi^2(\epsilon\eta)\Big(\int_{\mathbb{C}^n}\big|\hat{u}(w,\eta)1_I(\eta)\big|
e^{-\frac{1}{2}\sum\limits_{j,s=1}^{n}(M_\eta)_{js}(z_j-w_j)(\overline{z}_s-\overline{w}_s)}d\mu(w)\Big)^2\\
&\quad\quad\quad\quad d\mu(\eta)d\mu(z)\\
&\leq(2\pi)^m\int_{\mathbb{C}^n}\int_{\mathbb{R}^m}|\det(M_\eta)|^2\chi^2(\epsilon\eta)1_I(\eta)\Big(\int_{\mathbb{C}^n}\big|\hat{u}(w,\eta)\big|^2
e^{-\frac{1}{2}\sum\limits_{j,s=1}^{n}(M_\eta)_{js}(z_j-w_j)(\overline{z}_s-\overline{w}_s)}d\mu(w)\Big)\\
&\quad\quad\quad\quad \cdot\Big(\int_{\mathbb{C}^n}e^{-\frac{1}{2}\sum\limits_{j,s=1}^{n}(M_\eta)_{js}(z_j-w_j)(\overline{z}_s-\overline{w}_s)}d\mu(w)\Big)d\mu(\eta)d\mu(z)\\
&=(2\pi)^m\int_{\mathbb{C}^n}\int_{\mathbb{R}^m}|\det(M_\eta)|^2\chi^2(\epsilon\eta)1_I(\eta)\int_{\mathbb{C}^n}\big|\hat{u}(w,\eta)\big|^2
e^{-\frac{1}{2}\sum\limits_{j,s=1}^{n}(M_\eta)_{js}(z_j-w_j)(\overline{z}_s-\overline{w}_s)}d\mu(w)\\
&\quad\quad\quad\quad\cdot\frac{(2\pi)^n}{\det(M_\eta)}d\mu(\eta)d\mu(z)\\
&=(2\pi)^{n+m}\int_{\mathbb{C}^n}\int_{\mathbb{R}^m}\int_{\mathbb{C}^n}\det(M_\eta)|\chi(\epsilon\eta)1_I(\eta)\hat{u}(w,\eta)|^2
e^{-\frac{1}{2}\sum\limits_{j,s=1}^{n}(M_\eta)_{js}(z_j-w_j)(\overline{z}_s-\overline{w}_s)}\\
&\quad\quad\quad\quad\quad d\mu(w)d\mu(\eta)d\mu(z)\\
&\leq (2\pi)^{n+m}\int_{\mathbb{C}^n}\int_{\mathbb{R}^m}\int_{\mathbb{C}^n}\det(M_\eta)|1_I(\eta)\hat{u}(w,\eta)|^2
e^{-\frac{1}{2}\sum\limits_{j,s=1}^{n}(M_\eta)_{js}(z_j-w_j)(\overline{z}_s-\overline{w}_s)}d\mu(w)d\mu(\eta)d\mu(z)\\
&=(2\pi)^{2n+m}\int_{\mathbb{C}^n}1_I(\eta)|\hat{u}(w,\eta)|^2d\mu(w)d\mu(\eta)\\
&\leq(2\pi)^{2n+m}\int_{\mathbb{C}^n}|\hat{u}(w,\eta)|^2d\mu(w)d\mu(\eta)\\
&=(2\pi)^{2n+2m}\int_{\mathbb{C}^n}\int_{\mathbb{R}^m}|u(w,y'')|^2d\mu(y'')d\mu(w)\\
&=(2\pi)^{2n+2m}\|u\|^2.
\end{align*}
\end{proof}

Thus, we can extend $\tilde{P}$ to $L^2(\mathbb{H}^{n,m}),$ also denoted by $\tilde{P},$ in the standard way and we can find that the operator norm of $\tilde{P}$ is less or equal to 1; in particular, $\tilde{P} : L^2(\mathbb{H}^{n,m}) \rightarrow L^2(\mathbb{H}^{n,m})$ is a bounded operator. Equivalently, we have a bounded operator $P : L^2(\mathbb{H}^{n,m}, \psi_0) \rightarrow L^2(\mathbb{H}^{n,m}, \psi_0)$ whose operator norm $\|P\|\leq 1.$  We will show that $P$ is the Szeg\H{o} projection with respect to $\overline{\partial}_b$ and $\tilde{P}$ is the Szeg\H{o} projection with respect to $\overline{\partial}_{b, \psi_0}.$

\begin{lem} \label{imp}
For every $u \in L^2(\mathbb{H}^{n,m}, \psi_0)$, we have $P u \in \mathcal{H}^0_b(\mathbb{H}^{n,m})$.
\end{lem}
\begin{proof}
Fix $\epsilon > 0$. We first show that for all $u \in \mathcal{C}^\infty_c(\mathbb{H}^{n,m})$, we have
\begin{align*}
&P_\epsilon(u)\\ 
&:= \int_{I^{\Phi_0}_{M_1, \dots, M_m}} \int_{\mathbb{H}^{n,m}} \exp\Bigg[i\langle\eta,x''\rangle-i\langle\eta,y''\rangle+\frac{1}{2}\sum_{t=1}^m\beta_t x_{2n+t}-\frac{1}{2}\sum_{t=1}^m\beta_t y_{2n+t}+\sum_{j,s=1}^{n}(M_\eta)_{js}z_j\overline{w}_s\\
&\quad\quad-\sum_{j,s=1}^{n}(M_\eta)_{js}w_j\overline{w}_s
+\sum_{\ell=1}^m\Big(\frac{i\beta_\ell}{2}-\eta_\ell\Big)\sum_{j,s=1}^{n}\lambda^{\ell}_{js}w_j\overline{w}_s -\sum_{\ell=1}^m\Big(\frac{i\beta_\ell}{2}
  -\eta_\ell\Big)\sum_{j,s=1}^{n}\lambda^{\ell}_{js}z_j\overline{z}_s\Bigg]\\
&\quad\quad\cdot \chi(\epsilon\eta)\det(M_\eta) u(w,y')d\mu(y')d\mu(w)d\mu(\eta).
\end{align*}
is in $\mathcal{H}^0_b(\mathbb{H}^{n,m}).$

Let
\begin{align*}
E(x,y,\eta) &= i\langle\eta,x''\rangle-i\langle\eta,y''\rangle+\frac{1}{2}\sum_{t=1}^m\beta_t x_{2n+t}-\frac{1}{2}\sum_{t=1}^m\beta_t y_{2n+t}+\sum_{j,s=1}^{n}(M_\eta)_{js}z_j\overline{w}_s\\
&\quad-\sum_{j,s=1}^{n}(M_\eta)_{js}w_j\overline{w}_s +\sum_{\ell=1}^m\Big(\frac{i\beta_\ell}{2}-\eta_\ell\Big)\sum_{j,s=1}^{n}\lambda^{\ell}_{js}w_j\overline{w}_s -\sum_{\ell=1}^m\Big(\frac{i\beta_\ell}{2}-\eta_\ell\Big)\sum_{j,s=1}^{n}\lambda^{\ell}_{js}z_j\overline{z}_s.
\end{align*}

Let $I := I^{\Phi_0}_{M_1, \dots, M_m}.$ Let $u \in \mathcal{C}^\infty_c(\mathbb{H}^{n,m})$. For every $p=1, \dots, n$, we have
\begin{align*}
&\overline{U}_p \bigg( \int_I \int_{\mathbb{H}^{n,m}} e^{E(x,y,\eta)} \, \chi(\varepsilon \eta) \det(M_\eta) u(y) d\mu(y) d\mu(\eta) \bigg) \\
&= \int_I \bigg( \frac{\partial}{\partial \overline{z}_p} + i \sum_{\ell=1}^m \Big(\sum_{s=1}^n \lambda_{sj}^\ell z_s\Big) \frac{\partial}{\partial x_{2n+l}} \bigg) E(x,y,\eta) \int_{\mathbb{H}^{n,m}} e^{E(x,y,\eta)} \chi(\varepsilon \eta) \det(M_\eta) u(y) d\mu(y) d\mu(\eta) \\
&= 0,
\end{align*}
since $\overline{U}_p E(x,y,\eta) = 0$ for all $\eta \in \mathbb{R}^m$.

Thus, we have $P_\epsilon(u) \in \mathcal{H}^0(\mathbb{H}^{n,m})$ for all $u \in \mathcal{C}^\infty_c(\mathbb{H}^{n,m})$. Since $\lim_{\epsilon \to 0^+} P_\epsilon(u) = P(u)$ in $\mathcal{C}^\infty(\mathbb{H}^{n,m})$ topology,
we deduce that $P(u) \in \mathcal{H}^0(\mathbb{H}^{n,m})$ for all $u \in \mathcal{C}^\infty_c(\mathbb{H}^{n,m})$.

For $u \in L^2(\mathbb{H}^{n,m}, \Psi_0)$, take $u_j \in \mathcal{C}^\infty_c(\mathbb{H}^{n,m})$, $j=1,2,\dots$, $u_j \to u$ in $L^2(\mathbb{H}^{n,m}, \psi_0)$ as $j \to \infty$.
Then $P(u_j) \to P(u)$ in $L^2(\mathbb{H}^{n,m}, \psi_0)$ as $j \to \infty$ and hence $\overline{\partial}_{b, \mathbb{H}^{n,m}} (P(u_j)) \to \overline{\partial}_{b, \mathbb{H}^{n,m}} (P(u))$ in distribution
sense as $j \to \infty$. Thus, $P(u) \in \mathcal{H}_b^0(\mathbb{H}^{n,m})$.
\end{proof}

We can now prove that 

\begin{lem}\label{pu=u}
Let $u\in \mathcal{H}_b^0(\mathbb{H}^{n,m}).$ Then $Pu=u.$
\end{lem}

\begin{proof}
Fix $g\in C^{\infty}_c(\mathbb{H}^{n,m}).$ We only need to prove that
\begin{equation*}
    (Pu\mid g)_{\psi_0} = (u\mid g)_{\psi_0}.
\end{equation*}

By Fubini's theorem, we have
\begin{align*}
&(2\pi)^{n+m} (Pu\mid g)_{\psi_0}\\
&= \int_{\mathbb{H}^{n,m}}\bigg(\int_{I}\int_{\mathbb{H}^{n,m}}  \exp\Big[i\langle\eta,x''\rangle-i\langle\eta,y''\rangle+\frac{1}{2}\sum_{t=1}^m\beta_t x_{2n+t}-\frac{1}{2}\sum_{t=1}^m\beta_t y_{2n+t}+\sum_{j,s=1}^{n}(M_\eta)_{js}z_j\overline{w}_s\\
&\quad-\sum_{j,s=1}^{n}(M_\eta)_{js}w_j\overline{w}_s+\sum_{\ell=1}^m\Big(\frac{i\beta_\ell}{2}-\eta_\ell\Big)\sum_{j,s=1}^{n}\lambda^{\ell}_{js}w_j\overline{w}_s-\sum_{\ell=1}^m\Big(\frac{i\beta_\ell}{2}-\eta_\ell\Big)\sum_{j,s=1}^{n}\lambda^{\ell}_{js}z_j\overline{z}_s\Big]\\
&\quad\cdot\det(M_\eta)u(w,y')d\mu(y')d\mu(w)d\mu(\eta)\bigg)\overline{g(z,x'')}\,e^{-\sum\limits_t\beta_t x_{2n+t}-\sum\limits_{j,s}\mu_{js}z_j\overline{z}_s}\,d\mu(x)\\
&=\int_I\int_{\mathbb{C}^n_z}\int_{\mathbb{C}^n_w} e^{\sum\limits_{j,s=1}^{n}(M_\eta)_{js}z_j\overline{w}_s-\sum\limits_{j,s=1}^{n}(M_\eta)_{js}w_j\overline{w}_s}\,\det(M_\eta)\\
&\quad\cdot\bigg(\int_{\mathbb{R}^m}e^{-i\langle\eta,y''\rangle}\,u(w,y'')e^{-\frac{1}{2}\sum\limits_{t=1}^m\beta_t y_{2n+t}}d\mu(y'')e^{\sum\limits_{\ell=1}^m(\frac{i\beta_\ell}{2}-\eta_\ell)\sum\limits_{j,s=1}^{n}\lambda^{\ell}_{js}w_j\overline{w}_s}\bigg)\,e^{-\sum\limits_{\ell=1}^m(\frac{i\beta_\ell}{2}-\eta_\ell)\sum\limits_{j,s=1}^{n}\lambda^{\ell}_{js}z_j\overline{z}_s}\\
&\quad\cdot \bigg(\overline{\int_{\mathbb{R^m}}e^{-i\langle\eta,x''\rangle}\,g(z,x'')d\mu(x'')e^{-\frac{1}{2}\sum\limits_{t=1}^m\beta_t x_{2n+t}}}\bigg)e^{-\sum\limits_{j,s}\mu_{js}z_j\overline{z}_s}\,d\mu(w)d\mu(z)d\mu(\eta).
\end{align*}

Let
\begin{equation*}
    \alpha (w,\eta) = \int_{\mathbb{R}^m}e^{-i\langle\eta,y''\rangle}\,u(w,y'')e^{-\frac{1}{2}\sum\limits_{t=1}^m\beta_t y_{2n+t}}d\mu(y'')e^{\sum\limits_{\ell=1}^m(\frac{i\beta_\ell}{2}-\eta_\ell)\sum\limits_{j,s=1}^{n}\lambda^{\ell}_{js}w_j\overline{w}_s}.
\end{equation*}

By Proposition \ref{alphaholo} and the discussion for (\ref{berg}), we have
\begin{equation*}
    \alpha(z,\eta) = \Big(\frac{1}{2\pi}\Big)^n\int_{\mathbb{C}^n_w} \alpha(w,\eta)\det(M_\eta)\, e^{\sum\limits_{j,s=1}^{n}(M_\eta)_{js}z_j\overline{w}_s-\sum\limits_{j,s=1}^{n}(M_\eta)_{js}w_j\overline{w}_s}\,d\mu(w).
\end{equation*}

Hence, by Corollary \ref{vanish} and Parseval's formula,
\begin{align*}
&(2\pi)^{n+m} (Pu\mid g)_{\psi_0}\\
&=(2\pi)^n\int_I\int_{\mathbb{C}^n_z} \bigg(\alpha(z,\eta)\, e^{-\sum\limits_{\ell=1}^m(\frac{i\beta_\ell}{2}-\eta_\ell)\sum\limits_{j,s=1}^{n}\lambda^{\ell}_{js}z_j\overline{z}_s}\bigg)\\
&\quad\quad\cdot \bigg(\overline{\int_{\mathbb{R^m}}e^{-i\langle\eta,x''\rangle}\,g(z,x'')d\mu(x'')e^{-\frac{1}{2}\sum\limits_{t=1}^m\beta_t x_{2n+t}}}\bigg)e^{-\sum\limits_{j,s}\mu_{js}z_j\overline{z}_s}\,d\mu(z)d\mu(\eta)\\
&=(2\pi)^n\int_I\int_{\mathbb{C}^n_z} \mathcal{F}_{x''\rightarrow\eta}\Big[u(z,x'')e^{-\frac{1}{2}\sum\limits_{t=1}^m\beta_t x_{2n+t}}\Big] (z,\eta)\cdot\overline{\mathcal{F}_{x''\rightarrow\eta}\Big[g(z,x'')e^{-\frac{1}{2}\sum\limits_{t=1}^m\beta_t x_{2n+t}}\Big](z,\eta)}\\
&\quad\quad \cdot e^{-\sum\limits_{j,s}\mu_{js}z_j\overline{z}_s}\,d\mu(z)d\mu(\eta)\\
&= (2\pi)^{n+m} \int_{H^{n+m}} u(x)\overline{g(x)}\,d\mu(x)\\
&= (2\pi)^{n+m} (u\mid g).
\end{align*}
\end{proof}

\begin{thm}
For all $u\in L^2(\mathbb{H}^{n,m},\psi_0),$ we have $\Pi^{(0),\psi_0}_b u = Pu.$
\end{thm}

\begin{proof}
By Lemma \ref{imp}, we have $P(L^2(\mathbb{H}^{n,m},\psi_0)) \subset \mathcal{H}^0_b(\mathbb{H}^{n,m}).$ By Lemma \ref{pu=u}, we have $Pu = u$ for all $u\in \mathcal{H}^0_b(\mathbb{H}^{n,m}).$ Futhermore, by Proposition \ref{opbd} and the remark after Proposition \ref{opbd}, we have $\|P\|\leq 1.$ Thus, $P$ is an orthogonal projection onto $\mathcal{H}^0_b(\mathbb{H}^{n,m})$ and hence $\Pi^{(0),\psi_0}_b = P.$
\end{proof}
\begin{rem}
Equivalently, we have $\tilde{P} = \Pi^{(0)}_{b,\psi_0}.$
\end{rem}
In summery, let $I=I_{M_1,\dots,M_m}^{\Phi_0},$ we have
for $u\in L^2(\mathbb{H}^{n,m}, \psi_0),$  formally,
\begin{align*}
  &\big(\Pi^{(0),\psi_0}_b u\big)(z,x'')\\ 
  &=\Big(\frac{1}{2\pi}\Big)^{n+m} \int_{I}\int_{\mathbb{C}^n}\int_{\mathbb{R}^m} \det(M_\eta) \exp\bigg[i\langle\eta,x''\rangle-i\langle\eta,y''\rangle +\frac{1}{2}\sum_{t=1}^m\beta_t x_{2n+t}-\frac{1}{2}\sum_{t=1}^m\beta_t y_{2n+t}\nonumber\\
&\quad\quad+\sum_{j,s=1}^{n}(M_\eta)_{js}z_j\overline{w}_s-\sum_{j,s=1}^{n}(M_\eta)_{js}w_j\overline{w}_s+\sum_{\ell=1}^m\Big(\frac{i\beta_\ell}{2}-\eta_\ell\Big)\sum_{j,s=1}^{n}\lambda^{\ell}_{js}w_j\overline{w}_s\nonumber\\
&\quad\quad-\sum_{\ell=1}^m\Big(\frac{i\beta_\ell}{2}
  -\eta_\ell\Big)\sum_{j,s=1}^{n}\lambda^{\ell}_{js}z_j\overline{z}_s\bigg]\, u(w,y'')\,d\mu(y'')d\mu(w)d\mu(\eta).\nonumber
\end{align*}

and for $f\in L^2(\mathbb{H}^{n,m}),$ we have, formally,
\begin{align*}
  &(\Pi^{(0)}_{b,\psi_0} f)(z,x'') \\
  &= \Big(\frac{1}{2\pi}\Big)^{n+m} \int_{I}\int_{\mathbb{C}^n}\int_{\mathbb{R}^m} \det(M_\eta) \exp\bigg[i\langle\eta,x''\rangle-i\langle\eta,y''\rangle-\frac{1}{2}\sum_{j,s=1}^{n}(M_\eta)_{js}(z_j-w_j)(\overline{z}_s-\overline{w}_s)\nonumber\\
  &\quad\ +\frac{1}{2}\sum_{j,s=1}^{n}(M_\eta)_{js}(z_j\overline{w}_s-\overline{z}_s w_j)+\sum_{\ell=1}^m\frac{i\beta_\ell}{2}
  \sum_{j,s=1}^{n}\lambda^{\ell}_{js}(w_j\overline{w}_s-z_j\overline{z}_s)\bigg]f(w,y'')d\mu(y'')d\mu(w)d\mu(\eta).\nonumber
\end{align*}

\section{Asymptotics of Szeg\H{o} kernel for functions}\label{scaledszegokernelasy}

Let $(X^{2n+m},T^{1,0}X)$ be a $(2n+m)$-dimensional CR manifold of CR-codimension $m.$
Denote by $T^{*1,0}X$ and $T^{*0,1}X$ the dual bundles of $T^{1,0}X$ and $T^{0,1}X$ respectively. Define
the vector bundle of $(0, q)$ forms by $T^{*0,q}X = \Lambda^q(T^{*0,1}X)$. Let $D \subset X$ be an open set.
Let $\Omega^{0,q}(D)$ denote the space of smooth sections of $T^{*0,q}X$ over $D$ and let $\Omega^{0,q}_c(D)$ be the
subspace of $\Omega^{0,q}(D)$ whose elements have compact support in $D$. Similarly, if $E$ is a vector
bundle over $D$, then we let $\Omega^{0,q}(D, E)$ denote the space of smooth sections of $T^{*0,q}X \otimes E$
over $D$ and let $\Omega^{0,q}_c(D, E)$ be the subspace of $\Omega^{0,q}(D, E)$ whose elements have compact
support in $D$.
Let
\[
\overline{\partial}_b : \Omega^{0,q}(X) \rightarrow \Omega^{0,q+1}(X)
\]
be the tangential Cauchy-Riemann operator. Let $(L, h^{L})$ be a CR Hermitian line bundle on $X.$ Let $s$ be a local trivializing section of $L$ on an open subset $D\subset X.$ Then the local weight of $h^L$ with respect to $s$ is the function $\phi\in C^{\infty}(D;\mathbb{R})$ for which $|s(x)|^2 = e^{-\phi(x)}, x\in D.$ Let
\[
\overline{\partial}_{b,k} : \Omega^{0,q}(X,L^k) \rightarrow \Omega^{0,q+1}(X,L^k)
\]
be the tangential Cauchy-Riemann operator acting on forms with values in $L^k.$  We equip $X$ with a Hermitian metric, denoted by $\langle\,\cdot\,|\,\cdot\,\rangle,$ such that $T^{1,0}X \perp T^{0,1}X.$ Let $(\,\cdot\,|\,\cdot\,)$ and $(\,\cdot\,|\,\cdot\,)_k$ denote the global $L^2$ inner product on $\Omega^{0,q}(X)$ and $\Omega^{0,q}(X,L^k),$ respectively, induced from $\langle\,\cdot\,|\,\cdot\,\rangle$ and $h^L.$
Let 
\[
\overline{\partial}_{b,k}^* : \Omega^{0,q}(X,L^k) \rightarrow \Omega^{0,q-1}(X,L^k)
\]
be the formal adjoint of $\overline{\partial}_{b,k}$ with respect to $(\,\cdot\,|\,\cdot\,)_k.$
Let 
\[
\square^{(q)}_{b,k} = \overline{\partial}_{b,k}^*\overline{\partial}_{b,k} + \overline{\partial}_{b,k}\overline{\partial}_{b,k}^* : \Omega^{0,q}(X,L^k) \rightarrow \Omega^{0,q}(X,L^k)
\]
be the Kohn-Laplacian operator with values in $L^k.$
Let
\[
\Pi^{(q)}_{b,k} : L^2_{0,q}(X, L^k) \rightarrow \mathcal{H}^q_b(X, L^k) := \operatorname{Ker}(\square^{(q)}_{b,k})
\]
be the Szeg\H{o} projection with respect to $\square^{(q)}_{b,k}.$

\subsection{Local charts}\label{localchart}
We have $\mathbb{C}TX = T^{1,0}X\oplus T^{0,1}X\oplus N(X),$ where $T^{1,0}X$ is the CR structure on $X,$ $T^{0,1}X = \overline{T^{1,0}X}$ and $N(X)$ is the complexified totally real tangent bundle (orthogonal complement of $T^{1,0}X\oplus T^{0,1}X$) of $X.$ Let $N^*(X)$ be the dual bundle of $N(X).$ Let $\tau$ be a section of $N^*(X).$ Then $\tau$ annihilates $T^{1,0}X\oplus T^{0,1}X.$ The Levi form $\mathcal{L}_{p}$ of $X$ at a point $p$ is the Hermitian form on $T^{1,0}X$ with values in $N(X)$ defined by 
\begin{equation*}
    \mathcal{L}_p (U,\overline{V}) = \frac{1}{2i} \pi_p \big([U,\overline{V}]\big), \quad U,V\in T^{1,0}X,
\end{equation*}
where $[\cdot,\cdot]$ denote the Lie bracket and $\pi_p$ denote the orthogonal projection of $\mathbb{C}TX$ onto $N_p(X).$ The Levi form $\mathcal{L}_{\tau,p}$ in the direction $\tau$ is defined as the quadratic form
\begin{equation*}
    \mathcal{L}_{\tau,p} (U,\overline{V}) = \langle\mathcal{L}_p(U,\overline{V}), \tau\rangle_p, \quad U,V\in T^{1,0}X.
\end{equation*}

Given $x_0\in X,$ let $U_1(x),\dots,U_n(x)$ be an orthonormal frame of $T^{1,0}X$ over a neighborhood $D$ of $x_0.$ Let $Y_1(x),\dots,Y_m(x)$  be an orthonormal frame of $N(X)$ over $D.$ Identify $Y_1(x),\dots,Y_m(x)$ with an orthonormal frame of $N^*(X)$ via the Hermitian metric $\langle\,\cdot\,|\,\cdot\,\rangle.$ For all $j,s = 1,\ldots,n$ and $\ell = 1,\ldots,m,$ let
\[
\lambda^{\ell}_{js} = \frac{1}{2i}\langle[U_j,\overline{U}_s] \mid Y_\ell\rangle(x_0)
\]
which is the Levi form at $x_0$ in the direction $Y_\ell.$
Then $(\lambda^{\ell}_{js})$ are Hermitian matrices. We can find a local chart $(x_1,\dots,x_{2n},x_{2n+1},\dots,x_{2n+m}) =(x',x'') = (z,x'') = (z_1,\dots,z_n,x''), z_j = x_{2j-1}+ix_{2j}, j=1,\dots,n, x' = (x_1,\dots, x_{2n}), x''=(x_{2n+1},\dots,x_{2n+m}),$ for $D$ such that, for $1\leq j,s\leq 2n$ and $1\le \ell,q\le m,$
\[
(z(x_0),x''(x_0)) = 0,\quad \Big\langle\frac{\partial}{\partial x_j}(x_0)\mid\frac{\partial}{\partial x_s}(x_0)\Big\rangle = 2\delta_{js}, \quad \Big\langle\frac{\partial}{\partial x_{2n+\ell}}(x_0)\mid\frac{\partial}{\partial x_{2n+q}}(x_0)\Big\rangle = \delta_{\ell q}
\]
\begin{align*}
&U_j = \frac{\partial}{\partial z_j} - i\sum_{\ell=1}^m\Big(\sum_{s=1}^{n}\lambda^{\ell}_{js}\overline{z}_s\Big)\frac{\partial}{\partial x_{2n+\ell}}  + O(|(z,x'')|^2), \quad j=1,\dots,n,\\
&Y_\ell(x_0) = \frac{\partial}{\partial x_{2n+\ell}}(x_0), \quad \ell=1,\dots,m,
\end{align*}
and
\begin{align*}
   \phi(x) &= \sum_{j=1}^{n}(\alpha_jz_j+\overline{\alpha}_j\overline{z}_j) + \sum_{\ell=1}^m\beta_\ell x_{2n+\ell} + \sum_{j,s=1}^{n}(a_{js}z_jz_s + \overline{a}_{js}\overline{z}_j\overline{z}_s) + \sum_{j,s=1}^{n}\mu_{js}z_j\overline{z}_s \\
  &\quad + O(|z||x''|) + O(|x''|^2) + O(|(z,x'')|^3)
\end{align*}
where $\beta_\ell\in\mathbb{R}, \alpha_j, a_{js},\mu_{js}\in\mathbb{C}.$ We identify $D$ with some open subset in $\mathbb{R}^{2n+m}.$ Let
\begin{align*}
\phi_0(x) &= \sum_{\ell=1}^m\beta_\ell x_{2n+\ell} + \sum_{j,s=1}^{n}\mu_{js}z_j\overline{z}_s + O(|z||x''|) + O(|x''|^2) + O(|(z,x'')|^3),\\
R(z,x'') &=\sum_{j=1}^{n}(\alpha_jz_j) + \sum_{j,s=1}^{n}(a_{js}z_jz_s),\\
\psi_0(z,x'') &= \beta_1x_{2n+1} + \cdots + \beta_mx_{2n+m} + \sum_{j,s=1}^{n}\mu_{js}z_j\overline{z}_s.
\end{align*}
Let $L^2_{(0,q)}(D,k\phi)$ and $L^2_{(0,q)}(D,k\phi_0)$ be the completions of $\Omega^{0,q}_c(D)$ with respect to $(\,\cdot\,|\,\cdot\,)_{k\phi}$ and $(\,\cdot\,|\,\cdot\,)_{k\phi_0},$ respectively. We have the unitary identification.
\begin{align*}
  L^2_{(0,q)}(D,k\phi) &\leftrightarrow L^2_{(0,q)}(D,k\phi_0) \\
  \tilde{u} &\mapsto u = e^{-kR}\tilde{u}
\end{align*}
Let 
\[
\overline{\partial}_b^{*,k\phi} : \Omega^{0,q+1}(D) \rightarrow \Omega^{0,q}(D)
\]
be the formal adjoint of $\overline{\partial}_b$ with respect to $(\,\cdot\,|\,\cdot\,)_{k\phi}.$
Put 
\[
\square^{(q)}_{b,k\phi} = \overline{\partial}_b\overline{\partial}_b^{*,k\phi} + \overline{\partial}_b^{*,k\phi}\overline{\partial}_b : \Omega^{0,q}(D) \rightarrow \Omega^{0,q}(D).
\]
For $u\in \Omega^{0,q}(D, L^k),$ write $u = \hat{u}s^k$ where $\hat{u}\in \Omega^{0,q}(D).$ Then $\square^{(q)}_{b,k}u = (\square^{(q)}_{b,k\phi}\hat{u})s^k.$ Let
\[
\overline{\partial}_{b,k,s} = \overline{\partial}_b + k(\overline{\partial}_b R)\wedge.
\]
Then $\overline{\partial}_b\tilde{u} = \widetilde{\overline{\partial}_{b,k,s}u} = e^{kR}\overline{\partial}_{b,k,s}u$ for all $u\in \Omega^{0,q}(D)\cap L^2_{(0,q)}(D, k\phi_0).$ Let $\{e_j(z,x'')\}_{j=1}^n$ be the basis of $T^{*0,1}_{(z,x'')}X$ dual to $\big\{\overline{U}_j(z,x'')\big\}_{j=1}^n.$ Then
\begin{align*}
  \overline{\partial}_b &= \sum_{j=1}^{n} \big(e_j\wedge\overline{U}_j + \overline{\partial}_b e_j\wedge(e_j\wedge)^*\big), \\
  \overline{\partial}_{b,k,s} &= \sum_{j=1}^{n} e_j\wedge\big(\overline{U}_j + k(\overline{U}_j R)\big) + \sum_{j=1}^{n}\overline{\partial}_b e_j\wedge(e_j\wedge)^*,
\end{align*}
where $(e_j\wedge)^*$ is the formal adjoint of $e_j\wedge$ with respect to $\langle\cdot,\cdot\rangle.$ And
\begin{align*}
  \overline{\partial}_b^{*,k\phi} &= \sum_{j=1}^{n}(e_j\wedge)^*\overline{U}_j^{*,k\phi} + \sum_{j=1}^{n}e_j\wedge (\overline{\partial}_b e_j\wedge)^*, \\
  \overline{\partial}_{b,k,s}^{*,\phi_0} &= \sum_{j=1}^{n}(e_j\wedge)^*\big(\overline{U}_j^{*,k\phi_0} + k(U_j\overline{R})\big) + \sum_{j=1}^{n}e_j\wedge (\overline{\partial}_b e_j\wedge)^*,
\end{align*}
where $\overline{U}_j^{*,k\phi_0}$ is the formal adjoint of $\overline{U}_j$ with respect to $( \ \mid\ )_{k\phi_0}$.
Also, we have $\overline{\partial}_b^{*,k\phi}\tilde{u} = e^{kR}\overline{\partial}_{b,k,s}^{*,\phi_0}u.$ By calculation, we have
\[
\overline{U}_j^{*,k\phi_0} = -U_j + kU_j\phi_0 + s_j(z,x''),
\]
where $s_j\in \mathcal{C}^\infty(D),$ $s_j$ is independent of $k.$ Put
\[
\square^{(q)}_{b,k,s} = \overline{\partial}_{b,k,s}\overline{\partial}_{b,k,s}^{*,\phi_0} + \overline{\partial}_{b,k,s}^{*,\phi_0}\overline{\partial}_{b,k,s} : \Omega^{0,q}(D) \rightarrow \Omega^{0,q}(D).
\]
We have $\widetilde{\square^{(q)}_{b,k,s}u} = e^{kR}\square^{(q)}_{b,k,s}u = \square^{(q)}_{b,k\phi}\tilde{u}.$ 

\begin{prop}
\begin{align*}
\square^{(q)}_{b,k,s} &= \sum_{j=1}^{n}\big(\overline{U}_j^{*,k\phi_0} + k(U_j\overline{R})\big)\big(\overline{U}_j + k(\overline{U}_j R)\big)\\
&\quad + \sum_{j,t=1}^n (e_j\wedge)(e_j\wedge)^*\Big[\overline{U}_j + k(\overline{U}_j R) , \overline{U}_t^{*,k\phi_0} + k(U_t\overline{R})\Big]\\
&\quad +\varepsilon\big(\overline{U} + k(\overline{U} R)\big) + \varepsilon\big(\overline{U}^{*,k\phi_0} + k(U\overline{R})\big) + f(z,x''),
\end{align*}
where $\varepsilon\big(\overline{U} + k(\overline{U} R)\big)$ denotes remainder terms of the form $\sum a_j(z,x'')\big(\overline{U} + k(\overline{U} R)\big)$ with $a_j$ smooth, matrix-valued and independent of $k,$ for all $j,$ and similarly for $\varepsilon\big(\overline{U}^{*,k\phi_0} + k(U\overline{R})\big)$ and $f(z,x'')\in \mathcal{C}^\infty$ independent of $k.$ 
\end{prop}

\subsection{Scaled operators}
Let $D(r,r'') = \big\{(z,x'') = (x_1,\dots,x_{2n},x_{2n+1},\dots,x_{2n+m})\in\mathbb{C}^n\times\mathbb{R}^m;\, |x_j|< r, |x_{2n+\ell}|<r'', j=1,\dots,2n, \ell=1,\dots,m\big\}.$ If $r=r'',$ we write $D(r) = D(r,r'').$ Let $F_k$ be the scaling map
\begin{equation}\label{scaledmap}
F_k(z,x'') = \Big(\frac{z}{\sqrt{k}},\frac{x''}{k}\Big).
\end{equation}
Throughout this paper, we fix $0<\epsilon<\frac{1}{8}.$ From now on, we assume $k$ is large enough such that $F_k(D(k^\epsilon,k^{2\epsilon}))\subset D.$ We define the scaled bundle $F^*_k\Lambda^{0,q}T^*X$ on $D(k^\epsilon,k^{2\epsilon})$ to be the bundle whose fiber at $(z,x'')\in D(k^\epsilon,k^{2\epsilon})$ is
\[
F^*_k\Lambda^{0,q}T^*_{(z,x'')}X := \Bigg\{\sum_{|J|=q}a_Je_J\Big(\frac{z}{\sqrt{k}},\frac{x''}{k}\Big)\,;\, a_J\in \mathbb{C},\, |J|=q,\, J\ \text{strictly increasing }\Bigg\}.
\]
We define an Hermitian metric $\langle\,\cdot\,|\,\cdot\,\rangle_{F_k^*}$ on $F^*_k\Lambda^{0,q}T^*X$ by requiring, at each point $(z,x'')\in D(k^\epsilon,k^{2\epsilon}),$
\[
\bigg\{e_J\Big(\frac{z}{\sqrt{k}},\frac{x''}{k}\Big)\,;\,  |J|=q,\, J\ \text{strictly increasing }\bigg\}
\]
to be an orthonormal basis for $F^*_k\Lambda^{0,q}T^*_{(z,x'')}X.$ Let $F_k^*\Omega^{0,q}(D(k^\epsilon,k^{2\epsilon}))$ denote the space of smooth sections of $F^*_k\Lambda^{0,q}T^*X$ over $D(k^\epsilon,k^{2\epsilon}).$ Let $F_k^*\Omega^{0,q}_c(D(k^\epsilon,k^{2\epsilon}))$ be the subspace of $F_k^*\Omega^{0,q}(D(k^\epsilon,k^{2\epsilon}))$ whose elements have compact support in $D(k^\epsilon,k^{2\epsilon}).$ For each $f\in \Omega^{0,q}(F_k(D(k^\epsilon,k^{2\epsilon}))),$ we write $f=\sum'_{|J|=q}f_Je_J.$ We define the scaled form\\ $F_k^*f\in F_k^*\Omega^{0,q}(D(k^\epsilon,k^{2\epsilon}))$ by
\[
F_k^*f := \sum_{|J|=q} f_J\Big(\frac{z}{\sqrt{k}},\frac{x''}{k}\Big)e_J\Big(\frac{z}{\sqrt{k}},\frac{x''}{k}\Big).
\]

Let $P$ be a partial differential operator of order one on $F_k(D(k^\epsilon,k^{2\epsilon}))$ with smooth coefficients. We write
\[
P = \sum_{\ell=1}^m a_\ell(z,x'')\frac{\partial}{\partial x_{2n+\ell}} + \sum_{j=1}^{2n} b_j(z,x'')\frac{\partial}{\partial x_j},
\]
where $a_\ell,b_j\in \mathcal{C}^\infty\big(F_k(D(k^\epsilon,k^{2\epsilon}))\big).$ We define the partial differential operator $P_{(k)}$ on $D(k^\epsilon,k^{2\epsilon})$ by
\begin{align*}
P_{(k)} &= \sum_{\ell=1}^m\sqrt{k}F_k^*a_\ell\frac{\partial}{\partial x_{2n+\ell}} + \sum_{j=1}^{2n} F_k^*b_j\frac{\partial}{\partial x_j}\\
        &= \sum_{\ell=1}^m \sqrt{k}a_\ell\Big(\frac{z}{\sqrt{k}},\frac{x''}{k}\Big)\frac{\partial}{\partial x_{2n+\ell}} + \sum_{j=1}^{2n} b_j\Big(\frac{z}{\sqrt{k}},\frac{x''}{k}\Big)\frac{\partial}{\partial x_j}.
\end{align*}
For $f\in \mathcal{C}^\infty\big(F_k(D(k^\epsilon,k^{2\epsilon}))\big),$ we have 
\begin{equation}\label{scaled}
P_{(k)}(F_k^*f) = \frac{1}{\sqrt{k}}F_k^*(Pf).
\end{equation}
The scaled differential operator
\[
\overline{\partial}_{b,(k),s} : F_k^*\Omega^{0,q}(D(k^\epsilon,k^{2\epsilon})) \rightarrow F_k^*\Omega^{0,q+1}(D(k^\epsilon,k^{2\epsilon}))
\]
is given by 
\begin{equation}\label{scaleddbar}
\overline{\partial}_{b,(k),s} = \sum_{j=1}^{n} e_j\Big(\frac{z}{\sqrt{k}},\frac{x''}{k}\Big)\wedge\big(\overline{U}_{j(k)} + \sqrt{k}F_k^*(\overline{U}_j R)\big) + \sum_{j=1}^{n}\frac{1}{\sqrt{k}}\overline{\partial}_b e_j\Big(\frac{z}{\sqrt{k}},\frac{x''}{k}\Big)\wedge\Bigg(e_j\Big(\frac{z}{\sqrt{k}},\frac{x''}{k}\Big)\wedge\Bigg)^*.
\end{equation}
From \eqref{scaled}, we have for $f\in \Omega^{0,q}(F_k(D(k^\epsilon,k^{2\epsilon}))),$
\[
\overline{\partial}_{b,(k),s}F_k^*f = \frac{1}{\sqrt{k}}F_k^*(\overline{\partial}_{b,k,s}f).
\]
We can check that
\begin{equation}\label{scaledbarU_j}
\overline{U}_{j(k)} + \sqrt{k}F_k^*(\overline{U}_j R) = \frac{\partial}{\partial\overline{z}_j} + i\sum_{\ell=1}^m\Big(\sum_{s=1}^{n}\lambda^{\ell}_{sj}z_s\Big)\frac{\partial}{\partial x_{2n+\ell}} + \epsilon_kZ_{j,k}
\end{equation}
on $D(k^\epsilon,k^{2\epsilon})$, where $\epsilon_k$ is a sequence tending to zero with
$k\to\infty$ and $Z_{j,k}$ is a first order differential operator and all the
derivatives of the coefficients of $Z_{j,k}$ are uniformly bounded in $k$ on
$D(k^\epsilon,k^{2\epsilon})$, $j=1,\ldots,n$. 

Let $(\,\cdot\,|\,\cdot\,)_{(k)}$ be the inner product on the space $F_k^*\Omega^{0,q}_c(D(k^\epsilon,k^{2\epsilon}))$ defined as follows:
\[
(f\mid g)_{(k)} := \int_{D(k^\epsilon,k^{2\epsilon})} \langle f \mid g\rangle_{F_k^*} e^{-kF_k^*\phi_0}(F_k^*V)(z,x'')d\mu_{\mathbb{H}^{n,m}}(x),
\]
where $dV_X(x) = V(z,x'')d\mu_{\mathbb{H}^{n,m}}(x)$ is the volume form. Note that $V(0,0) = 1.$

Let
\[
\overline{\partial}^*_{b,(k),s} : F_k^*\Omega^{0,q+1}(D(k^\epsilon,k^{2\epsilon})) \rightarrow F_k^*\Omega^{0,q}(D(k^\epsilon,k^{2\epsilon}))
\]
be the formal adjoint of $\overline{\partial}_{b,(k),s}$ with respect to $(\,\cdot\,|\,\cdot\,)_{(k)}.$ 
we have for $f\in \Omega^{0,q+1}(F_k(D(k^\epsilon,k^{2\epsilon}))),$
\[
\overline{\partial}^*_{b,(k),s}F_k^*f = \frac{1}{\sqrt{k}}F_k^*(\overline{\partial}^*_{b,k,s}f).
\]
We also have
\begin{align}\label{scaleddbar*}
\bar\partial_{b,(k),s}^*
={}&
\sum_{j=1}^{n}\bigg(e_j\Big(\frac{z}{\sqrt{k}},\frac{x''}{k}\Big)\wedge\bigg)^*\bigg(-U_{j(k)}+\sqrt{k}F_k^*(U_j\overline{R})+\sqrt{k}F_k^*(U_j\phi_0)
+\frac{1}{\sqrt{k}}F_k^*s_j\bigg)\\
&+\sum_{j=1}^{n}\frac{1}{\sqrt{k}}e_j\Big(\frac{z}{\sqrt{k}},\frac{x''}{k}\Big)\wedge\bigg((\overline\partial_b e_j)\Big(\frac{z}{\sqrt{k}},\frac{x''}{k}\Big)\wedge\bigg)^*,\nonumber
\end{align}
where $s_j\in C^\infty(D(k^\epsilon,k^{2\epsilon}))$, $j=1,\ldots,n$, are independent of $k.$ We can check that
\begin{align}\label{scaledU_t}
&-U_{t(k)} + \sqrt{k}F_k^*(U_t \overline{R}) + \sqrt{k}F_k^*(U_t\phi_0) + \frac{1}{\sqrt{k}}F_k^*s_t\\
&= -\frac{\partial}{\partial z_t} + i\sum_{\ell=1}^m\Big(\sum_{s=1}^{n}\lambda^{\ell}_{ts}\overline{z}_s\Big)\frac{\partial}{\partial x_{2n+\ell}} - i\sum_{\ell=1}^m\Big(\sum_{s=1}^{n}\lambda^{\ell}_{ts}\overline{z}_s\Big)\beta_\ell + \sum_{s=1}^{n}\mu_{ts}\overline{z}_s + \delta_kV_{t,k}
\end{align}
on $D(k^\epsilon,k^{2\epsilon})$, where $\delta_k$ is a sequence tending to zero with$k\to\infty$ and $V_{j,k}$ is a first order differential operator and all the
derivatives of the coefficients of $V_{j,k}$ are uniformly bounded in $k$ on
$D(k^\epsilon,k^{2\epsilon})$, $t=1,\ldots,n$. 

We define the scaled Kohn-Laplacian
\[
\square^{(q)}_{b,(k),s} := \overline{\partial}^*_{b,(k),s}\overline{\partial}_{b,(k),s}+\overline{\partial}_{b,(k),s}\overline{\partial}^*_{b,(k),s} :  F_k^*\Omega^{0,q}_c(D(k^\epsilon,k^{2\epsilon})) \rightarrow F_k^*\Omega^{0,q}_c(D(k^\epsilon,k^{2\epsilon})).
\]
From \eqref{scaleddbar}, \eqref{scaleddbar*}, \eqref{scaledbarU_j} and \eqref{scaledU_t}, we can obtain
\begin{prop}
\begin{align*}
  \square^{(q)}_{b,(k),s}  &= \sum_{j=1}^{n}\Bigg[\Big(-\frac{\partial}{\partial z_j} + i\sum_{\ell=1}^m\Big(\sum_{s=1}^{n}\lambda^{\ell}_{js}\overline{z}_s\Big)\frac{\partial}{\partial x_{2n+\ell}} - i\sum_{\ell=1}^m\Big(\sum_{s=1}^{n}\lambda^{\ell}_{js}\overline{z}_s\Big)\beta_\ell + \sum_{s=1}^{n}\mu_{js}\overline{z}_s\Big)\\
  &\quad\quad\quad\quad \Big(\frac{\partial}{\partial\overline{z}_j} + i\sum_{\ell=1}^m\Big(\sum_{s=1}^{n}\lambda^{\ell}_{sj}z_s\Big)\frac{\partial}{\partial x_{2n+\ell}}\Big)\Bigg] \\
  &\quad + \sum_{j,t=1}^{n}e_j\Big(\frac{z}{\sqrt{k}},\frac{x''}{k}\Big)\wedge \Big(e_j\Big(\frac{z}{\sqrt{k}},\frac{x''}{k}\Big)\wedge\Big)^*\Big(\mu_{t,j} - i\sum_{\ell=1}^m\lambda^\ell_{tj}\beta_\ell + 2i\sum_{\ell=1}^m\lambda^\ell_{tj}\frac{\partial}{\partial x_{2n+\ell}}\Big)\\
   &\quad +\varepsilon_kP_k
\end{align*}
on $D(k^\epsilon,k^{2\epsilon}),$ where $\varepsilon_k$ is a sequence tending to zero with $k\to\infty$ and $P_k$ is a first order differential operator and all the
derivatives of the coefficients of $P_k$ are uniformly bounded in $k$ on
$D(k^\epsilon,k^{2\epsilon}).$
\end{prop}

\subsection{Scaled Szeg\H{o} projection}
In this section, we prove Theorem \ref{mainthmsp}. We first define the scaled Szeg\H{o} kernel and describe the explicit meaning of the limit \eqref{locspgap}.  Fix $u\in L^2(\mathbb{H}^{n,m},\psi_0).$ Let $\chi\in \mathcal{C}^\infty_c(\mathbb{H}^{n,m})$ be a smooth function such that $\chi=1$ on $D(\frac{1}{2}), \chi\geq0$ and $\operatorname{supp}(\chi)\subset D(1).$ Fix a real number $\epsilon$ with $0<\epsilon<\frac{1}{8}.$ Let
\begin{align}\label{scaledu}
  u_{(k)} &= \chi\Big(\frac{z}{k^{\epsilon}},\frac{x''}{k^{2\epsilon}}\Big)u \\
  u_{[k]} &= k^{\frac{n+m}{2}}u_{(k)}(\sqrt{k}z,kx'')e^{kR}s^k\notag\\
          &= k^{\frac{n+m}{2}}\chi\Big(\frac{\sqrt{k}z}{k^\epsilon},\frac{kx''}{k^{2\epsilon}}\Big)u(\sqrt{k}z,kx'')e^{kR}s^k.\notag
\end{align}
We define the scaled Szeg\H{o} projection $\Pi_{(k),s}$ on $L^2(\mathbb{H}^{n,m},\psi_0)$ by
\begin{equation}\label{scaledPiu}
\Pi_{(k),s} u := k^{-\frac{n+m}{2}}F_k^*\Big(\widehat{\Pi_k u_{[k]}}\cdot e^{-kR}\Big)
\end{equation}
for all $u\in L^2(\mathbb{H}^{n,m},\psi_0).$ By some simple calculations, we have the following propositions.
\begin{prop}
For all $u, v\in L^2(\mathbb{H}^{n,m}, \psi_0),$ we have $|( u_{[k]}\mid v_{[k]})_k| = |(u_{(k)}\mid v_{(k)})_{(k)}|.$
\end{prop}

\begin{prop}
For all $u\in L^2(\mathbb{H}^{n,m}, \psi_0),$ $\|\Pi_{(k),s} u\|^2_{(k)} \leq  \|u_{(k)}\|^2_{(k)}.$
\end{prop}

\begin{prop}\label{defdis}
Fix $u\in L^2(\mathbb{H}^{n,m}, \psi_0).$ For any $v\in \mathcal{C}^\infty_c(\mathbb{H}^{n,m}),$ we have
\[
\big|\big(\Pi_{(k),s} u \mid v\big)_{(k)}\big| \leq \|u_{[k]}\|_k \|v_{[k]}\|_k \rightarrow \|u\|_{\psi_0}\|v\|_{\psi_0}.
\]
This implies that $\big(\Pi_{(k),s} u\mid v\big)_{(k)}$ is bounded and hence has a convergent subsequence.
\end{prop}
By the diagonal argument and the density argument, we can take a sequence $\Pi_{(k_j),s}$ such that for all $u\in L^2(\mathbb{H}^{n,m},\psi_0),$ $\hat{S}(u) := \lim_{j\rightarrow\infty} \big(\Pi_{(k_j),s}u\mid\bullet \big)_{(k_j)}$ defines a distribution on $\mathcal{C}^\infty_c(\mathbb{H}^{n,m}).$ Actually, by Proposition \ref{defdis}, $\hat{S}$ defines a bounded anti-linear functional on $L^2(\mathbb{H}^{n,m}, \psi_0).$ By Riesz representation theorem, we have $\hat{S}$ is a linear map
\[
\hat{S} : L^2(\mathbb{H}^{n,m}, \psi_0) \rightarrow L^2(\mathbb{H}^{n,m}, \psi_0).
\]  
Proposition \ref{bounded} further shows that $\hat{S}$ is a bounded linear operator. In Theorem \ref{skerasy}, we will show that, under a spectral gap condition, $\Pi_{(k_j),s}u$ converges to $\Pi^{(0),\psi_0}_b u$ in the sense of distributions, that is, $\hat{S} = \Pi^{(0),\psi_0}_b$. Since $\Pi_{(k_j),s}$ is an arbitrary subsequence, we have $\Pi_{(k),s}$ itself converges to the Szeg\H{o} kernel for the model case. So, to simplify the notation, we may assume that $\hat{S}(u) = \lim_{k\rightarrow\infty}\big(\Pi_{(k),s}u\mid\bullet \big)_{(k)}.$

\begin{prop}\label{bounded}
$\hat{S} : L^2(\mathbb{H}^{n,m}, \psi_0) \rightarrow L^2(\mathbb{H}^{n,m}, \psi_0)$ is a bounded linear operator with operator norm $\|\hat{S}\| \leq 1.$
\end{prop}
\begin{proof}
Note that, for all $f\in L^2(\mathbb{H}^{n,m}, \psi_0),$ 
\[
\|\hat{S}f\|_{\psi_0} = \sup \big\{ \big(\hat{S}f\mid v\big)_{\psi_0};\, v\in \mathcal{C}^\infty_c(\mathbb{H}^{n,m})\ \text{with}\ \|v\|_{\psi_0} = 1 \big\}.
\]
\begin{align*}
\big|\big(\hat{S}f\mid v\big)_{\psi_0}\big| &= \big|\big(\hat{S}f\big)(v)\big|\\
&= \lim_{k\rightarrow\infty} \big|\big(\Pi_{(k),s}f\mid v\big)_{(k)} \big|
\leq \limsup_{k\rightarrow\infty} \big\|\Pi_{(k),s}f\big\|_{(k)}\cdot\|v\|_{(k)}\\
&\leq \limsup_{k\rightarrow\infty} \big\|f_{(k)}\big\|_{(k)}\cdot\|v\|_{(k)} = \|f\|_{\psi_0}\cdot\|v\|_{\psi_0} = \|f\|_{\psi_0}
\end{align*}
Thus, $\big\|\hat{S}\big\|\leq1.$
\end{proof}
\begin{prop}\label{ker}
  $\square^{(0),\psi_0}_b\hat{S}(u) = 0$ for all $u\in L^2(\mathbb{H}^{n,m}, \psi_0)$ in the sense of distributions. That is, we have
  \[
  \hat{S} : L^2(\mathbb{H}^{n,m}, \psi_0) \rightarrow \mathcal{H}^0_b(\mathbb{H}^{n,m}).
  \]
\end{prop}
\begin{proof}
Take arbitrarily $g\in \mathcal{C}^\infty_c(\mathbb{H}^{n,m}).$ Let $g_{[k]} = k^{\frac{n+m}{2}}g(\sqrt{k}z,kx'')e^{kR}s^k.$ Then we have
\[
\Big|\big(\Pi_{(k),s} u\,\big|\, \square_{b,(k),s}g\big)_{(k)}\Big| = \frac{1}{k}\Big|\big(\Pi_k u_{[k]}\,\big|\, \square_{b,k}g_{[k]}\big)_k\Big| = 0.
\]
Hence
\[
\big|\big(\Pi_{(k),s} u\,\big|\, \square^{(0),\psi_0}_b g - \square_{b,(k),s}g\big)_{(k)}\,\big|\, \leq \big\|\Pi_{(k),s} u\big\|_{(k)} \big\| \square^{(0),\psi_0}_b g - \square_{b,(k),s}g\big\|_{(k)} \rightarrow 0.
\]
Thus,
\[
\big(\hat{S}(u)\,\big|\,\square^{(0),\psi_0}_b g\big)_{\psi_0} = \lim_{k\rightarrow\infty} \big(\Pi_{(k),s} u\,\big|\, \square^{(0),\psi_0}_b g\big)_{(k)} = 0.
\]
\end{proof}
We now prove Theorem \ref{mainthmsp}.
\begin{thm}\label{skerasy}
Suppose that for all $u\in \mathcal{C}^\infty_c(D, L^k),$ we have
\[
\big\|\big(I - \Pi_k\big)u\big\|_k^2 \leq (k^{1-2\epsilon})^{-1} \big(\square_{b,k}u\mid u\big)_k.
\]
Then $\hat{S} = \Pi^{(0),\psi_0}_b.$
\end{thm}
\begin{proof}
Fix $u\in \mathcal{C}^\infty_c(\mathbb{H}^{n,m}).$ Let $\ddot{u}= \Pi^{(0),\psi_0}_b u.$ Then $\ddot{u}_{(k)} =\chi\big(\frac{z}{k^{\epsilon}},\frac{x''}{ k^{2\epsilon}}\big)\big(\Pi^{(0),\psi_0}_b u\big)$ and $\ddot{u}_{[k]} = k^{\frac{n+m}{2}}\ddot{u}_{(k)}(\sqrt{k}z,kx'')e^{kR}s^k.$ On the coordinate neighborhood $D,$ we write 
\begin{equation*}
dV_X(z,x'') = V(z,x'')d\mu_{\mathbb{H}^{n,m}}(x).
\end{equation*}
Then we have $V\big(\frac{z}{\sqrt{k}},\frac{x''}{k}\big)\to 1$ uniformly on $D(k^\epsilon,k^{2\epsilon})$ as $k\to\infty.$ Recall that $d\mu_{\mathbb{H}^{n,m}} = 2^ndx_1\wedge \cdots\wedge dx_{2n+m}.$
First we estimate $\big(\square_{b,k} \ddot{u}_{[k]}\mid \ddot{u}_{[k]}\big)_k.$
\begin{align*}
&\big(\square_{b,k} \ddot{u}_{[k]}\mid \ddot{u}_{[k]}\big)_k = \big\|\overline{\partial}_b \ddot{u}_{[k]}\big\|_k^2\\
&=\int_{D(k^\epsilon,k^{2\epsilon})} k\big|(\overline{\partial}_{b,(k),s}\ddot{u}_{(k)})(z,x'')\big|^2e^{-k\phi_0(\frac{z}{\sqrt{k}},\frac{x''}{k})}V\Big(\frac{z}{\sqrt{k}},\frac{x''}{k}\Big)d\mu_{\mathbb{H}^{n,m}}(x)\\
&= \int_{D(k^\epsilon,k^{2\epsilon})} k \Bigg|\Bigg(\sum_{j=1}^{n} e_j\Big(\frac{z}{\sqrt{k}},\frac{x''}{k}\Big)\wedge\big(\overline{U}_{j(k)} + \sqrt{k}F_k^*(\overline{U}_j R)\big)\Bigg)\Big(\chi\Big(\frac{z}{k^{\epsilon}},\frac{x''}{k^{2\epsilon}}\Big)\Pi^{(0),\psi_0}_b u\Big)\Bigg|^2\\
& \quad\cdot e^{-k\phi_0(\frac{z}{\sqrt{k}},\frac{x''}{k})}V\Big(\frac{z}{\sqrt{k}},\frac{x''}{k}\Big)d\mu_{\mathbb{H}^{n,m}}(x).
\end{align*}
It is enough to estimate
\begin{align*}
&\int_{D(k^\epsilon,k^{2\epsilon})} k \Big|\big(\overline{U}_{j(k)} + \sqrt{k}F_k^*(\overline{U}_j R)\big)\Big(\chi\Big(\frac{z}{k^{\epsilon}},\frac{x''}{ k^{2\epsilon}}\Big)\Pi^{(0),\psi_0}_b u\Big)\Big|^2e^{-k\phi_0(\frac{z}{\sqrt{k}},\frac{x''}{k})}V\Big(\frac{z}{\sqrt{k}},\frac{x''}{k}\Big)d\mu_{\mathbb{H}^{n,m}}(x)\\
&\leq 2\int_{D(k^\epsilon,k^{2\epsilon})} k \Big|\overline{U}_{j(k)}\Big(\chi\Big(\frac{z}{k^{\epsilon}},\frac{x''}{k^{2\epsilon}}\Big)\Pi^{(0),\psi_0}_b u\Big)\Big|^2e^{-k\phi_0(\frac{z}{\sqrt{k}},\frac{x''}{k})}V\Big(\frac{z}{\sqrt{k}},\frac{x''}{k}\Big)d\mu_{\mathbb{H}^{n,m}}(x)\\
&\quad + 2\int_{D(k^\epsilon,k^{2\epsilon})} k \Big|\sqrt{k}F_k^*(\overline{U}_j R)\Big(\chi\Big(\frac{z}{k^{\epsilon}},\frac{x''}{ k^{2\epsilon}}\Big)\Pi^{(0),\psi_0}_b u\Big)\Big|^2e^{-k\phi_0(\frac{z}{\sqrt{k}},\frac{x''}{k})}V\Big(\frac{z}{\sqrt{k}},\frac{x''}{k}\Big)d\mu_{\mathbb{H}^{n,m}}(x).
\end{align*}
For the first term in the right hand side, 
\begin{align*}
&\int_{D(k^\epsilon,k^{2\epsilon})} k \Big|\overline{U}_{j(k)}\Big(\chi\Big(\frac{z}{k^{\epsilon}},\frac{x''}{k^{2\epsilon}}\Big)\Pi^{(0),\psi_0}_b u\Big)\Big|^2e^{-k\phi_0(\frac{z}{\sqrt{k}},\frac{x''}{k})}V\Big(\frac{z}{\sqrt{k}},\frac{x''}{k}\Big)d\mu_{\mathbb{H}^{n,m}}(x) \\
&= \int_{D(k^\epsilon,k^{2\epsilon})} k \Big|\Big(\overline{U}_{j(k)}\chi\Big(\frac{z}{k^{\epsilon}},\frac{x''}{k^{2\epsilon}}\Big)\Big)\Pi^{(0),\psi_0}_b u + \chi\Big(\frac{z}{k^{\epsilon}},\frac{x''}{k^{2\epsilon}}\Big)\overline{U}_{j(k)}\Big(\Pi^{(0),\psi_0}_b u\Big)\Big|^2\\
&\quad\quad\cdot e^{-k\phi_0(\frac{z}{\sqrt{k}},\frac{x''}{k})}V\Big(\frac{z}{\sqrt{k}},\frac{x''}{k}\Big)d\mu_{\mathbb{H}^{n,m}}(x)  \\
&\leq 2\underbrace{ \int_{D(k^\epsilon,k^{2\epsilon})} k \Big|\Big(\overline{U}_{j(k)}\chi\Big(\frac{z}{k^{\epsilon}},\frac{x''}{k^{2\epsilon}}\Big)\Big)\Pi^{(0),\psi_0}_b u\Big|^2e^{-k\phi_0(\frac{z}{\sqrt{k}},\frac{x''}{k})}V\Big(\frac{z}{\sqrt{k}},\frac{x''}{k}\Big)d\mu_{\mathbb{H}^{n,m}}(x)}_{[1]}\\
&\quad + 2\underbrace{\int_{D(k^\epsilon,k^{2\epsilon})} k \Big|\chi\Big(\frac{z}{k^{\epsilon}},\frac{x''}{k^{2\epsilon}}\Big)\overline{U}_{j(k)}\Big(\Pi^{(0),\psi_0}_b u\Big)\Big|^2e^{-k\phi_0(\frac{z}{\sqrt{k}},\frac{x''}{k})}V\Big(\frac{z}{\sqrt{k}},\frac{x''}{k}\Big)d\mu_{\mathbb{H}^{n,m}}(x)}_{[2]}
\end{align*}

Recall that
\begin{align*}
\overline{U}_{j(k)} &= \frac{\partial}{\partial\overline{z}_j} + i\sum_{\ell=1}^m\Big(\sum_{s=1}^{n}\lambda^{\ell}_{sj}z_s\Big)\frac{\partial}{\partial x_{2n+\ell}} 
+ \sum_{t=1}^{2n}O\bigg(\Big|\Big(\frac{z}{\sqrt{k}},\frac{x''}{k}\Big)\Big|^2\bigg)\frac{\partial}{\partial x_t}\\
&\quad+ \sum_{\ell=1}^mO\bigg(\Big|\Big(\frac{z}{k^{\frac{1}{4}}},\frac{x''}{k^\frac{3}{4}}\Big)\Big|^2\bigg)\frac{\partial}{\partial x_{2n+\ell}}.
\end{align*}
For $[1],$ since
\begin{align*}
\overline{U}_{j(k)}\chi\Big(\frac{z}{k^{\epsilon}},\frac{x''}{k^{2\epsilon}}\Big)
&=\frac{1}{k^\epsilon}\frac{\partial\chi}{\partial\overline{z}_j}\Big(\frac{z}{k^{\epsilon}},\frac{x''}{k^{2\epsilon}}\Big) + i\sum_{\ell=1}^m\Big(\sum_{s=1}^{n}\lambda^{\ell}_{sj}\frac{z_s}{k^{2\epsilon}}\Big)\frac{\partial\chi}{\partial x_{2n+\ell}}\Big(\frac{z}{k^{\epsilon}},\frac{x''}{k^{2\epsilon}}\Big)\\
&\quad+ \sum_{t=1}^{2n}O\bigg(\Big|\Big(\frac{z}{k^{\frac{1}{2}+\frac{\epsilon}{2}}},\frac{x''}{k^{1+\frac{\epsilon}{2}}}\Big)\Big|^2\bigg)\frac{\partial\chi}{\partial x_t}\Big(\frac{z}{k^{\epsilon}},\frac{x''}{k^{2\epsilon}}\Big)\\
&\quad+\sum_{\ell=1}^mO\bigg(\Big|\Big(\frac{z}{k^{\frac{1}{4}+\epsilon}},\frac{x''}{k^{\frac{3}{4}+\epsilon}}\Big)\Big|^2\bigg)\frac{\partial\chi}{\partial x_{2n+\ell}}\Big(\frac{z}{k^{\epsilon}},\frac{x''}{k^{2\epsilon}}\Big),
\end{align*}
we have
\[
[1] \leq C_2 k^{1-2\epsilon} \int_{\substack{k^\epsilon<|z_j|<2k^\epsilon\\ k^{2\epsilon}<|x_\ell|<2k^{2\epsilon}}}\big|\Pi^{(0),\psi_0}_b u\big|^2 e^{-\psi_0}d\mu_{\mathbb{H}^{n,m}}(x) = k^{1-2\epsilon}a_k(u),
\]
where $C_2$ is a positive constant depending only on $\overline{U}_j$ and $\chi$, and 
\[
a_k(u) = C_2 \int_{\substack{k^\epsilon<|z_j|<2k^\epsilon \\ k^{2\epsilon}<|x_\ell|<2k^{2\epsilon}}}\big|\Pi^{(0),\psi_0}_b u\big|^2 e^{-\psi_0}d\mu_{\mathbb{H}^{n,m}}(x).
\]
Note that $a_k(u)\rightarrow 0$ as $k\rightarrow\infty.$ 

For $[2],$ since 
\[
\bigg(\frac{\partial}{\partial\overline{z}_j} + i\sum_{\ell=1}^m\Big(\sum_{s=1}^{n}\lambda^{\ell}_{sj}z_s\Big)\frac{\partial}{\partial x_{2n+\ell}}\bigg) \Big(\Pi^{(0),\psi_0}_b u\Big) = 0,
\]
we have
\begin{align*}
&[2]\\
&=\int_{D(k^\epsilon,k^{2\epsilon})} k \bigg|\chi\Big(\frac{z}{k^{\epsilon}},\frac{x''}{k^{2\epsilon}}\Big)\bigg[\sum_{t=1}^{2n}O\bigg(\Big|\Big(\frac{z}{\sqrt{k}},\frac{x''}{k}\Big)\Big|^2\bigg)\frac{\partial}{\partial x_t}
+ \sum_{\ell=1}^mO\bigg(\Big|\Big(\frac{z}{k^{\frac{1}{4}}},\frac{x''}{k^\frac{3}{4}}\Big)\Big|^2\bigg]\frac{\partial}{\partial x_{2n+\ell}}\bigg)\\
&\quad\quad\times\Big(\Pi^{(0),\psi_0}_b u\Big)\bigg|^2
\, e^{-k\phi_0(\frac{z}{\sqrt{k}},\frac{x''}{k})}V\Big(\frac{z}{\sqrt{k}},\frac{x''}{k}\Big)d\mu_{\mathbb{H}^{n,m}}(x).
\end{align*}

Recall that the exponent part of the integrand of the explicit representation of $\Pi^{(0),\psi_0}_b$ (see (\ref{weightfml})) is
\begin{align*}
E(x,y,\eta) &= \exp\bigg[i\langle\eta,x''\rangle-i\langle\eta,y''\rangle+\frac{1}{2}\sum_{t=1}^m\beta_t x_{2n+t}-\frac{1}{2}\sum_{t=1}^m\beta_t y_{2n+t}\\
&\quad+\sum_{j,s=1}^{n}(M_\eta)_{js}z_j\overline{w}_s-\sum_{j,s=1}^{n}(M_\eta)_{js}w_j\overline{w}_s\\
&\quad+\sum_{\ell=1}^m\big(\frac{i\beta_\ell}{2}-\eta_\ell\big)\sum_{j,s=1}^{n}\lambda^{\ell}_{js}w_j\overline{w}_s-\sum_{\ell=1}^m\big(\frac{i\beta_\ell}{2}
  -\eta_\ell\big)\sum_{j,s=1}^{n}\lambda^{\ell}_{js}z_j\overline{z}_s\bigg].
\end{align*}
Note that, for $p = 1,\dots,n$ and $q = 1,\dots,m,$
\begin{align*}
\frac{\partial E}{\partial x_{2p-1}} &= E\cdot \Bigg[\sum_{s=1}^{n}\mu_{ps}\overline{w}_s - \sum_{\ell=1}^m\eta_\ell \Big(\sum_{s=1}^{n}2\lambda^\ell_{ps}\Big)\overline{w}_s + \sum_{\ell=1}^m \frac{1}{i} \Big(\sum_{s=1}^n\lambda^\ell_{ps}\overline{z}_s\Big)\Big(\frac{\beta_\ell}{2}+i\eta_\ell\Big),\\
&\quad\quad\quad + \sum_{\ell=1}^m \frac{1}{i} \Big(\sum_{j=1}^n\lambda^\ell_{jp}z_j\Big)\Big(\frac{\beta_\ell}{2}+i\eta_\ell\Big)\Bigg]\\
\frac{\partial E}{\partial x_{2p}} &= E\cdot \Bigg[i\sum_{s=1}^{n}\mu_{ps}\overline{w}_s - \sum_{\ell=1}^mi\eta_\ell \Big(\sum_{s=1}^{n}2\lambda^\ell_{ps}\Big)\overline{w}_s + \sum_{\ell=1}^m \Big(\sum_{s=1}^n\lambda^\ell_{ps}\overline{z}_s\Big)\Big(\frac{\beta_\ell}{2}+i\eta_\ell\Big),\\
&\quad\quad\quad - \sum_{\ell=1}^m \Big(\sum_{j=1}^n\lambda^\ell_{jp}z_j\Big)\Big(\frac{\beta_\ell}{2}+i\eta_\ell\Big)\Bigg]\\
\frac{\partial E}{\partial x_{2n+q}} &= E\cdot \Big(\frac{1}{2}\beta_q + i\eta_q\Big).
\end{align*}
Then, by (\ref{weightfml}) and integration by parts, we have
\begin{align*}
  \frac{\partial \big(\Pi^{(0),\psi_0}_b u\big)}{\partial x_{2p-1}} &= \sum_{s=1}^{n}\Big(\mu_{ps}-\sum_{\ell=1}^m i\beta_\ell \lambda^\ell_{ps}\Big)\Big(\Pi^{(0),\psi_0}_b (\overline{z}_s u)\Big) 
   - \sum_{\ell=1}^m\frac{1}{i} \Big(\sum_{s=1}^{n}2\lambda^\ell_{ps}\Big)\Big(\Pi^{(0),\psi_0}_b \frac{\partial (\overline{z}_s u)}{\partial x_{2n+\ell}}\Big)\\
   &\quad + \sum_{\ell=1}^m\frac{1}{i}\Big(\sum_{s=1}^{n}\lambda^\ell_{ps}\overline{z}_s\Big) \Big(\Pi^{(0),\psi_0}_b \frac{\partial u}{\partial x_{2n+\ell}}\Big)
  + \sum_{\ell=1}^m\frac{1}{i}\Big(\sum_{j=1}^{n}\lambda^\ell_{jp}z_j \Big)\Big(\Pi^{(0),\psi_0}_b \frac{\partial u}{\partial x_{2n+\ell}}\Big),\\
   \frac{\partial \big(\Pi^{(0),\psi_0}_b u\big)}{\partial x_{2p}} &= \sum_{s=1}^{n}\Big(i\mu_{ps} + \sum_{\ell=1}^m\beta_\ell \lambda^\ell_{ps}\Big)\Big(\Pi^{(0),\psi_0}_b (\overline{z}_s u)\Big) - \sum_{\ell=1}^m \Big(\sum_{s=1}^{n}2\lambda^\ell_{ps}\Big)\Big(\Pi^{(0),\psi_0}_b \frac{\partial (\overline{z}_s u)}{\partial x_{2n+\ell}}\Big)\\
   &\quad + \sum_{\ell=1}^m\Big(\sum_{s=1}^{n}\lambda^\ell_{ps}\overline{z}_s\Big) \Big(\Pi^{(0),\psi_0}_b \frac{\partial u}{\partial x_{2n+\ell}}\Big) 
   - \sum_{\ell=1}^m\Big(\sum_{j=1}^{n}\lambda^\ell_{jp}z_j \Big)\Big(\Pi^{(0),\psi_0}_b \frac{\partial u}{\partial x_{2n+\ell}}\Big), \\
   \frac{\partial \big(\Pi^{(0),\psi_0}_b u\big)}{\partial x_{2n+q}} &= \Pi^{(0),\psi_0}_b \frac{\partial u}{\partial x_{2n+q}}.
\end{align*}
Hence, we have, for $x\in D(k^{\epsilon},k^{2\epsilon}),$
\begin{align*}
\Bigg|\sum_{t=1}^{2n}O\bigg(\Big|\Big(\frac{z}{\sqrt{k}},\frac{x''}{k}\Big)\Big|^2\bigg)\frac{\partial\Pi^{(0),\psi_0}_b u}{\partial x_t}\Bigg|
&\leq C_3 \frac{k^{3\epsilon}}{k}\, \Bigg(\sum_{s=1}^{2n}\big|\Pi^{(0),\psi_0}_b (x_su)\big| + \sum_{\ell=1}^m\Big|\Pi^{(0),\psi_0}_b \frac{\partial u}{\partial x_{2n+\ell}}\Big|\Bigg),\\
\Bigg|\sum_{\ell=1}^mO\bigg(\Big|\Big(\frac{z}{k^{\frac{1}{4}}},\frac{x''}{k^\frac{3}{4}}\Big)\Big|^2\bigg)\frac{\partial\Pi^{(0),\psi_0}_b u}{\partial x_{2n+\ell}}\Bigg| 
&\leq C_3 \frac{k^{2\epsilon}}{\sqrt{k}}\, \Bigg(\sum_{\ell=1}^m\Big|\Pi^{(0),\psi_0}_b \frac{\partial u}{\partial x_{2n+\ell}}\Big|\Bigg).
\end{align*}
where $C_3$ is positive constant independent of $u.$
Thus,
\begin{align*}
[2] &\leq \int_{D(k^\epsilon,k^{2\epsilon})}C k\, \bigg|\chi\Big(\frac{z}{k^\epsilon},\frac{x''}{k^{2\epsilon}}\Big)\frac{k^{2\epsilon}}{\sqrt{k}}\bigg(\sum_{\ell=1}^m\Big|\Pi^{(0),\psi_0}_b \frac{\partial u}{\partial x_{2n+\ell}}\Big|\bigg)\Bigg|^2
e^{-k\phi_0(\frac{z}{\sqrt{k}},\frac{x'}{k})}V\Big(\frac{z}{\sqrt{k}},\frac{x''}{k}\Big)d\mu_{\mathbb{H}^{n,m}}(x)\\
&\leq C_4\int_{D(k^\epsilon,k^{2\epsilon})} k^{4\epsilon} \bigg|\bigg(\big|\Pi^{(0),\psi_0}_b u\big| + \sum_{\ell=1}^m\Big|\Pi^{(0),\psi_0}_b \frac{\partial u}{\partial x_{2n+\ell}}\Big|\bigg)\bigg|^2 e^{-\psi_0(x)} d\mu_{\mathbb{H}^{n,m}}(x)\\
&= C_5(u)\cdot k^{4\epsilon}
\end{align*}
where $C_4$ is a positive constant independent of $u$ and $C_5(u)$ is a positive constant depending on $\overline{U}_j,$ $\psi_0$ and $u.$

We next estimate
\[
\int_{D(k^\epsilon,k^{2\epsilon})} k \Big|\sqrt{k}F_k^*(\overline{U}_j R)\Big(\chi\Big(\frac{z}{k^{\epsilon}},\frac{x''}{k^{2\epsilon}}\Big)\Pi^{(0),\psi_0}_b u\Big)\Big|^2e^{-k\phi_0(\frac{z}{\sqrt{k}},\frac{x''}{k})}V\Big(\frac{z}{\sqrt{k}},\frac{x''}{k}\Big)d\mu_{\mathbb{H}^{n,m}}(x).
\]
By the definition of $R(x)$, we have
\[
\sqrt{k}F_k^*(\overline{U}_j R) = O\bigg(\Big|\Big(\frac{z}{k^{\frac{1}{4}}},\frac{x''}{k^\frac{3}{4}}\Big)\Big|^2\bigg).
\]
Hence
\begin{align*}
&\int_{D(k^\epsilon,k^{2\epsilon})} k \Big|\sqrt{k}F_k^*(\overline{U}_j R)\Big(\chi\Big(\frac{z}{k^{\epsilon}},\frac{x''}{k^{2\epsilon}}\Big)\Pi^{(0),\psi_0}_b u\Big)\Big|^2e^{-k\phi_0(\frac{z}{\sqrt{k}},\frac{x''}{k})}V\Big(\frac{z}{\sqrt{k}},\frac{x''}{k}\Big)d\mu_{\mathbb{H}^{n,m}}(x) \\
& \leq \int_{D(k^\epsilon,k^{2\epsilon})} C_6\, k\, \Big(\frac{k^{2\epsilon}}{\sqrt{k}}\Big)^2  \Big|\Big(\chi\Big(\frac{z}{k^{\epsilon}},\frac{x''}{k^{2\epsilon}}\Big)\Pi^{(0),\psi_0}_b u\Big)\Big|^2 e^{-\psi_0(x)}d\mu_{\mathbb{H}^{n,m}}(x)\\
&\leq C_6k^{4\epsilon}\|\ddot{u}\|^2_{\psi_0}
\end{align*}
where $C_6$ is a positive constant independent of $u.$ Thus, we obtain
\begin{align}\label{derker}
&\int_{D(k^\epsilon,k^{2\epsilon})} k\big|(\overline{\partial}_{b,(k),s}\ddot{u}_{(k)})(z,x'')\big|^2e^{-k\phi_0(\frac{z}{\sqrt{k}},\frac{x''}{k})}V\Big(\frac{z}{\sqrt{k}},\frac{x''}{k}\Big)d\mu_{\mathbb{H}^{n,m}}(x)\\
&\leq \hat{a}_k(u) = \max\big\{k^{1-2\epsilon}a_k(u),C_5(u)k^{4\epsilon},C_6k^{4\epsilon}\|\ddot{u}\|^2_{\psi_0}\big\}.\nonumber
\end{align}
Recall that $\epsilon<\frac{1}{8}$ and hence $k^{1-2\epsilon} > k^{4\epsilon}.$

Suppose that for all $u\in \mathcal{C}^\infty_c(D, L^k),$ 
\[
\big\|\big(I - \Pi_k\big)u\big\|_k^2 \leq (k^{1-2\epsilon})^{-1} \big(\square_{b,k}u\mid u\big)_k.
\]
We will show that $\hat{S}(\ddot{u}) = \Pi^{(0),\psi_0}_b u$ for all $u\in \mathcal{C}^\infty_c(D, L^k).$ Take arbitrary $\varphi\in \mathcal{C}^\infty_c(\mathbb{H}^{n,m}).$ For $k$ large enough,
\begin{align*}
  &\big|\big(\ddot{u}_{(k)} - \Pi_{(k),s}(\ddot{u})\mid \varphi\big)_{(k)}\big|\\
  &=  \big|\big(\big(I - \Pi_k\big)\ddot{u}_{[k]}\mid \varphi_{[k]}\big)_k\big|
  \leq \big\|\big(I - \Pi_k\big)\ddot{u}_{[k]}\big\|_k\cdot\|\varphi\|_{(k)}\\
  &\leq \Big[(k^{1-2\epsilon})^{-1}\big(\square_{b,k}\ddot{u}_{[k]}\mid \ddot{u}_{[k]}\big)_k\Big]^{\frac{1}{2}}\cdot\|\varphi\|_{(k)}\\
  &= (k^{1-2\epsilon})^{-\frac{1}{2}} \big\|\overline{\partial}_{b}\ddot{u}_{[k]}\big\|_k\cdot\|\varphi\|_{(k)}\\
  &= \bigg[(k^{1-2\epsilon})^{-1}\int_{D(k^\epsilon,k^{2\epsilon})} k\big|(\overline{\partial}_{b,(k),s}\ddot{u}_{(k)})(z,x'')\big|^2e^{-k\phi_0(\frac{z}{\sqrt{k}},\frac{x''}{k})}V\Big(\frac{z}{\sqrt{k}},\frac{x''}{k}\Big)d\mu_{\mathbb{H}^{n,m}}(x)\bigg]^{\frac{1}{2}}\cdot\|\varphi\|_{(k)}\\
  &\leq \big[(k^{1-2\epsilon})^{-1}\hat{a}_k(u)\big]^{\frac{1}{2}}\cdot\|\varphi\|_{(k)}
\end{align*}
By the estimate (\ref{derker}), we have $\big|\big(\ddot{u}_{(k)} - \Pi_{(k),s}(\ddot{u})\mid \varphi\big)_{(k)}\big|\rightarrow0$ as $k\rightarrow\infty.$ Note that $(\ddot{u}\mid\varphi)_{\psi_0} = \lim_{k\rightarrow\infty}\big(\ddot{u}_{(k)}\mid\varphi\big)_{(k)}$ and $\big(\hat{S}(\ddot{u})\mid\varphi\big)_{\psi_0} = \lim_{k\rightarrow\infty}\big(\Pi_{(k),s}(\ddot{u})\mid\varphi\big)_{(k)}.$ Hence $\big(\ddot{u} - \hat{S}(\ddot{u})\mid\varphi\big)_{\psi_0} = 0$ for all $\varphi\in \mathcal{C}^\infty_c(\mathbb{H}^{n,m}),$ which implies $\hat{S}(\ddot{u}) = \ddot{u}.$

We next show that $\hat{S}\nu = \nu$ for all $\nu\in \mathcal{H}^0_b(\mathbb{H}^{n,m}).$ Fix $\nu\in\mathcal{H}^0_b(\mathbb{H}^{n,m}).$ Then there exists some $f\in L^2(\mathbb{H}^{n,m}, \psi_0)$ such that $\nu = \ddot{f} = \Pi^{(0),\psi_0}_b f.$ Take a sequence $\{u_j\}$ in $\mathcal{C}^\infty_c(\mathbb{H}^{n,m})$ such that $\{u_j\}$ converges to $f$ in $L^2(\mathbb{H}^{n,m}, \psi_0).$ We also have $\ddot{u}_j\rightarrow\ddot{f}$ in $L^2(\mathbb{H}^{n,m}, \psi_0).$ 
Observe that 
\begin{align}\label{ob1}
&\big\|\Pi_{(k),s}\ddot{u}_j - \Pi_{(k),s}\ddot{f}\big\|_{(k)}\\
&\leq \big\|\ddot{u}_{j(k)} - \ddot{f}_{(k)}\big\|_{(k)}\lesssim \big\|\ddot{u}_{j(k)} - \ddot{f}_{(k)}\big\|_{\psi_0}\leq \big\|\Pi^{(0),\psi_0}_b u_j - \Pi^{(0),\psi_0}_b f\big\|_{\psi_0} \rightarrow 0.\nonumber
\end{align}
Hence
\begin{align*}
&\big\| \Pi_{(k),s}\ddot{f} - \ddot{f}_{(k)} \big\|_{(k)}\\
&= \big\| \Pi_{(k),s}\ddot{f} - \Pi_{(k),s}\ddot{u}_j + \Pi_{(k),s}\ddot{u}_j - \ddot{u}_{j(k)} + \ddot{u}_{j(k)} - \ddot{f}_{(k)}\big\|_{(k)}\\
&\leq \big\| \Pi_{(k),s}\ddot{f} - \Pi_{(k),s}\ddot{u}_j\big\|_{(k)} + \big\| \Pi_{(k),s}\ddot{u}_j - \ddot{u}_{j(k)}\big\|_{(k)} + \big\|  \ddot{u}_{j(k)} - \ddot{f}_{(k)}\big\|_{(k)}.
\end{align*}
By (\ref{ob1}), the first term of the right hand side converges to 0 as $k \rightarrow \infty.$ The second term of the right hand side converges to 0 by our first part argument. The third term clearly converges to 0 as $k \rightarrow \infty.$ Thus we have shown that
\[\hat{S}\nu = \nu\] for all $\nu\in \mathcal{H}^0_b(\mathbb{H}^{n,m}).$ Together with Proposition \ref{bounded} and Proposition \ref{ker},
 we obtain $\hat{S} = \Pi^{(0),\psi_0}_b.$
\end{proof}

\section{Asymptotics of spectral kernel for functions}
In this section, we prove Theorem \ref{t-gue260829yyd}.
We now consider the spectral projection. By the spectral theorem, we have a unitary operator
\[
U : L^2(X, L^k) \rightarrow L^2(S\times\mathbb{N}, d\mu)
\]
where $S = \operatorname{Spec}\big(\square_{b,k}\big).$ For a Borel set $B\subset \mathbb{R},$ we denote by $\mathbf{1}_B\big(\square_{b,k}\big)$ the functional calculus of $\square_{b,k}$ with respect to the indicator function $\mathbf{1}_B.$ Let $\{\mu_k\}$ be a sequence of non-negative real numbers which converges to 0. Let
\[
\Pi_{\leq k\mu_k}=\Pi_{k,\leq k\mu_k} = \mathbf{1}_{[0,k\mu_k]}\big(\square_{b,k}\big) : L^2(X, L^k) \rightarrow E^{(0)}_{\leq k\mu_k}
\]
be the spectral projection where $E^{(0)}_{\leq k\mu_k} = \operatorname{Ran}\big(\mathrm{1}_{[0,k\mu_k]}\big(\square_{b,k}\big)\big).$

For $u\in L^2(\mathbb{H}^{n,m},\psi_0),$ let $u_{(k)} = \chi\big(\frac{z}{k^{\epsilon}},\frac{x''}{k^{2\epsilon}}\big)u$ and $u_{[k]} = k^{\frac{n+m}{2}}u_{(k)}(\sqrt{k}z,kx'')e^{kR}s^k.$ Define
\begin{equation}\label{scaledPimuu}
\Pi_{(k),\leq k\mu_k,s}u := k^{-\frac{n+m}{2}}F_k^*\Big(\widehat{\Pi_{\leq k\mu_k} u_{[k]}}\cdot e^{-kR}\Big).
\end{equation}
\begin{prop}\label{spdefdis}
Fix $u\in L^2(\mathbb{H}^{n,m}, \psi_0).$ For any $v\in \mathcal{C}^\infty_c(\mathbb{H}^{n,m}),$ we have
\[
\big|\big(\Pi_{(k),\leq k\mu_k,s} u\mid v\big)_{(k)}\big| \leq \|u_{[k]}\|_k \|v_{[k]}\|_k \rightarrow \|u\|_{\psi_0}\|v\|_{\psi_0}.
\]
This implies that $\big(\Pi_{(k),\leq k\mu_k,s} u\mid v\big)_{(k)}$ is  bounded and hence  has a convergent subsequence.
\end{prop}
By the diagonal argument and the density argument, we can take a sequence $\Pi_{(k_j),\leq k_j\mu_{k_j},s}$ such that for all $u\in L^2(\mathbb{H}^{n,m},\psi_0),$ $\tilde{S}(u) := \lim_{j\rightarrow\infty}\big(\Pi_{(k_j),\leq k_j\mu_{k_j},s}u\mid\bullet \big)_{(k_j)}$ defines a distribution on $\mathcal{C}^\infty_c(\mathbb{H}^{n,m}).$ Actually, by Proposition \ref{spdefdis}, $\tilde{S}$ defines a bounded anti-linear functional on $L^2(\mathbb{H}^{n,m}, \psi_0).$ By Riesz representation theorem, we have $\tilde{S}$ is a linear map
\[
\tilde{S} : L^2(\mathbb{H}^{n,m}, \psi_0) \rightarrow L^2(\mathbb{H}^{n,m}, \psi_0).
\]  
Proposition \ref{spbounded} further shows that $\tilde{S}$ is actually a bounded linear operator. In Theorem \ref{spkerasy}, we will show that, under a spectral gap condition, $\Pi_{(k_j),\leq k_j\mu_{k_j},s}u$ converges to $\Pi^{(0),\psi_0}_bu$ in the distribution sense, that is, $\tilde{S} = \Pi^{(0),\psi_0}_b$. Since $\Pi_{(k_j),\leq k_j\mu_{k_j},s}$ is an arbitrary subsequence, we have $\Pi_{(k),\leq k\mu_k,s}$ itself converges to the Szeg\H{o} kernel for the model case. So to simplify the notation, we may assume that $\tilde{S}(u) = \lim_{k\rightarrow\infty}\big(\Pi_{(k),\leq k\mu_k,s}u\mid\bullet \big)_{(k)}.$

\begin{prop}\label{spbounded}
$\tilde{S} : L^2(\mathbb{H}^{n,m}, \psi_0) \rightarrow L^2(\mathbb{H}^{n,m}, \psi_0)$ is a bounded linear operator with operator norm $\|\tilde{S}\| \leq 1.$
\end{prop}
\begin{proof}
Note that, for all $f\in L^2(\mathbb{H}^{n,m}, \psi_0),$ 
\[
\|\tilde{S}f\|_{\psi_0} = \sup \big\{ \big(\tilde{S}f\mid v\big)_{\psi_0} : v\in \mathcal{C}^\infty_c(\mathbb{H}^{n,m})\ \text{with}\ \|v\|_{\psi_0} = 1 \big\}.
\]
\begin{align*}
&\big|\big(\tilde{S}f\mid v\big)_{\psi_0}\big| = \big|\big(\tilde{S}f\big)(v)\big|\\
&= \lim_{k\rightarrow\infty}\Big|\big(\Pi_{(k),\leq k\mu_k,s}f\mid v\big)_{(k)} \Big|
\leq \limsup_{k\rightarrow\infty} \big\|\Pi_{(k),\leq k\mu_k,s}f\big\|_{(k)}\cdot\|v\|_{(k)}\\
&\leq \limsup_{k\rightarrow\infty} \big\|f_{(k)}\big\|_{(k)}\cdot\|v\|_{(k)} = \|f\|_{\psi_0}\cdot\|v\|_{\psi_0} = \|f\|_{\psi_0}
\end{align*}
Thus, $\big\|\tilde{S}\big\|\leq1.$
\end{proof}

\begin{prop}\label{spker}
Suppose that $\{\mu_k\}$ is sequence of non-negative real numbers which converges to 0. Then we have
 $\square^{(0),\psi_0}_b\tilde{S}(u) = 0$ for all $u\in L^2(\mathbb{H}^{n,m}, \psi_0)$ in the distribution sense. That is, we have
  \[
\tilde{S}(u) : L^2(\mathbb{H}^{n,m}, \psi_0) \rightarrow \mathcal{H}^0_b(\mathbb{H}^{n,m}).
\]
\end{prop}
\begin{proof}
We first observe that for all $f\in \operatorname{Dom}\big(\square_{b,k}\big)$ and $\alpha\in L^2(X, L^k),$ 
\[
\big|\big(\Pi_{k,\leq k\mu_k}\alpha\,\big|\, \square_{b,k}f\big)_k\big| \leq k\mu_k (\alpha\mid f)_k|.
\]
Next, for all $g\in \mathcal{C}^\infty_c(\mathbb{H}^{n,m}),$ 
\begin{align*}
&\big|\big(\Pi_{(k),\leq k\mu_k,s}u\,\big|\, \square_{b,(k),s}g\big)_{(k)}\big| = \frac{1}{k}\big|\big(\Pi_{k,\leq k\mu_k}u_{[k]}\,\big|\, \square_{b,k}g_{[k]}\big)_k\big|\\
 &\leq \mu_k\big|\big(u_{[k]}\mid g_{[k]}\big)_k\big| = \mu_k \big|\big(u_{(k)}\mid g\big)_{(k)}\big| \leq \mu_k\|u_{(k)}\|_{(k)}\|g\|_{(k)}\rightarrow0
\end{align*}
since $\mu_k\rightarrow0.$ Hence for all $v\in \mathcal{C}^\infty_c(\mathbb{H}^{n,m}),$
\begin{align*}
   &\big|\big(\Pi_{(k),\leq k\mu_k,s}u\,\big|\,\square^{(0),\psi_0}_bv - \square_{b,(k),s}v\big)_{(k)}\big|\\
   &\leq \big\|\Pi_{(k),\leq k\mu_k,s}u\big\|_{(k)}\cdot \big\|\square^{(0),\psi_0}_bv - \square_{b,(k),s}v\big\|_{(k)}\rightarrow0.
\end{align*}
Thus,
\[
\square^{(0),\psi_0}_b\tilde{S}(u)(v) = \tilde{S}(u)(\square^{(0),\psi_0}_bv) = \lim_{k\rightarrow\infty} \big(\Pi_{(k),\leq k\mu_k,s}u\,\big|\,\square^{(0),\psi_0}_bv \big)_{(k)} = 0.
\]
\end{proof}

We now prove Theorem \ref{t-gue260829yyd}.
\begin{thm}\label{spkerasy}
Suppose that the non-negative sequence $\{\mu_k\}$ satisfies
\begin{align*}
\lim_{k\rightarrow\infty} \mu_k &= 0,\\
\lim_{k\rightarrow\infty} \frac{1}{k^{2\epsilon} \mu_k} &= 0.
\end{align*}
Then we have $\tilde{S} = \Pi^{(0),\psi_0}_b.$
\end{thm}
\begin{proof}
Suppose that the non-negative sequence $\{\mu_k\}$ satisfies the assumptions. We first note that for all $\alpha\in \operatorname{Dom}\big(\square_{b,k}\big),$ we have
\[
\big\|\big(I - \Pi_{\leq k\mu_k}\big)\alpha\big\|^2_k \leq \frac{1}{k\mu_k} \big(\square_{b,k}\alpha\mid\alpha\big)_k.
\]
Then, for all $u\in \mathcal{C}^\infty_c(\mathbb{H}^{n,m}),$ by the proof of Theorem \ref{skerasy}, we have
\begin{align*}
\big\|\Pi_{(k),\leq k\mu_k,s}\ddot{u} - \ddot{u}_{(k)}\big\|^2_{(k)} &\leq \big\|\Pi_{k,\leq k\mu_k}\ddot{u}_{[k]} - \ddot{u}_{[k]}\big\|^2_k \leq  \frac{1}{k\mu_k} \big(\square_{b,k}\ddot{u}_{[k]}\mid\ddot{u}_{[k]}\big)_k\\
&= \frac{1}{k\mu_k} \big\|\overline{\partial}_b \ddot{u}_{[k]}\big\|_k^2 \leq \frac{k^{1-2\epsilon}a_k(u)}{k\mu_k} = \frac{a_k(u)}{k^{2\epsilon}\mu_k}
\end{align*}
for $k$ large enough. By our assumption on $\{\mu_k\}$, we have  $\big\|\Pi_{(k),\leq k\mu_k,s}\ddot{u} - \ddot{u}_{(k)}\big\|_{(k)}\rightarrow0$ as $k\rightarrow\infty.$ Now the result follows from the last part of the proof in Theorem \ref{skerasy} with Proposition \ref{bounded} and Proposition \ref{ker} replaced by Proposition \ref{spbounded} and Proposition \ref{spker}, respectively.
\end{proof}

\section{Special cases}
In this section, we provide a special case in which we can assume a much smaller spectral gap condition to obtain the asymptotics of Szeg\H{o} kernels.

\subsection{Codimension 1}

\begin{prop}
Suppose that $X$ is of CR codimension 1 and is strongly pseudoconvex at $p.$ Suppose that for all $u\in \mathcal{C}^\infty_c(D, L^k),$ we have
\[
\big\|\big(I - \Pi_k\big)u\big\|_k^2 \leq (k^{4\epsilon})^{-1} \big(\square_{b,k}u\mid u\big)_k.
\]
Then $\hat{S} = \Pi^{(0),\psi_0}_b.$
\end{prop}
\begin{proof}
We assume that $M = M_1$ is diagonal with diagonal entries $\lambda_1,\dots, \lambda_n>0.$ Note that $I = I^{\Phi_0}_{M_1} = (-\infty,a]$ for some $a\in\mathbb{R}.$ In this case, we have
\begin{align*}
&\Big(\Pi^{(0),\psi_0}_b u\Big)(z,x_{2n+1})e^{-\frac{1}{2}\psi_0(z,x_{2n+1})}\\
& = \lim_{\epsilon\rightarrow0}\Big(\frac{1}{2\pi}\Big)^{2n+1}\int^a_{-\infty}\int_{\mathbb{C}^n}\int_{\mathbb{R}}\chi(\epsilon\eta)\exp\Big[-\frac{1}{2}\beta y_{2n+1} + i\eta x_{2n+1} - i\eta y_{2n+1} + \sum_{j,s}(M_\eta)_{j,s}z_j\overline{w}_s\\ 
&\quad\quad - \sum_{j,s}(M_\eta)_{j,s}w_j\overline{w}_s - \frac{1}{2}\sum_{j,s}(M_\eta)_{j,s}z_j\overline{z}_s - \eta\sum_j \lambda_j|w_j|^2
+ i\beta\sum_j \lambda_j(|w_j|^2-|z_j|^2)\Big]\\
&\quad\quad \cdot \det(M_\eta)u(w,y_{2n+1})d\mu(y_{2n+1})d\mu(w)d\mu(\eta).
\end{align*}

We claim that, for all $N\in \mathbb{N},$ we have  
\begin{equation}\label{stpsest}
\Big|\big(\Pi^{(0),\psi_0}_b u\big)(z,x_{2n+1})e^{-\frac{1}{2}\psi_0(z,x_{2n+1})}\Big| \leq \frac{C_N(u)|z|^N + \cdots +C_1(u)|z| + C_0(u)}{\big|ix_{2n+1} + \sum_{j}\lambda_j|z_j|^2\big|^N},
\end{equation}
where $C_0(u), \dots, C_N(u)$ are positive constants depending only on $u.$

Let 
\begin{align*}
E^1(x,y,\eta) &= -\frac{1}{2}\beta y_{2n+1} + i\eta x_{2n+1} - i\eta y_{2n+1} + \sum_{j,s}(M_\eta)_{j,s}z_j\overline{w}_s- \sum_{j,s}(M_\eta)_{j,s}w_j\overline{w}_s\\
&\quad - \frac{1}{2}\sum_{j,s}(M_\eta)_{j,s}z_j\overline{z}_s - \eta\sum_j \lambda_j|w_j|^2 + i\beta\sum_j \lambda_j(|w_j|^2-|z_j|^2).
\end{align*}
Then 
\[
\frac{\partial}{\partial \eta}E^1(x,y,\eta) = ix_{2n+1} - iy_{2n+1} -2\sum_j \lambda_j z_j \overline{w}_j + \sum_j \lambda_j|w_j|^2 + \sum_j \lambda_j |z_j|^2.
\]
Fix $\epsilon>0,$ let 
\begin{align*}
&\big(P_\epsilon u\big) (z,x_{2n+1})e^{-\frac{1}{2}\psi_0(z,x_{2n+1})}\\
&:= \int^a_{-\infty}\int_{\mathbb{C}^n}\int_{\mathbb{R}}\chi(\epsilon\eta) e^{E^1(x,y,\eta)} \det(M_\eta)u(w,y_{2n+1})d\mu(y_{2n+1})d\mu(w)d\mu(\eta).
\end{align*}
Then
\begin{align*}
&\Big(ix_{2n+1} + \sum_j\lambda_j|z_j|^2\Big)\big(P_\epsilon u\big) (z,x_{2n+1})e^{-\frac{1}{2}\psi_0(z,x_{2n+1})}\\
&=  \int^a_{-\infty}\int_{\mathbb{C}^n}\int_{\mathbb{R}}\Big[\frac{\partial}{\partial \eta}E^1(x,y,\eta) + \Big(iy_{2n+1} + 2\sum_j \lambda_j z_j \overline{w}_j -\sum_j \lambda_j |w_j|^2 \Big)\Big]\\
&\quad\quad\cdot e^{E^1(x,y,\eta)}\chi(\epsilon\eta) \det(M_\eta)u(w,y_{2n+1})d\mu(y_{2n+1})d\mu(w)d\mu(\eta)\\
&= \int^a_{-\infty} \int_{H^{n+1}} \Big(\frac{\partial}{\partial \eta} e^{E^1(x,y,\eta)}\Big) \, \chi(\epsilon \eta) \det(M_\eta) u(y) d\mu(y) d\mu(\eta) + O(|z|) + O(1) \\
&= - \int^a_{-\infty} \int_{H^{n+1}} e^{E^1(x,y,\eta)} \Big(\frac{\partial (\chi(\epsilon \eta))}{\partial \eta}\Big) \det(M_\eta)  u(y) d\mu(y) d\mu(\eta)\\
&\quad- \int^a_{-\infty} \int_{H^{n+1}} e^{E^1(x,y,\eta)} \,\chi(\epsilon \eta) \Big(\frac{\partial}{\partial \eta} \det(M_\eta)\Big) u(y) d\mu(y) d\mu(\eta) + O(|z|) + O(1).
\end{align*}

The first integral in the last term goes to $0$ as $\epsilon\rightarrow 0.$ So it suffices to estimate 
\begin{equation*}
    \int^a_{-\infty} \int_{H^{n+1}} e^{E^1(x,y,\eta)} \,\chi(\epsilon \eta) \Big(\frac{\partial}{\partial \eta} \det(M_\eta)\Big) u(y) d\mu(y) d\mu(\eta).
\end{equation*}

Take $A < 0$ with $|A|$ large enough so that $\bigg| \frac{\frac{\partial}{\partial \eta} \det(M_\eta)}{\det(M_\eta)} \bigg| \leq 1$ for all $\eta < A.$ Write
\begin{align*}
&\int^a_{-\infty} \int_{H^{n+1}} e^{E^1(x,y,\eta)} \chi(\epsilon \eta) \Big(\frac{\partial}{\partial \eta} \det(M_\eta)\Big) u(y) d\mu(y) d\mu(\eta)\\
&= \int^A_{-\infty} \int_{H^{n+1}} e^{E^1(x,y,\eta)} \chi(\epsilon \eta) \Big(\frac{\partial}{\partial \eta} \det(M_\eta)\Big) u(y) d\mu(y) d\mu(\eta)\\
&\quad+ \int^a_A \int_{H^{n+1}} e^{E^1(x,y,\eta)} \chi(\epsilon \eta) \Big(\frac{\partial}{\partial \eta} \det(M_\eta)\Big) u(y) d\mu(y) d\mu(\eta).
\end{align*}

\noindent We have
\begin{align*}
&\bigg| \int^A_{-\infty} \int_{H^{n+1}} e^{E^1(x,y,\eta)} \chi(\epsilon \eta) \Big(\frac{\partial}{\partial\eta} \det(M_\eta)\Big) u(y) d\mu(y) d\mu(\eta) \bigg|\\
&= \bigg|\int^A_{-\infty} \int_{H^{n+1}} \frac{1}{(i\eta)^2} e^{E^1(x,y,\eta)} \chi(\epsilon \eta) \Big(\frac{\partial}{\partial \eta} \det(M_\eta)\Big) \frac{\partial^2u(y)}{\partial y_{2n+1}^2} d\mu(y) d\mu(\eta) \bigg| \\
&\leq C \int^A_{-\infty} \int_{\mathbb{C}^n} \frac{1}{\eta^2}\, e^{-\frac{1}{2} \sum\limits_{j,s} (M_\eta)_{js} (z_j - w_j) (\overline{z}_s - \overline{w}_s)} \Big| \frac{\partial}{\partial \eta} \det(M_\eta) \Big| d\mu(w) d\mu(\eta) \\
&\leq C \int^A_{-\infty} \frac{1}{\eta^2} \bigg| \frac{\frac{\partial}{\partial \eta} \det(M_\eta)}{\det(M_\eta)} \bigg| d\mu(\eta)\\
&\leq C\int^A_{-\infty} \frac{1}{\eta^2} d\mu(\eta) < \infty
\end{align*}

Thus,
\begin{align*}
\Big|\big(\Pi^{(0),\psi_0}_b u\big)(z, x_{2n+1}) e^{-\frac{1}{2} \psi_0(z, x_{2n+1})}\Big| \leq \frac{C_1(u)|z| + C_0(u)}{\big|i x_{2n+1} + \sum_j \lambda_j |z_j|^2\big|}
\end{align*}

 We next prove the case for $N=2.$
\begin{align*}
&\Big( ix_{2n+1} + \sum_j \lambda_j |z_j|^2\Big)^2 \big(P_\epsilon u\big)(z, x_{2n+1}) e^{-\frac{1}{2} \psi_0(z, x_{2n+1})} \\
&= \int^a_{-\infty} \int_{H^{n+1}} \bigg[ \frac{\partial E^1}{\partial \eta} + \Big(i y_{2n+1} + 2 \sum_j \lambda_j z_j \overline{w}_j - \sum_j \lambda_j |w_j|^2\Big) \bigg]^2 e^{E^1(x,y,\eta)}\\
&\quad\quad\cdot \chi(\epsilon \eta)\det(M_\eta) u(y) d\mu(y) d\mu(\eta) \\
&= \int^a_{-\infty} \int_{H^{n+1}} \Big(\frac{\partial^2}{\partial \eta^2} e^{E^1(x,y,\eta)}\Big) \chi(\epsilon \eta) \det(M_\eta) u(y) d\mu(y) d\mu(\eta)\\
&\quad+ 2 \int^a_{-\infty} \int_{H^{n+1}} \Big(\frac{\partial}{\partial \eta} e^{E^1(x,y,\eta)}\Big) \Big(i y_{2n+1} + 2 \sum_j \lambda_j z_j \overline{w}_j - \sum_j \lambda_j |w_j|^2\Big)\\
&\quad\quad\cdot\chi(\epsilon \eta)\det(M_\eta) u(y)d\mu(y) d\mu(\eta) \\
&\quad + \int^a_{-\infty} \int_{H^{n+1}} e^{E^1(x,y,\eta)}\Big(i y_{2n+1} + 2 \sum_j\lambda_j z_j \overline{w}_j - \sum_j\lambda_j |w_j|^2\Big)^2\\
&\quad\quad\cdot\chi(\epsilon \eta)\det(M_\eta)u(y) d\mu(y) d\mu(\eta)\\
&=\int^a_{-\infty} \int_{H^{n+1}} \Big(\frac{\partial^2}{\partial \eta^2} e^{E^1(x,y,\eta)}\Big) \chi(\epsilon \eta) \det(M_\eta) u(y) d\mu(y) d\mu(\eta) + O(|z|^2) + O(|z|) +O(1).
\end{align*}

We now estimate $\int^a_{-\infty} \int_{H^{n+1}} \Big(\frac{\partial^2}{\partial \eta^2} e^{E^1(x,y,\eta)}\Big) \chi(\epsilon \eta) \det(M_\eta) u(y) d\mu(y) d\mu(\eta).$

\begin{align*}
&\int^a_{-\infty} \int_{H^{n+1}} \Big(\frac{\partial^2}{\partial \eta^2} e^{E^1(x,y,\eta)}\Big) \chi(\epsilon \eta) \det(M_\eta) u(y) d\mu(y) d\mu(\eta)\\
&= \int^a_{-\infty} \int_{H^{n+1}} \frac{\partial}{\partial \eta} \bigg[ \Big(\frac{\partial}{\partial \eta} e^{E^1(x,y,\eta)}\Big)\, \chi(\epsilon \eta)\,\det(M_\eta) \bigg] \, u(y) d\mu(y) d\mu(\eta)\\
&\quad- \int^a_{-\infty} \int_{H^{n+1}}  \Big(\frac{\partial}{\partial \eta} e^{E^1(x,y,\eta)}\Big) \Big(\frac{\partial}{\partial \eta}\big(\chi(\epsilon \eta) \, \det(M_\eta)\big)\Big) \cdot u(y) d\mu(y) d\mu(\eta) \\
&= - \int^a_{-\infty} \int_{H^{n+1}} \Big(\frac{\partial}{\partial \eta} e^{E^1(x,y,\eta)}\Big) \, \frac{\partial \chi(\epsilon \eta)}{\partial \eta} \, \det(M_\eta) \, u(y)\, d\mu(y) d\mu(\eta)\\
&\quad- \int^a_{-\infty} \int_{H^{n+1}} \Big(\frac{\partial}{\partial \eta} e^{E^1(x,y,\eta)}\Big)\, \chi(\epsilon \eta) \, \Big(\frac{\partial}{\partial \eta} \det(M_\eta)\Big) \, u(y)\, d\mu(y) d\mu(\eta) \\
&= \int^a_{-\infty} \int_{H^{n+1}} e^{E^1(x,y,\eta)} \frac{\partial}{\partial \eta} \Big(\frac{\partial \chi(\epsilon \eta)}{\partial \eta} \det(M_\eta)\Big) \, u(y)\, d\mu(y) d\mu(\eta) \\
&\quad - \int^a_{-\infty} \int_{H^{n+1}} \frac{\partial}{\partial \eta} \Big(e^{E^1(x,y,\eta)} \, \chi(\epsilon \eta) \, \big(\frac{\partial}{\partial \eta} \det(M_\eta)\big)\Big) \, u(y)\, d\mu(y) d\mu(\eta)\\
&\quad+ \int^a_{-\infty} \int_{H^{n+1}} e^{E^1(x,y,\eta)} \frac{\partial}{\partial \eta} \Big(\chi(\epsilon \eta) \, \big(\frac{\partial}{\partial \eta} \det(M_\eta)\big)\Big) \, u(y)\, d\mu(y) d\mu(\eta) \\
&= O(1) - \int_{H^{n+1}} e^{E^1(x,y,a)} \, \Big(\frac{\partial \det (M_\eta)}{\partial\eta}\bigg|_{\eta = a}\Big) \,u(y)\, d\mu(y)\\
&\quad+ \int^a_{-\infty} \int_{H^{n+1}} e^{E^1(x,y,\eta)}\, \frac{\partial \chi(\epsilon \eta)}{\partial \eta}  \frac{\partial \det(M_\eta)}{\partial \eta}\, u(y)\, d\mu(y) d\mu(\eta) \\
&\quad + \int^a_{-\infty} \int_{H^{n+1}} e^{E^1(x,y,\eta)} \chi(\epsilon \eta) \, \Big(\frac{\partial^2}{\partial \eta^2} \det(M_\eta)\Big) \, u(y)\, d\mu(y) d\mu(\eta) \\
&= O(1).
\end{align*}
Hence
\begin{equation*} 
\Big|\big(\Pi^{(0),\psi_0}_b u\big)(z, x_{2n+1}) e^{-\frac{1}{2} \psi_0(z, x_{2n+1})}\Big| \leq \frac{C_2(u)|z|^2 + C_1(u)|z| + C_0(u)}{\big|i x_{2n+1} + \sum_j \lambda_j |z_j|^2\big|^2}.
\end{equation*}

By induction, we conclude (\ref{stpsest}). Replacing the estimation of $[1]$ in the proof of Theorem \ref{skerasy} by (\ref{stpsest}), we obtain the conclusion.
\end{proof}

\section{CR manifolds with \texorpdfstring{$\mathbb{R}^m$}{R\textasciicircum m}-action}\label{gpact}
\subsection{Fundamental vector fields}

Let $(X^{2n+m},T^{1,0}X)$ be a $(2n+m)$-dimensional CR manifold of CR-codimension $m$ endowed with an $\mathbb{R}^m$-action $\mathbb{R}^m \times X \rightarrow X, ((r_1,\dots,r_m),x) \mapsto (r_1,\dots,r_m)\circ x.$ Let $Y_1,\dots,Y_m\in \mathcal{C}^\infty(X, TX)$ be the infinitesimal generators induced by 
\[
(Y_ju)(x) = \frac{d}{d r}(u(ra_j\circ x))\Big|_{r=0},
\]
where $u\in \mathcal{C}^\infty(X)$ and $\{a_1,\dots,a_m\}$ is the standard basis of $\mathbb{R}^m.$

\begin{defn}
We say that the $\mathbb{R}^m$-action is \emph{CR} if
$[Y_j, \mathcal{C}^\infty(X,T^{1,0}X)] \subset \mathcal{C}^\infty(X,T^{1,0}X),$ $j = 1,\dots, m,$ and the $\mathbb{R}^m$-action is \emph{transversal} if for each $x \in X$,
$\mathbb{C}Y_1(x) \oplus \cdots \oplus \mathbb{C}Y_m(x) \oplus T^{1,0}_x X \oplus T^{0,1}_x X = \mathbb{C}T_x X .$
Moreover, we say that the $\mathbb{R}^m$-action is \emph{locally free} if $\{Y_j(x)\}_{j=1}^m$ is linear independent everywhere.
\end{defn}

In the rest of the paper, we assume that $(X,T^{1,0}X)$ is a compact connected $(2n+m)$-dimensional CR manifold of CR-codimension $m$ with a locally free transversal CR $\mathbb{R}^m$-action and we let $Y_1,\dots, Y_m$ be the global vector fields generated by the $\mathbb{R}^m$-action as above. Let $\omega_\ell\in \mathcal{C}^\infty(X,T^*X),\ \ell=1,\dots,m,$ be the global real one form determined by
\begin{align}\label{dire}
    &(\omega_\ell, Z) = 0,\ \forall Z \in T^{1,0}X \oplus T^{0,1}X,\\
    &(\omega_\ell, Y_q) = \delta_{\ell q},\ \ell, q=1,\dots,m.\nonumber
\end{align}
For $p\in X,$ we denote by $\mathcal{L}_{\omega_\ell,p}$ the Levi form at $p$ in the direction $\omega_\ell, \ell =1,\dots,m.$

\begin{defn}
Let $D \subset X$ be an open set.
We say that a function $u \in \mathcal{C}^\infty(D)$ is \emph{rigid} if $Y_\ell u = 0,\ \ell=1,\dots, m.$ We say that a function $u \in \mathcal{C}^\infty(X)$ is \emph{Cauchy--Riemann}
(\emph{CR} for short) if
$\overline{\partial}_b u = 0 .$ We say that $u \in \mathcal{C}^\infty(X)$ is \emph{rigid CR} if
\begin{equation*}
\overline{\partial}_b u = 0 \quad \text{and} \quad Y_\ell u = 0, \ \ell=1,\dots, m.
\end{equation*}
\end{defn}

\subsection{Rigid CR bundles.} In this section, we will state a simple generalization of some results in \cite{HHL22} without proof. Let $(X, T^{1,0}X)$, $\dim X = 2n+m$, be a CR manifold of codimension
$m \in \mathbb{N}$ and CR dimension $n \in \mathbb{N}$. The following definitions for CR vector bundles can be
found in \cite{HN00}.

\begin{defn} 
A complex vector bundle $(E, \pi, X)$ over $X$ is called \textit{CR vector bundle} if
\begin{align*}
(i)&\ E\ \text{is a CR manifold of codimension}\ d,\\
(ii)&\ \pi : E \to X\ \text{is a CR submersion},\\
(iii)&\ E \oplus E \ni (\xi_1, \xi_2) \to \xi_1 + \xi_2 \in E\ \text{and}\ \mathbb{C} \times E \ni (\lambda, \xi) \to \lambda\xi \in E\ \text{are CR maps}.
\end{align*}
A smooth section $s \in \mathcal{C}^\infty(U, E)$ defined on an open set $U \subset X$ is called \textit{CR section} if the map
$s : U \to E$ is CR.
\end{defn}

Let $(E_1, \pi_1, X)$ and $(E_2, \pi_2, X)$ be two CR vector bundles over $X$. A map $F : E_1 \to E_2$
is called a \textit{CR bundle isomorphism} if $F$ is a $\mathcal{C}^\infty$-diffeomorphism such that $F, F^{-1}$ are CR
maps, $\pi_2 \circ F = \pi_1$ and $F$ is fiberwise linear.

Given a CR vector bundle $(E, \pi, X)$ we find (see \cite{HN00}) the linear partial differential
operator $\overline{\partial}_b^E : \mathcal{C}^\infty(X, E) \to \mathcal{C}^\infty(X, E \otimes T^{*0,1}X)$ satisfying
\begin{align*}
(a)&\ \overline{\partial}_b^E(f \cdot s) = s\overline{\partial}_b(f) + f\overline{\partial}_b^E(s) \ \text{for all} \ f \in \mathcal{C}^\infty(X) \  \text{and} \ s \in \mathcal{C}^\infty(X, E),\\
(b)&\ s \in \mathcal{C}^\infty(U, E) \ \text{is a CR section if and only if} \ \overline{\partial}_b^E s = 0.
\end{align*}
\begin{defn} 
A CR vector bundle $(E, \pi, X)$ of rank $r$ is called \textit{locally CR trivializable} if for
any point $p \in X$ there exists an open neighborhood $U \subset X$ such that $E|_U$ is CR vector bundle
isomorphic to the trivial CR vector bundle $U \times \mathbb{C}^r$.
\end{defn}
The following lemma is well-known. For the proof, see \cite{HHL22}.

\begin{lem}\label{crtri}
Let $(E, \pi, X)$ be a CR vector bundle. The following are equivalent:
\begin{align*}
(i)&\ (E, \pi, X) \text{ is locally CR trivializable},\\
(ii)&\ \text{For any}\ p \in X \ \text{there exists a smooth frame} \ \{f_1, \dots, f_r\} \ \text{of}\ E|_U \ \text{on an open}\\
&\ \text{neighborhood} \ U \subset X \ \text{around} \ p \ \text{such that} \ f_1, \dots, f_r : U \to E \ \text{are CR sections}.
\end{align*}
\end{lem}

Let $\{f_1, \dots, f_r\}$ be a frame of $E|_U$ for some open set $U \subset X$. Then $\{f_1, \dots, f_r\}$ is called \textit{CR frame} if any $f_k, 1 \le k \le r$, is a CR section. Given two CR frames of $E|_U$ we find by (a) and (b) that the corresponding transition matrix is CR in the sense that any entry is a CR function.

\begin{defn}
Let $(X, T^{1,0}X)$ be a CR manifold of codimension $m$ and let $\{T_1,\dots,T_q\} \subset \mathcal{C}^\infty(X, TX)$ be a system of commutative CR vector fields (that is $[T_j, \mathcal{C}^\infty(X, T^{1,0}X)] \subset \mathcal{C}^\infty(X, T^{1,0}X)$ and $[T_i,T_j]=0, \ i,j=1,\dots,q$). A system of commutative CR bundle lifts with respect to  $\{T_j\}_{j=1}^q$ to $(E, \pi, X)$ is a family of linear partial differential operators $T_j^E : \mathcal{C}^\infty(X, E) \to \mathcal{C}^\infty(X, E)$ (with smooth coefficients), $j=1,\dots,q,$ such that
\begin{align*}
(i)& \ T_j^E(f \cdot s) = T_j(f) \cdot s + f T_j^E(s) \text{ for all } f \in \mathcal{C}^\infty(X) \text{ and } s \in \mathcal{C}^\infty(X, E), \ j=1,\dots,q,\\
(ii)& \ [T_j^E, \overline{\partial}_b^E] = 0, \ j=1,\dots,q,\\
(iii)&\ [T^E_i,T^E_j] = 0, \ \text{for all} \ i,j=1,\dots,q.
\end{align*}
\end{defn}

For a CR vector field $T,$ in order to define $[T^E, \overline{\partial}_b^E]$ we need to define $T^E$ on $(0, 1)$-forms with values in $E$ first. But this definition follows immediately from the fact that any $w \in \mathcal{C}^\infty(X, E \otimes T^{*0,1}X)$ locally can be written $w = \sum_{j=1}^r f_j \otimes \omega^j$ where $\{\omega^j\}$ are $(0, 1)$-forms and $\{f_j\}$ is a  local frame of $E,$ that is,
\begin{equation*}
T^E(w) = \sum_{j=1}^r \big(T^E(f_j)\otimes \omega^j + f_j\otimes \mathcal{L}_T(\omega^j)\big),
\end{equation*}
where $\mathcal{L}_T$ is the Lie derivative with respect to $T.$ $T^E$ is defined also for $(0, q)$-forms using the Lie derivative.

\begin{defn}\label{rigidbd}
Let $(X, T^{1,0}X)$ be a CR manifold of codimension $m$ and let $\{T_1,\dots,T_q\} \subset \mathcal{C}^\infty(X, TX)$ be a system of commutative CR vector fields. A CR vector bundle $(E, \pi, X)$ of rank $r$ over $X$ with a system of commutative CR bundle lifts with respect to  $\{T_j\}_{j=1}^q$ is called \textit{rigid CR} with respect to $\{T^E_j\}_{j=1}^q$ if for every point $p \in X$ there exists an open neighborhood $U$ around $p$ and a CR frame $\{f_1, \dots, f_r\}$ of $E|_U$ with $T^E_i(f_j) = 0$ for $1\leq i \leq q,\ 1 \le j \le r$.
\end{defn}

A function $f\in \mathcal{C}^\infty(X)$ is called rigid with respect to $\{T_j\}_{j=1}^q$ if $T_j f = 0$ for all $j=1,\dots,q.$  A section $s \in \mathcal{C}^\infty(X, E)$ is called a rigid CR section with respect to $\{T^E_j\}_{j=1}^q$ if $T^E_j s = 0,\ j=1,\dots,q$ and $\overline{\partial}_b^E s = 0$. The frame $\{f_j\}_{j=1}^r$ in Definition \ref{rigidbd} is called a rigid CR frame of $E|_U$. Note that it follows from Lemma \ref{crtri} that any rigid CR vector bundle is locally CR trivializable.

\begin{lem}\label{rigidcrequ}
Let $(E, \pi, X)$ be a CR vector bundle over a CR manifold $(X, T^{1,0}X)$ of codimension $m$ and let $\{T_1,\dots,T_q\} \subset \mathcal{C}^\infty(X, TX)$ be a system of commutative CR vector fields. The following are equivalent:
\begin{align*}
(i)& \ \text{There is a system of commutative CR bundle lifts}\  \{T^E_j\}_{j=1}^q \  \text{with respect to}\  \{T_j\}_{j=1}^q \\ 
& \ \text{ such that } (E, \pi, X) \text{ is rigid CR with respect to } \{T^E_j\}_{j=1}^q. \\
(ii)& \ \text{There exist an open cover } \{U_j\}_{j \in \mathbb{N}} \text{ of } X \text{ and CR frames } \{f_1^j, \dots, f_r^j\} \text{ for } E|_{U_j}, \ j \in \mathbb{N}, \\
&\ \text{such that the corresponding transition matrices are rigid CR with respect to}\  \{T_j\}_{j=1}^q \\
&\ \text{in the sense that any} \ \text{entry is a rigid CR function  with respect to}\  \{T_j\}_{j=1}^q.
\end{align*}
\end{lem}

\begin{lem}\label{liftimprigid}
Let $(X,T^{1,0}X)$ be a CR manifold of CR codimension $m$ with a transversal CR $\mathbb{R}^m$-action $\zeta:\mathbb{R}^m\times X\to X.$ Let $Y_1,\dots,Y_m$ be the infinitesimal generators of the $\mathbb{R}^m$-action induced from the standard basis for $\mathbb{R}^m.$ Let $(L,\pi,X)$ be a locally CR trivializable CR line bundle. Assume that $\{Y^L_\ell\}_{\ell=1}^m$ is a system of commutative CR bundle lifts with respect to  $\{Y_\ell\}_{\ell=1}^m$ to $(L, \pi, X).$ Then $(L,\pi,X)$ is rigid CR with respect to $\{Y^L_\ell\}_{\ell=1}^m.$
\end{lem}

\begin{lem}\label{liftequgpact}
Under the assumptions of Lemma \ref{liftimprigid}, the following assertions are equivalent:
\begin{itemize}
\item[(1)] There exists a system of commutative CR bundle lifts $\{Y_\ell^L\}_{\ell=1}^m$ of $\{Y_\ell\}_{\ell=1}^m$ to $L.$
\item[(2)] There exists a smooth $\mathbb{R}^m$-action $\Psi:\mathbb{R}^m\times L\longrightarrow L$
by CR bundle automorphisms covering $\eta;$ that is $\pi$ is  $\mathbb{R}^m$-equivariant.
\end{itemize}
Moreover, these systems of lifts and actions are in bijective correspondence, characterized by
\begin{equation}\label{indbunlift}
(Y_\ell^Lu)(x)=\frac{d}{dt}\bigg|_{t=0}\Psi_{-ta_\ell}\bigl(u(\zeta_{ta_\ell}(x))\bigr),\qquad u\in \mathcal{C}^\infty(X,L), \ \ell=1,\ldots,m.
\end{equation}
Recall that $\{a_1,\dots,a_m\}$ is the standard basis for $\mathbb{R}^m.$
\end{lem}
\begin{proof}
Suppose that $L$ admits a smooth $\mathbb{R}^m$-action $\Psi:\mathbb{R}^m\times L\longrightarrow L$ by CR bundle automorphisms covering $\zeta.$ Then it is easy to check that the system of linear partial differential operators $\{Y_\ell^L\}_{\ell=1}^m$ determined by \eqref{indbunlift} is a system of commutative CR bundle lifts $\{Y_\ell^L\}_{\ell=1}^m$ of $\{Y_\ell\}_{\ell=1}^m$ to $L.$
Conversely, assuming the existence of a system of commutative CR bundle lifts $\{Y_\ell^L\}_{\ell=1}^m$ of $\{Y_\ell\}_{\ell=1}^m$ to $L.$ By Lemma \ref{rigidcrequ} and  Lemma \ref{liftimprigid}, for every $p\in X,$ there exists an open neighborhood $U$ around $p$ and a local rigid CR frame $s$ over $U.$ For an orbit segment contained in $U,$ define the $\mathbb{R}^m$-action on $L$ by $\Psi_t(cs(p):=cs(\zeta_t(p))$ for suitable $t\in\mathbb{R}^m$ and all $c\in \mathbb{C}.$
\end{proof}

 In this paper, we will make systematic use of appropriate coordinates introduced by Baouendi-Rothschild-Treves \cite{BRT85}. For each point $p \in X,$ there exists a coordinate neighborhood $U$ with coordinates $(x_1, \dots, x_{2n+m})$, centered at $p=0$, and $\eta > 0$, $\eta_0 > 0$ such that, by setting $z_j = x_{2j-1} + ix_{2j}$, $j=1, \dots, n,$ $x'' = (x_{2n+1},\dots, x_{2n+m})$ and $D = \{(z,x'') \in U : |z| < \eta,\, |x_{2n+\ell}| < \eta_0,\, \ell =1,\dots, m\} \subset U$, we have 
\begin{equation}
Y_\ell = \frac{\partial}{\partial x_{2n+\ell}} \text{ on } D,
\end{equation}
$\ell=1,\dots, m,$ and the vector fields 
\begin{equation}\label{brtframe}
Z_j = \frac{\partial}{\partial z_j} - i\sum_{\ell=1}^m \frac{\partial \varphi_\ell}{\partial z_j}(z) \frac{\partial}{\partial x_{2n+\ell}},
\end{equation}
$j=1, \dots, n$, form a basis of $T_x^{1,0}X$ for each $x \in D$, where $\varphi_\ell \in \mathcal{C}^\infty(D, \mathbb{R})$ is independent of $x''$. By \cite{BRT85}, we can assume
\begin{equation}
    \varphi_\ell(z) = \sum_{j,s=1}^n \frac{\partial^2\varphi_\ell}{\partial z_j\partial \overline{z}_s}(0)z_j \overline{z}_s + O(|z|^3).
\end{equation}
We call $(x_1, \dots, x_{2n+m})$ canonical coordinates, $D$ a canonical coordinate patch and\\ $(D, (z, x''), \varphi)$ a BRT trivialization where $\varphi = (\varphi_1,\dots,\varphi_m)$. The frames in (\ref{brtframe}) are called BRT frames.

\begin{defn}
Let $L$ be a CR line bundle which is rigid  with respect to $\{Y^L_\ell\}_{\ell=1}^m.$ Let $\langle\,\cdot\,|\,\cdot\,\rangle_L$ be an Hermitian metric on $L.$ We say that $\langle\,\cdot\,|\,\cdot\,\rangle_L$ is a rigid Hermitian metric if for every local rigid frame $s$ of $L,$ we have $Y_\ell\,\langle s|s\rangle_L = 0$ for every $\ell=1,\dots, m.$
\end{defn}
\begin{rem}
By Lemma \ref{liftequgpact}, we can see that a rigid Hermitian metric is, in fact, an $\mathbb{R}^m$-invariant Hermitian metric on $L.$
\end{rem}

Let $(X^{2n+m},T^{1,0}X)$ be a connected, compact and oriented $(2n+m)$-dimensional CR manifold of CR-codimension $m$ endowed with an $\mathbb{R}^m$-action. Let $Y_1,\dots,Y_m$ be the infinitesimal generators of the $\mathbb{R}^m$-action induced from the standard basis for $\mathbb{R}^m.$ From now on, we say that something is rigid means that it is rigid with respect to a fixed  system of commutative CR bundle lifts  $\{Y^L_\ell\}_{\ell=1}^m.$

Let $L$ be a rigid CR line bundle over $X.$ Then there is an open cover $(U_j)_{j=1}^N$ and a family of rigid CR trivializing frames $\{s_j\}_{j=1}^N$ with each $s_j$ defined on $U_j$ and the transition functions between different rigid CR frames are rigid
CR functions. Let $L^k$ be the $k$-th tensor power of $L$. Then $\{s_j^k\}_{j=1}^N$ is a family of rigid CR
trivializing frames on each $U_j$. Let $\overline{\partial}_{b,k} : \Omega^{0,q}(X, L^k) \to \Omega^{0,q+1}(X, L^k)$ be the tangential
Cauchy-Riemann operator. Since $L^k$ is rigid CR we have $\overline{\partial}_{b,k} f = \overline{\partial}_b f_j \otimes s_j^k$, $Y_\ell f = (Y_\ell f_j) \otimes s_j^k$
for any $f = f_j \otimes s_j^k \in \Omega^{0,q}(X, L^k)$ and
\begin{align*}
 Y_\ell \overline{\partial}_{b,k} = \overline{\partial}_{b,k} Y_\ell \quad \text{on } \Omega^{0,q}(X, L^k).
\end{align*} Let $h^L$ be a rigid Hermitian metric on $L.$ The local weight of $h^L$ with respect to a local rigid CR trivializing section $s$ of $L$ over an open subset $D\subset X$ is the function $\phi\in \mathcal{C}^\infty(D; \mathbb{R})$ defined by
\begin{equation*}
|s(x)|^2_{h^L} = e^{-\phi(x)},\quad x\in D.
\end{equation*}
We denote by $\phi_j$ the weight of $h^L$ with respect to $s_j.$

\begin{defn}
Let $L$ be a rigid CR line bundle and let $h^L$ be a rigid Hermitian metric on $L$. The
curvature of $(L, h^L)$ is the Hermitian quadratic form $R^L = R^{(L,h^L)}$ on $T^{1,0}X$ defined by
\begin{align*}
 R^L_p(U, V) = \langle d(\overline{\partial}_b\phi_j - \partial_b\phi_j)(p), U \wedge \overline{V} \rangle, \quad U, V \in T^{1,0}_pX, \ p \in U_j.
\end{align*}

From the proof of Proposition 4.2 in \cite{HL12}($M^\phi_p$ is defined in Section 4 of \cite{HL12}), we have for all $p\in D\cap\widetilde{D},$
\begin{equation*}
    M^\phi_p = M^{\tilde{\phi}}_p + i\Big(\frac{Y_1 g}{g} - \frac{Y_1 \overline{g}}{\overline{g}}\Big)(p)\mathcal{L}_{\omega_1, p} - \cdots -  i\Big(\frac{Y_m g}{g}-\frac{Y_m \overline{g}}{\overline{g}}\Big)(p)\mathcal{L}_{\omega_m, p}
\end{equation*}
 for some rigid CR function $g.$
 Hence $R^L$ is a well-defined global Hermitian form, since the transition functions between different frames $s_j$ are annihilated by $Y_\ell$.
\end{defn}

\begin{defn}\label{pos}
 We say that $(L, h^L)$ is positive if there exists $r>0$ such that for all $t=(t_1,\dots,t_m)\in B(r),$ where $B(r)$ is the open ball of radius $r$ centered at the origin in $\mathbb{R}^m,$ the
associated curvature $R^L_x - 2t_1\mathcal{L}_{\omega_1,x}-\cdots-2t_m\mathcal{L}_{\omega_m,x}$ is positive definite at every $x \in X.$
\end{defn}

From Section 3 in \cite{HHL22}, in particular, Corollary 3.8 in \cite{HHL22}, we can deduce that

\begin{prop}\label{torusprop}
Let $(X, T^{1,0}X)$ be a compact connected $(2n+m)$-dimensional CR manifold of CR-codimension $m$ with a locally free transversal CR $\mathbb{R}^m$-action. Assume that $L \rightarrow X$ is a rigid positive  CR line bundle over $X.$ Then the $\mathbb{R}^m$-action on $X$ comes from a CR torus action. 
\end{prop}

\subsection{Notations}
Let $L \rightarrow X$ be a rigid positive CR line bundle over $X.$ By Proposition \ref{torusprop}, the $\mathbb{R}^m$-action on $X$ comes from a CR torus action $\mathbb T^N = (S^1)^N.$ We may assume that $\mathbb T^N$ has the coordinates $(e^{i\theta^1_1},\dots,e^{i\theta^1_{N_1}},\dots,e^{i\theta^m_1},\dots,e^{i\theta^m_{N_m}}),$ $N = N_1+\cdots +N_m.$ Let $T^1_1,\dots, T^1_{N_1}, \dots, T^m_1,\dots,T^m_{N_m}$ be the vector fields given by
\[
(T^\ell_{j})_x = \frac{\partial}{\partial\theta^\ell_j}(1,\dots,1,e^{i\theta^\ell_j},1,\dots,1)\circ x \Big|_{\theta^\ell_j=0},
\]
 $\ell=1,\dots,m$ and $j=1,\dots,N_\ell.$ 
 We also need 
\begin{prop}
Under the assumptions of Proposition \ref{torusprop}, we can find local CR rigid trivializations of $L$ defined on $D_j,\ j=1,\dots,M,$ such that $X= \bigcup_{j=1}^M D_j,$ and
\begin{align*}
    D_j = \bigcup_{t\in \mathbb T^N} t \circ D_j
\end{align*}
where $t\circ D_j = \big\{ t \circ x;\, x\in D_j\big\}.$
\end{prop}
\begin{proof}
Similar to the proof of Theorem 3.12 in \cite{HHL22}.
\end{proof}

\begin{cor}\label{invnbd}
Under the assumptions of Proposition \ref{torusprop}, we can find local CR rigid trivializations of $L$ defined on $D_j,\ j=1,\dots,M,$ such that $X= \bigcup_{j=1}^M D_j,$ and
\begin{align*}
D_j = \bigcup_{t\in \mathbb{R}^m} t \circ D_j, \quad j=1,\dots,M,
\end{align*}
where $t \circ D_j = \{t\circ x;\, x\in D_j\}.$
\end{cor}

We also have
\begin{prop}
Under the assumptions of Proposition \ref{torusprop}, let $h^L$ be any rigid Hermitian metric on $L.$ Then $h^L$ is $\mathbb T^N$-invariant.
\end{prop}
\begin{proof}
Similar to the proof of Lemma 3.14 in \cite{HHL22}.
\end{proof}
 
 We view the $\mathbb{R}^m$-action on $X$ as a Lie group homomorphism $\gamma : \mathbb{R}^m \to \text{Aut}_{\text{CR}}(X).$ Factor $\gamma$ as
 \begin{align*}
 \mathbb{R}^m \to \overline{\gamma(\mathbb{R}a_1)} \times \cdots \times \overline{\gamma(\mathbb{R}a_m)} \to \overline{\gamma(\mathbb{R}^m)} \hookrightarrow \text{Aut}_{\text{CR}}(X)
 \end{align*}
 which is given by
 \[
 t=(t_1,\dots,t_m)\mapsto (\gamma(t_1a_1),\dots,\gamma(t_ma_m))\mapsto \sum_{j=1}^m\gamma(t_ja_j),
 \]
 where $\{a_1,\dots,a_m\}$ is the standard basis for $\mathbb{R}^m.$ Note that $\overline{\gamma(\mathbb{R}a_1)}, \dots, \overline{\gamma(\mathbb{R}a_m)} , \overline{\gamma(\mathbb{R}^m)}$ are all tori. From the factorization and the Remark 4.2 in \cite{HL12},
 we can assume that there exist $d$ sets of real numbers, $\{\beta^1_1,\dots,\beta^1_{N_1}\},$ ..., $\{\beta^m_1,\dots,\beta^m_{N_m}\},$ each being linearly independent over $\mathbb Q$, such that
 \begin{equation}\label{Qli}
 Y_\ell = \sum_{j=1}^{N_\ell}\beta^\ell_jT^\ell_j, \quad \ell =1,\dots, m.
\end{equation}

Let $\langle\,\cdot\,|\,\cdot\,\rangle$ be the rigid Hermitian metric on $\mathbb{C}TX$ induced by $R^L$ such that
\begin{align*}
    T^{1,0}X \perp T^{0,1}X, \ Y_\ell\perp (T^{1,0}X \oplus T^{0,1}X), \ \langle Y_\ell \mid Y_q \rangle = \delta_{\ell q}, \quad \ell,q=1,\dots,m.
\end{align*}
The rigid Hermitian metric $\langle\,\cdot\,|\,\cdot\,\rangle$ on  $\mathbb{C}TX$ induces a rigid Hermitian metric, also denoted by $\langle\,\cdot\,|\,\cdot\,\rangle,$ on $\oplus_{q=1}^n \Lambda^q T^{*0,1}X.$ Let $dv_X$ denote the volume form induced from $\langle\,\cdot\,|\,\cdot\,\rangle.$ Let $(\,\cdot\,|\,\cdot\,)_k$ be the $L^2$ inner product on $\Omega^{0,q}(X,L^k)$ induced by $h^{L^k}$ and $dv_X$ and let $\|\cdot\|_k$ be the corresponding norm. Let $L^2_{(0,q)}(X,L^k)$ be the completion of $\Omega^{0,q}(X,L^k)$ with respect to  $(\,\cdot\,|\,\cdot\,)_k.$
For $(n^\ell_1,\dots,n^\ell_{N_\ell})\in {\mathbb Z}^{N_\ell},$ let 
\begin{align*}
    \Omega^{0,q}_{\ell, (n^\ell_1,\dots,n^\ell_{N_\ell})}(X, L^k) 
    &= \{u\in \Omega^{0,q}(X, L^k);\, T^\ell_ju = in^\ell_ju,\, j=1,\dots, N_\ell \}\\
    &= \{u\in \Omega^{0,q}(X, L^k);\, u((1,\dots,1,e^{i\theta^\ell_j},1,\dots,1)\circ x) = e^{in^\ell_j\theta^\ell_j}u,\\ 
    &\quad\quad \theta^\ell_j\in [0,2\pi],\, j=1,\dots, N_\ell \}.
\end{align*}
Let $L^{2,(0,q)}_{\ell, (n^\ell_1,\dots,n^\ell_{N_\ell})}(X, L^k)$ be the $L^2$-completion of $\Omega^{0,q}_{\ell, (n^\ell_1,\dots,n^\ell_{N_\ell})}(X, L^k)$ with respect to $(\,\cdot\,|\,\cdot\,)_k.$ Let 
\begin{align*}
&Q^{(q),k}_{\ell, (n^\ell_1,\dots,n^\ell_{N_\ell})} : L^2_{(0,q)}(X,L^k) \rightarrow L^{2,(0,q)}_{\ell, (n^\ell_1,\dots,n^\ell_{N_\ell})}(X, L^k),\\ 
&Q^{(q),k}_{ (n^1_1,\dots,n^m_{N_m})} : L^2_{(0,q)}(X,L^k) \rightarrow \bigcap_{\ell=1}^m L^{2,(0,q)}_{\ell, (n^\ell_1,\dots,n^\ell_{N_\ell})}(X, L^k)
\end{align*}
be the natural orthogonal projection with respect to $(\,\cdot\,|\,\cdot\,)_k.$ By the elementary Fourier analysis, we have for every $u\in \Omega^{0,q}(X,L^k),$
\begin{align*}
    & \sum_{\mathbf{n}^\ell = (n^\ell_1,\dots,n^\ell_{N_\ell})\in\mathbb{Z}^{N_\ell}, |\mathbf{n}^\ell| \leq M} Q^{(q),k}_{\ell, (n^\ell_1,\dots,n^\ell_{N_\ell})} u \to u \ \text{as $M\to\infty$ in the} \ \mathcal{C}^\infty \ \text{topology},\\
    & \sum_{\mathbf{n}^\ell = (n^\ell_1,\dots,n^\ell_{N_\ell})\in\mathbb{Z}^{N_\ell}, |\mathbf{n}^\ell| \leq M} \big\|Q^{(q),k}_{\ell, (n^\ell_1,\dots,n^\ell_{N_\ell})} u\big\|_k^2 \leq \|u\|_k^2,\ \text{for all} \ M\in \mathbb{N}
\end{align*}
and
\begin{align*}
    & \sum_{\mathbf{n} = (n^1_1,\dots,n^m_{N_m})\in\mathbb{Z}^{N}, |\mathbf{n}| \leq M} Q^{(q),k}_{ (n^1_1,\dots,n^m_{N_m})} u \to u \ \text{as $M\to\infty$ in the} \ \mathcal{C}^\infty \ \text{topology},\\
    & \sum_{\mathbf{n} = (n^1_1,\dots,n^m_{N_m})\in\mathbb{Z}^{N}, |\mathbf{n}| \leq M} \big\|Q^{(q),k}_{ (n^1_1,\dots,n^m_{N_m})} u\big\|_k^2 \leq \|u\|_k^2,\ \text{for all} \ M\in \mathbb{N}.
\end{align*}
Thus for $u\in L^2_{(0,q)}(X,L^k),$
\begin{align*}
    &\sum_{\mathbf{n}^\ell = (n^\ell_1,\dots,n^\ell_{N_\ell})\in\mathbb{Z}^{N_\ell}, |\mathbf{n}^\ell| \leq M} Q^{(q),k}_{\ell, (n^\ell_1,\dots,n^\ell_{N_\ell})} u \to u \ \text{as $M\to\infty$ in} \ L^2_{(0,q)}(X,L^k),\\
    & \sum_{\mathbf{n}^\ell = (n^\ell_1,\dots,n^\ell_{N_\ell})\in\mathbb{Z}^{N_\ell}, |\mathbf{n}^\ell| \leq M} \big\|Q^{(q),k}_{\ell, (n^\ell_1,\dots,n^\ell_{N_\ell})} u\big\|_k^2 \leq \|u\|_k^2,\ \text{for all} \ M\in \mathbb{N}
\end{align*}
and
\begin{align*}
    & \sum_{\mathbf{n} = (n^1_1,\dots,n^m_{N_m})\in\mathbb{Z}^{N}, |\mathbf{n}| \leq M} Q^{(q),k}_{ (n^1_1,\dots,n^m_{N_m})} u \to u \ \text{as $M\to\infty$ in} \ L^2_{(0,q)}(X,L^k),\\
    & \sum_{\mathbf{n} = (n^1_1,\dots,n^m_{N_m})\in\mathbb{Z}^{N}, |\mathbf{n}| \leq M} \big\|Q^{(q),k}_{ (n^1_1,\dots,n^m_{N_m})} u\big\|_k^2 \leq \|u\|_k^2,\ \text{for all} \ M\in \mathbb{N}.
\end{align*}
Let $B(r)$ be as in Definition \ref{pos}. We may take $\delta > 0$ such that the box $[-\delta,\delta]^m\subset B(r).$ For every $\alpha\in \text{Spec}(-iY_\ell),\ \ell = 1,\dots, m,$ put
\[
\Omega^{0,q}_{Y_\ell,\alpha}(X,L^k) := \{ u\in \Omega^{0,q}(X, L^k);\, -iY_\ell u = \alpha u \},
\]
and let $L^{2,(0,q)}_{Y_\ell,\alpha}(X,L^k)\subset L^2_{(0,q)}(X,L^k)$ be the eigenspace of $-iY_\ell$ with eigenvalue $\alpha.$ We can check that $L^{2,(0,q)}_{Y_\ell,\alpha}(X,L^k)$ is the completion of $\Omega^{0,q}_{Y_\ell,\alpha}(X,L^k)$ with respect to $(\,\cdot\,|\,\cdot\,)_k.$

Put
\begin{equation}\label{ylcrfun}
\mathcal{H}^{0,q}_{b,(Y_\ell,\alpha)}(X,L^k) := \{u\in L^{2,(0,q)}_{Y_\ell,\alpha}(X,L^k);\, \overline{\partial}_b u = 0 \},
\end{equation}
and, for $\boldsymbol{\lambda} = (\lambda_1,\dots,\lambda_m)\in (\mathbb{R}_{>0})^m,$ put
\begin{align*}
\Omega^{0,q}_{Y,\leq\boldsymbol{\lambda}}(X,L^k) &:= \bigoplus_{\alpha_\ell\in \text{Spec}(-iY_\ell), |\alpha_\ell|\leq \lambda_\ell, 1\leq\ell\leq m} \left(\bigcap_{1\leq\ell\leq m}\Omega^{0,q}_{Y_\ell,\alpha_\ell}(X,L^k)\right),\\
L^{2,(0,q)}_{Y,\leq\boldsymbol{\lambda}}(X,L^k) &:= \bigoplus_{\alpha_\ell\in \text{Spec}(-iY_\ell), |\alpha_\ell|\leq \lambda_\ell, 1\leq\ell\leq m} \left(\bigcap_{1\leq\ell\leq m}L^{2,(0,q)}_{Y_\ell,\alpha_\ell}(X,L^k)\right),\\
\mathcal{H}^{0,q}_{b,\leq\boldsymbol{\lambda}}(X,L^k) &:= \bigoplus_{\alpha_\ell\in \text{Spec}(-iY_\ell), |\alpha_\ell|\leq \lambda_\ell, 1\leq\ell\leq m} \left(\bigcap_{1\leq\ell\leq m}\mathcal{H}^{0,q}_{b,(Y_\ell,\alpha_\ell)}(X,L^k)\right).
\end{align*}
If each $\lambda_\ell = k\delta,$ we write
\begin{align*}
&\Omega^{0,q}_{Y,\leq k\delta}(X,L^k) := \Omega^{0,q}_{Y,\leq(k\delta,\dots,k\delta)}(X,L^k),\quad
L^{2,(0,q)}_{Y,\leq k\delta}(X,L^k) := L^{2,(0,q)}_{Y,\leq(k\delta,\dots,k\delta)}(X,L^k),\\
&\mathcal{H}^{0,q}_{b,\leq k\delta}(X,L^k) := \mathcal{H}^{0,q}_{b,\leq (k\delta,\dots,k\delta)}(X,L^k).
\end{align*}

Let
\begin{equation}
    Q^{(q)}_{(Y_\ell,\alpha),k} : L^2_{(0,q)}(X,L^k) \rightarrow L^{2,(0,q)}_{Y_\ell,\alpha}(X,L^k)
\end{equation}
be the orthogonal projection with respect to $(\,\cdot\,|\,\cdot\,)_k.$

Fix $\delta > 0$ and a function
\begin{equation*}
    \tau\in \mathcal{C}^\infty_c((-\delta,\delta)), 0\leq \tau \leq 1, \tau = 1\  \text{on} \left[-\frac{\delta}{2},\frac{\delta}{2}\right].
\end{equation*}
Let $Q^{(0)}_{X,\leq k\delta}:L^{2}(X,L^{k})\rightarrow L^{2}(X,L^{k})$ and $F_{k,\delta}:L^{2}(X,L^{k})\rightarrow L^{2}(X,L^{k})$ be the bounded operators given by, for all $u\in L^2(X,L^k)$
\begin{equation}\label{Q^0}
Q^{(0)}_{X,\leq k\delta}(u)=\sum_{\alpha_\ell\in \text{Spec}(-iY_\ell), |\alpha_\ell|\leq k\delta, 1\leq\ell \leq m}Q^{(0)}_{(Y_m,\alpha_m),k}\cdots Q^{(0)}_{(Y_1,\alpha_1),k}(u),
\end{equation}
and
\begin{equation}\label{F_k}
F_{k,\delta}(u)=\sum_{\alpha_\ell\in \text{Spec}(-iY_\ell), |\alpha_\ell|\leq k\delta, 1\leq\ell \leq m}\tau\Big(\frac{\alpha_m}{k}\Big)\cdots\tau\Big(\frac{\alpha_1}{k}\Big)Q^{(0)}_{(Y_m,\alpha_m),k}\cdots Q^{(0)}_{(Y_1,\alpha_1),k}(u).
\end{equation}
We will use the operator $F_{k,\delta}$ in the proof of our embedding theorem.

\medskip

For every $\lambda>0,$ we consider the partial Szeg\H{o} projector
\begin{equation}
\Pi^{(0)}_{k,({\mathbb R}^m\curvearrowright,\leq \lambda)}:L^{2}(X,L^{k})\rightarrow\mathcal{H}_{b,\le \lambda}^{0}(X,L^{k})
\end{equation}
which is the orthogonal projection on the space of $\mathbb{R}^m$-equivariant CR functions of degree less than $\lambda.$ Here we denote $\mathcal{H}_{b,\le \lambda}^{0,0}(X,L^{k})$ by $\mathcal{H}_{b,\le \lambda}^{0}(X,L^{k})$. Note that we have 
\begin{equation*}
\Pi^{(0)}_{k,({\mathbb R}^m\curvearrowright,\leq \lambda)} = Q^{(0)}_{X,\leq \lambda} \circ \Pi^{(0)}_{b,k} = \Pi^{(0)}_{b,k} \circ Q^{(0)}_{X,\leq \lambda}.
\end{equation*}

In this subsection, we list some propositions about the operators $-iY_1,\dots,-iY_m$ without proofs. We can follow the arguments in Section 4 of \cite{HHL22} with some modifications to check these properties.

\begin{lem}
The operator $-iY_\ell : \text{Dom} (-iY_\ell) \subset L^2_{(0,q)}(X,L^k) \to L^2_{(0,q)}(X,L^k)$ is self-adjoint for all $\ell = 1,\dots, m.$
\end{lem}

\begin{lem}
Fix $\ell\in\{1,\dots,m\}.$ Let $\beta^\ell_1,\dots,\beta^\ell_{N_\ell}$ be as in (\ref{Qli}). Given $p = (p^\ell_1,\dots,p^\ell_{N_\ell}) \in \mathbb{Z}^{N_\ell},$ we have that $p_{\beta^\ell} := \sum_{j=1}^{N_\ell} p_j \beta^\ell_j$ is an eigenvalue of $-iY_\ell$ and the corresponding eigenspace is $L^{2,(0,q)}_{\ell, (p^\ell_1,\dots,p^\ell_{N_\ell})}(X, L^k).$
\end{lem}

\begin{thm}
For every $\ell=1,\dots,m,$ $\text{Spec}(-iY_\ell)$ is countable and every element in $\text{Spec}(-iY_\ell)$ is an eigenvalue of $-iY_\ell.$ Moreover, for every $\alpha\in  \text{Spec}(-iY_\ell),$ we can find $(m_1,\dots,m_{N_\ell})\in \mathbb{Z}^{N_\ell}$ such that $\alpha = \sum_{j=1}^{N_\ell} \beta^\ell_j m_j,$ where $\beta^\ell_1,\dots,\beta^\ell_{N_\ell}$ are as in (\ref{Qli}).
\end{thm}

Thus, together with the ellipticity of $\Box_{b,k} - Y_1^2 - \cdots - Y_m^2,$ we have
\begin{thm}\label{findim}
For every $\boldsymbol{\lambda} \in (\mathbb{R}_{>0})^m,$ 
\begin{align*}
\mathcal{H}^{0,q}_{b,\leq\boldsymbol{\lambda}}(X,L^k)\subset\Omega^{0,q}(X, L^k)\quad \text{and}\quad  \dim\mathcal{H}^{0,q}_{b,\leq\boldsymbol{\lambda}}(X,L^k) < \infty.
\end{align*}
\end{thm}

\subsection{Asymptotics of partial Szeg\H{o} kernel} 

Fix a point $p\in X$ and take a BRT frame, $(D, (z, x''), \varphi),$ around $p.$ Let $s$ be a local rigid CR section of $L$ over $D.$ Write $|s|^2_{h^L} = e^{-\phi},\ \phi \in \mathcal{C}^\infty(D;\mathbb R).$ After a holomorphic change of variables if necessary, we may assume
\begin{equation}\label{libunwe}
    \phi(z) = \sum_{j,r=1}^{n}\mu_{jr}z_j\overline{z}_r  + O(|z|^3),
\end{equation}
which is independent of $x''.$ We have an isometry
\begin{equation}\label{iso}
    U_{k,s} : L^2_{(0,q)}(D) \rightarrow L^2_{(0,q)}(D, L^k),\quad u \mapsto ue^{\frac{k}{2}\phi}s^k,
\end{equation}
with inverse $U_{k,s}^{-1} : L^2_{(0,q)}(D, L^k)\rightarrow L^2_{(0,q)}(D),\ \alpha \mapsto e^{-\frac{k}{2}\phi}s^{-k}\alpha.$

For $u \in \Omega^{0,q}(D) \cap L^2_{(0,q)}(D)$, we define
\begin{align*}
\overline{\partial}_{b, k\phi} &= \overline{\partial}_b + \frac{k}{2} \overline{\partial}_b \phi = \sum_{j=1}^n e_j \wedge \Big(\overline{Z}_j + \frac{k}{2}(\overline{Z}_j \phi)\Big) + \sum_{j=1}^n (\overline{\partial}_b e_j) \wedge (e_j \wedge)^*, \\
\overline{\partial}_{b, k\phi}^* &= \sum_{j=1}^n (e_j \wedge)^* \Big(\overline{Z}_j^* + \frac{k}{2}(Z_j \phi)\Big) + \sum_{j=1}^n (e_j \wedge) (\overline{\partial}_b e_j \wedge)^*,\\
\Box^{(q)}_{b,k\phi} &= \overline{\partial}_{b, k\phi}\overline{\partial}_{b, k\phi}^* + \overline{\partial}_{b, k\phi}^*\overline{\partial}_{b, k\phi} : \Omega^{0,q}(D) \rightarrow \Omega^{0,q}(D)
\end{align*}
where $\overline{Z}_j^*$ is the formal adjoint of $\overline{Z}_j$ with respect to $(\,\cdot\,|\,\cdot\,)_\omega$ and $\overline{\partial}_{b, k\phi}^*$ is the formal adjoint of $\overline{\partial}_{b, k\phi}$ with respect to $(\,\cdot\,|\,\cdot\,)_\omega.$ Here $(\,\cdot\,|\,\cdot\,)_\omega$ is the inner product induced by the Hermitian metric on $X.$

\begin{prop}\label{gap}
We have 
\begin{align*}
&\Box_{b, k\phi}^{(q)}\\
&= \sum_{j=1}^n \Big(\overline{Z}_j^* + \frac{k}{2} Z_j \phi\Big)\Big(\overline{Z}_j + \frac{k}{2} \overline{Z}_j \phi\Big)+\sum_{j, t=1}^n (e_j \wedge) (e_t \wedge)^* (k M^{\phi} (Z_t, \overline{Z}_j)) \\
&\quad + \sum_{j, t=1}^n (e_j \wedge) (e_t \wedge)^* \Big(2i \sum_{\ell=1}^m \mathcal{L}_{\omega_\ell}(Z_t,\overline{Z}_j) Y_\ell\Big)  + \sum_{r=1}^n \left( \sum_{j, t=1}^n (e_j \wedge) (e_t \wedge)^* a^{t j}_r+ \sum_{j=1}^n \hat{\varepsilon}_j \right) \Big(\overline{Z}_r^* + \frac{k}{2} Z_r \phi\Big)\\ 
&\quad+ \sum_{r=1}^n \left( \sum_{j, t=1}^n (e_j \wedge) (e_t \wedge)^* b^{t j}_r + \sum_{j=1}^n \varepsilon_j \right) \Big(\overline{Z}_r + \frac{k}{2} \overline{Z}_r \phi\Big) + \varepsilon,
\end{align*}
where $a^{t j}_r, b^{t j}_r, \varepsilon_j, \hat{\varepsilon}_j$  and $\varepsilon$ are zero order operators.
\end{prop}

\begin{cor}
With the notations used above, let $q\in\{1,\dots,n\}.$ For $\delta$ small enough, there is a constant $C>1$ independent of $k$ such that
\begin{align}\label{localspgap}
   (\Box^{(q)}_{b,k}u \mid u)_k \geq \Big(\frac{k}{C}-C\Big)\|u\|^2_k
\end{align}
for all $u\in\Omega^{0,q}_{Y,\leq k\delta}(D,L^k)\cap L^2_{(0,q)}(D, L^k)$  and all $k\geq1.$
\end{cor}
\begin{proof}
By our assumption that $L$ is positive, we have that the curvature term $M^{\phi}$ is positive definite. Note that for all $u\in\Omega^{0,q}_{Y,\leq k\delta}(D,L^k)\cap L^2_{(0,q)}(D, L^k),$ we have $|(Y_\ell u \mid u)_k|\lesssim k\delta\|u\|_k^2.$ Hence, by Proposition \ref{gap} and the isometry (\ref{iso}), it is easy to see that for $\delta$ small enough, we have (\ref{localspgap}) for some positive constant $C$ independent of $k.$
\end{proof}

\begin{cor}
With the notations used above, let $q\in\{1,\dots,n\}.$ For $\delta$ small enough, there is a $k_0>0$ such that for all $k\geq k_0,$ we have
\begin{equation*}
    (\Box^{(q)}_{b,k}u \mid u)_k \geq Ck\|u\|^2_k
\end{equation*}
for all $u\in\Omega^{0,q}_{Y,\leq k\delta}(X,L^k),$ where $C$ is a positive constant independent of $k.$
\end{cor}

By using the standard functional analysis argument on $L^{2,(0,q)}_{Y,\leq\lambda}(X,L^k),$ we obtain
\begin{cor}\label{grspgap}
Fix $\delta$ small enough. For $k\gg1,$
\begin{equation*}
\Big\|\Big(I - \Pi^{(0)}_{k,({\mathbb R}^m\curvearrowright,\leq k\delta)}\Big)u\Big\|_k^2 \leq Ck^{-1}(\Box_{b,k}u \mid u)_k
\end{equation*}
for all $u\in \mathcal{C}^\infty_{Y,\leq k\delta}(X,L^k),$ where $C$ is a positive constant independent of $k.$
\end{cor}

Since $\mathcal{H}_{b,\le k\delta}^{0}(X,L^{k})\subset \mathcal{C}^\infty(X,L^k)$ and $\dim \left(\mathcal{H}_{b,\le k\delta}^{0}(X,L^{k})\right)<\infty,$ we have $\Pi^{(0)}_{k,({\mathbb R}^m\curvearrowright,\leq k\delta)}$ is a smoothing operator with the smooth kernel $P_{k,\delta}(x,y)\in \mathcal{C}^\infty(X\times X, L^k_x\boxtimes L^{{*}k}_y).$ Fix a point $p\in X$ and take a BRT frame, $(D, (z, x''), \varphi),$ around $p.$ Let $\{f_1,\dots,f_{d_k}\}$ be an orthonormal basis of $\mathcal{H}_{b,\le k\delta}^{0}(X,L^{k}).$ Then
\begin{equation}
    P_{k,\delta}(x,y) = \sum_{j=1}^{d_k} f_j(x)\otimes(f_j(y))^*.
\end{equation}

On $D,$ write $f_j(x) = \hat{f}_j\otimes s^k,\ \hat{f}_j\in \mathcal{C}^\infty(D),\ j=1,\dots,d_k.$ We can express $P_{k,\delta}(x,y)$ on $D$ as
\begin{equation}\label{locexp}
    P_{k, \delta,s}(x,y) = \sum_{j=1}^{d_k} e^{-\frac{k}{2}\phi(x)} \hat{f}_j(x)\overline{\hat{f}_j(y)}e^{-\frac{k}{2}\phi(y)} \in \mathcal{C}^\infty(D\times D).
\end{equation}

Let 
\begin{equation}
\begin{cases} 
\phi(x) = \psi_0(x) + \psi_1(x), \\ 
\psi_0(x) = \sum_{j,r=1}^{n}\mu_{jr}z_j\overline{z}_r, \\ 
\psi_1(x) = O(|z|^3). 
\end{cases} 
\end{equation}

Put
\begin{equation}
P_{(k),\delta,s}(x,y) := k^{-n-m} P_{k,\delta,s} \left( \Big(\frac{z}{\sqrt{k}},\frac{x''}{k}\Big), \Big(\frac{w}{\sqrt{k}},\frac{y''}{k}\Big) \right) \in \mathcal{C}^\infty(D(k^\epsilon,k^{2\epsilon})\times D(k^\epsilon,k^{2\epsilon})). 
\end{equation}

\begin{thm}\label{sup}
With the notations used above, let $K \subset D(k^\epsilon,k^{2\epsilon})$ be any compact subset. For every $M \in \mathbb{N}$, we can find a constant $C_{M,K} > 0$ independent of $k$ such that
\begin{equation}
\sup\{|\partial^\alpha_x\partial^\beta_yP_{(k),\delta,s}(x,y)|:\alpha,\beta \in ({\mathbb Z}_{\geq0})^{2n+m}, |\alpha|+|\beta|\leq M, (x,y)\in K\times K \}\leq C_{M,K}. 
\end{equation}
\end{thm}

\begin{proof}
Let $\{f_1,\dots,f_{d_k}\}$ be an orthonormal basis of $\mathcal{H}_{b,\le k\delta}^{0}(X,L^{k}).$ On $D,$ write $f_j(x) = \hat{f}_j\otimes s^k,\ \hat{f}_j\in \mathcal{C}^\infty(D),\ j=1,\dots,d_k.$ Fix $y_0\in K.$ Let $\beta \in ({\mathbb Z}_{\geq0})^{2n+m},\ |\beta|\leq M.$ We may suppose that
\begin{equation*}
\sum_{j=1}^{d_k}\bigg|\partial^\beta_y\bigg(\overline{\hat{f}_j\Big(\frac{w}{\sqrt{k}},\frac{y''}{k}\Big)}e^{-\frac{k}{2}\phi(\frac{w}{\sqrt{k}})}\bigg)(y_0)\bigg|^2 \neq 0.
\end{equation*}
Let
\begin{equation*}
u_{k}(x) = \frac{\sum_{j=1}^{d_{k}} f_{j}(x) \partial_{y}^{\beta} \Big( \overline{\hat{f}_{j}(\frac{w}{\sqrt{k}},\frac{y''}{k})} e^{-\frac{k}{2}\phi(\frac{w}{\sqrt{k}})} \Big) (y_{0})}{\Big( \sum_{j=1}^{d_{k}} \Big| \partial_{y}^{\beta} \Big( \overline{\hat{f}_{j}(\frac{w}{\sqrt{k}},\frac{y''}{k})} e^{-\frac{k}{2}\phi(\frac{w}{\sqrt{k}})} \Big) (y_{0}) \Big|^{2} \Big)^{1/2}} \in \mathcal{H}_{b,\le k\delta}^{0}(X,L^{k}).
\end{equation*}
We can verify that $\|u_{k}\|_{k} = 1$. Recall the scaling of $u_{k}$:
\begin{equation*}
u_{(k)}(x) = k^{-(n+m)/2} u_{k} \Big( \frac{z}{\sqrt{k}},\frac{x''}{k} \Big) \in C^{\infty}(D(k^\epsilon,k^{2\epsilon})).
\end{equation*}

 By using Kohn's estimates and Sobolev's embedding theorem, we have
\begin{equation}\label{alpha}
\sup \Big\{ \Big| \partial_{x}^{\alpha} \Big( u_{(k)}(x) e^{-\frac{k}{2}\phi(\frac{z}{\sqrt{k}})} \Big) \Big| \mid x \in K, \alpha \in ({\mathbb Z}_{\geq0})^{2n+m}, |\alpha| \le M \Big\} \le \hat{C}_{K,M},
\end{equation}
where $\hat{C}_{K,M}$ is a constant independent of $u_{k}$, $k$, and $y_{0} \in K$. From (\ref{alpha}), we obtain
\begin{align}\label{alpha2}
&\sup \big\{ |(\partial_{x}^{\alpha} \partial_{y}^{\beta} P_{(k),\delta,s})(x, y_{0})| \mid x \in K, \alpha,\beta \in ({\mathbb Z}_{\geq0})^{2n+m}, |\alpha| + |\beta| \le M \big\} \\
&\le \hat{C}_{K,M} {\bigg( \sum_{j=1}^{d_{k}} \Big| \partial_{y}^{\beta} \Big( \overline{\hat{f}_{j}\big(\frac{w}{\sqrt{k}},\frac{y''}{k}\big)} e^{-\frac{k}{2}\phi(\frac{w}{\sqrt{k}})} \Big) (y_{0}) \Big|^{2} \bigg)^{1/2}}. \nonumber
\end{align}
 Similarly, we can show that
\begin{equation}\label{beta}
\sum_{j=1}^{d_{k}} \Big| \partial_{y}^{\beta} \Big( \overline{\hat{f}_{j}\big(\frac{w}{\sqrt{k}},\frac{y''}{k}\big)} e^{-\frac{k}{2}\phi(\frac{w}{\sqrt{k}})} \Big) (y_{0}) \Big|^{2}  \le C, 
\end{equation}
where $C > 0$ is a constant independent of $k$ and $y_{0}$. The theorem follows from (\ref{alpha2}) and (\ref{beta}).
\end{proof}

From Theorem \ref{sup} and Arzel\`a-Ascoli Theorem, we obtain:

\begin{thm}\label{smoothcon}
There exists a subsequence $1 \le k_{1} < k_{2} < \dots,$ with $\lim_{j \to +\infty} k_{j} = +\infty$ and $P(x,y)\in \mathcal{C}^\infty(\mathbb{H}^{n,m}\times \mathbb{H}^{n,m})$ such that for any compact set $K\subset \mathbb{H}^{n,m}$ and every $\alpha,\beta\in ({\mathbb Z}_{\geq0})^{2n+m},$ we have 
\begin{equation*}
    \lim_{j\rightarrow\infty} \sup \{\big|\partial^\alpha_x\partial^\beta_y(P_{(k_j),\delta,s}-P)(x,y)\big|\mid (x,y)\in K\times K\} = 0.
\end{equation*}
\end{thm}

Let $\lambda^\ell_{j,r} = \frac{\partial^2\varphi_\ell}{\partial z_j\partial \overline{z}_r}(0).$ As in Section \ref{heisenberggpcrstr}, we now consider the following CR-structure on the Heisenberg group $\mathbb{H}^{n,m}:$
\[
T^{1,0}\mathbb{H}^{n,m} := \text{Span}_{\mathbb{C}}\bigg\{Z_{j,\mathbb{H}^{n,m}} = \frac{\partial}{\partial z_j} - i\sum_{\ell=1}^m\Big(\sum_{r=1}^{n}\lambda^{\ell}_{jr}\overline{z}_r\Big)\frac{\partial}{\partial x_{2n+\ell}};\, j =1,\dots,n\bigg\}.
\]

Fix $\delta>0$ small such that $(\mu_{jr}) - 2\sum_{\ell=1}^m\eta_\ell(\lambda^{\ell}_{jr})$ is positive definite for all $|\eta_\ell|\leq\delta, \ell = 1,\dots, m.$ Let $M_\ell = (\lambda^{\ell}_{jr}), \ell = 1,\dots,m,$ be $m$ Hermitian matrices. For $\eta = (\eta_1,\dots,\eta_m)\in {\mathbb R}^m,$ let $M_\eta = (\mu_{jr}) - 2\sum_{\ell=1}^m\eta_\ell M_\ell.$
We consider the integral operator $\Pi^{(0),\psi_0}_{b,\leq\delta}$ on $L^2(\mathbb{H}^{n,m}, \psi_0)$ which is given by, for $u\in L^2(\mathbb{H}^{n,m}, \psi_0),$
\begin{align}\label{grszeasy}
\big(\Pi^{(0),\psi_0}_{b,\leq\delta}u\big)(x) &= \Big(\frac{1}{2\pi}\Big)^{n+m} \int_{[-\delta,\delta]^m}\int_{\mathbb{C}^n}\int_{\mathbb{R}^m} \det(M_\eta) \exp\bigg[i\langle\eta,x''\rangle-i\langle\eta,y''\rangle +\sum_{j,r=1}^{n}(M_\eta)_{jr}z_j\overline{w}_r\\
&\quad -\sum_{j,r=1}^{n}(M_\eta)_{jr}w_j\overline{w}_r -\sum_{\ell=1}^m\eta_\ell\sum_{j,r=1}^{n}\lambda^{\ell}_{jr}w_j\overline{w}_r+\sum_{\ell=1}^m\eta_\ell\sum_{j,r=1}^{n}\lambda^{\ell}_{jr}z_j\overline{z}_r\bigg]\nonumber\\
&\quad\cdot u(w,y'')d\mu(y'')d\mu(w)d\mu(\eta).\nonumber
\end{align}

$\Pi^{(0),\psi_0}_{b,\leq\delta}$ has the smooth integral kernel 
\begin{align}
P_{\delta}(x,y) &:= \Big(\frac{1}{2\pi}\Big)^{n+m} \int_{[-\delta,\delta]^m} \exp\bigg[i\langle\eta,x''\rangle-i\langle\eta,y''\rangle +\sum_{j,r=1}^{n}(M_\eta)_{jr}z_j\overline{w}_r-\sum_{j,r=1}^{n}(M_\eta)_{jr}w_j\overline{w}_r\\ 
&\quad -\sum_{\ell=1}^m\eta_\ell\sum_{j,r=1}^{n}\lambda^{\ell}_{jr}w_j\overline{w}_r+\sum_{\ell=1}^m\eta_\ell\sum_{j,r=1}^{n}\lambda^{\ell}_{jr}z_j\overline{z}_r\bigg]\det(M_\eta) d\mu(\eta).\nonumber
\end{align}
 We can check that $\Pi^{(0),\psi_0}_{b,\leq\delta}$ is the orthogonal projection of $L^2(\mathbb{H}^{n,m}, \psi_0)$ onto the space $E$ where
\begin{equation*}
    E := \{ u \in \operatorname{Ker}(\overline{\partial}_{b, \mathbb{H}^{n,m}});\, \hat{u}(z,\eta) = 0 \ \text{if} \ |\eta_\ell| \geq \delta,\ \ell = 1,\dots, m\}.
\end{equation*}
We can now prove Theorem \ref{3rdmain}. 
\begin{thm}\label{nongrasy}
The sequence $P_{(k),\delta,s}(x,y)$ converges to $e^{-\frac{1}{2}\psi_0(x)}P_\delta(x,y)e^{\frac{1}{2}\psi_0(y)}$ locally uniformly in the $\mathcal{C}^\infty$-topology.
\end{thm}
Note that
\begin{align}\label{nonwegrszeker}
&e^{-\frac{1}{2}\psi_0(x)}P_\delta(x,y)e^{\frac{1}{2}\psi_0(y)}\\
&= \Big(\frac{1}{2\pi}\Big)^{n+m} \int_{[-\delta,\delta]^m} \exp\bigg[i\langle\eta,x''\rangle-i\langle\eta,y''\rangle-\frac{1}{2}\sum_{j,r=1}^{n}(M_\eta)_{jr}(z_j-w_j)(\overline{z}_r-\overline{w}_r)\nonumber\\
&\quad\quad\quad\quad\quad\quad+\frac{1}{2}\sum_{j,r=1}^{n}(M_\eta)_{jr}(z_j\overline{w}_r-\overline{z}_r w_j)\bigg]\det(M_\eta)d\mu(\eta).\nonumber
\end{align}
We denote the integral operator determined by the smooth kernel $e^{-\frac{1}{2}\psi_0(x)}P_\delta(x,y)e^{\frac{1}{2}\psi_0(y)}$ by $\Pi^{(0)}_{b,\psi_0,\leq\delta}.$

By Theorem \ref{smoothcon}, we only need to show that $\Pi^{(0)}_{k,({\mathbb R}^m\curvearrowright,\leq k\delta)}$ converges to $\Pi^{(0),\psi_0}_{b,\leq\delta}$ in the sense of distributions. Following the arguments as in Section \ref{scaledszegokernelasy}, we only need to show the following lemma. We follow the notation used in Section \ref{scaledszegokernelasy}.
\begin{lem}\label{gdiscon}
 Fix $u\in \mathcal{C}^\infty_c(\mathbb{H}^{n,m}).$ Let $\ddot{u}= \Pi^{(0),\psi_0}_{b,\leq \delta}u.$ Let $\ddot{u}_{(k)} =\chi\big(\frac{z}{k^{\epsilon}},\frac{x''}{k^{2\epsilon}}\big)\big(\Pi^{(0),\psi_0}_{b,\leq \delta}u\big)$ and $\ddot{u}_{[k]} = k^{\frac{n+m}{2}}\ddot{u}_{(k)}(\sqrt{k}z,kx'')s^k.$ Then for all $v\in \mathcal{C}^\infty_c(\mathbb{H}^{n,m}),$ we have
\begin{equation}
    \lim_{k\rightarrow\infty}\left| ( \Pi^{(0)}_{k,({\mathbb R}^m\curvearrowright,\leq k\delta)} \ddot{u}_{[k]} - \ddot{u}_{[k]}\mid v_{[k]} )_k \right| = 0.
\end{equation}
\end{lem}

To prove this lemma, we need to introduce some operators on $\mathcal{C}^\infty(D_0), D_0\subset\subset D.$ Let $\chi(y'')\in \mathcal{C}^\infty_c\big((-\epsilon,\epsilon)^m\big)$ such that $\chi(y'') = 1$ for every $(y',y'')\in D_0.$ 

Define $\hat{Q}^{(0)}_{\leq k\delta } : \Omega_c(D, L^k) \to \Omega(D, L^k)$ by
\begin{equation}
\Big(\hat{Q}^{(0)}_{\leq k\delta } u\Big)(x) = \frac{1}{(2\pi)^m} \int e^{i\langle x'' - y'',\eta\rangle} 1_{[-k\delta , k\delta]^m}(\eta)\chi(y'') u(z, y'') d\mu(y'') d\mu(\eta).
\end{equation}

\begin{lem}\label{Q+R}
We have $Q^{(0)}_{X, \leq k\delta} u = \hat{Q}^{(0)}_{\leq k\delta} u + \hat{R}_{\leq k\delta} u$ on $D_0$, where
\begin{align*} 
&(2\pi)^m\big(\hat{R}_{\leq k\delta} u\big)(x)\\
&= \sum_{(m^1_1,\dots,m^m_{N_m})\in {\mathbb Z}^N} \int_{\mathbb T^N} e^{i \langle x'' - y'', \eta\rangle + i\big(\sum_{j=1}^{N_1}m^1_j\beta^1_j\big)y_{2n+1}+\cdots+i\big(\sum_{j=1}^{N_m}m^m_j\beta^m_j\big)y_{2n+m} - i m^1_1\theta^1_1-\cdots- i m^m_{N_m}\theta^m_{N_m}}\\ 
&\qquad\qquad\qquad\qquad\qquad\cdot 1_{[-k\delta , k\delta]^m}(\eta)(1-\chi(y'')) u\big((e^{i\theta^1_1},\dots,e^{i\theta^m_{N_m}})\circ x'\big) d\mathbb{T}_N d\mu(y'') d\mu(\eta).
\end{align*}
Here, $x'$ denotes $(x',0)$ and $d\mathbb{T}_N = \big(\frac{1}{2\pi}\big)^Nd\theta^1_1\wedge\cdots\wedge d\theta^1_{N_1}\wedge\cdots\wedge d\theta^m_1\wedge\cdots\wedge d\theta^m_{N_m}.$
\end{lem}
\begin{proof}
    Similar to the proof of Lemma 6.5 in \cite{HZ23}.
\end{proof}

We can now prove Lemma \ref{gdiscon}.
\begin{proof}
By Lemma \ref{Q+R}, we have
\begin{align*}
&\left| \big( \Pi^{(0)}_{k,({\mathbb R}^m\curvearrowright,\leq k\delta)}\ddot{u}_{[k]} - \ddot{u}_{[k]} \mid v_{[k]}\big)_k \right| 
= \left| \big( Q^{(0)}_{X, \leq k\delta} \circ\Pi^{(0)}_{k,({\mathbb R}^m\curvearrowright,\leq k\delta)} (\ddot{u}_{[k]}) - \ddot{u}_{[k]} \mid v_{[k]} \big)_k \right| \\
&= \left| \big( Q^{(0)}_{X, \leq k\delta} \circ \Pi^{(0)}_{k,({\mathbb R}^m\curvearrowright,\leq k\delta)} (\ddot{u}_{[k]}) - Q^{(0)}_{X, \leq k\delta} (\ddot{u}_{[k]}) +Q^{(0)}_{X, \leq k\delta} (\ddot{u}_{[k]}) - \ddot{u}_{[k]} \mid v_{[k]} \big)_k \right| \\
&\leq \left| \big( Q^{(0)}_{X, \leq k\delta} (\Pi^{(0)}_{k,({\mathbb R}^m\curvearrowright,\leq k\delta)} - I)(\ddot{u}_{[k]}) \mid v_{[k]} \big)_k \right| + \left| \big( Q^{(0)}_{X, \leq k\delta} (\ddot{u}_{[k]}) - \ddot{u}_{[k]} \mid v_{[k]} \big)_k \right| \\
&\leq \left\| ( Q^{(0)}_{X, \leq k\delta} (\Pi^{(0)}_{k,({\mathbb R}^m\curvearrowright,\leq k\delta)} - I) )(\ddot{u}_{[k]}) \right\|_k \| v_{[k]} \|_k + \left| \big( (\hat{Q}^{(0)}_{\leq k\delta} + \hat{R}_{\leq k\delta})(\ddot{u}_{[k]}) - \ddot{u}_{[k]} \mid v_{[k]} \big)_k \right| \\
&\leq \underbrace{ \left\| (\Pi^{(0)}_{k,({\mathbb R}^m\curvearrowright,\leq k\delta)} - I)(Q^{(0)}_{X, \leq k\delta}\ddot{u}_{[k]}) \right\|_k \| v_{[k]} \|_k }_{(1)} + \underbrace{ \left| ( \hat{Q}^{(0)}_{\leq k\delta} (\ddot{u}_{[k]}) - \ddot{u}_{[k]} \mid v_{[k]} )_k \right| }_{(2)} \\
&\quad+ \underbrace{ \left| ( \hat{R}_{\leq k\delta} (\ddot{u}_{[k]}) \mid v_{[k]} )_k \right| }_{(3)}.
\end{align*}

\begin{align*}
(1) = \left\| (\Pi^{(0)}_{k,({\mathbb R}^m\curvearrowright,\leq k\delta)} - I)(Q^{(0)}_{X, \leq k\delta}\ddot{u}_{[k]}) \right\|_k \| v_{[k]} \|_k
&\lesssim k^{-\frac{1}{2}}(\Box_{b,k}Q^{(0)}_{X, \leq k\delta}\ddot{u}_{[k]} \mid Q^{(0)}_{X, \leq k\delta}\ddot{u}_{[k]})_k^{\frac{1}{2}} \| v_{[k]} \|_k\\
&= k^{-\frac{1}{2}}(\overline{\partial}_{b,k}Q^{(0)}_{X, \leq k\delta}\ddot{u}_{[k]} \mid \overline{\partial}_{b,k}Q^{(0)}_{X, \leq k\delta}\ddot{u}_{[k]})_k^{\frac{1}{2}} \| v_{[k]} \|_k\\
&= k^{-\frac{1}{2}}(Q^{(0)}_{X, \leq k\delta}\overline{\partial}_{b,k}\ddot{u}_{[k]} \mid Q^{(0)}_{X, \leq k\delta}\overline{\partial}_{b,k}\ddot{u}_{[k]})_k^{\frac{1}{2}} \| v_{[k]} \|_k\\
&\leq k^{-\frac{1}{2}}(\overline{\partial}_{b,k}\ddot{u}_{[k]} \mid \overline{\partial}_{b,k}\ddot{u}_{[k]})_k^{\frac{1}{2}} \| v_{[k]} \|_k\\
&\longrightarrow 0\ (\text{by the argument as in Section \ref{scaledszegokernelasy}.})
\end{align*}
\begin{align*}
&(2)\\
&=\big| ( \hat{Q}^{(0)}_{\leq k\delta} (\ddot{u}_{[k]}) - \ddot{u}_{[k]} \mid v_{[k]} )_k \big|\\
&= \bigg| \int_{D(\frac{k^\epsilon}{\sqrt{k}},\frac{k^{2\epsilon}}{k})} \Big( \frac{1}{(2\pi)^m} \int e^{i\langle x'' - y'',\eta\rangle} 1_{[-k\delta, k\delta]^m}(\eta) \ddot{u}_{[k]}(z, y'') dy'' d\eta - \ddot{u}_{[k]}(z,x'') \Big)   \\ 
&\quad \cdot\overline{v_{[k]}(z,x'')} e^{-k\phi(z)} m(z) d\mu_{\mathbb{H}^{n,m}}(x) \bigg| \\
&= \bigg| \int_{D(\frac{k^\epsilon}{\sqrt{k}},\frac{k^{2\epsilon}}{k})} \Big( \frac{1}{(2\pi)^m} \int e^{i\langle x'' - y'',\eta\rangle} 1_{[-k\delta, k\delta]^m}(\eta) \cdot k^{\frac{n+m}{2}} \ddot{u}_{(k)}(\sqrt{k}z, ky'')dy''d\eta \\
&\quad - k^{\frac{n+m}{2}} \ddot{u}_{(k)}(\sqrt{k}z, kx'') \Big) \cdot k^{\frac{n+m}{2}} \overline{v(\sqrt{k}z, kx'')} e^{-k \phi(z)} m(z) d\mu_{\mathbb{H}^{n,m}}(x) \bigg| \\
&= \bigg| \int_{D(k^\epsilon,k^{2\epsilon})} \Big( \frac{1}{(2\pi)^m} \int e^{i\langle\frac{x''}{k} - y'',\eta\rangle} 1_{[-k\delta, k\delta]^m}(\eta) \ddot{u}_{(k)}(z, y'') dy'' d\eta - \ddot{u}_{(k)}(z,x'') \Big) \\
&\quad\cdot\overline{v(z,x'')} e^{-k\phi(\frac{z}{\sqrt{k}})}V\big(\frac{z}{\sqrt{k}}\big) d\mu_{\mathbb{H}^{n,m}}(x) \bigg| \\
&= \bigg| \int_{D(k^\epsilon,k^{2\epsilon})} \Big( \frac{1}{(2\pi)^m} \int e^{i\langle x'' - y'',\eta\rangle} 1_{[-\delta, \delta]^m}(\eta) \ddot{u}_{(k)}(z, y'') dy''d\eta - \ddot{u}(z,x'') \Big)\\
&\quad \cdot\overline{v(z,x'')} e^{-k\phi(\frac{z}{\sqrt{k}})} V\big(\frac{z}{\sqrt{k}}\big) d\mu_{\mathbb{H}^{n,m}}(x) \bigg| \\
&\leq C \bigg| \int_{D(k^\epsilon,k^{2\epsilon})} \Big( \frac{1}{(2\pi)^m} \int e^{i\langle x'' - y'',\eta\rangle} 1_{[-\delta, \delta]^m}(\eta) \ddot{u}_{(k)}(z, y'') dy''d\eta - \ddot{u}(z,x'') \Big)\\
&\quad \cdot\overline{v(z,x'')} e^{-\psi_0(z)}d\mu_{\mathbb{H}^{n,m}}(x) \bigg| \\
&\le C \bigg| \int_{D(k^{\epsilon}, k^{2\epsilon})} \Big|\mathcal{F}^{-1}_{\eta \to x''} \big[ 1_{[-\delta, \delta]^m}(\eta) \mathcal{F}_{y''\to \eta} [\ddot{u}_{(k)}] \big] (z, x'') - \ddot{u}(z, x'') \Big|^2 e^{-\psi_0(z)} d\mu_{\mathbb{H}^{n,m}}(x) \bigg|^{1/2} \cdot \|v\|_{L^2} \\
&\le C \bigg| \int_{\mathbb{H}^{n,m}} \Big|\mathcal{F}^{-1}_{\eta \to x''} \big[ 1_{[-\delta, \delta]^m}(\eta) \mathcal{F}_{y''\to \eta} [\ddot{u}_{(k)}] \big] (z, x'') - \ddot{u}(z, x'') \Big|^2 e^{-\psi_0(z)} d\mu_{\mathbb{H}^{n,m}}(x) \bigg|^{1/2} \cdot \|v\|_{L^2} \\
&= \frac{C}{\sqrt{(2\pi)^m}} \Bigg| \int_{\mathbb{H}^{n,m}} \big| 1_{[-\delta, \delta]^m}(\eta) \mathcal{F}_{y'' \to \eta} [\ddot{u}_{(k)}] (z, \eta) - \mathcal{F}_{x''\to\eta}[\ddot{u}](z, \eta) \big|^2 e^{-\psi_0(z)} d\mu(z) d\mu(\eta) \Bigg|^{1/2} \cdot \|v\|_{L^2} \\
&= \frac{C}{\sqrt{(2\pi)^m}} \bigg|\int_{\mathbb{C}^n} \int_{[-\delta,\delta]^m} \big| \mathcal{F}_{y'' \to \eta} [\ddot{u}_{(k)}] (z, \eta) - \mathcal{F}_{x''\to\eta}[\ddot{u}](z, \eta) \big|^2 e^{-\psi_0(z)} d\mu(z) d\mu(\eta) \bigg|^{1/2} \cdot \|v\|_{L^2} \\
&\le \frac{C}{\sqrt{(2\pi)^m}} \bigg| \int_{\mathbb{H}^{n,m}} \big| \mathcal{F}_{y'' \to \eta} [\ddot{u}_{(k)}] (z, \eta) - \mathcal{F}_{x''\to\eta}[\ddot{u}](z, \eta) \big|^2 e^{-\psi_0(z)} d\mu(z) d\mu(\eta) \bigg|^{1/2} \cdot \|v\|_{L^2} \\
&= C \bigg| \int_{\mathbb{H}^{n,m}} \big| \ddot{u}_{(k)}(z, x'') - \ddot{u}(z, x'') \big|^2 e^{-\psi_0(z)} d\mu_{\mathbb{H}^{n,m}}(x) \bigg|^{1/2} \cdot \|v\|_{L^2} \to 0 \text{ as } k \to \infty.
\end{align*}

\begin{align*}
&(3) \\
&=\big|(\hat{R}_{\leq k\delta} (\ddot{u}_{[k]}) \mid v_{[k]})_k\big|\\
&= \bigg| \frac{1}{(2\pi)^m} \sum_{(m^1_1,\dots,m^m_{N_m})\in {\mathbb Z}^N} \int_{D(\frac{k^{\epsilon}}{\sqrt{k}}, \frac{k^{2\epsilon}}{k})} e^{i\langle x'' - y'', \eta\rangle + i\sum\limits_{\ell=1}^m\big(\sum\limits_{j=1}^{N_\ell}m^\ell_j\beta^\ell_j\big)y_{2n+\ell} - i \sum\limits_{\ell=1}^m\big(\sum\limits_{j=1}^{N_\ell}m^\ell_j \theta^\ell_j\big)}  \\
&\quad \cdot 1_{[-k\delta, k\delta]^m}(\eta) (1 - \chi(y'')) \ddot{u}_{[k]} \big((e^{i\theta^1_1}, \dots, e^{i\theta^m_{N_m}}) \circ x'\big)d\mathbb{T}_Nd\mu(\eta) d\mu(y'')\\ 
&\quad \cdot \overline{v_{[k]}(z, x'')} e^{-k\phi(z)} m(z) d\mu_{\mathbb{H}^{n,m}}(x) \bigg| \\
&= \bigg| \frac{1}{(2\pi)^m} \sum_{(m^1_1,\dots,m^m_{N_m})\in {\mathbb Z}^N} \int_{D(k^\epsilon,k^{2\epsilon})} e^{i\langle x'' - y'', \eta\rangle + i\sum\limits_{\ell=1}^m\big(\sum\limits_{j=1}^{N_\ell}m^\ell_j\beta^\ell_j\big)\frac{y_{2n+\ell}}{k} - i \sum\limits_{\ell=1}^m\big(\sum_{j=1}^{N_\ell}m^\ell_j \theta^\ell_j\big)}  \\
&\quad \cdot 1_{[-\delta, \delta]^m}(\eta) \cdot \big(1 - \chi\big(\frac{y''}{k}\big)\big) \ddot{u}_{(k)} \big(F_k^{-1}\big((e^{i\theta^1_1}, \dots, e^{i\theta^m_{N_m}}) \circ \frac{x'}{k}\big)\big)d\mathbb{T}_Nd\eta dy''\\
&\quad \cdot \overline{v(z, x'')} e^{-k\phi\big(\frac{z}{\sqrt{k}}\big)} V\big(\frac{z}{\sqrt{k}}\big) d\mu_{\mathbb{H}^{n,m}}(x) \bigg| \\
&= \bigg| \frac{1}{(2\pi)^m} \sum_{(m^1_1,\dots,m^m_{N_m})\in {\mathbb Z}^N} \int e^{i\sum\limits_{\ell=1}^m\big(\sum\limits_{j=1}^{N_\ell}m^\ell_j\beta^\ell_j\big)\frac{y_{2n+\ell}}{k} - i \sum\limits_{\ell=1}^m\big(\sum_{j=1}^{N_\ell}m^\ell_j \theta^\ell_j\big)}  \mathcal{F}_{\eta\to\tau}\big[1_{[-\delta, \delta]^m}\big](\tau) \\
&\quad \cdot \overline{v(z, y'' - \tau)} \big(1 - \chi\big(\frac{y''}{k}\big)\big) \ddot{u}_{(k)} \big(F_k^{-1}\big((e^{i\theta^1_1}, \dots, e^{i\theta^m_{N_m}}) \circ \frac{x'}{k}\big)\big)\\
&\quad\cdot e^{-k\phi(\frac{z}{\sqrt{k}})} m(\frac{z}{\sqrt{k}}) d\mathbb{T}_N d\mu(\tau) d\mu(y'')d\mu(z) \bigg| \\
&= \bigg| \int \ddot{u}_{(k)} (z, y'')  \mathcal{F}_{\eta\to\tau}\big[1_{[-\delta, \delta]^m}\big](\tau)  \overline{v(z, y'' - \tau)} \big(1 - \chi\big(\frac{y''}{k}\big)\big)  e^{-k\phi(\frac{z}{\sqrt{k}})} m(\frac{z}{\sqrt{k}})  d\mu(\tau) d\mu(z) d\mu(y'') \bigg| \\
&= \bigg| \int \ddot{u}_{(k)} (z, y'') \mathcal{F}_{\eta\to y''}\big[1_{[-\delta, \delta]^m} \overline{\mathcal{F}_{x''\to\eta}[v]}\big](z,y'') \big(1 - \chi\big(\frac{y''}{k}\big)\big) e^{-k\phi(\frac{z}{\sqrt{k}})} m(\frac{z}{\sqrt{k}}) d\mu(z) d\mu(y'') \bigg| \\
&\le C \bigg| \int |\ddot{u}(z, y'')|^2 \big(1 - \chi\big(\frac{y''}{k}\big)\big) e^{-\psi_0(z)}d\mu(z) d\mu(y'') \bigg|^{1/2} \\
&\quad \cdot \bigg|\int \big|\mathcal{F}_{\eta\to y''}\big[1_{[-\delta, \delta]^m}\cdot \overline{\mathcal{F}_{x''\to\eta}[v]}\big](z,y'')\big|^2 \big(1 - \chi\big(\frac{y''}{k}\big)\big) e^{-\psi_0(z)} d\mu(z) d\mu(y'') \bigg|^{1/2} \\
&\to 0 \quad \text{as } k \to \infty.
\end{align*}

\end{proof}

\section{Equivariant Kodaira embedding theorem}
\subsection{Proof of embedding Theorem}
In this section, we prove Theorem \ref{embedthmintro}.
Throughout this section, we assume that $(X,T^{1,0}X)$ is a compact CR manifold of dimension $2n+m$ and of CR codimension $m$ equipped with a transversal locally free CR $\mathbb{R}^m$-action on $X$ and there is a rigid positive CR line bundle $L$ over $X.$
We first prove that there exists a well-defined Kodaira map from $X$ to complex projective space when $\delta$ small and $k$ large. 
Let 
\begin{align*}
    f_1\in \mathcal{H}^0_{b,\le k\delta}(X,L^k),\dots,f_{d_k}\in \mathcal{H}^0_{b,\le k\delta}(X,L^k)
\end{align*}
be an orthonormal basis of $\mathcal{H}^0_{b,\le k\delta}(X,L^k)$ with respect to $(\,\cdot\,|\,\cdot\,)_k.$

\begin{lem}\label{wedef}
Let $\delta>0$ be a positive constant small enough. Then there exist $C_0>0$ and $k_0\in \mathbb{N}$ such that for all $k\ge k_0,$ and all $x\in X,$ we have
\begin{equation}\label{dialead}
    \sum_{j=1}^{d_k} |f_j(x)|_{h^k}^2 \ge C_0 k^{n+m}.
\end{equation}
Moreover, there exist $c_0>0$ and $k_0\in \mathbb{N}$ such that for all $k\ge k_0,$ and all $x\in X,$ we can find some $j_0\in \{1,\dots,d_k\}$ such that
\begin{equation}\label{sindialead}
    |f_{j_0}(x)|_{h^k}^2 \ge c_0.
\end{equation}
\end{lem}
\begin{proof}
(\ref{dialead}) follows by Theorem \ref{nongrasy} and (\ref{nonwegrszeker}). For the second assertion, we can see that, following similar arguments as in \cite{HL18}, there exists $C_1>0$ independent of $k$ such that
\begin{equation*}
    \dim \mathcal{H}^0_{b,\le k\delta}(X,L^k) = d_k \leq C_1 k^{n+m}.
\end{equation*}
Together with (\ref{dialead}), we can conclude (\ref{sindialead}).
\end{proof}

On an open subset $D\subset X$ satisfying Corollary \ref{invnbd}, we write
\begin{align*}
    f_j = \hat{f}_j\otimes s^k, \quad \hat{f}_j \in \mathcal{C}^\infty(D), \quad j=1,\dots,d_k.
\end{align*}
By Lemma \ref{wedef}, for $\delta$ small enough and $k$ large enough, we have a well-defined global map
\begin{equation}\label{embedmap}
    \Phi_{k,\delta} : X \to \mathbb{CP}^{d_k-1}, \quad x \mapsto \big[\hat{f}_1(x) : \cdots : \hat{f}_{d_k}(x)\big].
\end{equation}
We can choose $f_j,\ j=1,\dots, d_k,$ such that $f_j\in \bigcap_{1\leq\ell\leq m}\mathcal{H}^0_{b,(Y_\ell,\alpha^j_\ell)}(X,L^k)$ where each $\alpha^j_\ell \in \text{Spec}(-iY_\ell)$ with $\big|\alpha^j_\ell\big| \leq k\delta.$

We now define an $\mathbb{R}^m$-action on $\mathbb{CP}^{d_k-1}.$ For each $j = 1,\dots, d_k,$ we assign a multi-weight $\alpha^j=(\alpha^j_1,\dots,\alpha^j_{m})\in \mathbb{R}^m.$ For all $\eta = (\eta_1,\dots,\eta_m)\in \mathbb{R}^m,$ we define, for $[z_1 : \cdots : z_{d_k}]\in \mathbb{CP}^{d_k-1},$
\begin{align}\label{rdact}
    \eta \circ [z_1 :\cdots :z_{d_k}] 
    &:= \big[e^{i\langle\alpha^1,\eta\rangle} z_1 : \cdots :e^{i\langle\alpha^{d_k},\eta\rangle} z_{d_k}\big],\\
    & = \big[e^{i(\alpha^1_1\eta_1+\cdots+\alpha^1_m\eta_m)}z_1 : \cdots : e^{i(\alpha^{d_k}_1\eta_1+\cdots+\alpha^{d_k}_m\eta_m)}z_{d_k}\big].\nonumber
\end{align}

\begin{thm}\label{koembedthm}
 Under the above assumptions, there exists $\delta_0 > 0$ such that for all $\delta\in (0,\delta_0)$ there exists a positive integer $k(\delta)$ so that for $k>k(\delta)$ and any orthonormal basis $\{f_j\}_{j=1}^{d_k}$ of $\mathcal{H}^0_{b,\leq k\delta}(X, L^k)$ with respect to $(\,\cdot\,|\,\cdot\,)_k$ such that $f_j\in \bigcap_{1\leq\ell\leq m}\mathcal{H}^0_{b,(Y_\ell,\alpha^j_\ell)}(X,L^k)$ where each $\alpha^j_\ell \in \text{Spec}(-iY_\ell)$ with $\big|\alpha^j_\ell\big| \leq k\delta$, the map $\Phi_{k,\delta}$ introduced by (\ref{embedmap}) is a smooth CR embedding which is $\mathbb{R}^m$-equivariant with respect to the multi-weighted $\mathbb{R}^m$-action on $\mathbb{CP}^{d_k-1}$ defined by $\alpha^1=(\alpha^1_1,\dots,\alpha^1_m),\dots,\alpha^{d_k}=(\alpha^{d_k}_1,\dots,\alpha^{d_k}_m)$ as in (\ref{rdact}).
\end{thm}

Take $\delta_0>0$ small enough to satisfy Lemma \ref{wedef} and fix $\delta\in (0,\delta_0).$ We have shown that $\Phi_{k,\delta}$ is globally defined for $k$ large enough. Since $X$ is compact, it suffices to show that $\Phi_{k,\delta}$ is an injective immersion. We first show that $\Phi_{k,\delta}$ is an immersion.
\begin{thm}\label{imm}
For $k$ large enough, the map $\Phi_{k,\delta}$ is an immersion.
\end{thm}

To prove Theorem \ref{imm}, we need some preparations. Let $\chi \in \mathcal{C}^\infty_c(\mathbb{R}, \mathbb{R})$ be a smooth even bump function with $0 \le \chi \le 1$, $\chi(t) = 1$ on $[-\frac{1}{2}, \frac{1}{2}]$ and $\chi(t)=0$ when $|t|\geq1.$  Consider the 2-dimensional bump function $\rho(v_1,v_2) \in \mathcal{C}^\infty_c(\mathbb{R}^2)$ such that $\rho(v) = \chi(|v|).$ For $\theta > 0$, let $\rho_\theta(v) := \theta^{-2} \rho(\theta^{-1}v)$ for $v \in \mathbb{R}^2 \cong \mathbb{C}$.
 We can check that, as oscillatory integrals,
\begin{align*}
& \lim_{\theta \to 0} \int_{\mathbb{H}^{n,m}} \int_{[-\delta, \delta]^m} e^{-i \langle y'', \eta \rangle - \frac{1}{2} \sum\limits_{j,r} (M_\eta)_{jr} w_j \overline{w}_r}  \det(M_\eta) \, \rho_\theta(w_1) \cdots \rho_\theta(w_n) d\mu(\eta) d\mu_{\mathbb{H}^{n,m}}(y) \\
&= \lim_{\theta \to 0} \int_{\mathbb{H}^{n,m}} \int_{[-\delta, \delta]^m} e^{-i \langle y'', \eta \rangle - \frac{1}{2} \sum\limits_{j,r} (M_\eta)_{jr} \theta^2 w_j\overline{w}_r}  \det(M_\eta) \, \rho(w_1) \cdots \rho(w_n) d\mu(\eta) d\mu_{\mathbb{H}^{n,m}}(y) \\
&= \Big( \int_{\mathbb{R}^m} \int_{[-\delta, \delta]^m} e^{-i \langle y'', \eta \rangle}  \det(M_\eta) d\mu(\eta) d\mu(y'') \Big) \cdot \Big( \int_{\mathbb{R}^{2n}} \rho(w_1) \cdots \rho(w_{n}) d\mu(w) \Big) \\
&= (2\pi)^m \, \big(\det(M_\eta)\big|_{\eta=0}\big)  \int_{\mathbb{R}^{2n}} \rho(w_1) \cdots \rho(w_n) d\mu(w) > 0.
\end{align*}

\begin{align*}
& \lim_{\theta \to 0} \int_{\mathbb{H}^{n,m}} \int_{[-\delta, \delta]^m} e^{-i \langle y'', \eta \rangle - \frac{1}{2} \sum\limits_{j,r} (M_\eta)_{jr} w_j \overline{w}_r} (i\eta_\ell) \det(M_\eta)\, \rho_\theta(w_1) \cdots \rho_\theta(w_n) y_{2n+\ell} d\mu(\eta) d\mu_{\mathbb{H}^{n,m}}(y) \\
&= \Big( \int_{\mathbb{R}^m} \int_{[-\delta, \delta]^m} e^{-i \langle y'', \eta \rangle} (i \eta_\ell y_{2n+\ell}) \det(M_\eta) d\mu(\eta) d\mu(y'') \Big) \cdot \Big( \int_{\mathbb{R}^{2n}} \rho(w_1) \cdots \rho(w_n) d\mu(w) \Big) \\
&= (2\pi)^m \, \big(\det(M_\eta)\big|_{\eta=0}\big)  \int_{\mathbb{R}^{2n}} \rho(w_1) \cdots \rho(w_n) d\mu(w) > 0, \quad \text{for all }\ \ell=1, \dots, m.
\end{align*}

We deduce that there exists $\theta_0 > 0$ small enough such that, for  $j=1, \dots, 2n\ \text{and} \  \ell=1, \dots, m,$
\begin{align} \label{baimest}
&\int_{\mathbb{H}^{n,m}} \int_{[-\delta, \delta]^m} e^{-i \langle y'', \eta \rangle - \frac{1}{2} \sum\limits_{j,r}(M_\eta)_{jr} w_j \overline{w}_r}  \det(M_\eta) \rho_{\theta_0}(w_1) \cdots \rho_{\theta_0}(w_n) d\mu(\eta) d\mu_{\mathbb{H}^{n,m}}(y) =: V_0 > 0 \\
&\int_{\mathbb{H}^{n,m}} \int_{[-\delta, \delta]^m} e^{-i \langle y'', \eta \rangle - \frac{1}{2} \sum\limits_{j,r}(M_\eta)_{jr} w_j \overline{w}_r} (i \eta_\ell y_{2n+\ell}) \det(M_\eta) \rho_{\theta_0}(w_1) \cdots \rho_{\theta_0}(w_n)d\mu(\eta) d\mu_{\mathbb{H}^{n,m}}(y)\nonumber\\
&=: V_\ell > 0.\nonumber
\end{align}

Fix a point $p \in X$. Let $s$ be a local rigid CR frame of $L$ on a BRT coordinates neighborhood $(D,(z,x''),\varphi)$ centered at $p$ with $|s|^2_h=e^{-\phi}$  as in (\ref{libunwe}). We first prove the following lemma.
\begin{lem} \label{imlem}
With the notations above, there exists a positive integer $K_0 > 0$ independent of the point $p$ such that for all $k \ge K_0$, we can find
\[
g^j_k \in \mathcal{H}^0_{b, \le k\delta} (X, L^k), \ j=1, \dots, n+m,
\]
such that if we write $g^j_k = \hat{g}^j_k \otimes s^k$ on $D$, $j=1, \dots, n+m$, then we have
\begin{align}\label{lemimin}
& \sum_{j=1}^{n+m} \big| \big(e^{-\frac{k}{2}\phi} \hat{g}^j_k\big) \big|(0) = 0, \\
& \bigg|\frac{1}{\sqrt{k}} \partial_{\overline{z}_t} (e^{-\frac{k}{2}\phi} \hat{g}^j_k)(0)\bigg| < \theta_k, \quad j=1, \dots, n+m, \ t=1, \dots, n,\nonumber \\
& \bigg|\frac{1}{\sqrt{k}} \partial_{z_t} (e^{-\frac{k}{2}\phi} \hat{g}^j_k)(0)\bigg| < \theta_k, \quad j=1, \dots, n+m, \ t=1, \dots, n, \ j \neq t,\nonumber \\
&\bigg|\frac{1}{k} \partial_{x_{2n+q}} (e^{-\frac{k}{2}\phi} \hat{g}^{n+\ell}_k)(0)\bigg| < \theta_k, \quad \ell,q=1, \dots, m,\ \ell\neq q,\nonumber\\
& \bigg|\frac{1}{\sqrt{k}} \partial_{z_j} (e^{-\frac{k}{2}\phi} \hat{g}^j_k)(0)\bigg| \ge C_0, \quad j=1, \dots, n,\nonumber \\
& \bigg|\frac{1}{k} \partial_{x_{2n+\ell}} (e^{-\frac{k}{2}\phi} \hat{g}^{n+\ell}_k)(0)\bigg| \ge C_0, \quad \ell=1, \dots, m,\nonumber
\end{align}
where $C_0 > 0$ is a constant independent of $k$ and the point $p,$ and $\theta_k$ is a sequence independent of $p$ with $\lim\limits_{k \to \infty} \theta_k = 0.$
\end{lem}

\begin{proof}
We first construct a base function $\hat{u}_k$ using the cutoff functions. 
Let $\chi, \ \rho$ and $\theta_0 > 0$ be as in (\ref{baimest}).
 Let
\begin{align*}
u_{(k)}(y) := \rho_{\theta_0}(w_1) \cdots \rho_{\theta_0}(w_n) \, \chi\left(\frac{y_{2n+1}}{k^{2\epsilon}}\right) \cdots \chi\left(\frac{y_{2n+m}}{k^{2\epsilon}}\right) \in \mathcal{C}^\infty_c(\mathbb{R}^{2n+m}).
\end{align*}
Put
\begin{align*}
u_{\langle k\rangle}(y) &:= u_{(k)}(\sqrt{k}w, ky'')  e^{\frac{k}{2}\phi(y)}\\ 
&\ = \rho_{\theta_0}(\sqrt{k}w_1) \cdots \rho_{\theta_0}(\sqrt{k}w_n) \, \chi\left(\frac{k y_{2n+1}}{k^{2\epsilon}}\right) \cdots \chi\left(\frac{k y_{2n+m}}{k^{2\epsilon}}\right) e^{\frac{k}{2}\phi(y)}.
\end{align*}
Let
\begin{align*}
u_k := \Pi^{(0)}_{k,({\mathbb R}^m\curvearrowright,\leq k\delta)}  \big(u_{\langle k\rangle}(y)\otimes s^k\big).
\end{align*}
If we write $u_k = \hat{u}_k \otimes s^k$, then by Theorem \ref{nongrasy}, (\ref{nonwegrszeker}) and (\ref{baimest}) we have
\begin{align*}
&\lim_{k \to \infty} (2\pi)^{n+m} e^{-\frac{k}{2}\phi(0)} \hat{u}_k(0)\\ 
&\quad= \int_{\mathbb{H}^{n,m}} \int_{[-\delta, \delta]^m} e^{-i\langle y'', \eta \rangle - \frac{1}{2} \sum\limits_{j,r}(M_\eta)_{jr} w_j \overline{w}_r} \, \det(M_\eta) \, \rho_{\theta_0}(w_1) \cdots \rho_{\theta_0}(w_n)d\mu(\eta)d\mu_{\mathbb{H}^{n,m}}(y) > 0.
\end{align*}
Since $X$ is compact, there exists $k_0 > 0$ independent of the point $p$ such that for all $k \ge k_0$, we have
\begin{equation}\label{ubound}
\frac{1}{A_0} \le \Big|e^{-\frac{k}{2}\phi(0)} \hat{u}_k(0)\Big| \le A_0
\end{equation}
where $A_0 > 1$ is independent of $k$ and the point $p$.

From (\ref{grszeasy}), we can check that, for $\ell=1, \dots, m$,
\begin{align}\label{y2n+l=0}
&\lim_{k \to \infty} \frac{1}{k} \frac{\partial}{\partial x_{2n+\ell}} (e^{-\frac{k}{2}\phi} \hat{u}_k)(0)\\
&= (\frac{1}{2\pi})^{n+m} \frac{\partial}{\partial x_{2n+\ell}} \bigg|_{x=0} \int_{\mathbb{H}^{n,m}} \int_{[-\delta, \delta]^m} e^{i\langle x'', \eta \rangle - i\langle y'', \eta \rangle - \frac{1}{2} \sum\limits_{j,r} (M_\eta)_{jr} (z_j - w_j)(\overline{z}_r - \overline{w}_r) + \frac{1}{2} \sum\limits_{j,r}(M_\eta)_{jr} (z_j\overline{w}_r - \overline{z}_r w_j)}\nonumber \\
&\quad \quad \quad\quad\quad\quad\quad\quad\quad\quad\cdot \det(M_\eta) \, \rho_{\theta_0}(w_1) \cdots \rho_{\theta_0}(w_n) d\mu(\eta) d\mu_{\mathbb{H}^{n,m}}(y)\nonumber \\
&= (\frac{1}{2\pi})^{n+m} \int_{\mathbb{H}^{n,m}} \int_{[-\delta, \delta]^m} e^{-i\langle y'', \eta \rangle - \frac{1}{2} \sum\limits_{j,r} (M_\eta)_{jr} w_j \overline{w}_r} (i\eta_\ell)\det(M_\eta) \rho_{\theta_0}(w_1) \cdots \rho_{\theta_0}(w_n) d\mu(\eta) d\mu_{\mathbb{H}^{n,m}}(y)\nonumber \\
&= (2\pi)^m \cdot (\frac{1}{2\pi})^{n+m} (i\eta_\ell \det(M_\eta)) \big|_{\eta=0} \ \int_{\mathbb{R}^{2n}} \rho_{\theta_0}(w_1) \cdots \rho_{\theta_0}(w_n) d\mu(w) = 0,\nonumber
\end{align}
and for $t=1, \dots, n$,
\begin{align}\label{zt=0}
&\lim_{k \to \infty} \frac{1}{\sqrt{k}} \partial_{z_t} (e^{-\frac{k}{2}\phi} \hat{u}_k)(0)\\ 
&= (2\pi)^{-(n+m)} \int_{\mathbb{H}^{n,m}} \int_{[-\delta, \delta]^m} e^{-i\langle y'', \eta \rangle - \frac{1}{2} \sum\limits_{j,r} (M_\eta)_{jr} w_j \overline{w}_r} \, \Big(\sum_r (M_\eta)_{tr} \overline{w}_r\Big)  \det(M_\eta)\nonumber\\
&\quad\quad\quad\quad\quad\quad\cdot \rho_{\theta_0}(w_1) \cdots \rho_{\theta_0}(w_n) d\mu(\eta) d\mu_{\mathbb{H}^{n,m}}(y) = 0, \nonumber
\end{align}

\begin{align}\label{barzt=0}
\lim_{k \to \infty} \frac{1}{\sqrt{k}} \partial_{\overline{z}_t} (e^{-\frac{k}{2}\phi} \hat{u}_k)(0) = 0.
\end{align}
Hence, we have, for $\ell=1, \dots, m$ and $t=1, \dots, n$,
\begin{align}\label{uksmall}
&\bigg|\frac{1}{\sqrt{k}} \partial_{z_t} (e^{-\frac{k}{2}\phi} \hat{u}_k)(0)\bigg| < \delta_k,\nonumber\\
&\bigg|\frac{1}{\sqrt{k}} \partial_{\overline{z}_t} (e^{-\frac{k}{2}\phi} \hat{u}_k)(0)\bigg| < \delta_k,\\
&\bigg|\frac{1}{k} \partial_{x_{2n+\ell}} (e^{-\frac{k}{2}\phi} \hat{u}_k)(0)\bigg| < \delta_k,\nonumber
\end{align}
where $\delta_k$ is a sequence independent of the point $p$ with $\lim\limits_{k \to \infty} \delta_k = 0$.

For $\ell=1,\dots,m,$ put 
\begin{align*}
&v_{(k)}^{n+\ell} = \rho_{\theta_0}(w_1) \cdots \rho_{\theta_0}(w_n) \, \chi\left(\frac{y_{2n+1}}{k^{2\epsilon}}\right) \cdots \chi\Big(\frac{y_{2n+m}}{k^{2\epsilon}}\Big)\,y_{2n+\ell}.\\
&v_{\langle k\rangle}^{n+\ell} = v_{(k)}^{n+\ell}(\sqrt{k}w, k y'') e^{\frac{k}{2}\phi(y)}\\ 
&\quad\quad\,  = \rho_{\theta_0}(\sqrt{k}w_1) \cdots \rho_{\theta_0}(\sqrt{k}w_n)\,\chi\Big(\frac{ky_{2n+1}}{k^{2\epsilon}}\Big)\cdots\chi\Big(\frac{ky_{2n+m}}{k^{2\epsilon}}\Big) \, (k y_{2n+\ell}) \, e^{\frac{k}{2}\phi(y)}\\
&v_{k}^{n+\ell} = \Pi^{(0)}_{k,({\mathbb R}^m\curvearrowright,\leq k\delta)} \Big( v_{\langle k\rangle}^{n+\ell}(y) \cdot s^k(y) \Big)
\end{align*}

Write $v_{k}^{n+\ell} = \hat{v}_k^{n+\ell} \otimes s^k$. We can see that
\begin{align*}
&\lim_{k \to \infty} (2\pi)^{n+m}e^{-\frac{k}{2}\phi(0)} \hat{v}_k^{n+\ell}(0)\\ 
&= \int_{\mathbb{H}^{n,m}} \int_{[-\delta, \delta]^m} e^{-i\langle y'', \eta \rangle - \frac{1}{2}\sum\limits_{j,r}(M_\eta)_{jr} w_j\overline{w}_r} \det(M_\eta) \rho_{\theta_0}(w_1) \cdots\rho_{\theta_0}(w_n) y_{2n+\ell}d\mu(\eta) d\mu_{\mathbb{H}^{n,m}}(y).
\end{align*}

It is easy to see that there exists $E_\ell > 0$ independent of $k$ and the point $p$ such that
\begin{equation}\label{vn+lbound}
\Big|e^{-\frac{k}{2}\phi(0)} \hat{v}_k^{n+\ell}(0) \Big| \le E_\ell.
\end{equation}

From (\ref{grszeasy}) and (\ref{baimest}), we can check that
\begin{align}\label{x2n+lv=0}
&\lim_{k \to \infty} (2\pi)^{n+m} \frac{1}{k} \partial_{x_{2n+\ell}} \big(e^{-\frac{k}{2}\phi} \hat{v}_k^{n+\ell}\big)(0)\\
&= \int_{\mathbb{H}^{n,m}} \int_{[-\delta, \delta]^m} e^{-i\langle y'', \eta \rangle - \frac{1}{2}\sum\limits_{j,r}(M_\eta)_{jr} w_j \overline{w}_r} \det(M_\eta) (i \eta_\ell y_{2n+\ell}) \rho_{\theta_0}(w_1) \cdots \rho_{\theta_0}(w_n) d\mu(\eta) d\mu_{\mathbb{H}^{n,m}}(y)\nonumber\\
&= V_\ell \neq 0.\nonumber
\end{align}
For $q=1, \dots, m$ with $q \neq \ell,$
\begin{align}\label{x2n+mv=0}
&\lim_{k \to \infty} (2\pi)^{n+m}\frac{1}{k} \partial_{x_{2n+q}} (e^{-\frac{k}{2}\phi} \hat{v}_k^{n+\ell})(0) \\
&= \int_{\mathbb{H}^{n,m}} \int_{[-\delta, \delta]^m} e^{-i\langle y'', \eta \rangle - \frac{1}{2}\sum\limits_{j,r}(M_\eta)_{jr} w_j \overline{w}_r}  \det(M_\eta)  (i \eta_q y_{2n+\ell}) \chi_{\theta_0}(y_1) \dots \chi_{\theta_0}(y_{2n}) d\mu(\eta) d\mu_{\mathbb{H}^{n,m}}(y)\nonumber \\
&= (2\pi)  \Bigg( \eta_q \frac{\partial (\det M_\eta)}{\partial \eta_\ell} \Bigg|_{\eta=0} \Bigg)  \int_{\mathbb{R}^{2n}} \rho_{\theta_0}(w_1) \dots \rho_{\theta_0}(w_n) d\mu(w) = 0,\nonumber
\end{align}
and for $t=1, \dots, n$,
\begin{align}\label{ztv=0}
&\lim_{k \to \infty} \frac{1}{\sqrt{k}} \partial_{z_t} (e^{-\frac{k}{2}\phi} \hat{v}_k^{n+\ell})(0)\\
&=  \Big(\frac{1}{2\pi}\Big)^{n+m}\int_{\mathbb{H}^{n,m}} \int_{[-\delta, \delta]^m} e^{-i\langle y'', \eta \rangle - \frac{1}{2}\sum\limits_{j,r} (M_\eta)_{jr} w_j \overline{w}_r} \Big( \sum_r (M_\eta)_{tr} \overline{w}_r \Big) \det(M_\eta)
y_{2n+\ell}\nonumber\\
&\quad\quad\quad\quad\quad\cdot\rho_{\theta_0}(w_1) \cdots \rho_{\theta_0}(w_n) d\mu(\eta) d\mu_{\mathbb{H}^{n,m}}(y)\nonumber \\
&= 0,\nonumber
\end{align}
and similarly,
\begin{equation}\label{barztv=0}
\lim_{k \to \infty} \frac{1}{\sqrt{k}} \partial_{\overline{z}_t} (e^{-\frac{k}{2}\phi} \hat{v}_k^{n+\ell})(0) = 0.
\end{equation}
Hence we can see that there exists $k_1 \in \mathbb{N}$ with $k_1 \ge k_0$, independent of the point $p$, such that for $k \ge k_1$, we have
\begin{align}\label{vn+lcon}
&\bigg| \frac{1}{k} \partial_{x_{2n+\ell}} (e^{-\frac{k}{2}\phi} \hat{v}_k^{n+\ell})(0) \bigg| \ge B_1, \quad \ell=1, \dots, m\nonumber \\
&\bigg| \frac{1}{k} \partial_{x_{2n+q}} (e^{-\frac{k}{2}\phi} \hat{v}_k^{n+\ell})(0) \bigg| \le \hat{\delta}_k, \quad \ell, q=1, \dots, d, \quad \ell \neq q, \\
&\bigg|\frac{1}{\sqrt{k}} \partial_{z_t} (e^{-\frac{k}{2}\phi} \hat{v}_k^{n+\ell})(0) \bigg| \le \hat{\delta}_k, \quad t=1, \dots, n,\nonumber \\
&\bigg|\frac{1}{\sqrt{k}} \partial_{\overline{z}_t} (e^{-\frac{k}{2}\phi} \hat{v}_k^{n+\ell})(0) \bigg| \le \hat{\delta}_k, \quad t=1, \dots, n,\nonumber
\end{align}
where $B_1 > 0$ is a constant independent of $k$ and the point $p,$ and $\hat{\delta}_k$ is a sequence independent of the point $p$ with $\lim\limits_{k \to \infty} \hat{\delta}_k = 0$. From now on, we assume that $k \ge k_1$. Put
\begin{equation*}
g_k^{n+\ell} := v_k^{n+\ell} - \frac{(e^{-\frac{k}{2}\phi} \hat{v}_k^{n+\ell})(0)}{(e^{-\frac{k}{2}\phi} \hat{u}_k)(0)} u_k \in \mathcal{H}^0_{b, \le k\delta}(X, L^k), \quad \ell=1, \dots, m.
\end{equation*}

Write $g_k^{n+\ell} = \hat{g}_k^{n+\ell} \cdot s^k$. Then by (\ref{ubound}), (\ref{uksmall}), (\ref{vn+lbound}) and (\ref{vn+lcon}) we have
\begin{align}\label{gn+lcon}
&\Big| (e^{-\frac{k}{2}\phi} \hat{g}_k^{n+\ell} )(0) \Big| = 0, \quad \ell=1, \dots, m,\nonumber \\
&\bigg| \frac{1}{\sqrt{k}} \partial_{z_t} (e^{-\frac{k}{2}\phi} \hat{g}_k^{n+\ell} )(0) \bigg| \le \epsilon_k, \quad t=1, \dots, n, \quad \ell=1, \dots, m,\nonumber \\
&\bigg| \frac{1}{\sqrt{k}} \partial_{\overline{z}_t} (e^{-\frac{k}{2}\phi} \hat{g}_k^{n+\ell} )(0) \bigg| \le \epsilon_k, \quad t=1, \dots, n, \quad \ell=1, \dots, m, \\
&\bigg| \frac{1}{k} \partial_{x_{2n+\ell}} (e^{-\frac{k}{2}\phi} \hat{g}_k^{n+\ell} )(0) \bigg| \ge C_0, \quad \ell=1, \dots, m, \nonumber\\
&\bigg| \frac{1}{k} \partial_{x_{2n+q}} (e^{-\frac{k}{2}\phi} \hat{g}_k^{n+\ell} )(0) \bigg| \le \epsilon_k, \quad \ell, q=1, \dots, d, \quad \ell \neq q,\nonumber
\end{align}
where $C_0 > 0$ is a constant independent of $k$ and the point $p,$ and $\epsilon_k$ is a sequence independent of the point $p$ with $\lim\limits_{k \to \infty} \epsilon_k = 0$.

We can find a measurable unitary matrix valued-function $Q(\tau)$ on $[-\delta,\delta]^m$ such that
\begin{equation*}
Q^*(\tau) M_\tau Q(\tau) = \begin{pmatrix} \lambda_1(\tau) &  & \\ & \ddots  & \\ & &\lambda_n(\tau) \end{pmatrix} \text{ is diagonal,}
\end{equation*}
where $Q^*(\tau)$ is the adjoint of $Q(\tau).$
Let
\begin{equation*}
\begin{pmatrix} \widetilde{w}_1 \\ \vdots \\ \widetilde{w}_n \end{pmatrix} = Q^T(\tau) \begin{pmatrix} w_1 \\ \vdots \\ w_n \end{pmatrix} = \begin{pmatrix} f_1(w, \tau) \\ \vdots \\ f_n(w, \tau) \end{pmatrix},
\end{equation*}
where $Q^T(\tau)$ is the transpose of $Q(\tau).$ It is easy to see that $f_1, \dots, f_n$ are homogeneous of degree 1 in $w$.

Fix $t \in \{1, \dots, n\}$. Put
\begin{align*}
&v_{(k)}^t:= \int_{[-\delta,\delta]^m} e^{i \langle y'', \tau \rangle} \rho_{\theta_0}(f_1(w, \tau)) \cdots \rho_{\theta_0}(f_n(w, \tau))\, \chi\left(\frac{y_{2n+1}}{k^{2\epsilon}}\right) \dots \chi\left(\frac{y_{2n+m}}{k^{2\epsilon}}\right)\\
&\quad\quad\quad \times\theta_0^{-2} \Big( \sum_r (M_\tau)_{rt} w_r \Big) d\mu(\tau), \\
&v_{\langle k\rangle}^t := v_{(k)}^t(\sqrt{k}w, k y'')e^{\frac{k}{2}\phi(y)}, \\
&v_k^t := \Pi^{(0)}_{k,({\mathbb R}^m\curvearrowright,\leq k\delta)} \left( v_{\langle k\rangle}^t(y) \otimes s^k(y) \right).
\end{align*}

Write $v_k^t = \hat{v}_k^t \otimes s^k$. We can see that
\begin{align}\label{vtat0}
&\lim_{k \to \infty} e^{-\frac{k}{2}\phi(0)} \hat{v}_k^t(0)\\ 
&= \Big(\frac{1}{2\pi}\Big)^{n+m} \int_{\mathbb{H}^{n,m}} \int_{[-\delta, \delta]^m} e^{-i \langle y'', \eta \rangle - \frac{1}{2} \sum\limits_{j,r=1}^n (M_\eta)_{jr} w_j \overline{w}_r}  \det(M_\eta) \nonumber \\
&\qquad \cdot \int_{\mathbb{R}^m} e^{i \langle y'', \tau \rangle}\rho_{\theta_0}(f_1(w, \tau)) \cdots \rho_{\theta_0}(f_n(w, \tau))\theta_0^{-2} \Big( \sum_{r=1}^n (M_\tau)_{rt} w_r \Big) d\mu(\tau)d\mu(\eta) d\mu_{\mathbb{H}^{n,m}}(y)\nonumber \\
&= \Big(\frac{1}{2\pi}\Big)^{n} \int_{\mathbb{R}^{2n}} \int_{[-\delta, \delta]^m} e^{- \frac{1}{2} \sum\limits_{j,r=1}^n (M_\eta)_{jr} w_j \overline{w}_r} \det(M_\eta) \nonumber\\
&\qquad \cdot \rho_{\theta_0}(f_1(w, \eta)) \cdots \rho_{\theta_0}(f_n(w, \eta))  \theta_0^{-2} \Big( \sum_{r=1}^n (M_\eta)_{rt} w_r \Big) d\mu(\eta) d\mu(w).\nonumber
\end{align}
So there exists $E_t > 0$ independent of $k$ and the point $p$ such that
\begin{equation}\label{vtbd}
\big|e^{-\frac{k}{2}\phi(0)} \hat{v}_k^t(0)\big| \le E_t.
\end{equation}

From (\ref{grszeasy}), we have
\begin{align}\label{ztvt}
&\lim_{k \to \infty} \frac{1}{\sqrt{k}} \partial_{z_t} (e^{-\frac{k}{2}\phi} \hat{v}_k^t)(0)\\
&= \Big(\frac{1}{2\pi}\Big)^{n+m} \int_{\mathbb{H}^{n,m}} \int_{[-\delta, \delta]^m} e^{-i \langle y'', \eta \rangle - \frac{1}{2} \sum\limits_{j,r} (M_\eta)_{jr} w_j \overline{w}_r} \, \det(M_\eta) \Big( \sum_r (M_\eta)_{tr} \overline{w}_r \Big)\nonumber\\
&\quad \cdot \int e^{i \langle y'', \tau \rangle} \rho_{\theta_0}(f_1(w, \tau)) \cdots \rho_{\theta_0}(f_n(w, \tau))\, \theta_0^{-2} \, \Big( \sum_r (M_\tau)_{rt} w_r \Big) d\mu(\tau) d\mu(\eta) d\mu_{\mathbb{H}^{n,m}}(y) \nonumber\\
&\xrightarrow[\theta_0 \to 0]{} \Big(\frac{1}{2\pi}\Big)^{n} \int_{\mathbb{R}^{2n}} \int_{[-\delta, \delta]^m} \det(M_\eta)  \Big| \sum_r (M_\eta)_{rt} w_r \Big|^2 \rho(f_1(w, \eta)) \cdots \rho(f_n(w, \eta)) d\mu(\eta) d\mu(w) > 0.\nonumber
\end{align}

\noindent Hence we can choose $\theta_0 > 0$ such that
\begin{equation}\label{ztvt>0}
\lim_{k \to \infty} \frac{1}{\sqrt{k}} \partial_{z_t} (e^{-\frac{k}{2}\phi} \hat{v}_k^t)(0) > 0.
\end{equation}

For $q \in \{1, \dots, n\}$ with $q \neq t$,
\begin{align}\label{zrvt=0}
&\lim_{k \to \infty} \frac{1}{\sqrt{k}} \partial_{z_q} (e^{-\frac{k}{2}\phi} \hat{v}_k^t)(0)\\
&= \Big(\frac{1}{2\pi}\Big)^{n+m} \int_{\mathbb{H}^{n,m}} \int_{[-\delta, \delta]^m} e^{-i \langle y'', \eta \rangle - \frac{1}{2} \sum\limits_{j,r} (M_\eta)_{jr} w_j \overline{w}_r} \, \det(M_\eta)  \Big( \sum_r (M_\eta)_{qr} \overline{w}_r \Big)\nonumber\\
&\quad\cdot \int e^{i \langle y'', \tau \rangle} \rho_{\theta_0}(f_1(w, \tau)) \cdots \rho_{\theta_0}(f_n(w, \tau)) \, \theta_0^{-2} \, \Big( \sum_r (M_\tau)_{rt} w_r \Big) d\mu(\tau) d\mu(\eta) d\mu_{\mathbb{H}^{n,m}}(y)\nonumber \\
&\xrightarrow[\theta_0 \to 0]{} \Big(\frac{1}{2\pi}\Big)^n \int_{\mathbb{R}^{2n}} \int_{[-\delta, \delta]^m} \det(M_\eta) \Big( \sum_r (M_\eta)_{qr} \overline{w}_r \Big) \Big( \sum_r (M_\eta)_{rt} w_r \Big)\nonumber\\
&\quad\quad\quad\cdot \rho(f_1(w, \eta)) \cdots \rho(f_n(w, \eta)) d\mu(\eta) d\mu(w)\nonumber \\
&= \Big(\frac{1}{2\pi}\Big)^n \int_{\mathbb{R}^{2n}} \int_{[-\delta, \delta]^m} \det(M_\eta)\,(\lambda_q(\eta) \overline{\widetilde{w}}_q)\, (\lambda_t(\eta) \widetilde{w}_t) \, \rho(\widetilde{w}_1) \cdots \rho(\widetilde{w}_n) d\mu(\widetilde{w}) d\mu(\eta)\nonumber \\
&= \Big(\frac{1}{2\pi}\Big)^n \int_{[-\delta, \delta]^m} \det(M_\eta)  \lambda_t(\eta) \lambda_q(\eta) \int_{\mathbb{R}^{2n}} \overline{\widetilde{w}}_q \widetilde{w}_t \rho(\widetilde{w}_1) \cdots \rho(\widetilde{w}_n) d\mu(\widetilde{w}) d\mu(\eta)\nonumber \\
&= 0.\nonumber
\end{align}

\noindent Similarly, we have
\begin{equation}\label{barzrvt=0}
\lim_{k \to \infty} \frac{1}{\sqrt{k}} \partial_{\overline{z}_q} (e^{-\frac{k}{2}\phi} \hat{v}_k^t)(0) = 0, \quad q=1, \dots, n,
\end{equation}
\begin{equation}\label{x2n+lvt=0}
\lim_{k \to \infty} \frac{1}{k} \partial_{x_{2n+\ell}} (e^{-\frac{k}{2}\phi} \hat{v}_k^t)(0) = 0, \quad \ell=1, \dots, m.
\end{equation}
Hence, by (\ref{ztvt>0}),  (\ref{zrvt=0}), (\ref{barzrvt=0}) and (\ref{x2n+lvt=0}), we can see that there exists $k_2 \in \mathbb{N}$ with $k_2 \ge k_1$, independent of the point $p$, such that for $k \ge k_2$, we have
\begin{align}\label{vtcon}
&\bigg| \frac{1}{\sqrt{k}} \partial_{z_t} (e^{-\frac{k}{2}\phi} \hat{v}_k^{t})(0) \bigg|\ge B_2,\quad t=1,\dots,n,\nonumber\\
&\bigg| \frac{1}{\sqrt{k}} \partial_{z_q} (e^{-\frac{k}{2}\phi} \hat{v}_k^{t})(0) \bigg| \le \widetilde{\delta}_k, \quad q,t=1,\dots,n, \quad q\neq t, \\
&\bigg| \frac{1}{\sqrt{k}} \partial_{\overline{z}_q} (e^{-\frac{k}{2}\phi} \hat{v}_k^{t})(0) \bigg| \le \widetilde{\delta}_k, \quad q,t =1,\dots,n,\nonumber\\
&\bigg| \frac{1}{k} \partial_{x_{2n+\ell}} (e^{-\frac{k}{2}\phi} \hat{v}_k^{t})(0) \bigg| \le \widetilde{\delta}_k, \quad \ell=1, \dots, m, \nonumber
\end{align}
where $B_2 > 0$ is a constant independent of $k$ and the point $p$ and $\widetilde{\delta}_k$ is a sequence independent of the point $p$ with $\lim\limits_{k \to \infty} \widetilde{\delta}_k = 0$. From now on, we assume that $k \ge k_2$. Put
\begin{equation*}
g_k^t := v_k^t - \frac{(e^{-\frac{k}{2}\phi} \hat{v}_k^t)(0)}{(e^{-\frac{k}{2}\phi} \hat{u}_k)(0)} u_k \in \mathcal{H}^0_{b, \le k\delta}(X, L^k), \quad t=1, \dots, n.
\end{equation*}
Write $g_k^t = \hat{g}_k^t \cdot s^k$. Then by (\ref{ubound}), (\ref{uksmall}), (\ref{vtbd}) and (\ref{vtcon}) we have
\begin{align}\label{gtcon}
&\Big| (e^{-\frac{k}{2}\phi} \hat{g}_k^t )(0) \Big| = 0, \quad t=1, \dots, n,\nonumber \\
&\bigg| \frac{1}{\sqrt{k}} \partial_{z_t} (e^{-\frac{k}{2}\phi} \hat{g}_k^t )(0) \bigg| \ge C_0, \quad t=1, \dots, n,\nonumber \\
&\bigg| \frac{1}{\sqrt{k}} \partial_{z_q} (e^{-\frac{k}{2}\phi} \hat{g}_k^t )(0) \bigg| \le \epsilon_k, \quad q,t=1, \dots, n, \quad q\neq t,\\
&\bigg| \frac{1}{\sqrt{k}} \partial_{\overline{z}_q} (e^{-\frac{k}{2}\phi} \hat{g}_k^t )(0) \bigg| \le \epsilon_k, \quad q,t=1, \dots, n,\nonumber \\
&\bigg| \frac{1}{k} \partial_{x_{2n+\ell}} (e^{-\frac{k}{2}\phi} \hat{g}_k^t )(0) \bigg| \le \epsilon_k, \quad t=1,\dots,n, \quad \ell=1, \dots, m,\nonumber 
\end{align}
where $C_0 > 0$ is a constant independent of $k$ and the point $p$ and $\epsilon_k$ is a sequence independent of the point $p$ with $\lim\limits_{k \to \infty} \epsilon_k = 0$.
\end{proof}

We can now prove Theorem \ref{imm}.
\begin{proof}
Fix a point $p \in X$. Let $s$ be a local rigid CR frame of $L$ on a BRT coordinates neighborhood $(D,(z,x''),\varphi)$ centered at $p$ with $|s|^2_h=e^{-\phi}$  as in (\ref{libunwe}). From Lemma \ref{wedef}, we may assume that
\begin{align}\label{f1}
|(e^{-k\phi} \hat{f}_1)(p)|^2 \ge c_0,
\end{align}
where $f_j = \hat{f}_j s^k$ on $D$, $j = 1, \dots, d_k$ and $c_0 > 0$ is a constant independent of $k$ and the
point $p$. Let $g_k^1, \dots, g_k^{n+m} \in \mathcal{H}^0_{b, \le k\delta}(X, L^k)$ be as in Lemma \ref{imlem}. From the first equality in (\ref{lemimin}), we can see that there is a $K_0 > 0$ independent of the point $p$ such that $f_1, g_k^1, \dots, g_k^{n+m}$ are linearly 
independent over $\mathbb{C}$ whenever $k \ge K_0$. Put
\begin{align*}
&p_k^j = \frac{e^{-\frac{k}{2}\phi} \hat{g}_k^j}{e^{-\frac{k}{2}\phi} \hat{f}_1}, \quad j = 1, \dots, n+m, \\
&p_k^j = \alpha_k^{2j-1} + i \alpha_k^{2j}, \quad \alpha_k^{2j-1} = \text{Re } p_k^j, \quad \alpha_k^{2j} = \text{Im } p_k^j, \quad j = 1, \dots, n+m,
\end{align*}
where $g_k^j = \hat{g}_k^j \otimes s^k$ on $D$, $j = 1, \dots, n+m$. From (\ref{lemimin}) and (\ref{f1}), we can see that
there is a $\widetilde{K}_0 > 0$ independent of the point $p$ such that for all $k \ge \widetilde{K}_0$, we have
\begin{align}\label{difplar}
|\partial_{z_t} p_k^j(p)| \ge c_1 \sqrt{k}, \quad t = 1, \dots, n, \quad |\partial_{x_{2n+\ell}} p_k^{n+\ell}(p)| \ge c_1 k, \quad \ell =1,\dots,m
\end{align}
and
\begin{align}\label{difpsma}
&\sup \Big\{ \frac{1}{\sqrt{k}}|\partial_{z_t} p_k^j(p)|, \frac{1}{k}|\partial_{x_{2n+q}} p_k^t(p)|, \frac{1}{k}|\partial_{x_{2n+q}} p_k^{n+\ell}(p)|\, \mid\,t=1,\dots,n,\ \ell,q=1,\dots, d, \\
&\quad \quad j=1,\dots,n+m,\ q\ne \ell, j\ne t  \Big\}\nonumber  \\
&+ \sup \Big\{ \frac{1}{\sqrt{k}}|\partial_{\overline{z}_t} p_k^j(p)|, |p_k^j(p)|  \mid t = 1, \dots, n,\ j=1,\dots, n+m \Big\} \le \theta_k,\nonumber
\end{align}
where $\theta_k$ is a sequence independent of the point $p$ with $\lim\limits_{k \to \infty} \theta_k = 0$.

From (\ref{difplar}), (\ref{difpsma})
and some elementary linear algebra argument, we conclude that there is a $K_1 > 0$ independent of the point $p$ such that for every $k \ge K_1$, the linear map $A_k : \mathbb{R}^{2n+m} \to \mathbb{R}^{2n+2m}$ represented by the matrix
\begin{align*}
\begin{bmatrix}
\partial_{x_1} \alpha_k^1(p) & \partial_{x_2} \alpha_k^1(p) & \cdots & \partial_{x_{2n+m}} \alpha_k^1(p) \\
\partial_{x_1} \alpha_k^2(p) & \partial_{x_2} \alpha_k^2(p) & \cdots & \partial_{x_{2n+m}} \alpha_k^2(p) \\
\vdots & \vdots & \vdots & \vdots \\
\partial_{x_1} \alpha_k^{2n+2m}(p) & \partial_{x_2} \alpha_k^{2n+2m}(p) & \cdots & \partial_{x_{2n+m}} \alpha_k^{2n+2m}(p)
\end{bmatrix},
\end{align*}
is injective. Hence the differential of the map
\begin{align*}
X \ni x \mapsto \Bigg( \frac{\hat{g}_k^1}{\hat{f}_1}(x), \dots, \frac{\hat{g}_k^{n+m}}{\hat{f}_1}(x) \Bigg) \in \mathbb{C}^{n+m}
\end{align*}
at $p$ is injective if $k \ge K_1$. From this and some elementary linear algebra arguments, we
conclude that the differential of the map
\begin{align*}
X \ni x \mapsto \Bigg( \frac{\hat{f}_2}{\hat{f}_1}(x), \dots, \frac{\hat{f}_{d_k}}{\hat{f}_1}(x) \Bigg) \in \mathbb{C}^{d_k-1}
\end{align*}
at $p$ is injective if $k \ge K_1$. Theorem \ref{imm} follows.
\end{proof}

We next prove the injectivity. Before proving this, we need a lemma. For every $k \in \mathbb{N}$, let
\begin{align*}
A_k := \{(x, y) \in X \times X; \Phi_{k,\delta}(x) = \Phi_{k,\delta}(y)\}. 
\end{align*}

\begin{lem}\label{nest}
There is a $k_0 \in \mathbb{N}$ such that for all $k \in \mathbb{N}$ and all $j \ge k_0$, we have
\begin{align*}
A_{k+j} \subset A_k. 
\end{align*}
\end{lem}

\begin{proof} 
Let $k_0 \in \mathbb{N}$ be as in Lemma \ref{wedef}. Let $j \ge k_0$ and let $k \in \mathbb{N}$. Let $\Box^{(0)}_{b,k\delta} := \Box_{b,k}\big|_{L^2_{Y,\leq k\delta}(X,L^k)}.$ Note that $\operatorname{Ker} \Box^{(0)}_{b,k\delta} = \mathcal{H}^0_{b, \le k\delta}(X, L^k).$ Let $(x_0, y_0) \in
A_{k+j}$. Then $\Phi_{k+j,\delta}(x_0) = \Phi_{k+j,\delta}(y_0)$. If $(x_0, y_0) \notin A_k$, then $[f_1(x_0): \cdots: f_{d_k}(x_0)] \ne
[f_1(y_0): \cdots: f_{d_k}(y_0)].$ Then
we can find $j_1,  j_2 \in \{1, \dots, d_k\}$ such that the two vectors $(f_{j_1}(x_0), f_{j_2}(x_0))$ and $(f_{j_1}(y_0), f_{j_2}(y_0))$
are linearly independent. By linear algebra, we can find $u_k$ and $v_k$ in $\mathcal{H}^0_{b, \le k\delta}(X, L^k)$ such that
\begin{align}\label{li}
u_k(x_0) = 0,\quad & u_k(y_0) \ne 0. \\
v_k(x_0) \ne 0,\quad & v_k(y_0) = 0.\nonumber
\end{align} 

By Lemma \ref{wedef}, we can find $g_j$ and $h_j$ in $\mathcal{H}^0_{b, \le j\delta}(X, L^k)$ such that $g_j(x_0) \ne 0$ and $h_j(y_0) \ne 0$. Let
$\alpha_{k+j} := h_j u_k \in \mathcal{H}^0_{b, \le (k+j)\delta}(X, L^k)$ and $\beta_{k+j} := g_j v_k \in \mathcal{H}^0_{b, \le (k+j)\delta}(X, L^k)$. Then we have
\begin{align}\label{li2}
\alpha_{k+j}(x_0) = 0,\quad & \alpha_{k+j}(y_0) \ne 0, \\
\beta_{k+j}(x_0) \ne 0,\quad & \beta_{k+j}(y_0) = 0.\nonumber
\end{align}

From (\ref{li2}), it is not difficult to see that
\begin{align*}
\Phi_{k+j,\delta}(x_0) \ne \Phi_{k+j,\delta}(y_0)
\end{align*}
which is a contradiction. Hence we get $A_{k+j} \subset A_k$. 
\end{proof}

We also recall the following non-weighted partial Szeg\H{o} kernel on $\mathbb{H}^{n,m}$ (See Theorem \ref{nongrasy} and (\ref{nonwegrszeker})):
\begin{align*}
&\Pi^{(0)}_{b,\psi_0,\le\delta} (x, y) = \Big(\frac{1}{2\pi}\Big)^{n+m} \int_{[-\delta, \delta]^m} \exp \Big[ i\langle x'', \eta \rangle - i\langle y'', \eta \rangle - \frac{1}{2} \sum_{j,r} (M_\eta)_{jr} (z_j - w_j)(\overline{z}_r - \overline{w}_r)\\
&\quad\quad\quad\quad\quad\quad\quad\quad\quad\quad\quad+\frac{1}{2} \sum_{j,r} (M_\eta)_{jr} (z_j \overline{w}_r - \overline{z}_r w_j) \Big]  \det(M_\eta) d\mu(\eta).
\end{align*}

We have
\begin{equation} \label{nowegrasy}
P_{(k), \delta, s}(x, y) = \Pi^{(0)}_{b, \psi_0, \le \delta}(x, y) + \mathcal{E}_k(x, y),
\end{equation}
where $\mathcal{E}_k(x, y) \in \mathcal{C}^\infty\big(D(k^\epsilon, k^{2\epsilon}), D(k^\epsilon, k^{2\epsilon})\big)$ satisfying $\lim\limits_{k \to \infty} \|\mathcal{E}_k(x, y)\|_{C^m(K \times K)} = 0$ for all $m \in \mathbb{N}$ and $k$-independent compact subset $K$ of $\mathbb{H}^{n,m}$.

\begin{thm}
For $k$ large enough, the map $\Phi_{k,\delta}$ is injective.
\end{thm}

Suppose that the result is not true. Then we can find $x_{k_j}, y_{k_j}\in X, \ x_{k_j}\neq y_{k_j}, \ 1\leq k_1<k_2<\cdots,$ such that $\Phi_{k_j}(x_{k_j})\neq \Phi_{k_j}(y_{k_j})$ for all $j.$ For simplicity, we may assume that we can find $x_k, y_k\in X,\ k\in\mathbb{N},$ such that $\Phi_k(x_k)\neq\Phi_k(y_k)$ for all $k\in \mathbb{N}$ and $\lim\limits_{k\to \infty}x_k=p\in X, \lim\limits_{k\to \infty}y_k=q\in X.$ We split our proof into two cases: $p=q$ and $p\neq q.$ In the following, we write $x_k=(z^k,x^k_{2n+1},\dots,x^k_{2n+m})=(z^k_1,\dots,z^k_n,x^k_{2n+1},\dots,x^k_{2n+m})$ and $y_k=(w^k,y^k_{2n+1},\dots,y^k_{2n+m})=(w^k_1,\dots,w^k_n,y^k_{2n+1},\dots,y^k_{2n+m}).$\\

\noindent Case I: $p =  q.$
\begin{proof}
In this case, by Lemma \ref{nest}, we can assume that
\begin{equation}
\lim\limits_{k \to \infty} k |z^k - w^k|^2 = 0 \quad \text{and} \quad \lim\limits_{k \to \infty} k |x^k_{2n+\ell} - y^k_{2n+\ell}| = 0, \quad \ell=1, \dots, m.
\end{equation}
Let 
\begin{align*}
&a_k = (\sqrt{k}z^k_1, \dots, \sqrt{k}z^k_n, kx^k_{2n+1}, \dots, kx^k_{2n+m})\\
&b_k = (\sqrt{k}w^k_1, \dots, \sqrt{k}w^k_n, ky^k_{2n+1}, \dots, ky^k_{2n+m}).
\end{align*}
Then we have $\lim\limits_{k \to \infty} a_k = 0$ and $\lim\limits_{k \to \infty} b_k = 0$.

Define
\begin{align*}
A_k(u) &= |P_{k, \delta, s} (u x_k + (1-u)y_k, y_k)|^2, \quad 0 \le u \le 1, \\
B_k(u) &= P_{k, \delta, s} (u x_k + (1-u)y_k, u x_k + (1-u)y_k) \cdot P_{k, \delta, s}(y_k, y_k), \quad 0 \le u \le 1, \\
H_k(u) &= \frac{A_k(u)}{B_k(u)}, \quad 0 \le u \le 1,
\end{align*}
where $P_{k,\delta,s}$ is defined as in (\ref{locexp}).
By Cauchy-Schwarz inequality and the definition of $P_{k, \delta, s}$, we can see that $0 \le H_k(u) \le 1$. Since $\Phi_k(x_k) = \Phi_k(y_k)$, we also have $H_k(0) = H_k(1) = 1$. Hence, for every $k \in \mathbb{N}$, there exists $u_k \in (0, 1)$ such that $H''_k(u_k) \ge 0$.

We denote $\Pi^{(0)}_{b, \psi_0, \le \delta}(x,y)$ by $P_{\psi_0,\delta}(x, y).$ Set
\begin{equation*}
\widehat{H}_k(u) = \frac{|P_{\psi_0, \delta} ((1-u)a_k + u b_k, b_k)|^2}{P_{\psi_0, \delta}((1-u)a_k + u b_k, (1-u)a_k + u b_k) \cdot P_{\psi_0, \delta}(b_k, b_k)}, \quad 0 \le u \le 1.
\end{equation*}
From the definition of $P_{k, \delta, s}$ and (\ref{nowegrasy}), it is straightforward to see that
\begin{equation*}
\liminf_{k \to \infty} \frac{H''_k(u_k)}{|a_k - b_k|^2} = \liminf_{k \to \infty} \frac{\widehat{H}''_k(u_k)}{|a_k - b_k|^2} \ge 0.
\end{equation*}
Write $a_k = (a^k_1, \dots, a^k_n, a^k_{2n+1}, \dots, a^k_{2n+m})$, $b_k = (b^k_1, \dots, b^k_n, b^k_{2n+1}, \dots, b^k_{2n+m}) \in \mathbb{C}^n \times \mathbb{R}^m$ and $a''_k = (a^k_{2n+1}, \dots, a^k_{2n+m})$, $b''_k = (b^k_{2n+1}, \dots, b^k_{2n+m})$. We set

\begin{align*}
(2\pi)^{n+m} \Gamma_k(u) := (2\pi)^{n+m} P_{\psi_0, \delta}((1-u)a_k + u b_k, b_k).
\end{align*}

\noindent Then
\begin{align*}
(2\pi)^{n+m} \Gamma_k(u) 
&= \int_{[-\delta, \delta]^m} \exp \Big[ i(1-u)\langle a''_k - b''_k, \eta \rangle - (1-u)^2 \cdot \frac{1}{2} \sum_{j,r} (M_\eta)_{jr} (a^k_j - b^k_j)(\overline{a}^k_r - \overline{b}^k_r) \\
&\quad + (1-u) \frac{1}{2} \sum_{j,r} (M_\eta)_{jr} (a^k_j \overline{b}^k_r - \overline{a}^k_r b^k_j) \Big]  \det(M_\eta) d\mu(\eta) \\
&= \int_{[-\delta, \delta]^m} e^{E_k(u, \eta)}  \det(M_\eta)  d\mu(\eta),
\end{align*}
where 
\begin{align*}
E_k(u, \eta) = &i(1-u)\langle a''_k - b''_k, \eta \rangle - (1-u)^2 \cdot \frac{1}{2} \sum_{j,r} (M_\eta)_{jr} (a^k_j - b^k_j)(\overline{a}^k_r - \overline{b}^k_r) \\
& + (1-u) \frac{1}{2} \sum_{j,r} (M_\eta)_{jr} (a^k_j \overline{b}^k_r - \overline{a}^k_r b^k_j).
\end{align*}

\noindent Note that 
\begin{align*}
(2\pi)^{n+m} P_{\psi_0, \delta}((1-u)a_k + u b_k, (1-u)a_k + u b_k) = (2\pi)^{n+m} P_{\psi_0, \delta}(b_k, b_k) = \int_{[-\delta, \delta]^m} \det(M_\eta) d\mu(\eta)  
\end{align*}
is a positive constant for all $u \in [0, 1].$
Hence
\begin{equation*}
\widehat{H}''_k(u) = \frac{1}{\big( \int_{[-\delta, \delta]^m} \det(M_\eta) d\mu(\eta) \big)^2} \cdot \Big( \Gamma''_k(u) \overline{\Gamma_k(u)} + 2|\Gamma'_k(u)|^2 + \Gamma_k(u) \overline{\Gamma''_k(u)} \Big).
\end{equation*}

We now calculate $\Gamma'_k$ and $\Gamma''_k.$
\begin{align*}
(2\pi)^{n+m} \Gamma'_k(u) &= \int_{[-\delta, \delta]^m} e^{E_k(u, \eta)} \Big[ -i \langle a''_k - b''_k, \eta \rangle + (1-u) \sum_{j,r} (M_\eta)_{jr} (a^k_j - b^k_j)(\overline{a}^k_r - \overline{b}^k_r) \\
&\quad - \frac{1}{2} \sum_{j,r} (M_\eta)_{jr} (a^k_j \overline{b}^k_r - \overline{a}^k_r b^k_j) \Big]  \det(M_\eta)  d\mu(\eta).
\end{align*}

\begin{align*}
(2\pi)^{n+m} \Gamma''_k(u) &= \int_{[-\delta, \delta]^m} e^{E_k(u, \eta)} \Big[ -i \langle a''_k - b''_k, \eta \rangle + (1-u) \sum_{j,r} (M_\eta)_{jr} (a^k_j - b^k_j)(\overline{a}^k_r - \overline{b}^k_r) \\
&\quad - \frac{1}{2} \sum_{j,r} (M_\eta)_{jr} (a^k_j \overline{b}^k_r - \overline{a}^k_r b^k_j) \Big]^2  \det(M_\eta) d\mu(\eta) \\
&\quad + \int_{[-\delta, \delta]^m} e^{E_k(u, \eta)}  \Big[ - \sum_{j,r} (M_\eta)_{jr} (a^k_j - b^k_j)(\overline{a}^k_r - \overline{b}^k_r) \Big]  \det(M_\eta)  d\mu(\eta).
\end{align*}

\noindent Hence
\begin{align}\label{conest}
&\\
&(2\pi)^{2n+2m} \Big( \Gamma''_k(u) \overline{\Gamma_k(u)} + 2|\Gamma'_k(u)|^2 + \Gamma_k(u) \overline{\Gamma''_k(u)} \Big)\nonumber \\
&= \bigg( \int_{[-\delta, \delta]^m} e^{E_k(u, \eta)} \big(-\langle a''_k - b''_k, \eta \rangle^2\big) \det(M_\eta) d\mu(\eta) \bigg) \bigg( \int_{[-\delta, \delta]^m} e^{\overline{E_k(u, \eta)}} \det(M_\eta) d\mu(\eta) \bigg)\nonumber \\
&\quad + \bigg( \int_{[-\delta, \delta]^m} e^{E_k(u, \eta)} \Big[ -\sum_{j,r} (M_\eta)_{jr} (a^k_j - b^k_j)(\overline{a}^k_r - \overline{b}^k_r) \Big]\det(M_\eta) d\mu(\eta) \bigg)\nonumber\\
&\quad\quad\cdot \bigg( \int_{[-\delta, \delta]^m} e^{\overline{E_k(u, \eta)}}  \det(M_\eta) d\mu(\eta) \bigg)\nonumber \\
&\quad + 2 \bigg| \int_{[-\delta, \delta]^m} e^{E_k(u, \eta)}  (-i \langle a''_k - b''_k, \eta \rangle)  \det(M_\eta) d\mu(\eta) \bigg|^2\nonumber \\
&\quad + \bigg( \int_{[-\delta, \delta]^m} e^{\overline{E_k(u, \eta)}}  \big(-\langle a''_k - b''_k, \eta \rangle^2\big)  \det(M_\eta) d\mu(\eta) \bigg) \left( \int_{[-\delta, \delta]^m} e^{E_k(u, \eta)} \det(M_\eta) d\mu(\eta) \right)\nonumber \\
&\quad + \bigg( \int_{[-\delta, \delta]^m} e^{\overline{E_k(u, \eta)}}  \Big[ -\sum_{j,r} (M_\eta)_{jr} (a^k_j - b^k_j)(\overline{a}^k_r - \overline{b}^k_r) \Big]  \det(M_\eta) d\mu(\eta) \bigg)\nonumber\\
&\quad\quad\cdot \bigg( \int_{[-\delta, \delta]^m} e^{E_k(u, \eta)} \det(M_\eta) d\mu(\eta) \bigg)\nonumber \\
&\quad + O(|a_k - b_k|^3).\nonumber
\end{align}

\noindent We now apply two facts.
\begin{itemize}
    \item[(1)] There is a positive constant $C_1 > 0$ such that for all $\eta \in [-\delta, \delta]^m$ and $z = (z_1, \dots, z_n) \in \mathbb{C}^n$, we have
    \[ \sum_{j,r} (M_\eta)_{jr} z_j \overline{z}_r \ge C_1 |z|^2. \]
    \item[(2)] If $f(\eta)$ is a positive smooth function on $[-\delta, \delta]^m$, then there is $C_2 > 0$ depending only on $f(\eta)$ and $m$ such that for all $a = (a_1, \dots, a_m) \in \mathbb{R}^m$, we have
    \[ \left( \int_{[-\delta, \delta]^m} \langle a, \eta \rangle f(\eta) d\eta \right)^2 - \left( \int_{[-\delta, \delta]^m} \langle a, \eta \rangle^2 f(\eta) d\eta \right) \cdot \left( \int_{[-\delta, \delta]^m} f(\eta) d\eta \right) \le -C_2 |a|^2. \]
\end{itemize}

\noindent From the two facts and (\ref{conest}), we can see that
\begin{equation*}
\liminf_{k \to \infty} \frac{\hat{H}''_k(u_k)}{|a_k - b_k|^2} < 0
\end{equation*}
which is a contradiction.
\end{proof}

\noindent Case II: $p\neq q.$
\begin{proof}
Let $\big(D, (z,x''), \varphi=(\varphi_1,\dots,\varphi_m)\big)$ be a BRT trivialization such that $p\in D$ and $q\notin D.$
Identify $D$ with $W_z \times (-2a, 2a)^m$ and $p$ with $(0,0)$ where $W_z$ is an open subset of $\mathbb{C}^n$ and $a > 0$. Let $D' = W_z \times (-a, a)^m$. Let $B_z(r)$ be the open ball of radius $r > 0$ centered at the origin in $\mathbb{C}^n$. Let $\chi \in \mathcal{C}^\infty_c(B_z(1))$ with $0 \le \chi \le 1$, $\chi = 1$ on $B_z(\frac{1}{2})$. Note that the subscript $z$ in the open sets above is to emphasize that those open sets lie in $\mathbb{C}^n.$ For each $k \in \mathbb{N}$, set
\begin{equation*}
u_k(z) = \chi\Big(\frac{\sqrt{k}z}{k^\epsilon}\Big).
\end{equation*}
For $\eta \in \mathbb{R}^m$ with $|\eta| \le k\delta$, we have
\begin{align}\label{ukest}
&\int_{B_z(\frac{k^\epsilon}{\sqrt{k}})} |u_k(z)|^2 e^{-k\phi(z) - 2\eta_1\varphi_1(z) - \dots - 2\eta_m \varphi_m(z)}V(z) d\mu(z) \\
&= \int_{B_z(\frac{k^\epsilon}{\sqrt{k}})} \bigg|\chi\Big(\frac{\sqrt{k}z}{k^\epsilon}\Big)\bigg|^2 e^{-k\phi(z) - 2\eta_1\varphi_1(z) - \dots - 2\eta_m \varphi_m(z)} V(z)d\mu(z) \notag \\
&= \int_{B_z(k^\epsilon)} \bigg|\chi\Big(\frac{z}{k^\epsilon}\Big)\bigg|^2 e^{-k\big( \phi(\frac{z}{\sqrt{k}}) - 2\frac{\eta_1}{k} \varphi_1(\frac{z}{\sqrt{k}}) - \dots - 2\frac{\eta_m}{k} \varphi_m(\frac{z}{\sqrt{k}})\big)} k^{-n} V\Big(\frac{z}{\sqrt{k}}\Big) d\mu(z) \notag \\
&\lesssim k^{-n} \int_{\mathbb{C}^n} e^{-c|z|^2} d\mu(z) \approx k^{-n},\notag
\end{align}
where $c > 0$ is a positive constant given by the positive definiteness of $M_\eta$ at $p$ for $|\eta|\leq\delta$. 

Moreover,
\begin{align} \label{partukest}
&\int_{B_z(\frac{k^\epsilon}{\sqrt{k}})} |\overline{\partial} u_k(z)|^2  e^{-k\phi(z) - 2\eta_1\varphi_1(z) - \dots - 2\eta_m \varphi_m(z)} V(z)d\mu(z)\\ 
&= \int_{B_z(\frac{k^\epsilon}{\sqrt{k}})} \bigg| \frac{\sqrt{k}}{k^\epsilon} (\overline{\partial}\chi)\Big(\frac{\sqrt{k}z}{k^\epsilon}\Big) \bigg|^2 e^{-k\phi(z) - 2\eta_1\varphi_1(z) - \dots - 2\eta_m \varphi_m(z)} V(z)d\mu(z) \notag \\
&= \int_{\frac{1}{2}k^\epsilon \le |z| \le k^\epsilon} \frac{k}{k^{2\epsilon}} \bigg|(\overline{\partial}\chi)\Big(\frac{z}{k^\epsilon}\Big) \bigg|^2 e^{-k\big( \phi(\frac{z}{\sqrt{k}}) - 2\frac{\eta_1}{k} \varphi_1(\frac{z}{\sqrt{k}}) - \dots - 2\frac{\eta_m}{k} \varphi_m(\frac{z}{\sqrt{k}})\big)} \, k^{-n}\, V\Big(\frac{z}{\sqrt{k}}\Big) d\mu(z) \notag \\
&= O(k^{-\infty}),\nonumber
\end{align}
since the integration domain in the last step is supported where $|z| \ge \frac{1}{2}k^\epsilon$, causing the exponential term to decay at least like $e^{-c' k^{2\epsilon}}$ for some constant $c' > 0$.

For $(z, x'') \in D$, we set
\begin{equation*}
g_k(z, x'') = \int e^{i\langle x'', \eta \rangle}  u_k(z)  e^{-\langle \eta, \varphi\rangle}  \tau\Big(\frac{\eta_1}{k}\Big) \cdots \tau\Big(\frac{\eta_m}{k}\Big)  d\mu(\eta),
\end{equation*}
where $\varphi = (\varphi_1, \dots, \varphi_m)$.

Let $\sigma \in \mathcal{C}^\infty_c((-2a, 2a)^m)$ such that $0 \le \sigma \le 1$ and $\sigma = 1$ on $[-a, a]^m$. Put
\begin{equation*}
G_k(x) = (\sigma(x'')  g_k(z, x'')) \otimes s^k(x).
\end{equation*}
For $x \in D$, we have
\begin{align*}
|G_k(x)|_{h^k} = |\sigma(x'')|  |g_k(z, x'')| \, e^{-\frac{k}{2}\phi(z)} 
&\le \bigg| \int |u_k(z)|  \tau\Big(\frac{\eta_1}{k}\Big) \cdots \tau\Big(\frac{\eta_m}{k}\Big)  e^{-k(\frac{1}{2}\phi(z) + \langle \frac{\eta}{k}, \varphi \rangle)} d\mu(\eta) \bigg| \\
&\le \int \tau\Big(\frac{\eta_1}{k}\Big) \cdots \tau\Big(\frac{\eta_m}{k}\Big) d\mu(\eta) \le (2k\delta)^m,
\end{align*}
and at $x=0$,
\begin{equation*}
|G_k(0)|_{h^k} = g_k(0,0) = \int u_k(0) \, \tau\Big(\frac{\eta_1}{k}\Big) \cdots \tau\Big(\frac{\eta_m}{k}\Big) d\mu(\eta) \approx k^m.
\end{equation*}

We now estimate $\|\overline{\partial}_b G_k\|^2_k$.
\begin{equation*}
\overline{\partial}_b G_k(x) = \big[ (\overline{\partial}_b \sigma(x'')) \, g_k(z, x'') + \sigma(x'') \, (\overline{\partial}_b g_k(z, x'')) \big] \otimes s^k(x).
\end{equation*}
Hence
\begin{align*}
\|\overline{\partial}_b G_k\|^2_{k} \le \
& 2 \underbrace{\int \big|\big(\overline{\partial}_b \sigma(x'')\big)  g_k(z, x'')\big|^2 e^{-k\phi(z)} V(z)d\mu(z) d\mu(x'')}_{(1)} \\
&+ 2 \underbrace{\int \big|\sigma(x'') \, \overline{\partial}_b g_k(z, x'')\big|^2 e^{-k\phi(z)} V(z)d\mu(z) d\mu(x'')}_{(2)}.
\end{align*}

\noindent (1): Fix a point $(z, x'') \in D \setminus D' = W_z \times ((-2a, 2a)^m \setminus (-a, a)^m)$. Then we have $|x_{2n+j_0}| \ge a$ for some $j_0 \in \{1, \dots, m\}$, say $j_0=1$.
For every $N \in \mathbb{N}$, by using integration by parts with respect to $\eta_1$ for $N$ times, we have
\begin{align*}
&\quad\big|g_k(z, x'')\big|\, e^{-\frac{k}{2}\phi(z)}\\ 
&= \bigg| \int e^{i\langle x'', \eta \rangle} u_k(z)  e^{-\langle \eta, \varphi(z) \rangle}  \tau\Big(\frac{\eta_1}{k}\Big) \cdots \tau\Big(\frac{\eta_m}{k}\Big) d\mu(\eta) \bigg|  e^{-\frac{k}{2}\phi(z)} \\
&= \bigg| \int \Big(\frac{1}{ix_{2n+1} - \varphi_1(z)}\Big)^N e^{i\langle x'', \eta \rangle}  u_k(z)  \frac{1}{k^N}  \Big(\tau^{(N)}\Big(\frac{\eta_1}{k}\Big)\Big)  \tau\Big(\frac{\eta_2}{k}\Big) \cdots \tau\Big(\frac{\eta_m}{k}\Big)  e^{-\frac{k}{2}\phi(z) - \langle \eta, \varphi(z) \rangle} d\mu(\eta) \bigg| \\
&\le \Big(\frac{1}{ak}\Big)^N \int |u_k(z)| \, \Big|\tau^{(N)}(\frac{\eta_1}{k})\Big|  \tau\Big(\frac{\eta_2}{k}\Big) \cdots \tau\Big(\frac{\eta_m}{k}\Big) e^{-\frac{k}{2}\phi(z) - \langle \eta, \varphi(z) \rangle} d\mu(\eta) \\
&\le \Big(\frac{1}{ak}\Big)^N \int \Big|\tau^{(N)}\Big(\frac{\eta_1}{k}\Big)  \tau\Big(\frac{\eta_2}{k}\Big) \dots \tau\Big(\frac{\eta_m}{k}\Big)\Big| d\mu(\eta) \lesssim \Big(\frac{1}{k}\Big)^{N-d}.
\end{align*}
Hence
\begin{align*}
(1) &\lesssim \int_{B_z(\frac{k^\epsilon}{\sqrt{k}})} \int_{a \le |x''| \le 2\sqrt{d} a} \big|g_k(z, x'')\big|^2 e^{-k\phi(z)} dx'' d\mu(z)\\
&\lesssim \Big(\frac{1}{k}\Big)^{2(N-d)} \cdot k^m \cdot \Big(\frac{k^\epsilon}{\sqrt{k}}\Big)^{2n}
= k^{-2N+3d+2n\epsilon-n}.
\end{align*}
Since $N$ is arbitrary, we have
\[
(1) = O(k^{-\infty}).
\]

\noindent (2): We can check that, for $j = 1, \dots, n$,
\begin{equation*}
\overline{Z}_j g_k(z, x'') = \int e^{i\langle x'', \eta \rangle} \frac{\partial u_k}{\partial \overline{z}_j}(z)  e^{-\langle \eta, \varphi \rangle} \tau\Big(\frac{\eta_1}{k}\Big) \cdots \tau\Big(\frac{\eta_m}{k}\Big) d\mu(\eta).
\end{equation*}
Hence, by Plancherel's theorem and (\ref{partukest}),
\begin{align*}
(2) &\lesssim \sum_{j=1}^n \int \big|\overline{Z}_j g_k(z, x'')\big|^2\, e^{-k\phi(z)} V(z)d\mu(z) d\mu(x'') \\
&= \sum_{j=1}^n \int \bigg|\int e^{i\langle x'', \eta \rangle} \frac{\partial u_k}{\partial \overline{z}_j}(z)  e^{-\langle \eta, \varphi \rangle} \tau\Big(\frac{\eta_1}{k}\Big) \cdots \tau\Big(\frac{\eta_m}{k}\Big) d\mu(\eta) \bigg|^2  e^{-k\phi(z)} V(z)d\mu(z) d\mu(x'') \\
&\lesssim \sum_{j=1}^n \int \bigg|\frac{\partial u_k}{\partial \overline{z}_j}(z)\bigg|^2 \bigg|\tau\Big(\frac{\eta_1}{k}\Big) \cdots \tau\Big(\frac{\eta_m}{k}\Big)\bigg|^2 e^{-k\phi(z) - 2\langle \eta, \varphi \rangle} V(z)d\mu(z) d\mu(\eta) \\
&= O(k^{-\infty}).
\end{align*}
Thus, we have $\|\overline{\partial}_b G_k\|^2_k = O(k^{-\infty})$.
\\

Next, we estimate $\big\|\Box_{b,k} G_k\big\|^2_k$.
\begin{equation*}
\Box_{b,k} G_k(x) = \big(\Box_{b,k} (\sigma(x'') g_k(z, x''))\big) \otimes s^k = \bigg(\sum_{j=1}^n \overline{Z}^{*,k\phi}_j \overline{Z}_j \big(\sigma(x'') g_k(z, x'')\big)\bigg) \otimes s^k.
\end{equation*}

For $j, t \in \{1, \dots, n\}$,
\begin{align*}
&\overline{Z}^{*,k\phi}_t \overline{Z}_j\, \big(\sigma(x'') g_k(z, x'')\big)\\ 
&= \big(-Z_t + k Z_t(\phi) + S_t(z)\big) \overline{Z}_j \big(\sigma(x'') g_k(z, x'')\big) \\
&= -Z_t \overline{Z}_j \big(\sigma(x'') g_k(z, x'')\big) + k Z_t(\phi) \, \overline{Z}_j \big(\sigma(x'') g_k(z, x'')\big) + S_t(z) \, \overline{Z}_j \big(\sigma(x'') g_k(z, x'')\big).
\end{align*}

It suffices to estimate
\begin{equation*}
\int \Big|Z_t \overline{Z}_j \big(\sigma(x'') g_k(z, x'')\big)\Big|^2 e^{-k\phi(z)} V(z)d\mu(z) d\mu(x'').
\end{equation*}

We have
\begin{align*}
Z_t \overline{Z}_j \big(\sigma(x'') g_k(z, x'')\big) &= \underbrace{\big(Z_t \overline{Z}_j \sigma(x'')\big) \, g_k(z, x'')}_{(1)'} + \underbrace{\big(\overline{Z}_j \sigma(x'')\big) \big(Z_t g_k(z, x'')\big)}_{(2)'} \\
&\quad + \underbrace{\big(Z_t \sigma(x'')\big) \big(\overline{Z}_j g_k(z, x'')\big)}_{(3)'} + \underbrace{\sigma(x'') (Z_t \overline{Z}_j g_k(z, x''))}_{(4)'}.
\end{align*}
Hence,
\begin{align*}
&\int \Big|Z_t \overline{Z}_j \big(\sigma(x'') g_k(z, x'')\big)\Big|^2 e^{-k\phi(z)} V(z)d\mu(z) d\mu(x'') \\
&\le 4 \int \big(|(1)'|^2 + |(2)'|^2 + |(3)'|^2 + |(4)'|^2\big) e^{-k\phi(z)} V(z)d\mu(z) d\mu(x'').
\end{align*}
Using a similar argument as for (1), we have the following:
\begin{equation*}
\int |(1)'|^2 e^{-k\phi(z)} V(z)d\mu(z) d\mu(x'') = O(k^{-\infty}).
\end{equation*}
For $(2)'$, we first check that
\begin{equation*}
Z_t g_k(z, x'') = \int e^{i\langle x'', \eta \rangle} \Big( \frac{\partial u_k}{\partial z_t} - 2 u_k(z)\sum_{\ell=1}^m \eta_\ell \frac{\partial \varphi_\ell}{\partial z_t} \Big) e^{-\langle \eta, \varphi \rangle}  \tau\Big(\frac{\eta_1}{k}\Big) \cdots \tau\Big(\frac{\eta_m}{k}\Big) d\mu(\eta).
\end{equation*}
Fix a point $(z, x'') \in D \setminus D'$ and $N \in \mathbb{N}.$ We may assume $|x_{2n+1}| \ge a.$ By using integration by parts for $N$ times, we have
\begin{align*}
&Z_t g_k(z, x'')\\
&= \int \Big( \frac{-1}{ix_{2n+1} - \varphi_1} \Big)^N  e^{i\langle x'', \eta \rangle - \langle \eta, \varphi \rangle} \bigg[ N \Big(-2u_k(z) \frac{\partial \varphi_1}{\partial z_t}(z) \tau^{(N-1)}\Big(\frac{\eta_1}{k}\Big) \tau\Big(\frac{\eta_2}{k}\Big) \cdots \tau\Big(\frac{\eta_m}{k}\Big)\Big) \frac{1}{k^{N-1}}\\
&\quad\quad+\Big( \frac{\partial u_k}{\partial z_t} - 2u_k \sum_{\ell=1}^m \eta_\ell \frac{\partial \varphi_\ell}{\partial z_t} \Big)  \tau^{(N)}\Big(\frac{\eta_1}{k}\Big) \tau\Big(\frac{\eta_2}{k}\Big)\cdots \tau\Big(\frac{\eta_m}{k}\Big) \frac{1}{k^N} \bigg] d\mu(\eta).
\end{align*}
From this, we can show that
\begin{equation*}
\int |(2)'|^2 e^{-k\phi(z)} V(z)d\mu(z) d\mu(x'') = O(k^{-\infty}).
\end{equation*}
Using an argument similar to (2) in estimating  $\|\overline{\partial}_b G_k\|^2_k$, we have the following.
\begin{equation*}
\int |(3)'|^2 e^{-k\phi(z)} V(z)d\mu(z) d\mu(x'') = O(k^{-\infty}).
\end{equation*}
For $ (4)'$, we first check that
\begin{equation*}
Z_t \overline{Z}_j g_k(z, x'') = \int e^{i\langle x'', \eta \rangle} \Big( \frac{\partial^2 u_k}{\partial z_t \partial \overline{z}_j} - 2\frac{\partial u_k}{\partial \overline{z}_j} \sum_{\ell=1}^m \eta_\ell \frac{\partial \varphi_\ell}{\partial z_t} \Big) e^{-\langle \eta, \varphi \rangle}  \tau\Big(\frac{\eta_1}{k}\Big) \cdots \tau\Big(\frac{\eta_m}{k}\Big) d\mu(\eta).
\end{equation*}
Hence
\begin{align*}
&\int |(4)'|^2 e^{-k\phi(z)} V(z)d\mu(z) d\mu(x'') \\
&= \int |\sigma(x'')|^2 \cdot \bigg| \int e^{i\langle x'', \eta \rangle} \Big(\frac{\partial^2 u_k}{\partial z_t \partial \overline{z}_j} - 2\frac{\partial u_k}{\partial \overline{z}_j} \sum_{\ell=1}^m \eta_\ell \frac{\partial \varphi_\ell}{\partial z_t}\Big) e^{-\langle \eta, \varphi \rangle}  \tau\Big(\frac{\eta_1}{k}\Big) \cdots \tau\Big(\frac{\eta_m}{k}\Big) d\mu(\eta) \bigg|^2\\
&\quad \quad \cdot e^{-k\phi(z)} V(z)d\mu(z) d\mu(x'') \\
&\lesssim \int \bigg|\frac{\partial^2 u_k}{\partial z_t \partial \overline{z}_j} - 2\frac{\partial u_k}{\partial \overline{z}_j} \sum_{\ell=1}^m \eta_\ell \frac{\partial \varphi_\ell}{\partial z_t}\bigg|^2  \bigg|\tau\Big(\frac{\eta_1}{k}\Big) \cdots \tau\Big(\frac{\eta_m}{k}\Big)\bigg|^2 e^{-k\phi(z) - 2\langle \eta, \varphi \rangle} V(z)d\mu(z) d\mu(\eta) \\
&= O(k^{-\infty}).
\end{align*}
Thus, we have
\begin{equation*}
\big\| \Box_{b,k} G_k \big\|^2_k = O(k^{-\infty}).
\end{equation*}
In this manner, we can show that
\begin{equation*}
\big\| (\Box_{b,k})^j G_k \big\|^2_k = O(k^{-\infty}) \quad \text{for all } j \in \mathbb{N}.
\end{equation*}

 We have the following type of Kohn's estimate, for every $s\in \mathbb{N},$ there is a constant $C_s>0$ independent of $k$ and $N_s\in \mathbb{N}$ such that
\begin{equation} \label{kohnest}
\|u\|_{s+1, k} \le C_s k^{N_s} \bigg( \big\|\Box_{b,k} u\big\|_{s,k} + \sum_{\ell=1}^m \|Y_\ell u\|_{s,k} + \|u\|_k \bigg), \quad u \in \mathcal{C}^\infty(X, L^k),
\end{equation}
where $\|\cdot\|_{s,k}$ denotes the standard Sobolev norm of order $s$ on the Sobolev space $H^s(X,L^k).$
We apply this inequality to $F_{k,\delta} \big(I -\Pi^{(0)}_{k,({\mathbb R}^m\curvearrowright,\leq k\delta)}\big) G_k$ and obtain
\begin{align*}
&\big\|F_{k,\delta} \big(I -\Pi^{(0)}_{k,({\mathbb R}^m\curvearrowright,\leq k\delta)}\big) G_k \big\|_{s+1, k}\\ 
&\le C_s k^{N_s} \bigg( \big\|\Box_{b,k} F_{k,\delta} \big(I -\Pi^{(0)}_{k,({\mathbb R}^m\curvearrowright,\leq k\delta)}\big) G_k \big\|_{s,k} + \sum_{\ell=1}^m \big\|Y_\ell F_{k,\delta} \big(I -\Pi^{(0)}_{k,({\mathbb R}^m\curvearrowright,\leq k\delta)}\big) G_k \big\|_{s,k}\\
&\quad + \big\|F_{k,\delta} \big(I -\Pi^{(0)}_{k,({\mathbb R}^m\curvearrowright,\leq k\delta)}\big) G_k \big\|_k\bigg).
\end{align*}
We will show that for all $s \in \mathbb{N}$, there exist $\widetilde{N}_s \in \mathbb{N}$ and $\widetilde{C}_s > 0$ independent of $k$ such that
\begin{equation}\label{kohnest2}
\big\|F_{k,\delta} \big(I -\Pi^{(0)}_{k,({\mathbb R}^m\curvearrowright,\leq k\delta)}\big) G_k \big\|_{s, k} \le \widetilde{C}_s k^{\widetilde{N}_s} \bigg( \sum_{j=0}^s \big\| \big(\Box_{b,k}\big)^j F_{k,\delta} \big(I -\Pi^{(0)}_{k,({\mathbb R}^m\curvearrowright,\leq k\delta)}\big) G_k \big\|_k \bigg).
\end{equation}
Plugging $s=0$ into (\ref{kohnest}), we have
\begin{align*}
&\big\|F_{k,\delta} \big(I -\Pi^{(0)}_{k,({\mathbb R}^m\curvearrowright,\leq k\delta)}\big) G_k \big\|_{1, k}\\
&\le C_0 k^{N_0} \bigg( \big\|\Box_{b,k} F_{k,\delta} \big(I -\Pi^{(0)}_{k,({\mathbb R}^m\curvearrowright,\leq k\delta)}\big) G_k \big\|_k + \sum_{\ell=1}^m \big\|Y_\ell F_{k,\delta} \big(I -\Pi^{(0)}_{k,({\mathbb R}^m\curvearrowright,\leq k\delta)}\big) G_k \big\|_k\\
&\quad + \big\|F_{k,\delta} \big(I -\Pi^{(0)}_{k,({\mathbb R}^m\curvearrowright,\leq k\delta)}\big) G_k \big\|_k \bigg).
\end{align*}
Using $\big\|Y_\ell F_{k,\delta} \big(I -\Pi^{(0)}_{k,({\mathbb R}^m\curvearrowright,\leq k\delta)}\big)G_k \big\|_k \le k \delta \big\|F_{k,\delta} \big(I -\Pi^{(0)}_{k,({\mathbb R}^m\curvearrowright,\leq k\delta)}\big) G_k \big\|_k$ for all $\ell=1, \dots, m$, we get
\begin{align*}
&\big\|F_{k,\delta} \big(I -\Pi^{(0)}_{k,({\mathbb R}^m\curvearrowright,\leq k\delta)}\big) G_k \big\|_{1, k}\\
&\le C_0 (1 + d\delta) \, k^{N_0+1} \, \Big( \big\|\Box_{b,k} F_{k,\delta} \big(I -\Pi^{(0)}_{k,({\mathbb R}^m\curvearrowright,\leq k\delta)}\big) G_k \big\|_k + \big\|F_{k,\delta} \big(I -\Pi^{(0)}_{k,({\mathbb R}^m\curvearrowright,\leq k\delta)}\big) G_k \big\|_k \Big)
\end{align*}
which proves (\ref{kohnest2}) for the case $s=1$.
Assume that (\ref{kohnest2}) holds for all $s \in \mathbb{N}$ with $s \le s_0$ for some $s_0 \in \mathbb{N}$. So we have
\begin{align*}
\big\|F_{k,\delta} \big(I -\Pi^{(0)}_{k,({\mathbb R}^m\curvearrowright,\leq k\delta)}\big) G_k \big\|_{s_0, k} \le \widetilde{C}_{s_0} k^{\widetilde{N}_{s_0}} \bigg( \sum_{j=1}^{s_0} \big\| (\Box_{b,k})^j \big(I -\Pi^{(0)}_{k,({\mathbb R}^m\curvearrowright,\leq k\delta)}\big) G_k \big\|_k \bigg).
\end{align*}
Taking $s = s_0 $ in (\ref{kohnest}), we get
\begin{align}\label{kohnestlem1}
&\big\|F_{k,\delta} \big(I -\Pi^{(0)}_{k,({\mathbb R}^m\curvearrowright,\leq k\delta)}\big) G_k \big\|_{s_0+1, k}\\
&\le C_{s_0} k^{N_{s_0}} \bigg( \big\|\Box_{b,k} F_{k,\delta}\big(I -\Pi^{(0)}_{k,({\mathbb R}^m\curvearrowright,\leq k\delta)}\big) G_k \big\|_{s_0, k} + \sum_{\ell=1}^m \big\|Y_\ell F_{k,\delta} \big(I -\Pi^{(0)}_{k,({\mathbb R}^m\curvearrowright,\leq k\delta)}\big) G_k \big\|_{s_0, k}\nonumber\\
&\quad + \big\|F_{k,\delta} \big(I -\Pi^{(0)}_{k,({\mathbb R}^m\curvearrowright,\leq k\delta)}\big) G_k \big\|_k \bigg).\nonumber
\end{align}

Applying the induction hypothesis to $\Box_{b,k} F_{k, \delta} \big(I -\Pi^{(0)}_{k,({\mathbb R}^m\curvearrowright,\leq k\delta)}\big) G_k$, we get
\begin{align*}
&\big\| \Box_{b,k} F_{k, \delta} \big(I -\Pi^{(0)}_{k,({\mathbb R}^m\curvearrowright,\leq k\delta)}\big) G_k \big\|_{s_0, k}\\
&\leq \widetilde{C}_{s_0} k^{\widetilde{N}_{s_0}} \bigg( \sum_{j=0}^{s_0} \big\| (\Box_{b,k})^j \big(I -\Pi^{(0)}_{k,({\mathbb R}^m\curvearrowright,\leq k\delta)}\big) \Box_{b,k} F_{k, \delta} \big(I -\Pi^{(0)}_{k,({\mathbb R}^m\curvearrowright,\leq k\delta)}\big) G_k \big\|_k \bigg).\nonumber
\end{align*}
Note that $\Pi^{(0)}_{k,({\mathbb R}^m\curvearrowright,\leq k\delta)} \Box_{b,k} = 0$. Hence we have
\begin{align}\label{kohnestlem2}
\big\| \Box_{b,k} F_{k, \delta} \big(I -\Pi^{(0)}_{k,({\mathbb R}^m\curvearrowright,\leq k\delta)}\big) G_k \big\|_{s_0, k} \leq \widetilde{C}_{s_0} k^{\widetilde{N}_{s_0}} \bigg( \sum_{j=1}^{s_0+1} \big\| \big(\Box_{b,k}\big)^j F_{k, \delta} \big(I -\Pi^{(0)}_{k,({\mathbb R}^m\curvearrowright,\leq k\delta)}\big) G_k \big\|_k \bigg).
\end{align}
Applying induction hypothesis to $Y_\ell G_k, \ell=1, \dots, m$, we get
\begin{align}\label{kohnestlem3}
\big\| Y_\ell F_{k, \delta} \big(I -\Pi^{(0)}_{k,({\mathbb R}^m\curvearrowright,\leq k\delta)}\big) G_k \big\|_{s_0, k} 
&\leq \widetilde{C}_{s_0} k^{\widetilde{N}_{s_0}} \bigg( \sum_{j=0}^{s_0} \big\| (\Box_{b,k})^j F_{k, \delta} \big(I -\Pi^{(0)}_{k,({\mathbb R}^m\curvearrowright,\leq k\delta)}\big) Y_\ell G_k \big\|_k \bigg) \\
&\leq \widetilde{C}_{s_0} k^{\widetilde{N}_{s_0}} \, (k\delta) \bigg( \sum_{j=0}^{s_0} \big\| (\Box_{b,k})^j F_{k, \delta} \big(I -\Pi^{(0)}_{k,({\mathbb R}^m\curvearrowright,\leq k\delta)}\big) G_k \big\|_k \bigg).\nonumber
\end{align}
From (\ref{kohnestlem1}), (\ref{kohnestlem2}) and (\ref{kohnestlem3}), we have
\begin{align*}
&\big\| F_{k, \delta} \big(I -\Pi^{(0)}_{k,({\mathbb R}^m\curvearrowright,\leq k\delta)}\big) G_k \big\|_{s_0+1, k}\\
&\leq C_{s_0} k^{N_{s_0}} \Big( \big\| \Box_{b,k} F_{k, \delta} \big(I -\Pi^{(0)}_{k,({\mathbb R}^m\curvearrowright,\leq k\delta)}\big) G_k \big\|_{s_0, k} + \sum_{\ell=1}^m \big\| Y_\ell F_{k, \delta} \big(I -\Pi^{(0)}_{k,({\mathbb R}^m\curvearrowright,\leq k\delta)}\big) G_k \big\|_{s_0, k} \\
&\quad + \big\| F_{k, \delta} \big(I -\Pi^{(0)}_{k,({\mathbb R}^m\curvearrowright,\leq k\delta)}\big) G_k \big\|_k \Big) \\
&\leq C_{s_0} k^{N_{s_0}} \bigg( \widetilde{C}_{s_0} k^{\widetilde{N}_{s_0}} \Big( \sum_{j=1}^{s_0+1} \big\| (\Box_{b,k})^j F_{k, \delta} \big(I -\Pi^{(0)}_{k,({\mathbb R}^m\curvearrowright,\leq k\delta)}\big) G_k \big\|_k \Big)\\
&\quad + d\widetilde{C}_{s_0} k^{\widetilde{N}_{s_0}} \, (k\delta) \sum_{j=0}^{s_0} \big\| (\Box_{b,k})^j F_{k, \delta} \big(I -\Pi^{(0)}_{k,({\mathbb R}^m\curvearrowright,\leq k\delta)}\big) G_k \big\|_k 
+ \big\| F_{k, \delta} \big(I -\Pi^{(0)}_{k,({\mathbb R}^m\curvearrowright,\leq k\delta)}\big) G_k \big\|_k \bigg) \\
&= \widetilde{C}_{s_0+1} k^{\widetilde{N}_{s_0+1}} \bigg( \sum_{j=0}^{s_0+1} \big\| (\Box_{b,k})^j F_{k, \delta} \big(I -\Pi^{(0)}_{k,({\mathbb R}^m\curvearrowright,\leq k\delta)}\big) G_k \big\|_k \bigg).
\end{align*}
By induction, the inequality (\ref{kohnest2}) follows.

From $\big\| (\Box_{b,k})^j G_k \big\|_k = O(k^{-\infty})$, $\big\| F_{k, \delta} \big(I -\Pi^{(0)}_{k,({\mathbb R}^m\curvearrowright,\leq k\delta)}\big) G_k \big\|_k = O(k^{-\infty})$ and (\ref{kohnest2}), we have
\begin{align*}
F_{k, \delta} \big(I -\Pi^{(0)}_{k,({\mathbb R}^m\curvearrowright,\leq k\delta)}\big) G_k = O(k^{-\infty}),
\end{align*}
and hence
\begin{align*}
F_{k, \delta}\Pi^{(0)}_{k,({\mathbb R}^m\curvearrowright,\leq k\delta)} G_k = F_{k, \delta} G_k + O(k^{-\infty}).
\end{align*}
By Lemma \ref{peakest} (proved below) and the definition of $B_k$ and $G_k,$ we can see that 
\begin{align*}
    \big|F_{k, \delta} G_k(0)\big|_{h^k} \approx k^m \quad \text{and} \quad \big|F_{k, \delta} G_k(y_k)\big|_{h^k} = O(k^{-\infty}).
\end{align*}
Thus, we can find $u_{k_j} = F_{k_j,\delta}f_{k_j},\ f_{k_j}\in \mathcal{H}^{(0)}_{b,\leq k_j\delta}(X, L^{k_j}),\ 1\leq k_1 < k_2 <\cdots,$ such that
\begin{equation}\label{xlarge}
    |u_{k_j}(x_{k_j})|_{h^{k_j}} \ge 1, \quad |u_{k_j}(y_{k_j})|_{h^{k_j}} \le \frac{1}{2}.
\end{equation}
Similarly, we can find $v_{k_{j_s}} = F_{k_{j_s},\delta}h_{k_{j_s}},\ h_{k_{j_s}}\in \mathcal{H}^{(0)}_{b,\leq k_{j_s}\delta}(X, L^{k_{j_s}}),\ s\in \mathbb{N},$ such that
\begin{equation}\label{ylarge}
    |v_{k_{j_s}}(x_{k_{j_s}})|_{h^{k_{j_s}}} \le \frac{1}{2}, \quad |v_{k_{j_s}}(y_{k_{j_s}})|_{h^{k_{j_s}}} \ge 1.
\end{equation}
For simplicity, we may assume that (\ref{xlarge}) and (\ref{ylarge}) hold for all $k$ large. In summary, we may assume that we can find $u_k = F_{k,\delta}f_k,\ v_k = F_{k,\delta}h_k,\ f_k,h_k\in \mathcal{H}^{(0)}_{b,\leq k\delta}(X, L^k)$ such that for $k$ large, we have
\begin{align}\label{xlarge2}
|u_k(x_k)|_{h^k} \ge 1, \quad |u_k(y_k)|_{h^k} \le \frac{1}{2},
\end{align}
and
\begin{align}\label{ylarge2}
|v_k(x_k)|_{h^k}  \le \frac{1}{2}, \quad |v_k(y_k)|_{h^k} \ge 1.
\end{align}
Now, $\Phi_{k,\delta}(x_k) = \Phi_{k,\delta}(y_k)$ implies that
\begin{align*}
|u_k(x_k)|_{h^k} = r_k |u_k(y_k)|_{h^k} , \quad |v_k(x_k)|_{h^k} = r_k |v_k(y_k)|_{h^k} ,
\end{align*}
where $r_k \in \mathbb{R}_+$, for each $k$. We deduce from (\ref{xlarge2}) that $r_k \ge 2$, for $k$ large. But (\ref{ylarge2}) implies
that $r_k \le \frac{1}{2}$, for $k$ large. We get a contradiction. 

\end{proof}

\begin{lem} \label{peakest}
Let $u_k \in \mathcal{C}^\infty(X, L^k)$, $k \in \mathbb{N}$. Assume that there are constants $C>0$ and $A>0$ independent
of $k$ such that $|u_k(x)|_{h^k}^2 \leq C k^A$ for every $x \in X$. Put
\begin{equation*}
    B_k= \Big(\frac{k}{2\pi}\Big)^{2n+m} \int e^{i k \langle x - y, \xi\rangle} \tau(\xi_1) \cdots \tau(\xi_m) d\mu(\xi).
\end{equation*}
Then we have
\begin{align*}
(F_{k, \delta} - B_k) u_k(x) = O(k^{-\infty}) \text{ on } D' \subset\subset D.
\end{align*}
Moreover, if $\operatorname{Supp}(u_k) \cap D = \emptyset$, then
\begin{align*}
F_{k, \delta} u_k = O(k^{-\infty}) \text{ on } D.
\end{align*}
\end{lem}

\begin{proof}
Identify $D = W_z \times (-3a, 3a)^m$ where $W_z$ is an open subset of $\mathbb{C}^n$ and $a>0$.
Let $D' = U' \times (-2a, 2a)^m$ where $U' \subset\subset U$ is an open subset of $U$ and let $D'' = U' \times (-a, a)^m$.
Let $\rho \in \mathcal{C}^\infty_c((-3a, 3a)^m)$ with $\rho = 1$ on $(-2a, 2a)^m$. Let $R_k : \mathcal{C}^\infty_c(D'') \to \mathcal{C}^\infty(D'')$
be the continuous operator given by
\begin{align*}
(R_k u)(x) &= \Big(\frac{1}{2\pi}\Big)^m \sum_{(m_1^1, \dots, m_{N_m}^m) \in \mathbb{Z}^N} \int_{\mathbb T^N} e^{i \langle x''-y'', \xi\rangle + i \sum\limits_{\ell=1}^m \big(\sum\limits_{j=1}^{N_\ell} m^\ell_j \beta^\ell_j\big)y_{2n+\ell} - i \sum\limits_{\ell=1}^m \big(\sum\limits_{j=1}^{N_\ell} m^\ell_j \theta^\ell_j\big)} \\ &\quad\quad\cdot\tau\Big(\frac{\xi_1}{k}\Big) \cdots \tau\Big(\frac{\xi_m}{k}\Big) (1 - \rho(y''))  u\big((e^{i \theta_1^1}, \dots, e^{i \theta_{N_m}^m}) \circ x'\big)  d\mathbb{T}_N  d\mu(\xi)  d\mu(y'').
\end{align*}
Recall that $d\mathbb{T}_N = \big(\frac{1}{2\pi}\big)^Nd\theta^1_1\wedge\cdots\wedge d\theta^1_{N_1}\wedge\cdots\wedge d\theta^m_1\wedge\cdots\wedge d\theta^m_{N_m}.$ 
For $x \in D''$ and $y \in U \times ((-3a, 3a)^m \setminus (-2a, 2a)^m)$, we have $|x_{2n+j_0} - y_{2n+j_0}| > a$ for some $j_0 \in \{1, \dots, m\}$.

Fix $x \in D''$. For each $(m_1^1, \dots, m_{N_m}^m) \in \mathbb{Z}^N$, by integration by parts with respect
to $\xi$ and $y''$ for several times, we have
\begin{align*}
&(2\pi)^m(R_k u)(x)\\
&=\int_{\mathbb T^N} \frac{1}{\sum_{\ell=1}^m |x_{2n+\ell} - y_{2n+\ell}|^{4Md}} \cdot \frac{1}{\prod_{\ell=1}^m \big(1 + (\sum_{j=1}^{N_\ell} m^\ell_j \beta^\ell_j)^2\big)^M}\\
&\quad\cdot \frac{1}{k^{4Md}}\, e^{i \sum\limits_{\ell=1}^m \big(\sum\limits_{j=1}^{N_\ell} m^\ell_j \beta^\ell_j\big)y_{2n+\ell} - i \sum\limits_{\ell=1}^m \big(\sum\limits_{j=1}^{N_\ell} m^\ell_j \theta^\ell_j\big)} \bigg( \sum_{\ell=1}^m \tau\Big(\frac{\xi_1}{k}\Big)\cdots\tau^{(4Md)}\Big(\frac{\xi_\ell}{k}\Big) \cdots \tau\Big(\frac{\xi_m}{k}\Big) \bigg)\\
&\quad\cdot \Big[ \langle D_{y_{2n+m}}\rangle^{2M} \cdots \langle D_{y_{2n+1}}\rangle^{2M} \big( e^{i\langle x'' - y'', \xi\rangle} (1 - \rho(y'')) \big) \Big] \\
&\quad \cdot u\big((e^{i \theta_1^1}, \dots, e^{i \theta_{N_m}^m}) \circ x'\big)  d\mathbb{T}_N  d\mu(\xi)  d\mu(y''),
\end{align*}
for every $M \in \mathbb{N},$ where $\langle D_{y_{2n+\ell}}\rangle = \sqrt{1 + (\frac{1}{i} \frac{\partial}{\partial y_{2n+\ell}})^2}$ for $\ell=1, \dots, m$.

Thus,
\begin{align*}
&|(R_k u)(x)|\\
&\le C_M k^{-d(2M-1)} \sum_{(m_1^1, \dots, m_{N_m}^m) \in \mathbb{Z}^N} \frac{1}{\prod_{\ell=1}^m \big(1 + (\sum_{j=1}^{N_\ell} m^\ell_j \beta^\ell_j)^2\big)^M}  \int \big|u\big(( e^{i \theta_1^1}, \dots, e^{i \theta_{N_m}^m}) \circ x'\big)\big| d\mathbb{T}_N
\end{align*}
for all $x \in D''$. By our assumption that $|u_k(x)|_{h^k}^2 \le C k^A$ for all $k \in \mathbb{N}$, we have
\begin{align*}
|(R_k u)(x)| = O(k^{-\infty}) \text{ on } x \in D''.
\end{align*}

We next show that $F_{k, \delta} = B_k + R_k$. For $u \in \mathcal{C}^\infty_c(D'')$, by Fourier inversion formula, 
\begin{align*}
(B_k u)(x) &= \Big(\frac{k}{2\pi}\Big)^m \int e^{i k \langle x'' - y'', \xi\rangle} \tau(\xi_1) \cdots \tau(\xi_m) \rho(y'') u(x', y'') d\mu(y'') d\mu(\xi) \\
&= \Big(\frac{k}{2\pi}\Big)^m \int e^{i k \langle x'' - y'', \xi\rangle} \tau(\xi_1) \cdots \tau(\xi_m) \rho(y'') \sum_{(m_1^1, \dots, m_{N_m}^m) \in \mathbb{Z}^N}
e^{i \sum\limits_{\ell=1}^m \big(\sum\limits_{j=1}^{N_\ell} m^\ell_j \beta^\ell_j\big) y_{2n+\ell}} \\
&\quad\quad \cdot \int_{\mathbb T^N} e^{-i \sum\limits_{\ell=1}^m \big(\sum\limits_{j=1}^{N_\ell} m^\ell_j \theta^\ell_j\big)} u\big((e^{i \theta_1^1}, \dots, e^{i \theta_{N_m}^m}) \circ x'\big)  d\mathbb{T}_N  d\mu(y'')  d\mu(\xi) \\
&=\Big(\frac{1}{2\pi}\Big)^m \sum_{(m_1^1, \dots, m_{N_m}^m) \in \mathbb{Z}^N} \int_{\mathbb T^N} e^{i \langle x''-y'', \xi\rangle + i \sum\limits_{\ell=1}^m \big(\sum\limits_{j=1}^{N_\ell} m^\ell_j \beta^\ell_j\big)y_{2n+\ell} - i \sum\limits_{\ell=1}^m \big(\sum\limits_{j=1}^{N_\ell} m^\ell_j \theta^\ell_j\big)} \\ &\quad\quad\cdot\tau\Big(\frac{\xi_1}{k}\Big) \cdots \tau\Big(\frac{\xi_m}{k}\Big) \rho(y'')  u\big((e^{i \theta_1^1}, \dots, e^{i \theta_{N_m}^m}) \circ x'\big)  d\mathbb{T}_N  d\mu(\xi)  d\mu(y'').
\end{align*}
So
\begin{align*}
(B_k + R_k) u(x) 
& =\Big(\frac{1}{2\pi}\Big)^m \sum_{(m_1^1, \dots, m_{N_m}^m) \in \mathbb{Z}^N} \int_{\mathbb T^N} e^{i \langle x''-y'', \xi\rangle + i \sum\limits_{\ell=1}^m \big(\sum\limits_{j=1}^{N_\ell} m^\ell_j \beta^\ell_j\big)y_{2n+\ell} - i \sum\limits_{\ell=1}^m \big(\sum\limits_{j=1}^{N_\ell} m^\ell_j \theta^\ell_j\big)} \\ 
&\quad\quad\cdot\tau\Big(\frac{\xi_1}{k}\Big) \cdots \tau\Big(\frac{\xi_m}{k}\Big)  u\big((e^{i \theta_1^1}, \dots, e^{i \theta_{N_m}^m}) \circ x'\big)  d\mathbb{T}_N  d\mu(\xi) d\mu(y'').
\end{align*}
Note that we have for every $\alpha \in \mathbb{R}^m$,
\begin{align}\label{delta}
\int e^{i \langle\alpha, y''\rangle} \cdot e^{-i \langle y'', \xi\rangle} d\mu(y'') = (2\pi)^m \delta_\alpha(\xi),
\end{align}
where the integral is defined as an oscillatory integral.

By using (\ref{delta}) and the Fourier inversion formula, we have
\begin{align*}
(B_k + R_k) u(x)
&= \sum_{(m_1^1, \dots, m_{N_m}^m) \in \mathbb{Z}^N} \tau\bigg(\frac{\sum_{j=1}^{N_1} m^1_j \beta^1_j}{k}\bigg) \cdots \tau\bigg(\frac{\sum_{j=1}^{N_m} m^m_j \beta^m_j }{k}\bigg)  e^{i \sum\limits_{\ell=1}^m \big(\sum\limits_{j=1}^{N_\ell} m^\ell_j \beta^\ell_j\big)x_{2n+\ell}}\\
&\quad\quad\quad\quad\cdot\int_{\mathbb T^N} e^{- i \sum\limits_{\ell=1}^m \big(\sum\limits_{j=1}^{N_\ell} m^\ell_j \theta^\ell_j\big)}  u\big(( e^{i \theta_1^1}, \dots, e^{i \theta_{N_m}^m}) \circ x'\big)  d\mathbb{T}_N \\
&= F_{k, \delta} u(x).
\end{align*}
\end{proof}

\subsection{Examples}
In this section, we give some examples to illustrate our Kodaira type embedding theorem.
\begin{example}
Let $M$ be a connected compact projective manifold of dimension $n$ and let $(A,h^A)$ be a positive holomorphic line bundle over $M.$ Choose Hermitian holomorphic line bundles
\[(B_1,h^1),\ldots,(B_m,h^m)\] over $M.$ Define
\[
X:=\bigg\{(v_1,\ldots,v_m)\in\bigoplus_{\ell=1}^mB_\ell;\,|v_\ell|_{h^\ell}=1,\,\ell=1,\dots,m\bigg\}.
\]
Then $X$ is the fiber product of the unit circle bundles of $B_\ell,$ with the natural projection $p:X\to M.$
It is easy to see that $X$ is a compact generic CR submanifold of the holomorphic vector bundle $\bigoplus B_\ell$ of  dimension $2n+m$ and of CR codimension $m.$ The action
\[
\zeta_t(v_1,\ldots,v_m):=(e^{it_1}v_1,\ldots,e^{it_m}v_m), \qquad t\in \mathbb{R}^m,
\]
is a transversal locally free CR $\mathbb{R}^m$-action on $X$ and factors through a free $\mathbb T^m$-action ($\mathbb T^m$ is the $m$-dimensional torus).
Take
\[
L=p^*A,\qquad h^L=p^*h^A.
\]
Then we can see that the metric $h^L$ is rigid, and
\[
R^L=2p^*R^A>0.
\]
Hence, by Theorem \ref{koembedthm}, $X$ can be embedded into a projective space by an $\mathbb{R}^m$-equivariant CR embedding.
\end{example}

\begin{example}
Put
\[
H_1=
\begin{pmatrix}1&0\\0&-1\end{pmatrix},
\qquad
H_2=
\begin{pmatrix}0&1\\1&0\end{pmatrix}.
\]
On $\mathbb{H}^{2,2}=\mathbb C^2\times\mathbb R^2,$ define
\[
(z,t)\cdot(w,s)=\left(z+w,\ t+s-2
\begin{pmatrix}
\operatorname{Im}(z^TH_1\bar w)\\
\operatorname{Im}(z^TH_2\bar w)
\end{pmatrix}
\right).
\]
Here we write $z=(z_1,z_2),\ w=(w_1,w_2)\in \mathbb{C}^2$ and $t=(t_1,t_2), s=(s_1,s_2)\in \mathbb{R}^2.$
Since $H_1$ and $H_2$ have integer entries, we can see that
\[
\Gamma=\mathbb Z[i]^2\times\mathbb Z^2
\]
is a subgroup of $\mathbb{H}^{2,2}.$ Also $\Gamma$ is discrete and cocompact. Hence 
\[
X:= \Gamma\backslash\mathbb H^{n,m}
\]
is a connected compact smooth manifold. Write $Y_\ell=\partial_{t_\ell},\ \ell=1,2,$ and define
\begin{align*}
Z_1&=\partial_{z_1}-i\overline z_1Y_1-i\overline z_2Y_2,\\
Z_2&=\partial_{z_2}+i\overline z_2Y_1-i\overline z_1Y_2.
\end{align*}
These are left-invariant vector fields, so they descend to $X.$ Set
\[
T^{1,0}X=\operatorname{span}_{\mathbb C}\{Z_1,Z_2\}.
\]
Then $(X, T^{1,0}X)$ is a connected compact CR manifold of dimension 6 and of CR codimension 2 with the natural $\mathbb{R}^2$-action
\[
\zeta_r([(z,t)])=[(z,t+r)], \qquad r\in \mathbb{R}^2.
\]
The commutators are
\begin{align}\label{commu}
&[Z_1,Z_2]=0,\qquad [Y_\ell,Z_j]=0,\qquad [Y_1,Y_2]=0,\\
&[Z_1,\overline Z_1]=2iY_1,\qquad [Z_2,\overline Z_2]=-2iY_1,\nonumber\\
&[Z_1,\overline Z_2]=[Z_2,\overline Z_1]=2iY_2.\nonumber
\end{align}
We have $\{dz_1,\, dz_2,\, d\overline{z}_1,\, d\overline{z}_2,\, \omega_1,\, \omega_2 \}$ where 
\begin{align*}
\omega_1={}&dt_1+i(\overline z_1\,dz_1-\overline z_2\,dz_2)-i(z_1\,d\overline z_1-z_2\,d\overline z_2),\\
\omega_2={}&dt_2+i(\overline z_2\,dz_1+\overline z_1\,dz_2)-i(z_2\,d\overline z_1+z_1\,d\overline z_2),
\end{align*}
is dual to $\{Z_1,\, Z_2,\, \overline{Z}_1, \, \overline{Z}_2,\, -Y_1,\ -Y_2\}.$
Let
\[
\mathbb{T}^2=\mathbb C/(\mathbb Z+i\mathbb Z), \qquad M=\mathbb{T}^2\times \mathbb{T}^2.
\]
There is a well-defined CR projection
\[
p:X\longrightarrow M,\qquad p([(z,t)])=[z].
\]
Consider the standard positive theta line bundle $P\to \mathbb{T}.$ It is defined by $\mathbb{C}\times\mathbb{C}/\sim$ where the equivalence relation $\sim$ is given by
\[
(z,\xi)\sim (z+a+bi, e^{\pi b^2 -2\pi i az}\xi), \qquad a+bi \in \mathbb{Z}[i],
\]
with the descended Hermitian metric
\[
\|[(z,\xi)]\|^2=e^{-2\pi(\operatorname{Im}z)^2}|\xi|^2.
\]
Set
\[
A=P\boxtimes P\longrightarrow M,\qquad L=p^*A,\qquad h^L=p^*h^A.
\]
Then we can check that $L$ is a rigid positive CR line bundle over $X$ with $R^L=2\pi I_2>0$ where $I_2$ is the identity matrix in $M_2(\mathbb{R}).$
Hence, by Theorem \ref{koembedthm}, $X$ can be embedded into a projective space by an $\mathbb{R}^m$-equivariant CR embedding.

We next show that all smooth CR functions on $X$ are constant.  Let $d\mu$ be the volume form on $X$ induced from $d\mu_{\mathbb{H}^{2,2}}$ (see Section \ref{heisenberggpcrstr}). Then we have
\[
Z_\ell^*=-\overline Z_\ell \qquad \ell=1,2.
\]
Let $u$ be a smooth CR function on $X.$ By definition,  we have
\[
\overline{Z}_1u=\overline Z_2u=0.
\]
Note that, from \eqref{commu}, we have
\[
[\overline Z_1,Z_1]+[\overline Z_2,Z_2]=0.
\]
Hence integration by parts gives
\begin{align*}
\sum_{j=1}^2\|Z_ju\|_{L^2}^2
&=-\int_X\bigg(\sum_{j=1}^2\overline Z_jZ_ju\bigg)\bar u\,d\mu\\
&=-\int_X\bigg(\sum_{j=1}^2[\overline Z_j,Z_j]u\bigg)\bar u\,d\mu\\
&=0.
\end{align*}
Therefore \(Z_1u=Z_2u=0\). This implies
\[
[Z_1,\overline Z_1]u=2iY_1u=0,\qquad [Z_1,\overline Z_2]u=2iY_2u=0.
\]
Thus, we obtain $du=0$ which implies that $u$ is constant. In particular, every CR map $X\to\mathbb C^N$ is constant. So $X$ cannot be CR embedded into complex Euclidean space. Nevertheless, all the hypotheses of  Theorem \ref{koembedthm} hold, so it provides projective CR embeddings constructed from joint low-frequency CR sections of $L^k$ for sufficiently large $k.$

The preceding construction and proof extend directly to the following
generalization; we therefore omit the details. Choose arbitrarily finitely many Hermitian matrices
\[H_1,\ldots,H_m\in \operatorname{Herm}_n(\mathbb C)\cap M_n(\mathbb Z[i]).\] Equip $\mathbb H^{n,m}$
with the corresponding group law and left-invariant CR structure. Let $\Gamma=\mathbb Z[i]^n\times\mathbb Z^m$ and $X=\Gamma\backslash\mathbb H^{n,m}.$ Then we have
\begin{itemize}
    \item[(1)] The compact CR nilmanifold $X$ admits an $\mathbb R^m$-equivariant CR embedding into $\mathbb{CP}^N$ for some $N.$
    \item[(2)] If $H_1,\ldots,H_m$ are trace free, then every smooth CR function on $X$ is annihilated by every smooth section of
    $T^{1,0}X\oplus T^{0,1}X$. In particular, $X$ admits no CR embedding into $\mathbb C^N$ for any $N.$
    \item[(3)] Moreover, if $H_1,\ldots,H_m$ are trace free and linearly independent over $\mathbb{R},$ then every smooth CR function on $X$ is constant.
\end{itemize}

\end{example}

\end{document}